\documentclass[11pt,reqno]{amsart}

\usepackage[utf8]{inputenc}
\usepackage[T1]{fontenc}
\usepackage{lmodern}
\usepackage{amsmath,amssymb,amsthm,amscd,mathtools}
\usepackage[mathscr]{euscript}
\usepackage{mathrsfs}
\usepackage{enumitem}
\usepackage{xcolor}
\usepackage{etoolbox}
\usepackage{float}
\usepackage[colorlinks=true,linkcolor=blue,citecolor=black,urlcolor=black]{hyperref}
\usepackage{aliascnt}
\usepackage[nameinlink,noabbrev]{cleveref}
\let\standardtableofcontents\tableofcontents

\usepackage{tikz}
\usetikzlibrary{arrows.meta}

\usepackage{array}
\usepackage{booktabs}

\newcolumntype{L}[1]{>{\raggedright\arraybackslash}p{#1}}

\theoremstyle{plain}
\newtheorem{theorem}{Theorem}[section]
\newtheorem{example}[theorem]{Example}

\newaliascnt{lemma}{theorem}
\newtheorem{lemma}[lemma]{Lemma}
\aliascntresetthe{lemma}
\crefname{lemma}{lemma}{lemmas}
\Crefname{lemma}{Lemma}{Lemmas}

\newaliascnt{proposition}{theorem}
\newtheorem{proposition}[proposition]{Proposition}
\aliascntresetthe{proposition}
\crefname{proposition}{proposition}{propositions}
\Crefname{proposition}{Proposition}{Propositions}

\newaliascnt{corollary}{theorem}
\newtheorem{corollary}[corollary]{Corollary}
\aliascntresetthe{corollary}
\crefname{corollary}{corollary}{corollaries}
\Crefname{corollary}{Corollary}{Corollaries}

\newaliascnt{claim}{theorem}

\aliascntresetthe{claim}
\crefname{claim}{claim}{claims}
\Crefname{claim}{Claim}{Claims}

\theoremstyle{definition}
\newaliascnt{definition}{theorem}
\newtheorem{definition}[definition]{Definition}
\aliascntresetthe{definition}
\crefname{definition}{definition}{definitions}
\Crefname{definition}{Definition}{Definitions}

\newaliascnt{remark}{theorem}
\newtheorem{remark}[remark]{Remark}
\aliascntresetthe{remark}
\crefname{remark}{remark}{remarks}
\Crefname{remark}{Remark}{Remarks}

\newaliascnt{question}{theorem}
\newtheorem{question}[question]{Question}
\aliascntresetthe{question}
\crefname{question}{question}{questions}
\Crefname{question}{Question}{Questions}

\newcommand{\N}{\mathbb N}
\newcommand{\R}{\mathbb R}
\newcommand{\C}{\mathbb C}
\newcommand{\T}{\mathbb T}

\newcommand{\eps}{\varepsilon}
\newcommand{\distT}{d_{\mathbb T}}

\newcommand{\spancl}[1]{\overline{\operatorname{span}}\{#1\}}
\newcommand{\spann}{\operatorname{span}}

\DeclareMathOperator{\supp}{supp}

\newcounter{problempaper}

\makeatletter

\def\leftmark{\expandafter\@firstoftwo\botmark{}{}}
\def\rightmark{\expandafter\@secondoftwo\botmark{}{}}

\newcommand{\problempaper}[1]{%
  \stepcounter{problempaper}%
  \markboth{#1}{#1}%
  \section{#1}%
  \markboth{#1}{#1}%
}

\newenvironment{problempaperbody}{%
  \setcounter{subsection}{0}%
  \setcounter{subsubsection}{0}%
  \setcounter{paragraph}{0}%
  \renewcommand{\thesubsection}{\arabic{subsection}}%
  \renewcommand{\thesubsubsection}{\thesubsection.\arabic{subsubsection}}%
  \renewcommand{\theparagraph}{\thesubsubsection.\arabic{paragraph}}%
  \let\problem@outersection\section
  \let\problem@outersubsection\subsection
  \let\problem@outersubsubsection\subsubsection
  \let\problem@outerparagraph\paragraph
  \let\section\problem@papersection
  \let\subsection\problem@papersubsection
  \let\subsubsection\problem@papersubsubsection
}{%
  \let\section\problem@outersection
  \let\subsection\problem@outersubsection
  \let\subsubsection\problem@outersubsubsection
  \let\paragraph\problem@outerparagraph
}

\newcommand{\problem@target}[2]{%
  \hypertarget{problem-\theproblempaper-#1-#2}{}%
}

\newcommand{\problem@paperheading}[3]{%
  \par\bigskip
  \problem@target{#1}{#2}%
  \noindent{\large\bfseries #2.\quad #3\par}%
  \nobreak\smallskip
}

\newcommand{\problem@papersubheading}[3]{%
  \par\medskip
  \problem@target{#1}{#2}%
  \noindent{\normalsize\bfseries #2.\quad #3\par}%
  \nobreak\smallskip
}

\newcommand{\problem@papersubsubheading}[3]{%
  \par\smallskip
  \problem@target{#1}{#2}%
  \noindent{\normalsize\itshape #2.\quad #3.\ }%
}

\def\problem@papersection{%
  \@ifstar{\problem@papersectionstar}{\problem@papersectionnostar}%
}

\def\problem@papersectionstar#1{%
  \par\bigskip\noindent{\large\bfseries #1\par}\nobreak\smallskip
}

\def\problem@papersectionnostar#1{%
  \refstepcounter{subsection}%
  \setcounter{subsubsection}{0}%
  \setcounter{paragraph}{0}%
  \problem@paperheading{section}{\thesubsection}{#1}%
  \problem@writecontentsline{section}{\thesubsection}{#1}%
}

\def\problem@papersubsection{%
  \@ifstar{\problem@papersubsectionstar}{\problem@papersubsectionnostar}%
}

\def\problem@papersubsectionstar#1{%
  \par\medskip\noindent{\normalsize\bfseries #1\par}\nobreak\smallskip
}

\def\problem@papersubsectionnostar#1{%
  \refstepcounter{subsubsection}%
  \setcounter{paragraph}{0}%
  \problem@papersubheading{subsection}{\thesubsubsection}{#1}%
  \problem@writecontentsline{subsection}{\thesubsubsection}{#1}%
}

\def\problem@papersubsubsection{%
  \@ifstar{\problem@papersubsubsectionstar}{\problem@papersubsubsectionnostar}%
}

\def\problem@papersubsubsectionstar#1{%
  \par\smallskip\noindent{\normalsize\itshape #1.\ }%
}

\def\problem@papersubsubsectionnostar#1{%
  \refstepcounter{paragraph}%
  \problem@papersubsubheading{subsubsection}{\theparagraph}{#1}%
  \problem@writecontentsline{subsubsection}{\theparagraph}{#1}%
}

\newcommand{\problem@writecontentsline}[3]{%
  \begingroup
  \protected@write\@auxout{}{%
    \string\problemcontentsentry{\theproblempaper}{#1}{#2}{#3}{\thepage}%
  }%
  \endgroup
}

\newcommand{\problemcontentsentry}[5]{%
  \@ifundefined{problemcontents@#1}{%
    \expandafter\gdef\csname problemcontents@#1\endcsname{}%
  }{}%
  \expandafter\g@addto@macro\csname problemcontents@#1\endcsname{%
    \problemcontentsline{#1}{#2}{#3}{#4}{#5}%
  }%
}

\newcommand{\problemcontentslink}[4]{%
  \hyperlink{problem-#1-#2-#3}{#4}%
}

\newcommand{\problemcontentsline}[5]{%
  \ifstrequal{#2}{section}{%
    \par\noindent
    \problemcontentslink{#1}{#2}{#3}{\makebox[2.5em][l]{#3}#4}%
    \nobreak\dotfill
    \problemcontentslink{#1}{#2}{#3}{#5}\par
  }{%
    \ifstrequal{#2}{subsection}{%
      \par\noindent\hspace*{1.5em}%
      \problemcontentslink{#1}{#2}{#3}{\makebox[3.5em][l]{#3}#4}%
      \nobreak\dotfill
      \problemcontentslink{#1}{#2}{#3}{#5}\par
    }{%
      \par\noindent\hspace*{3em}%
      \problemcontentslink{#1}{#2}{#3}{\makebox[4.5em][l]{#3}#4}%
      \nobreak\dotfill
      \problemcontentslink{#1}{#2}{#3}{#5}\par
    }%
  }%
}

\newcommand{\problemcontents}{%
  \begingroup
  \par\medskip\noindent\textbf{Contents of this problem paper}\par\smallskip\small
  \@ifundefined{problemcontents@\theproblempaper}{%
    \noindent\emph{Run LaTeX again to generate this local table of contents.}\par
  }{%
    \csname problemcontents@\theproblempaper\endcsname
  }%
  \endgroup
}

\makeatother

\newenvironment{problemabstract}
  {\begin{quote}\small\noindent\textbf{Abstract.}\ }
  {\end{quote}}

\newcommand{\websiteurl}{https://mathematical-discoveries-in-the-wild.github.io/}
\newcommand{\website}[1][project website]{%
  \href{\websiteurl}{\textcolor{blue}{#1}}%
}


\title[Mathematical Discovery in the Wild]{Mathematical Discovery in the Wild:\\ AI-Guided Proofs in Banach Space Theory}

\author[A. Acuaviva]{Antonio Acuaviva}
\address{School of Mathematical Sciences,
Fylde College,
Lancaster University,
LA1 4YF,
United Kingdom} \email{ahacua@gmail.com}

\author[P. Acuaviva]{Pablo Acuaviva}
\address{Institute of Computer Science,
University of Bern,
Neubrückstrasse 10,
3012 Bern,
Switzerland} \email{pablohacuaviva@gmail.com}

\date{19 July 2026}

\hypersetup{
  pdftitle={Mathematical Discovery in the Wild: AI-Guided Proofs in Banach Space Theory},
  pdfauthor={Antonio Acuaviva and Pablo Acuaviva},
  pdfsubject={AI-guided mathematical discovery and proofs in Banach space theory},
  pdfkeywords={
    AI-assisted mathematics,
    large language models,
    mathematical discovery,
    proof search,
    proof verification,
    Banach space theory
  },
  pdfdisplaydoctitle=true
}

\begin{document}

\begin{abstract}
    We investigate the capacity of current language models to contribute to mathematical research. In Banach space theory, AI systems generated key ideas and proofs for five new results, which were then verified and refined by humans. We also developed an automated system that searches the literature for open problems and attempts solutions at scale. Our results show both the potential of language models for mathematical discovery and the continuing importance of expert verification.
\end{abstract}

\maketitle

\begin{flushright}
\begin{minipage}{0.7\textwidth}
\itshape
``C'est par la logique qu'on démontre, c'est par l'intuition qu'on invente.''

\medskip
\upshape\hfill Henri Poincaré, \emph{Science et méthode}
\end{minipage}
\end{flushright}

\setcounter{tocdepth}{1}
\standardtableofcontents

\part{Survey and Discussion}\label{part1}

\section{Introduction and discussion}
\label{sec:introduction-discussion}

The main purpose of this article is to draw attention to a development which, in our view, the mathematical community can no longer regard as merely speculative. Current language model systems, when placed in suitable workflows and checked by experts, can already do useful work on live research level mathematics. The claim is not merely that such systems can solve polished exercises, reproduce known arguments, or assist with exposition, but that they can sometimes generate serious proof candidates for questions that arise naturally in the research literature. This point is made concrete in \Cref{part2}, where we present selected Banach space problems whose proofs were generated essentially by the model before human verification, editing, and integration; the provenance of the individual problem papers is discussed more carefully in \Cref{sec:problems-and-provenance}.

A second purpose is to begin raising the questions that follow from this fact. Autonomous or semi-autonomous proof search with general-purpose models is no longer only a matter for benchmark design or future speculation.  We place this development in recent and historical context in \Cref{sec:historical-background}. The important point is that these systems are beginning to affect how mathematical problems can be found, attempted, checked, organized, and written up. If this continues, the mathematical community will need to think carefully about its consequences: standards of verification, attribution and credit, the role of human judgement, the value of open problem lists, pressure on peer-review, and the ways in which research training and collaboration may change.  We return to these broader issues in \Cref{sec:new-age} and~\ref{sec:limitations-credit-responsibility}.

\subsection{Origin and scope of the project}
\label{sec:origin-scope} This article grew out of an automatic\allowbreak-search project in Banach space theory.  The original objective was to search the arXiv literature for questions, conjectures, and open problems that appeared sufficiently local to be attacked, and then to try to settle them with minimal human guidance. The motivation was to work systematically through unresolved points left in papers.  Such questions are often not the central open problems of a field, but answering them can still clarify the literature: a proposed problem may have a proof, a counterexample, an answer already implicit in known results, or an obstruction explaining why it remains open.  In this sense, the project was concerned with accelerating the gradual closure of mathematical knowledge.

As the project developed, the model-generated outputs were stronger than expected.  In particular, many of the claims presented by the model as full solutions survived subsequent mathematical checking.  This changed the scope of the experiment.  We continued the automatic search, but also began to test the same style of automatic proof search on harder problems selected by humans: problems suggested in discussions with other mathematicians, or chosen from the authors' knowledge of the field.  Thus the project came to have two distinct components.

\Cref{part2} contains the human-selected component.  The problems treated there were not discovered by the automatic pipeline. They were chosen precisely because they seemed to be substantial, non-routine research problems: difficult enough that a solution would, in our judgement, constitute a meaningful advance in Banach-space theory. Their purpose is therefore not to exhibit model performance on routine exercises, but to test autonomous proof search on problems that specialists would regard as serious mathematical targets. Thus, the selection was human and deliberately demanding, but the proof search itself was fully model-driven. The arguments printed in \Cref{part2} were autonomously produced by the model before human verification, editing, and integration into self-contained mathematical notes.

\Cref{part3} records the automatic-search component. Here, both the selection of targets and the attempted solutions came from the pipeline.  The purpose of this part is not to present polished, standalone papers, but to document the wider machine-selected, machine-attempted exploration from which the project began. 

The choice of field should not be interpreted as a claim that the phenomenon is specific to Banach space theory. It reflects a practical constraint on this kind of work: generated mathematics requires verification by someone who understands the relevant definitions, the cited results, and the local literature.  In both parts of the article, human verification enters after generation.  The authors formulate or select problems in the human-chosen cases, run the exploratory process, identify outputs worth checking, verify the mathematics, assess significance, and edit the final exposition, fixing some minor mistakes whenever necessary.

\subsection{Recent context}\label{sec:historical-background}

Several recent projects evaluate AI on advanced or research-level mathematics. FrontierMath currently has two components: Tiers 1--4, a benchmark of unpublished problems authored and peer-reviewed by expert mathematicians, and Open Problems, a separate collection of unsolved research problems \cite{GlazerEtAl2024FrontierMath,EpochAIFrontierMath}. RealMath derives automatically verifiable mathematical tasks from research papers and mathematical forums and is designed as a continually refreshable benchmark \cite{ZhangEtAl2025RealMath}. LemmaBench similarly constructs a live research-level benchmark by extracting recent lemmas from arXiv and rewriting them as self-contained statements \cite{PeyronnetEtAl2026LemmaBench}. IMProofBench is a private benchmark of peer-reviewed proof problems developed by expert mathematicians; models operate in an agentic environment with tools such as web search and SageMath, while complete proofs are graded by human experts \cite{SchmittEtAl2025IMProofBench}. First Proof focuses more specifically on questions that arose in the work of professional mathematicians, had known solutions, and had not previously appeared online. In its first, informal batch, the organisers released ten such questions and temporarily withheld the answers. In the organisers' preliminary tests, each query was run once, without iterative interaction or rerunning \cite{AbouzaidEtAl2026FirstProofFirstBatch}. The second batch was a formal benchmark using ten new questions and four systems: ChatGPT 5.5 Pro and three academic harnesses. Each system received the problems under a fixed one-shot protocol and returned solutions without further interaction from the organisers; the problems, human-generated and AI-generated solutions, costs, evaluations including referee reports, logs, and harness code were subsequently released \cite{AbouzaidEtAl2026FirstProofSecondBatch}.

The aim of the present paper is different. We do not seek to compare systems under a fixed protocol or compute budget, or to produce a leaderboard result. Instead, we ask what current models can contribute when used for repeated exploration of live problems, with final human verification of claims that may be genuinely new. A complementary collaborative paradigm is represented by the AI co-mathematician, an interactive and stateful workbench through which mathematicians guide agents across ideation, literature search, computational exploration, theorem proving, and theory building \cite{ZhengEtAl2026AICoMathematician}. The process studied here was not that kind of specialist-guided collaboration: a mathematician did not supply the next lemma or choose how to use each intermediate observation. Problems could be revisited, and partial outputs could inform later attempts, but the mathematical content of the promoted proofs was generated in the model attempts rather than supplied by a mathematician; see \Cref{sec:problems-and-provenance}.

A related but methodologically distinct line of work uses automated evaluators to guide discovery. FunSearch combines an LLM with evolutionary search over programs and an automated evaluator; it produced new cap-set constructions and improved online bin-packing heuristics \cite{RomeraParedesEtAl2024FunSearch}. Subsequent work made this style of search more accessible to working mathematicians \cite{EllenbergEtAl2025GenerativeModelling}. AlphaEvolve generalizes evaluator-guided evolutionary search to direct modification of algorithmic code and a broader range of scientific and computational tasks \cite{NovikovEtAl2025AlphaEvolve}. Georgiev, G\'omez-Serrano, Tao, and Wagner applied AlphaEvolve to 67 mathematical problems, recovering best-known solutions in most cases and improving them in several others \cite{GeorgievEtAl2025MathDiscoveryScale}. Such methods apply most directly in settings where candidate objects can be represented computationally and scored automatically.

The problems considered here generally do not have that form. The outputs are natural-language proof candidates for questions extracted from papers or selected by mathematical judgement. Their correctness is not fully captured by a numerical objective or a short verification computation. They must instead be checked by reading the proof, verifying the hypotheses of cited results, comparing the claim with the original problem, and rewriting the argument in a clear mathematical form.

Other recent work has developed longer-horizon scientific and mathematical agents with memory, tool use, literature search, theorem retrieval, and verification. The AI Scientist and The AI Scientist v2 automate stages of machine-learning research including idea or hypothesis generation, experiment execution, analysis, manuscript writing, and automated or simulated review \cite{LuEtAl2024AIScientist,YamadaEtAl2025AIScientistV2}. Kosmos coordinates literature search, hypothesis generation, and data analysis through a structured world model \cite{MitchenerEtAl2025Kosmos}, while Arbor uses Hypothesis-Tree Refinement to preserve hypotheses, artifact versions, evidence, and distilled insights across long research runs \cite{JinEtAl2026Arbor}. Iteris uses an explore--plan--execute loop for numerical experimentation, construction, and proof development on open problems in computational mathematics; in its case studies, the resulting artifacts were followed by expert review and correction \cite{ChenLiuHeDong2026Iteris}.

In mathematics more specifically, Aletheia iteratively generates, verifies, and revises natural-language solutions and has also been evaluated on the inaugural First Proof problems \cite{FengEtAl2026AletheiaFirstProof,FengEtAl2026AutonomousMathResearch}. Rethlas combines informal proof search with the Matlas theorem-search engine, while Archon uses LeanSearch to formalize and verify the resulting arguments in Lean 4 \cite{JuEtAl2026RethlasArchon}. Recent AxiomProver-backed papers report a Lean formalization of the combinatorial identity used in one proof and autonomous Lean formalizations and machine-checkable proofs of six conjectures in another \cite{ChenEtAl2026ParityKDifferentials,ChenOnoZhang2026PartitionPolynomials}. Harmonic's Aristotle combines informal reasoning with Lean proof search and a dedicated geometry solver \cite{AchimEtAl2025Aristotle}.

The setup used in the present project is much lighter: direct interaction with ChatGPT 5.5 Pro and Codex agents operating over a filesystem. An automatic pipeline extracts open-problem signals from arXiv source files and organizes them into candidate questions; the agents then attempt the problems, and their outputs are submitted for human review; see \Cref{sec:technical-details-pipeline}. We do not claim that this is an optimal architecture for automated mathematics. The point is that a relatively ordinary setup can already produce useful proof candidates in both settings considered here: difficult problems selected by humans and literature questions selected by machines; see \Cref{part2} and \Cref{tab:human-review-counts}.

\subsection{Contribution and organization}

The paper has two main contributions. The first is to help raise awareness within the mathematical community that current language-model systems can already play a substantive role in research-level mathematical exploration, and to open a discussion about how such tools should be understood, used, disclosed, verified, and credited.

The second contribution is mathematical.  The paper presents selected results in Banach space theory for human-selected problems. These results are not included merely as demonstrations of model behaviour. In our view, they are meaningful mathematical contributions in their own right, and several of them could plausibly have supported independent papers. At the same time, their provenance gives them an additional role: they provide concrete evidence that current systems can sometimes engage with serious research-level problems.

The project was not fully autonomous in the strongest sense.  The authors chose which outputs to promote, verified the mathematical content, checked references, assessed significance, and edited the exposition; they take responsibility for the final text and the mathematical correctness. Nor are the proofs formally verified.  They are ordinary mathematical proofs, checked by human mathematical reading. At the same time, the model's role was not confined to exposition or stylistic assistance. It generated proof ideas, proof structures, and in several cases, complete or essentially complete arguments for problems drawn from the literature.

The paper is organised as follows. \Cref{part1} gives the survey-level discussion, including the list of selected problems and comments on verification, disclosure, credit, and responsibility.  \Cref{part2} contains the selected problem papers. These are intended to be readable as self-contained mathematical notes. \Cref{part3} records both technical details of the automatic-search pipeline and selected outputs from that pipeline.

\section{Selected problems and provenance}
\label{sec:problems-and-provenance}

The project is problem-driven. The problems listed below are not toy examples or prompt demonstrations, but mathematical problems which could reasonably be considered serious research questions in Banach space theory.

The provenance of the proofs is therefore central. For the selected problems, the proofs arose primarily from model-generated proof search, followed by human verification, editing, and rewriting. We discuss this in more detail in Subsection~\ref{susec:proofs-and-provenance}. The original AI outputs can be found on the \website[project website]. Thus, the problem papers should be read neither as conventional unaided proofs nor as unverified model transcripts, but as AI-assisted mathematical arguments for which the authors take responsibility after human verification and editing.

The first two problems were proposed by Tomasz Kania, the third and fourth problems were proposed by Kevin Beanland, and the last problem was proposed by the first-named author. Since the individual problem papers are written mainly as proof notes rather than as fully introduced articles, we give here some brief background on each problem and its significance. When available, we also include comments supplied by the proposer.

\subsection*{Problem summaries and background}

\noindent\textbf{P1. Toroidal Elton--Odell theorem.} \emph{Proposed by Tomasz Kania.} 

The following background and comments were given by Tomasz Kania and are lightly edited here for style.

Kottman's theorem asserts that the unit sphere of every infinite-dimensional normed space contains a sequence whose mutual distances are strictly greater than one. The classical Elton--Odell theorem strengthens this by providing a uniform margin: there are $\varepsilon>0$ and a sequence in the unit sphere whose mutual distances are all at least $1+\varepsilon$.

Over the complex field, the natural projective analogue identifies vectors which differ by multiplication by a unimodular scalar. Accordingly, for $x,y\in S_X$, one considers the toroidal distance
\begin{equation*}
    \distT(x,y)=\inf_{\theta\in\T}\|x-\theta y\|.
\end{equation*}
The toroidal Elton--Odell problem asks whether every infinite-dimensional complex normed space admits $\varepsilon>0$ and a sequence $(x_n)\subset S_X$ such that
\begin{equation*}
    \distT(x_n,x_m)\geq 1+\varepsilon
    \qquad(n\ne m).
\end{equation*}
This problem was raised explicitly in \cite[\S5.3.1]{HajekKaniaRusso2018}.

A strict but non-uniform analogue holds in complete generality: every infinite-dimensional complex normed space contains a sequence $(x_n)\subset S_X$ such that
\begin{equation*}
    \distT(x_n,x_m)>1
    \qquad(n\ne m).
\end{equation*}
This does not settle the uniform problem, since the excess over one may depend on the pair and tend to zero. Uniform positive results were also known under additional geometric hypotheses, including finite cotype, lower $q$-estimates, and asymptotic uniform convexity, together with computations of the toroidal separation constant for several classical spaces \cite{Kania2026Toroidal}. The remaining question was whether the additional hypotheses could be removed and a single positive margin obtained in every infinite-dimensional complex normed space.

The proof included here gives a complete solution to this problem.

\medskip

\noindent\textbf{P2. Non-Calkin unital Banach algebras.}
\emph{Proposed by Tomasz Kania.}

This problem concerns the realisation problem for Calkin algebras. For a Banach space $X$, its Calkin algebra is the quotient
\begin{equation*}
    \mathcal B(X)/\mathcal K(X)
\end{equation*}
of the bounded operators on $X$ by the compact operators. The corresponding realisation problem asks which unital Banach algebras can occur, up to Banach-algebra isomorphism, as Calkin algebras. This question was recorded in Tarbard's thesis and subsequently studied by Horv\'ath and Kania
\cite{Tarbard2013,HorvathKania2021}.

The scalar-plus-compact theorem of Argyros and Haydon provided a striking early solution to a particular realisation problem \cite{ArgyrosHaydon2011}. They constructed an infinite-dimensional Banach space $X$ on which every bounded operator is a scalar multiple of the identity plus a compact operator. Consequently,
\begin{equation*}
    \mathcal B(X)/\mathcal K(X)\cong\mathbb C.
\end{equation*}
Thus even the scalar field can occur as the Calkin algebra of an infinite-dimensional Banach space.

Subsequent constructions demonstrated considerable flexibility in the class of algebras which can be realised as Calkin algebras. Motakis, Puglisi, and Zisimopoulou realised $C(K)$ for every countable compact metric space $K$ \cite{MotakisPuglisiZisimopoulou2016}, and Motakis later extended this to every compact metric space \cite{Motakis2024Calkin}. Further examples include broad classes of diagonal scalar-plus-compact algebras \cite{MotakisPuglisiTolias2020}, infinite-dimensional reflexive Calkin algebras \cite{MotakisPelczarBarwacz2025}, and the unitisation of the noncommutative algebra $\mathcal K(c_0)$ \cite{MotakisPuglisi2025}.

Despite these positive realisation results, it remained unknown whether every unital Banach algebra is isomorphic to the Calkin algebra of some Banach space. Horv\'ath and Kania obtained a partial negative result: they constructed simple unital AF $C^*$-algebras which are not isomorphic to the Calkin algebra of any separable Banach space \cite{HorvathKania2021}. Their argument did not, however, rule out a representation over a nonseparable Banach space. This left open whether there exists a unital Banach algebra $A$ such that
\begin{equation*}
    A\not\cong\mathcal B(X)/\mathcal K(X)
\end{equation*}
for every Banach space $X$.

The proof included here answers this question affirmatively. The principal difficulty is the nonseparable case, where density alone provides no obstruction. The constructions combine large density and topological simplicity with carefully chosen matrix-unit or shift relations. If either algebra were the Calkin algebra of a nonseparable Banach space, topological simplicity would force every separable-range operator to be compact. The additional algebraic relations are then used to construct a noncompact operator with separable range, giving the required contradiction.

\medskip

\noindent\textbf{P3. Strict cosingularity and adjoints.} \emph{Proposed by Kevin Beanland.} 

The following background and comments were given by Kevin Beanland and are lightly edited here for style. This problem concerns the relationship between strictly singular and strictly cosingular operators. Pe{\l}czy\'nski introduced and studied strictly cosingular operators as a dual counterpart to strictly singular operators, and proved several duality results between the two ideals. In particular, if $T \colon X\to Y$ is strictly cosingular and weakly compact, then $T^*$ is strictly singular; conversely, if $T$ is strictly singular and $X$ is reflexive, then $T^*$ is strictly cosingular. He also showed that the weak compactness assumption in the first implication cannot be omitted in general: the canonical inclusion $c_0\hookrightarrow \ell_\infty$ is strictly cosingular, while its adjoint is not strictly singular.

The separable-range case arose from Beanland's 2008 paper with George Androulakis on descriptive set theoretic methods for strictly singular and strictly cosingular operators \cite{AndroulakisBeanland2008}. Earlier work had introduced an ordinal hierarchy of $\mathcal S_\alpha$-strictly singular operators using the transfinite Schreier families, and it was natural to ask whether a corresponding hierarchy could be obtained for strictly cosingular operators. A natural route was to define such classes through duality, but this required understanding whether Pe{\l}czy\'nski's implication remains true when the range space is separable and the weak compactness hypothesis is omitted. The main obstruction was the following question: if $T^*$ restricts to an isomorphism on an infinite-dimensional subspace of $Y^*$, must there be an infinite-dimensional weak-star closed subspace of $Y^*$ on which $T^*$ is still an isomorphism? Beanland later posted this question on MathOverflow \cite{BeanlandMO98449}.

The proof included here gives an affirmative answer in the required form and proves that, when $Y$ is separable, $T$ is strictly cosingular if and only if $T^*:Y^*\to X^*$ is strictly singular. \medskip

\noindent\textbf{P4. Weakly compact basis factorization.} \emph{Proposed by Kevin Beanland.} The following background and comments were given by Kevin Beanland and are lightly edited here for style.

Davis, Figiel, Johnson, and Pe{\l}czy\'nski, in their seminal paper introducing the interpolation method used to prove the weakly compact factorization theorem, also established conditions under which the DFJP interpolation space admits a Schauder basis \cite{DavisFigielJohnsonPelczynski1974}. Specifically, they showed that if $T\colon X\to Y$ is weakly compact and the range space $Y$ has a shrinking Schauder basis, then the interpolation space may be constructed so that it also has a Schauder basis.

This hypothesis cannot simply be omitted. Not all separable Banach spaces have the bounded approximation property, while a separable Banach space has the bounded approximation property if and only if its identity operator factors through a Banach space with a Schauder basis \cite{Pelczynski1971}. Consequently, one cannot expect the DFJP interpolation space to admit a basis in complete generality.

Obtaining a basis is not immediate from the definition of the DFJP interpolation space. If one simply constructs the interpolation space from the relatively weakly compact set $T(B_X)$, there is no evident basis structure. Instead, Davis, Figiel, Johnson, and Pe{\l}czy\'nski enlarge the generating set by adjoining the images of $T(B_X)$ under the coordinate projections associated with the shrinking basis of $Y$. The shrinking property is then used to show that this enlarged set remains relatively weakly compact. This argument is highly specific to shrinking bases; in general, the same enlargement process need not preserve weak compactness, even when $T$ itself is weakly compact.

The corresponding result in the unconditional case was subsequently established by Figiel, Johnson, and Tzafriri \cite{FigielJohnsonTzafriri1975}. Later, Ghoussoub, Maurey, and Schachermayer proved that the DFJP interpolation space may likewise be constructed with a Schauder basis whenever the range space is isomorphic to $C[0,1]$ \cite{GhoussoubMaureySchachermayer1992}.

Beanland's interest in these questions arose from attempts to combine the DFJP interpolation method with the descriptive set theoretic framework developed by Dodos and others \cite{Dodos2010}, in order to obtain uniform factorization theorems for classes of weakly compact operators. A principal motivation is the theorem of Dodos and Ferenczi \cite{DodosFerenczi2007}, who proved that the class of separable reflexive Banach spaces is strongly bounded. Thus, although no universal separable reflexive Banach space exists, every analytic family of separable reflexive Banach spaces embeds isomorphically into a single separable reflexive Banach space. Moreover, when each member of the family has a Schauder basis, these embeddings may be chosen to be complemented.

Passing from Banach spaces to weakly compact operators naturally suggests an analogous theory in which uniform embedding theorems are replaced by uniform factorization theorems. The principal remaining obstacle is that, even for analytic classes of weakly compact operators whose range spaces possess Schauder bases, one needs the associated DFJP interpolation spaces to admit Schauder bases in order to apply the existing descriptive set theoretic machinery. Prior to the present work, this was known only in the special situations described above.

In June 2016, Beanland asked on MathOverflow whether the DFJP interpolation space could always be chosen to have a Schauder basis when the range space is isomorphic to $L_1[0,1]$ \cite{BeanlandMO2016}. Bill Johnson gave an elegant argument establishing this case. Using that observation, the descriptive set theoretic machinery was extended to this setting. In particular, it was proved that the collection of all weakly compact operators $T\colon X\to L_1$, where $X$ is separable, is analytic, and consequently there exists a single separable reflexive Banach space through which every such operator factors.

The more general question of whether the DFJP interpolation space could always be chosen to admit a Schauder basis whenever the range space has a Schauder basis remained open until the proof presented here.

\medskip

\noindent\textbf{P5. Primariness of $L_p(L_1)$.} \emph{Proposed by the first named author.} A Banach space is primary if, whenever it is decomposed as a direct sum of two complemented subspaces, one of those subspaces is isomorphic to the whole space. The problem considered here asks whether the mixed norm space $L_p(L_1)$ is primary for $1<p<\infty$.

This question belongs to the broader programme of understanding primariness for classical Banach spaces and their mixed norm variants. Lechner, Motakis, M\"uller, and Schlumprecht proved that $L_1(L_p)$ is primary for $1<p<\infty$ \cite{LechnerMotakisMullerSchlumprecht2022L1Lp}. In the same paper, they identified the primariness of $L_p(L_1)$ as one of the prominent remaining open cases. They subsequently studied this problem further in their work on multipliers on bi-parameter Haar system Hardy spaces \cite{LechnerMotakisMullerSchlumprecht2024}, where they established factorization results constituting a first step towards a proof of primariness, while leaving the full question unresolved.

The first named author's interest in this question grew out of this line of work, which led him to study primariness for biparameter constructions involving $C(K)$ spaces \cite{Acuaviva2026}. From that perspective, $L_p(L_1)$ was a natural remaining case, although one somewhat outside his original background.

When the problem was proposed to the model, it was included with little expectation that it would actually be solved. It was meant partly as a technical test of whether the exploratory process could handle a long and highly structured problem in Banach space theory. In hindsight, this makes the outcome especially interesting, while the provenance discussion below explains why the original raw output should still not be confused with a finished proof.

\medskip

\subsection{Proof provenance and AI writing}
\label{susec:proofs-and-provenance}

This subsection aims to record more precise comments on the AI outputs, the final proofs, and the human intervention and modification that connected them. We begin with some general comments about the outputs produced by the AI.

For Problems~1, 2, 3, and~4, the original AI outputs could be regarded as essentially correct, up to verification, small modifications, and editing. Somewhat surprisingly, most of the required work was more a matter of form than substance: the issues were often in the writing, presentation, and organisation of the arguments rather than in the underlying mathematics. The main proof was already present. In some places, however, the arguments were written in a nonstandard or unnatural mathematical style. In a few other places, the model misattributed a theorem, cited a result imprecisely, or made a small error. These issues did not require substantial mathematical repair. Another recurring feature was that the model sometimes left parts of an argument not fully fleshed out, presenting them as ``immediate'' even when some verification was still needed.

Problem~5 was different, largely because of its length and complexity. The raw output could not be regarded as a complete proof as it stood. It contained the right proof architecture, and most of the technical details and mathematical content were already present, but the assembly left much to be desired. Some parts of the proof were only indicated rather than fully written out, some terminology was used before being defined, and the exposition was at times difficult to follow. In this case, the human intervention was not merely local editing. It involved reorganising the proof, making the indicated arguments precise, and turning the architecture provided by the model into a complete written argument.
\section{Implications for mathematical exploration}
\label{sec:new-age}

The examples in this paper come from Banach space theory merely from practical limitations when verifying the significance and correctness of the solutions. The broader lesson is not that AI-assisted proof discovery is particular to this field, but that similar modes of exploration may become available in many areas where there is enough human expertise to verify, repair, and contextualise the generated arguments.

\subsection{From solving to exploring}

AI systems may make it cheaper to generate many possible routes through a problem. The value of this is not only that a model may occasionally produce a complete proof. It may also suggest intermediate lemmas, illuminate alternative decompositions of a problem, or produce failed attempts whose failure is itself informative. In this sense, AI may shift part of mathematical research from solving isolated problems to exploring large families of possible approaches. Many generated ideas will fail, some will require substantial repair, and only a few may ultimately become complete mathematical arguments. Nevertheless, the ability to explore a much larger mathematical search space may itself become a valuable research capability.

\subsection{Acceleration rather than replacement}

Current AI systems do not solve every mathematical problem, and many difficult questions remain well beyond their present capabilities. Rather than replacing mathematical research, they may accelerate it. A mathematician can formulate promising questions, provide context, reject unproductive directions, and repeatedly redirect the exploration towards approaches that appear mathematically meaningful. Even when a complete proof is not obtained, the generated intermediate ideas may shorten the path to a human solution. We therefore view AI not as a replacement for mathematical research, but as a tool that may substantially accelerate parts of the discovery process when guided by human expertise.

\subsection{Human mathematical judgement}

If proof generation becomes easier, human judgement does not disappear. It shifts. Choosing worthwhile problems, recognizing genuinely interesting ideas, identifying hidden gaps, deciding which proof strategies deserve further development, connecting new results to existing theory, and judging whether a theorem is mathematically important may become increasingly central tasks. In particular, the ability to distinguish between mathematically significant discoveries and technically correct but uninteresting statements may become one of the defining skills of future mathematical research.

\subsection{Verification and formal mathematics}

One of the principal bottlenecks encountered during this project was not proof generation itself, but verification. Producing plausible mathematical arguments became substantially easier than establishing, with confidence, that every step was correct and every cited ingredient was applicable. Formal proof assistants such as Lean may therefore represent a natural next step. If AI-generated proofs can eventually be translated into formally verified mathematics with relatively little additional effort, then the verification bottleneck may be substantially reduced. Such developments would not diminish the importance of human mathematicians. Rather, they would further shift human effort towards selecting important problems, evaluating ideas, organizing mathematical knowledge, and deciding which directions are worth pursuing.

\subsection{Distillation as a mathematical skill}

As the volume of mathematically plausible text continues to grow, the ability to distill ideas into clear, elegant, and verifiable mathematics may become increasingly valuable. A proof that is technically correct but opaque is less useful than one that reveals the underlying mechanism, exposes the essential ideas, and can be readily understood and reused by other mathematicians. Clear exposition has always been a central mathematical skill. If proof generation becomes substantially cheaper, its importance may only increase. One part of the mathematician's role may increasingly involve recognising important ideas, verifying them, organising them into coherent theories, and communicating them in a form that advances mathematical understanding.

\section{Limitations, credit, and responsibility}
\label{sec:limitations-credit-responsibility}

The examples in this article should be read as evidence that current AI systems can contribute substantially to proof discovery in suitable mathematical workflows. They should not be read as evidence that raw model outputs can be trusted without expert checking, or that the present project gives a general measurement of AI reliability in mathematics. The central lesson is more specific: when combined with human problem selection, verification, repair, and exposition, model generated proof search can already produce mathematically serious results. Precisely because this contribution can be substantial, it also raises a series of challenges about verification, credit, attribution, disclosure, and responsibility. The rest of this section is devoted to those challenges.

\subsection{Disclosure and incentives}

The possibility of substantial AI assistance raises a delicate question about disclosure. If theorems obtained with such assistance are automatically devalued, while undisclosed AI assistance remains difficult to detect, then the mathematical community risks creating a perverse incentive: honest disclosure may be penalized, whereas nondisclosure may be rewarded. This concern is especially serious for early-career researchers, for whom a small number of results can have a large professional impact.

The issue is not merely hypothetical. Some of the results recorded in the selected problem papers are, in our judgement, substantial enough to stand as standalone mathematical contributions. Such results can affect very concrete forms of professional evaluation: hiring, fellowships, grants, invitations, collaborations, and related opportunities. It is therefore important to address the incentive problem directly. If mathematically valuable work receives less professional recognition solely because AI assistance played a substantial, or even decisive, role in its discovery, then authors may have an incentive to present the history of the work selectively rather than transparently. One possible way forward is to keep the emphasis primarily on the usual mathematical standards: correctness, novelty, significance, clarity of presentation, and responsibility for verification.

There is also a second, less visible incentive problem. If substantial AI assistance is taken to diminish the professional value of a result, researchers may become reluctant to use these systems on their most promising ideas. They may reserve their best questions, conjectures, and proof strategies for unaided work, not because AI assistance would be mathematically inappropriate, but because they fear “losing” the result in professional terms if the model contributes decisively. This would be an especially unfortunate outcome. Mathematical progress depends on testing ideas with the strongest available tools, and a norm that discourages researchers from exploring their best ideas with AI could slow discovery, delay the resolution of important problems, and reduce the collective value produced by the field. Transparency should not require researchers to choose between using an effective method of exploration and receiving appropriate recognition for the mathematics that results.

In short, the lasting value of a theorem should ultimately depend on its mathematical content: the strength of its statement and proof, the clarity and depth of the ideas, the quality of the exposition, and the new avenues it opens. There is something personally and culturally meaningful about a theorem being discovered by a particular mathematician through their own insight, and that aspect of mathematics is real. Nevertheless, once a theorem has been stated, proved, and responsibly verified, its primary mathematical value should not depend on whether the path to it involved pencil and paper, conversation with colleagues, computer search, formal verification, or interaction with an AI system. This distinction also matters for incentives: if researchers, especially early-career applicants, expect less credit for results obtained with substantial AI assistance, they may be discouraged from disclosing that assistance. A healthier norm would evaluate the theorem and proof on their mathematical merits, independently of how the result came to be, while still requiring transparency about the process and full responsibility for the correctness of the final written work.

\subsection{Understanding and stewardship}

Mathematics is not merely the production of correct theorem statements and proofs. A result becomes part of mathematics through understanding: one must know why the statement is natural, which ideas make the proof work, how it relates to the surrounding theory, and what new questions it makes possible. An automatically generated proof candidate should therefore normally be regarded as the beginning of a mathematical process rather than its end. Even when it contains the decisive idea or an essentially complete argument, turning it into durable mathematical knowledge requires people to verify it, identify its conceptual content, place it in context, and explain it in a form that others can use and develop further.

There is a useful analogy with software. Software developed by, or in close collaboration with, people who understand its architecture is generally easier to inspect, maintain, repair, and extend. By contrast, code generated largely by AI can be difficult to work with when no person understands the reasons for its design choices or can explain how its components fit together. The problem is not that people have moved on from a system that they once understood, but that the system may never have been properly understood by anyone in the first place. Open-source access to the code does not by itself solve this problem: a system can be completely visible and yet remain practically opaque.

A similar danger arises in mathematics. A discovery pipeline may generate correct arguments while producing a body of work whose structure and significance have not been absorbed by the mathematical community. It would be irresponsible to build systems for producing mathematics while accepting that neither the systems nor their outputs need to be understood by the people presenting and relying on them. Scale must therefore be accompanied by distillation and stewardship. The aim should not be to accumulate an ever larger collection of claims, but to transform selected outputs into mathematics that can be checked, taught, reused, and extended.

People consequently remain part of the system at every stage. Mathematical problems are embedded in communities and histories, and the production of a final proof is only one contribution among many. Formulating a fruitful question, developing the surrounding theory, identifying the relevant literature, recognising a useful idea, verifying and repairing an argument, and presenting it clearly are all substantive parts of mathematical work. The role of an AI system should be disclosed accurately, but authorship and responsibility must remain with people who understand the result sufficiently well to explain it, defend it, and correct it when necessary.

The same principles should govern any public platform for generated attempts. Unverified material should be unmistakably labelled as unverified, changes and corrections should be recorded, the origins of the questions should be credited, and experts should be encouraged to inspect and improve the arguments. Such a platform should support collaboration and reduce duplicated effort, rather than overwhelm the literature with unfiltered claims or create the impression that generated outputs can replace the community whose accumulated knowledge made them possible. If proof generation becomes cheaper, human understanding, verification, organisation, and communication become more important, not less.

\subsection{Responsibility for correctness}

AI-assisted mathematics does not remove responsibility from the authors.  On the contrary, it increases the burden of responsibility.  The authors of an AI-assisted paper remain responsible for the correctness of the statements, the validity of the proofs, the accuracy of the references, and the clarity of the exposition.  AI output should be treated as a source of possible ideas, not as a certificate of truth.

\subsection{Pressure on peer-review}

There is also a practical danger. If AI systems make it easier to generate large numbers of plausible proofs, then the already strained peer-review system may face an additional burden.  Referees should not be expected to clean up unclear AI-generated arguments or to discover basic gaps that the authors could have found themselves.  For this reason, AI-assisted work should be held to a high standard of exposition and verification. Authors using such tools should make special efforts to write arguments in a form that is easy to check, to separate standard ingredients from new ones, to give precise references, and to indicate where human verification has entered.

In this possible future, clear distillation may become one of the central skills of the working mathematician. As proof search becomes cheaper, the ability to recognize the useful idea, discard misleading fragments, repair gaps, organize the argument, and present it in a readable and verifiable form may become even more important than it already is.

\subsection{A cautious optimism} Despite these challenges, our view of this direction is ultimately optimistic. Used carefully, these systems need not diminish mathematical research. They may enlarge the range of problems we can seriously explore, help us follow long technical paths, test approaches that would otherwise seem too costly, and reach parts of the mathematical landscape that might otherwise remain out of practical reach. This optimism does not lessen the need for verification, attribution, and responsibility. Rather, it places those obligations in service of a positive goal: accelerating discovery while preserving the standards of the subject. In this sense, we are sympathetic to Hilbert's maxim, \emph{``Wir m\"ussen wissen. Wir werden wissen.''} Properly used, AI systems may become one of the tools that helps us move further toward that ideal.

\subsection{A moving target.}
One final caution is in order. All of the exploratory work underlying the results discussed here was carried out with ChatGPT 5.5 Pro, before the release of the newer generation of models. These results should therefore be read as a snapshot of what was possible with that earlier generation, rather than as an estimate of the present frontier. The release of the new models shortly before the completion of this manuscript, together with our preliminary and necessarily anecdotal experiments with them, has only strengthened this impression. The improvements are perhaps most immediately apparent in matters of form (for example, in the clarity, organization, and overall quality of the mathematical writing) but our early experience suggests that they extend to matters of mathematical substance as well. Thus, if anything, the examples in this paper may understate what can now be achieved with rapidly improving AI systems and further reinforce our conviction that the questions raised here deserve serious attention.

\newpage
\part{Selected Problems}\label{part2}

\problempaper{Problem 1. Toroidal Elton--Odell theorem}
\label{probpaper:toroidal-elton-odell}

\begin{problemabstract}
We prove a toroidal form of the Elton--Odell theorem. More precisely, every infinite-dimensional complex normed space contains a sequence of unit vectors which are separated by a distance strictly larger than one, even after multiplication by unimodular scalars.
\end{problemabstract}

\problemcontents

\begin{problempaperbody}

\section{Statement and notation}

Throughout, all spaces are complex. For unit vectors $x,y$ in a normed space, set
\begin{equation*}
        \distT(x,y)=\inf_{|\theta|=1}\|x-\theta y\|.
\end{equation*}
We will prove the following statement.

\begin{theorem}[Toroidal Elton--Odell]
\label{p1-first:thm:main}
Let $X$ be an infinite-dimensional complex normed space. Then there are $\eps>0$ and unit vectors $(x_n)_{n=1}^\infty\subset X$ such that
\begin{equation*}
        \|x_n-\theta x_m\|\ge 1+\eps
        \qquad(n\ne m,\ |\theta|=1).
\end{equation*}
Equivalently, $\distT(x_n,x_m)\ge 1+\eps$ for all $n\ne m$.
\end{theorem}

The model produced two different proofs of the theorem. This is itself noteworthy: although the proofs share some techniques, their central arguments are sufficiently distinct to be regarded as different proof methods rather than variations of a single argument. We include only the first proof, which has been manually verified and edited. The second has not yet undergone the level of manual verification required for inclusion in the paper, but is made publicly available on the \website[project website] as part of the record of the exploratory process.

\section{Proof strategy and organisation}

The proof is organised as a sequence of reductions. In \Cref{p1-first:sec:easy-c0-l1}, we treat the two classical easy cases. If $X$ contains an isomorphic copy of $c_0$ or $\ell_1$, then the desired toroidally separated sequence is obtained from explicit sequences in $c_0$ and $\ell_1$, and then transferred to $X$ by James' distortion theorem.

In \Cref{p1-first:sec:flat-block-alternative}, we introduce asymptotically monotone bases and flat blocks. In \Cref{p1-first:sec:rapid-flat-block-alternative}, we prove the main combinatorial ingredient, the rapid flat-block alternative, \Cref{p1-first:prop:rapid}. Roughly speaking, it says that, inside the closed span of an asymptotically monotone basic sequence, failure of toroidal separation forces a rigid obstruction: there must exist successive almost-flat blocks $d_1<d_2<\ldots$ and phases $\alpha_j\in\T$ such that
\begin{equation*}
        \sup_{N\in\N}\left\|\sum_{j=1}^{N}\alpha_jd_j\right\|<\infty.
\end{equation*}

In \Cref{p1-first:sec:excluding-bad-alternative}, we show that this obstruction cannot occur in the two remaining structural cases needed later. First, if $X$ contains a boundedly complete basic sequence, we pass to an asymptotically monotone boundedly complete block basis and apply \Cref{p1-first:prop:rapid}; the bounded twisted partial sums produced by the proposition contradict bounded completeness. Second, if $X$ contains a non-weakly-convergent strongly summing basic sequence, we form normalised differences of far-apart terms. These differences are weakly null, so we pass to an asymptotically monotone basic subsequence and again apply \Cref{p1-first:prop:rapid}. The resulting bounded flat-block sums are then expanded back in the original strongly summing sequence, forcing convergence of a scalar series whose terms do not tend to zero.

Finally, in \Cref{p1-first:sec:banach-proof}, we assemble the cases. If $X$ contains $c_0$ or $\ell_1$, \Cref{p1-first:sec:easy-c0-l1} applies. If $X$ is reflexive, then it contains a boundedly complete basic sequence, so \Cref{p1-first:sec:excluding-bad-alternative} applies. If $X$ is nonreflexive and contains neither $c_0$ nor $\ell_1$, Rosenthal's $\ell_1$ theorem gives a non-weakly-convergent weak-Cauchy sequence, and Rosenthal's $c_0$ theorem gives a strongly summing subsequence; hence \Cref{p1-first:sec:excluding-bad-alternative} applies again. This exhausts all infinite-dimensional Banach spaces.

\section{The easy $c_0$ and $\ell_1$ cases}\label{p1-first:sec:easy-c0-l1}

We first treat the simple case where the space $X$ contains an isomorphic copy of either $c_0$ or $\ell_1$. We shall use the following elementary transfer estimate.

\begin{lemma}[Almost isometric transfer]
\label{p1-first:lem:transfer}
Let $E$ and $X$ be normed spaces, and let $(u_n)_{n=1}^{\infty}\subset S_E$ satisfy
\begin{equation*}
        \inf_{\substack{n,m\in\N\\ n\ne m}}\distT(u_n,u_m)\ge \sigma>1.
\end{equation*}
Suppose that $T\colon E\to X$ is linear and that, for some $0<\delta<1$,
\begin{equation*}
        (1-\delta)\|u\|\leq \|Tu\|\leq (1+\delta)\|u\| \qquad(u\in E).
\end{equation*}
For $n\in\N$, set
\begin{equation*}
        x_n=\frac{Tu_n}{\|Tu_n\|}.
\end{equation*}
Then
\begin{equation*}
        \distT(x_n,x_m)\ge \frac{(1-\delta)\sigma-2\delta}{1+\delta} \qquad(n,m\in\N,\ n\ne m).
\end{equation*}
In particular, if $\delta<(\sigma-1)/(\sigma+3)$, then the right-hand side is strictly larger than $1$.
\end{lemma}

\begin{proof}
For each $k\in\N$, put
\begin{equation*}
        a_k=\|Tu_k\|.
\end{equation*}
Since $\|u_k\|=1$, the hypothesis on $T$ gives
\begin{equation*}
        1-\delta\leq a_k\leq 1+\delta \qquad(k\in\N).
\end{equation*}
In particular, $a_k>0$, and $Tu_k=a_kx_k$.

Fix distinct indices $n,m\in\N$ and fix $\theta\in\T$. On the one hand,
\begin{equation*}
        \|Tu_n-\theta Tu_m\|\geq (1-\delta)\|u_n-\theta u_m\|\geq (1-\delta)\sigma.
\end{equation*}
On the other hand,
\begin{equation*}
\begin{aligned}
        \|Tu_n-\theta Tu_m\|&=\|a_nx_n-\theta a_mx_m\|\\
        &\leq a_n\|x_n-\theta x_m\|+|a_n-a_m|\\
        &\leq (1+\delta)\|x_n-\theta x_m\|+2\delta.
\end{aligned}
\end{equation*}
Combining the two estimates gives
\begin{equation*}
        \|x_n-\theta x_m\|\geq \frac{(1-\delta)\sigma-2\delta}{1+\delta}.
\end{equation*}
Taking the infimum over $\theta\in\T$ proves the asserted estimate.
\end{proof}

As an easy consequence, we get the following.

\begin{lemma}
\label{p1-first:lem:c0-l1}
If a Banach space $X$ contains an isomorphic copy of $c_0$ or an isomorphic copy of $\ell_1$, then $X$ satisfies \Cref{p1-first:thm:main}.
\end{lemma}

\begin{proof}
In $\ell_1$, the unit vector basis $(e_n)_{n=1}^{\infty}$ satisfies
\begin{equation*}
        \|e_n-\theta e_m\|_1=2 \qquad(n,m\in\N,\ n\ne m,\ |\theta|=1).
\end{equation*}
Hence the conclusion follows for every Banach space containing $\ell_1$ from James' distortion theorem \cite{James1964} and \Cref{p1-first:lem:transfer}.

In $c_0$, define
\begin{equation*}
        u_n=e_1+\ldots+e_n-e_{n+1} \qquad(n\in\N).
\end{equation*}
Then $(u_n)_{n=1}^{\infty}\subset S_{c_0}$. If $n<m$, then
\begin{equation*}
        \|u_n-\theta u_m\|_\infty\geq \max\{|1-\theta|,|1+\theta|\} \qquad(|\theta|=1).
\end{equation*}
Since
\begin{equation*}
        \inf_{|\theta|=1}\max\{|1-\theta|,|1+\theta|\}=\sqrt2,
\end{equation*}
we have $\distT(u_n,u_m)\geq\sqrt2$ whenever $n\ne m$. Again, James' distortion theorem \cite{James1964} and \Cref{p1-first:lem:transfer} give the conclusion for every Banach space containing $c_0$.
\end{proof}

\section{Asymptotically monotone bases and flat blocks}\label{p1-first:sec:flat-block-alternative}

We shall need the following definitions.

\begin{definition}
\label{p1-first:def:asymptotically-monotone}
Let $(e_n)_{n=1}^{\infty}$ be a normalized basic sequence. For $N\in\N$, set
\begin{equation*}
        M_N=\sup\left\{
        \frac{\left\|\sum_{i=p}^{q} a_i e_i\right\|}
             {\left\|\sum_{i=p}^{r} a_i e_i\right\|}:
        N\leq p\leq q\leq r,\ (a_i)_{i=p}^{r}\subset\C,\ 
        \sum_{i=p}^{r}a_i e_i\neq 0
        \right\}.
\end{equation*}
We say that $(e_n)_{n=1}^{\infty}$ is \emph{asymptotically monotone} if
\begin{equation*}
        M_N\longrightarrow 1 \qquad(N\to\infty).
\end{equation*}
Equivalently, for every $\gamma>0$ there is $N_0\in\N$ such that, whenever $N_0\leq p\leq q\leq r$ and $(a_i)_{i=p}^{r}\subset\C$,
\begin{equation*}
        \left\|\sum_{i=p}^{q} a_i e_i\right\|\leq (1+\gamma)\left\|\sum_{i=p}^{r} a_i e_i\right\|.
\end{equation*}
\end{definition}

This tail-initial form of asymptotic monotonicity is the only projection property used below.

\begin{definition}[Flat blocks]
\label{p1-first:def:flat-blocks}
Let $(e_n)_{n=1}^{\infty}$ be a basic sequence. A \emph{flat block} of $(e_n)_{n=1}^{\infty}$ is a nonzero vector of the form
\begin{equation*}
        d=\sum_{i\in F}\lambda_i e_i,
\end{equation*}
where $F\subset\N$ is finite and nonempty, and $|\lambda_i|=1$ for every $i\in F$.

If $0<\eta<1$, an \emph{$\eta$-flat unit block} is a unit vector of the form
\begin{equation*}
        b=\frac{d}{\|d\|},
\end{equation*}
where $d$ is a flat block satisfying
\begin{equation*}
        |\|d\|-1|\leq \eta.
\end{equation*}
For two finite-support vectors $u$ and $v$, write $u<v$ if $\max\supp u<\min\supp v$.
\end{definition}

Observe that, in the previous definition, the finite set $F$ need not be an interval of integers and thus gaps are allowed.

\section{The rapid flat-block alternative}\label{p1-first:sec:rapid-flat-block-alternative}

The next proposition is the main combinatorial step in the proof. It says that, if toroidal separation fails inside the closed span of an asymptotically monotone basic sequence, then this failure forces a very rigid alternative: one can find successive flat blocks $d_1<d_2<\ldots$, almost of norm one, and phases $\alpha_j\in\T$ whose partial sums are uniformly bounded. The idea is that failure of separation gives, on each sufficiently far-out tail, a finite cover of the almost-flat unit blocks, up to multiplication by unimodular scalars. These finite covers are then chained backwards: starting from a far-out block, one repeatedly chooses an earlier cover element close to the normalised block already built and subtracts a suitable phase multiple of it. This produces finite phased sums with uniformly bounded initial pieces; a diagonal argument then gives the infinite sequence. The rest of the proof is devoted to showing that this alternative is impossible in the cases we need.

\begin{proposition}[Rapid flat-block alternative]
\label{p1-first:prop:rapid}
Let $(e_n)_{n=1}^{\infty}$ be a normalized asymptotically monotone basic sequence. Suppose that, for every $\eps>0$, there is no sequence $(x_n)_{n=1}^{\infty}\subset \spancl{e_n:n\in\N}$ of unit vectors satisfying
\begin{equation*}
        \distT(x_n,x_m)>1+\eps \qquad(n,m\in\N,\ n\ne m).
\end{equation*}
Then there are successive flat blocks $(d_j)_{j=1}^{\infty}$ and phases $(\alpha_j)_{j=1}^{\infty}\subset\T$ such that
\begin{equation*}
        \|d_j\|\longrightarrow 1
\end{equation*}
and
\begin{equation*}
        \sup_{N\in\N}\left\|\sum_{j=1}^{N}\alpha_jd_j\right\|<\infty.
\end{equation*}
\end{proposition}

\begin{proof}
Choose two positive null sequences $(\eta_j)_{j=1}^{\infty}$ and $(\gamma_j)_{j=1}^{\infty}$ such that
\begin{equation*}
        0<\eta_j<2^{-j-10}
        \qquad\text{and}\qquad
        0<\gamma_j<2^{-j-10}
        \qquad(j\in\N).
\end{equation*}
For $j\in\N$, set
\begin{equation*}
        \rho_j=\sum_{k=j}^{\infty}(2\eta_k+\gamma_k).
\end{equation*}

By choosing the terms recursively and sufficiently small at each step, we may assume that
\begin{equation*}
        \rho_{j+1}\leq \eta_j \qquad(j\in\N).
\end{equation*}

We first construct, by induction on $j\in\N$, integers $m_j$ and finite families
\begin{equation*}
        B_j=\{b_{j,1},\ldots,b_{j,r_j}\}
\end{equation*}
of $\eta_j$-flat unit blocks, satisfying the following properties:

\begin{enumerate}[label=\textup{(\roman*)}]
\item For every $j\in\N$, the tail after $m_j$ is $(1+\gamma_j)$-monotone: whenever $m_j\leq p\leq q\leq r$ and $a_p,\ldots,a_r\in\C$,
\begin{equation}
\label{p1-first:eq:tail-projection}
        \left\|\sum_{i=p}^{q}a_ie_i\right\|\leq (1+\gamma_j)\left\|\sum_{i=p}^{r}a_ie_i\right\|.
\end{equation}

\item\label{p1-first:item:maximal-family} For every $j\in\N$, the family $B_j$ is a finite maximal family of $\eta_j$-flat unit blocks supported after $m_j$ such that
\begin{equation*}
        \distT(b_{j,i},b_{j,\ell})>1+\eta_j \qquad(1\leq i<\ell\leq r_j).
\end{equation*}

\item The families are successive: every member of $B_j$ is supported before $m_{j+1}$.
\end{enumerate}

Here, a vector is supported after $m_j$ if its support is contained in $\{m_j+1,m_j+2,\ldots\}$. Maximality in \ref{p1-first:item:maximal-family} is meant with respect to inclusion among families satisfying the displayed separation condition. Thus, once $B_j$ has been chosen, no further $\eta_j$-flat unit block supported after $m_j$ can be added to $B_j$ while preserving pairwise toroidal separation greater than $1+\eta_j$. Equivalently, every $\eta_j$-flat unit block supported after $m_j$ lies within distance at most $1+\eta_j$, up to multiplication by a unimodular scalar, of some member of $B_j$.

We now carry out the construction. By asymptotic monotonicity, for each $j\in\N$ there is an integer $N_j$ such that \eqref{p1-first:eq:tail-projection} holds whenever $N_j\leq p\leq q\leq r$.

Choose $m_1\geq N_1$. Having chosen $m_j$, consider the class of all $\eta_j$-flat unit blocks supported after $m_j$. By the standing assumption, this class contains no infinite subfamily whose distinct members are pairwise at $\distT$-distance greater than $1+\eta_j$. Hence a greedy selection process gives a finite maximal family
\begin{equation*}
        B_j=\{b_{j,1},\ldots,b_{j,r_j}\}
\end{equation*}
with the separation property in \ref{p1-first:item:maximal-family}. For each $1\leq i\leq r_j$, choose a flat block $d_{j,i}$ such that
\begin{equation}
\label{p1-first:eq:representing-block}
        b_{j,i}=\frac{d_{j,i}}{\|d_{j,i}\|}, \qquad |\|d_{j,i}\|-1|\leq \eta_j.
\end{equation}
Since $B_j$ is finite, we may choose $m_{j+1}\geq N_{j+1}$ larger than every index appearing in the supports of the vectors in $B_j$. This defines $B_j$ and $m_{j+1}$, and hence completes the recursive construction.

The maximality of $B_j$ gives the following covering property. If $u$ is any $\eta_j$-flat unit block supported after $m_j$, then there are $1\leq i\leq r_j$ and $\theta\in\T$ such that
\begin{equation}
\label{p1-first:eq:covering}
        \|b_{j,i}-\theta u\|\leq 1+\eta_j.
\end{equation}
Indeed, otherwise $u$ could be added to $B_j$, contradicting maximality.

This finishes the construction of the finite covers $B_j$. We now use them to build finite flat block combinations with uniformly bounded partial sums. The construction is done first at a fixed finite length $L$, and then a diagonal argument will let $L\to\infty$.

Fix $L\in\N$. We shall construct, by backward induction on $j=L,L-1,\ldots,1$, indices
\begin{equation*}
        \iota_j^{(L)}\in\{1,\ldots,r_j\},
\end{equation*}
phases $\theta_j^{(L)}\in\T$ for $1\leq j<L$, and flat blocks
\begin{equation*}
        z_L^{(L)},z_{L-1}^{(L)},\ldots,z_1^{(L)}
\end{equation*}
such that the following hold.

\begin{enumerate}[label=\textup{(\alph*)}]
\item\label{p1-first:item:last-level} At the last level,
\begin{equation*}
        z_L^{(L)}=d_{L,\iota_L^{(L)}}.
\end{equation*}

\item\label{p1-first:item:backward-recursion} For every $1\leq j<L$, if
\begin{equation*}
        u_{j+1}^{(L)}=\frac{z_{j+1}^{(L)}}{\|z_{j+1}^{(L)}\|},
\end{equation*}
then
\begin{equation}
\label{p1-first:eq:covering-step}
        \left\|b_{j,\iota_j^{(L)}}-\theta_j^{(L)}u_{j+1}^{(L)}\right\|\leq 1+\eta_j,
\end{equation}
and
\begin{equation}
\label{p1-first:eq:z-recursion}
        z_j^{(L)}=d_{j,\iota_j^{(L)}}-\theta_j^{(L)}z_{j+1}^{(L)}.
\end{equation}

\item\label{p1-first:item:norm-control} For every $1\leq j\leq L$,
\begin{equation}
\label{p1-first:eq:z-error-target}
        |\|z_j^{(L)}\|-1|\leq \rho_j.
\end{equation}
\end{enumerate}

We start at level $L$. Choose any $\iota_L^{(L)}\in\{1,\ldots,r_L\}$ and set
\begin{equation*}
        z_L^{(L)}=d_{L,\iota_L^{(L)}}.
\end{equation*}
This gives \ref{p1-first:item:last-level}. Moreover, by \eqref{p1-first:eq:representing-block},
\begin{equation*}
        |\|z_L^{(L)}\|-1|\leq \eta_L\leq \rho_L,
\end{equation*}
so \ref{p1-first:item:norm-control} holds at level $L$.

Now suppose that $1\leq j<L$ and that the construction has been completed at level $j+1$. In particular, \ref{p1-first:item:norm-control} gives
\begin{equation*}
        |\|z_{j+1}^{(L)}\|-1|\leq \rho_{j+1}\leq \eta_j.
\end{equation*}
Since the families $B_1,B_2,\ldots$ are successive, $z_{j+1}^{(L)}$ is supported after $m_j$. Hence the normalized vector
\begin{equation*}
        u_{j+1}^{(L)}=\frac{z_{j+1}^{(L)}}{\|z_{j+1}^{(L)}\|}
\end{equation*}
is an $\eta_j$-flat unit block supported after $m_j$. By the covering property \eqref{p1-first:eq:covering}, there are $\iota_j^{(L)}\in\{1,\ldots,r_j\}$ and $\theta_j^{(L)}\in\T$ such that \eqref{p1-first:eq:covering-step} holds. We then define
\begin{equation*}
        z_j^{(L)}=d_{j,\iota_j^{(L)}}-\theta_j^{(L)}z_{j+1}^{(L)}.
\end{equation*}
Thus \ref{p1-first:item:backward-recursion} holds at level $j$. Since the two terms have successive disjoint supports and unimodular coefficients on their supports, $z_j^{(L)}$ is again a flat block.

It remains to verify \ref{p1-first:item:norm-control} at level $j$. Put
\begin{equation*}
        s_j^{(L)}=\|d_{j,\iota_j^{(L)}}\|,
        \qquad
        t_{j+1}^{(L)}=\|z_{j+1}^{(L)}\|.
\end{equation*}
By \eqref{p1-first:eq:representing-block} and the induction hypothesis,
\begin{equation*}
        |s_j^{(L)}-1|\leq\eta_j,
        \qquad
        |t_{j+1}^{(L)}-1|\leq\rho_{j+1}.
\end{equation*}
By \ref{p1-first:item:backward-recursion}, and by the definitions of $s_j^{(L)}$, $t_{j+1}^{(L)}$ and $u_{j+1}^{(L)}$, we have
\begin{equation*}
        z_j^{(L)}=s_j^{(L)}b_{j,\iota_j^{(L)}}-\theta_j^{(L)}t_{j+1}^{(L)}u_{j+1}^{(L)}.
\end{equation*}
Therefore,
\begin{equation}
\label{p1-first:eq:upper-z}
\begin{aligned}
        \|z_j^{(L)}\|
        &=\left\|s_j^{(L)}b_{j,\iota_j^{(L)}}-\theta_j^{(L)}t_{j+1}^{(L)}u_{j+1}^{(L)}\right\|\\
        &\leq \left\|b_{j,\iota_j^{(L)}}-\theta_j^{(L)}u_{j+1}^{(L)}\right\|+|s_j^{(L)}-1|+|t_{j+1}^{(L)}-1|\\
        &\leq 1+\eta_j+\eta_j+\rho_{j+1}\\
        &=1+2\eta_j+\rho_{j+1}.
\end{aligned}
\end{equation}
For the lower estimate, $d_{j,\iota_j^{(L)}}$ is an initial interval projection of $z_j^{(L)}$. Since $m_j\geq N_j$, \eqref{p1-first:eq:tail-projection} gives
\begin{equation*}
        \|d_{j,\iota_j^{(L)}}\|\leq (1+\gamma_j)\|z_j^{(L)}\|.
\end{equation*}
Therefore
\begin{equation}
\label{p1-first:eq:lower-z}
        \|z_j^{(L)}\|\geq \frac{1-\eta_j}{1+\gamma_j}\geq 1-\eta_j-\gamma_j,
\end{equation}
where in the last inequality we used the elementary estimate
\begin{equation*}
        \frac{1-a}{1+b}\geq 1-a-b \qquad(a,b>0).
\end{equation*}
Combining \eqref{p1-first:eq:upper-z} and \eqref{p1-first:eq:lower-z}, we obtain
\begin{equation*}
        |\|z_j^{(L)}\|-1|\leq 2\eta_j+\gamma_j+\rho_{j+1}=\rho_j.
\end{equation*}
Thus \ref{p1-first:item:norm-control} holds at level $j$. This proves the induction step and hence completes the backward construction.

Unwinding \eqref{p1-first:eq:z-recursion}, and setting
\begin{equation*}
        \alpha_1^{(L)}=1, \qquad \alpha_j^{(L)}=(-1)^{j-1}\theta_1^{(L)}\cdots\theta_{j-1}^{(L)} \quad(2\leq j\leq L),
\end{equation*}
we obtain
\begin{equation}
\label{p1-first:eq:finite-expansion}
        z_1^{(L)}=\sum_{j=1}^{L}\alpha_j^{(L)}d_{j,\iota_j^{(L)}}.
\end{equation}
In particular, by \eqref{p1-first:eq:z-error-target},
\begin{equation}
\label{p1-first:eq:z1-bound}
        \|z_1^{(L)}\|\leq 1+\rho_1 \qquad(L\in\N).
\end{equation}

Thus, for each fixed $L\in\N$, the construction gives a finite sequence
\begin{equation*}
        d_{1,\iota_1^{(L)}}<d_{2,\iota_2^{(L)}}<\ldots<d_{L,\iota_L^{(L)}}
\end{equation*}
and phases $\alpha_1^{(L)},\ldots,\alpha_L^{(L)}\in\T$. If $N\leq L$, then the partial sum over the first $N$ levels is an initial interval projection of $z_1^{(L)}$. Applying \eqref{p1-first:eq:tail-projection} at level $1$, we get
\begin{equation}
\label{p1-first:eq:finite-partial-bound}
        \left\|\sum_{j=1}^{N}\alpha_j^{(L)}d_{j,\iota_j^{(L)}}\right\|\leq (1+\gamma_1)(1+\rho_1)=:C.
\end{equation}

Finally, we pass from finite lengths to an infinite sequence. For each fixed $j\in\N$ and $L \geq j$, the indices $\iota_j^{(L)}$ take values in the finite set $\{1,\ldots,r_j\}$, while the phases $\alpha_j^{(L)}$ lie in the compact set $\T$. A diagonal compactness argument gives a subsequence $L_k\to\infty$ such that, for each fixed $j\in\N$,
\begin{equation*}
        \iota_j^{(L_k)}=\iota_j \text{ eventually},
        \qquad
        \alpha_j^{(L_k)}\to\alpha_j\in\T.
\end{equation*}
Set
\begin{equation*}
        d_j=d_{j,\iota_j}\qquad(j\in\N).
\end{equation*}
Then $d_1<d_2<\ldots$, and \eqref{p1-first:eq:representing-block} gives $\|d_j\|\to1$. Fixing $N\in\N$ in \eqref{p1-first:eq:finite-partial-bound} and passing to the limit along $L_k$ yields
\begin{equation*}
        \left\|\sum_{j=1}^{N}\alpha_jd_j\right\|\leq C.
\end{equation*}
Since $N\in\N$ was arbitrary,
\begin{equation*}
        \sup_{N\in\N}\left\|\sum_{j=1}^{N}\alpha_jd_j\right\|<\infty.
\end{equation*}
This proves the proposition.
\end{proof}

\section{Excluding the bad alternative}\label{p1-first:sec:excluding-bad-alternative}

The rapid flat-block alternative leaves one obstruction: bounded twisted partial sums of successive almost-flat blocks. We now record the selection principle that allows us to arrange asymptotic monotonicity without losing the structural properties needed later.

\begin{lemma}[Asymptotically monotone selection]
\label{p1-first:lem:am-selection}
The following holds.
\begin{enumerate}[label=\textup{(\roman*)}]
\item\label{p1-first:item:am-selection-weakly-null} Every seminormalized weakly null sequence has a subsequence whose termwise normalization is an asymptotically monotone basic sequence.

\item\label{p1-first:item:am-selection-bc} Every boundedly complete basic sequence has a normalised asymptotically monotone block basis. Moreover, the block basis is boundedly complete.
\end{enumerate}
\end{lemma}
\begin{proof}
\ref{p1-first:item:am-selection-weakly-null} is precisely the asymptotically monotone selection principle; see \cite[Lemma 3.1]{Barroso2019}. \ref{p1-first:item:am-selection-bc} follows from \cite[Lemma 2.4]{HajekKaniaRusso2018}. Finally, normalising a block basis does not change its canonical projections, and bounded completeness is preserved under such normalisation.
\end{proof}

Using this, we are ready to prove the boundedly complete case.

\begin{lemma}[Boundedly complete case]
\label{p1-first:lem:boundedly-complete}
Let $X$ contain a boundedly complete basic sequence. Then $X$ satisfies \Cref{p1-first:thm:main}.
\end{lemma}

\begin{proof}
Assume toward a contradiction that no sequence of unit vectors in $X$ is toroidally $(1+\eps)$-separated for any $\eps>0$. Choose a boundedly complete basic sequence in $X$. By \Cref{p1-first:lem:am-selection}\ref{p1-first:item:am-selection-bc}, pass to a normalised asymptotically monotone boundedly complete block basis $(e_n)_{n=1}^{\infty}$.

Apply \Cref{p1-first:prop:rapid} to $(e_n)_{n=1}^{\infty}$. We obtain successive flat blocks $(d_j)_{j=1}^{\infty}$ with $\|d_j\|\to1$ and phases $(\alpha_j)_{j=1}^{\infty}\subset\T$ such that
\begin{equation*}
        \sup_{N \in \N} \left\|\sum_{j=1}^N \alpha_j d_j\right\|<\infty .
\end{equation*}
Since $(d_j)_{j=1}^{\infty}$ is a seminormalised block basic sequence of the boundedly complete basis $(e_n)_{n=1}^{\infty}$, it is boundedly complete. Hence the series $\sum_{j=1}^{\infty} \alpha_j d_j$ converges. Its terms must then tend to $0$, contradicting
\begin{equation*}
        \|\alpha_j d_j\|=\|d_j\|\longrightarrow 1
        \qquad (j\to\infty).
\end{equation*}
Thus, the desired toroidally separated sequence must exist.
\end{proof}

We shall need the following definition.

\begin{definition}
\label{p1-first:def:strongly-summing}
A basic sequence $(s_n)_{n=1}^{\infty}$ is called \emph{strongly summing} if it is weak-Cauchy and, whenever a scalar sequence $(c_n)_{n=1}^{\infty}$ satisfies
\begin{equation*}
        \sup_{N\in\N}\left\|\sum_{n=1}^{N} c_n s_n\right\|<\infty,
\end{equation*}
the scalar series
\begin{equation*}
        \sum_{n=1}^{\infty} c_n
\end{equation*}
converges.
\end{definition}

We now treat the case in which the Banach space $X$ contains a nontrivial strongly summing basic sequence.

\begin{lemma}[Strongly summing case]
\label{p1-first:lem:strongly-summing}
Suppose that a Banach space $X$ contains a strongly summing basic sequence $(s_n)_{n=1}^{\infty}$ which is not weakly convergent. Then $X$ satisfies \Cref{p1-first:thm:main}.
\end{lemma}
\begin{proof}
Assume, toward a contradiction, that no sequence of unit vectors in $X$ is toroidally $(1+\eps)$-separated for any $\eps>0$.

Since $(s_n)_{n=1}^{\infty}$ is weak-Cauchy and not weakly convergent, it is not norm Cauchy. Hence there are $r>0$ and indices
\begin{equation*}
        p_1<q_1<p_2<q_2<\ldots
\end{equation*}
such that
\begin{equation}
\label{p1-first:eq:diff-lower}
        \|s_{p_n}-s_{q_n}\|\geq r \qquad(n\in\N).
\end{equation}
Define
\begin{equation*}
        y_n=\frac{s_{p_n}-s_{q_n}}{\|s_{p_n}-s_{q_n}\|} \qquad(n\in\N).
\end{equation*}
Then $(y_n)_{n=1}^{\infty}$ is normalized and weakly null. Indeed, fix $f\in X^*$. Since $(s_n)_{n=1}^{\infty}$ is weak-Cauchy, the scalar sequence $(f(s_n))_{n=1}^{\infty}$ is Cauchy. As $p_n,q_n\to\infty$, it follows that
\begin{equation*}
        f(s_{p_n})-f(s_{q_n})\longrightarrow 0.
\end{equation*}
Since the denominators in the definition of $y_n$ are bounded below by $r$, we get $f(y_n)\to0$. Thus $(y_n)_{n=1}^{\infty}$ is weakly null.

By \Cref{p1-first:lem:am-selection} \ref{p1-first:item:am-selection-weakly-null}, pass to a normalized asymptotically monotone basic subsequence of $(y_n)_{n=1}^{\infty}$. Relabelling this subsequence and the corresponding pairs, we may assume that $(y_n)_{n=1}^{\infty}$ itself is asymptotically monotone.

Apply \Cref{p1-first:prop:rapid} to $(y_n)_{n=1}^{\infty}$. We obtain successive flat blocks
\begin{equation*}
        d_j=\sum_{n\in F_j}\lambda_{j,n}y_n, \qquad |\lambda_{j,n}|=1 \quad(n\in F_j), \qquad F_1<F_2<\ldots,
\end{equation*}
and phases $(\alpha_j)_{j=1}^{\infty}\subset\T$ such that
\begin{equation}
\label{p1-first:eq:bad-strong}
        A:=\sup_{N\in\N}\left\|\sum_{j=1}^{N}\alpha_jd_j\right\|<\infty.
\end{equation}
Define scalars $(a_n)_{n=1}^{\infty}$ by
\begin{equation*}
        a_n=\alpha_j\lambda_{j,n} \quad(n\in F_j), \qquad a_n=0 \quad\left(n\notin\bigcup_{j=1}^{\infty}F_j\right).
\end{equation*}
This is well defined because the sets $F_j$ are successive, hence disjoint. Moreover, every nonzero $a_n$ has modulus $1$, and there are infinitely many nonzero $a_n$.

Let $K$ be the basis constant of $(y_n)_{n=1}^{\infty}$. If $P_m$ denotes the $m$-th coordinate projection relative to $(y_n)_{n=1}^{\infty}$, then $\|P_m\|\leq K$ for every $m\in\N$. For each $m\in\N$, choose $N$ large enough that $m\leq \max F_N$. Then
\begin{equation*}
        \sum_{n=1}^{m}a_ny_n=P_m\left(\sum_{j=1}^{N}\alpha_jd_j\right).
\end{equation*}
Thus \eqref{p1-first:eq:bad-strong} gives
\begin{equation}
\label{p1-first:eq:y-partials-bounded}
        B:=\sup_{m\in\N}\left\|\sum_{n=1}^{m}a_ny_n\right\|\leq KA<\infty.
\end{equation}

We now expand these sums in the original strongly summing sequence. For each $n\in\N$,
\begin{equation*}
        a_ny_n=\frac{a_n}{\|s_{p_n}-s_{q_n}\|}s_{p_n}-\frac{a_n}{\|s_{p_n}-s_{q_n}\|}s_{q_n}.
\end{equation*}
Define scalar coefficients $(c_k)_{k=1}^{\infty}$ by
\begin{equation*}
        c_{p_n}=\frac{a_n}{\|s_{p_n}-s_{q_n}\|}, \qquad c_{q_n}=-\frac{a_n}{\|s_{p_n}-s_{q_n}\|} \qquad(n\in\N),
\end{equation*}
and set $c_k=0$ whenever $k\notin\{p_n,q_n:n\in\N\}$. If $a_n=0$, then the two displayed coefficients are also $0$.

We claim that the vector partial sums $\sum_{k=1}^{m}c_ks_k$ are bounded. Since
\begin{equation*}
        p_1<q_1<p_2<q_2<\ldots,
\end{equation*}
every partial sum stops either after a completed pair or halfway through the next pair. More precisely, suppose first that, for some $n\in\N$,
\begin{equation*}
        q_n\leq m<p_{n+1}.
\end{equation*}
Then the pairs $(p_i,q_i)$ have been completely included in the partial sum for $1\leq i\leq n$, and no later pair has begun. Hence
\begin{equation*}
        \sum_{k=1}^{m}c_ks_k=\sum_{i=1}^{n}\left(c_{p_i}s_{p_i}+c_{q_i}s_{q_i}\right)=\sum_{i=1}^{n}a_iy_i,
\end{equation*}
so the norm is at most $B$ by \eqref{p1-first:eq:y-partials-bounded}.

Suppose instead that, for some $n\in\N$,
\begin{equation*}
        p_n\leq m<q_n.
\end{equation*}
Then the pairs $(p_i,q_i)$ have been completely included for $1\leq i<n$, while the $n$-th pair has contributed only its first term. Therefore
\begin{equation*}
        \sum_{k=1}^{m}c_ks_k=\sum_{i=1}^{n-1}a_iy_i+\frac{a_n}{\|s_{p_n}-s_{q_n}\|}s_{p_n}.
\end{equation*}
Since $(s_n)_{n=1}^{\infty}$ is weak-Cauchy, it is bounded. Writing $M=\sup_{n\in\N}\|s_n\|<\infty$, the second term has norm at most $M/r$ by \eqref{p1-first:eq:diff-lower}. Thus every partial sum has norm at most $B+M/r$, and hence
\begin{equation}
\label{p1-first:eq:s-partials-bounded}
        \sup_{m\in\N}\left\|\sum_{k=1}^{m}c_ks_k\right\|<\infty.
\end{equation}

Strong summability of $(s_n)_{n=1}^{\infty}$ implies that the scalar series $\sum_{k=1}^{\infty}c_k$ converges. This is impossible. Indeed, for every $n\in\N$ such that $a_n\ne0$,
\begin{equation*}
        |c_{p_n}|=\frac{1}{\|s_{p_n}-s_{q_n}\|}\geq \frac{1}{2M}.
\end{equation*}
There are infinitely many such $n$, so the scalar sequence $(c_k)_{k=1}^{\infty}$ does not converge to $0$. This contradicts the convergence of $\sum_{k=1}^{\infty}c_k$. Therefore, the assumed failure of toroidal separation is impossible.
\end{proof}

\section{Proof of the main result}\label{p1-first:sec:banach-proof}

We can finally obtain the proof of \Cref{p1-first:thm:main}.

\begin{proof}[Proof of \Cref{p1-first:thm:main}]
Assume first that $X$ is an in\-fi\-nite-di\-men\-sion\-al Banach space. If $X$ contains an isomorphic copy of $c_0$ or $\ell_1$, then \Cref{p1-first:lem:c0-l1} applies. Hence assume that $X$ contains neither $c_0$ nor $\ell_1$.

If $X$ is reflexive, then, by the basic sequence theorem, $X$ contains a basic sequence. Its closed linear span is reflexive, and therefore the basis is boundedly complete. \Cref{p1-first:lem:boundedly-complete} applies.

It remains to consider the case when $X$ is nonreflexive and contains neither $c_0$ nor $\ell_1$. Since $X$ is nonreflexive, its closed unit ball is not weakly compact. By the Eberlein--Šmulian theorem, the closed unit ball is therefore not weakly sequentially compact. Hence there is a bounded sequence $(x_n)_{n=1}^{\infty}\subset X$ with no weakly convergent subsequence.

Since $X$ contains no copy of $\ell_1$, the complex version of Rosenthal's $\ell_1$ theorem due to Dor \cite{Dor1975}, gives a weak-Cauchy subsequence of $(x_n)_{n=1}^{\infty}$. This weak-Cauchy subsequence is not weakly convergent, because otherwise $(x_n)_{n=1}^{\infty}$ would have a weakly convergent subsequence. Since $X$ contains no copy of $c_0$, Rosenthal's $c_0$ theorem \cite{Rosenthal1994} gives a strongly summing subsequence. \Cref{p1-first:lem:strongly-summing} applies.

A standard completion-and-density argument now yields the result for an arbitrary in\-fi\-nite-di\-men\-sion\-al, not necessarily complete, normed space.
\end{proof}

\end{problempaperbody}

\clearpage
\providecommand{\PtwoB}{\mathcal B}
\providecommand{\PtwoK}{\mathcal K}
\providecommand{\PtwoQ}{\mathcal Q}
\providecommand{\PtwoS}{\mathcal S}
\providecommand{\PtwoC}{\mathfrak c}
\providecommand{\Ptwodens}{\mathop{\mathrm{dens}}}
\providecommand{\Ptworan}{\mathop{\mathrm{ran}}}

\problempaper{Problem 2. Non-Calkin unital Banach algebras}
\label{probpaper:calkin-nonrepresentation}

\begin{problemabstract}
We construct unital complex Banach algebras which are not isomorphic, as Banach algebras, to $\PtwoB(X)/\PtwoK(X)$ for any Banach space $X$. We give two constructions. The first uses a large Leavitt-type quotient and an $\ell_1$ matrix-unit obstruction. The second uses a large shift algebra on $c_0(\Gamma^{<\omega})$ and forces a copy of $c_0$ inside a nonseparable space.
\end{problemabstract}

\problemcontents

\begin{problempaperbody}

\section{Statement and notation}

All Banach spaces and Banach algebras are complex. For a Banach space $X$, write
\begin{equation*}
        \PtwoB(X)=\{T \colon X\to X:T\text{ is bounded and linear}\},
\end{equation*}
and let $\PtwoK(X)$ be the closed ideal of compact operators on $X$. The Calkin algebra of $X$ is
\begin{equation*}
        \PtwoQ(X)=\PtwoB(X)/\PtwoK(X).
\end{equation*}
A Banach-algebra isomorphism means a bounded complex-linear algebra isomorphism; by the open mapping theorem, its inverse is then bounded.

\begin{theorem}
\label{p2:thm:main}
There is a unital Banach algebra $A$ which is not isomorphic, as a Banach algebra, to $\PtwoB(X)/\PtwoK(X)$ for any Banach space $X$.
\end{theorem}

We present two proofs. Both constructions follow the same strategy. We first build a Banach algebra which is large enough to rule out being the Calkin algebra of a separable Banach space, and which is topologically simple. If such an algebra were the Calkin algebra of a nonseparable Banach space, topological simplicity would force the ideal of separable-range operators to coincide with the compact operators. The two proofs then use different algebraic structures to contradict this conclusion: in the first, a matrix-unit configuration forces a noncompact separable-range operator through a map into $\ell_1$; in the second, the structure forces a copy of $c_0$ to appear in the underlying Banach space.

\section{First proof}

\subsection{Proof strategy and organisation}

The first proof starts from a large Leavitt-type algebra. In \Cref{p2-first:sec:quotient-preliminaries}, we record two elementary facts about quotient norms and about absorbing countably many compact errors into a separable invariant subspace. In \Cref{p2-first:sec:leavitt-construction}, we construct a topologically simple Banach algebra $A_\kappa$ from the one-vertex graph with $\kappa=(2^{\aleph_0})^+$ loops. In \Cref{p2-first:sec:density}, we compute its density character.

The obstruction to being a Calkin algebra is proved in \Cref{p2-first:sec:matrix-obstruction}. If $X$ is nonseparable and $\PtwoQ(X)$ is topologically simple, then separable-range operators on $X$ must coincide with compact operators. We then show that this is incompatible with a countable matrix-unit system whose first row has uniformly bounded $c_0$-type finite sums. Finally, \Cref{p2-first:lem:Akappa-has-matrix-units} shows that $A_\kappa$ contains precisely such a matrix-unit system, and \Cref{p2-first:sec:first-conclusion} completes the proof.

\subsection{Quotient-norm preliminaries}
\label{p2-first:sec:quotient-preliminaries}

We shall use the following elementary quotient-norm observation to pass from estimates on a corrected operator on $X$ to estimates for the induced operator on $X/N$. The point is that adding an operator whose range lies in the quotient kernel changes the lift but not the quotient action.

\begin{lemma}[Induced operator norm from a corrected lift]
\label{p2-first:lem:quotient-correction}
Let $X$ be a Banach space, let $N\subset X$ be a closed subspace, and let $\pi \colon X\to X/N$ be the quotient map. Suppose that $U,K\in\PtwoB(X)$ satisfy $U(N)\subset N$ and $K(X)\subset N$. Then $U+K$ also leaves $N$ invariant, the operators induced by $U$ and $U+K$ on $X/N$ are equal, and
\begin{equation*}
        \|\overline U\|\leq \|U+K\|,
\end{equation*}
where $\overline U$ denotes the operator induced by $U$ on $X/N$.
\end{lemma}

\begin{proof}
Since $K(X)\subset N$, in particular $K(N)\subset N$, so $U+K$ leaves $N$ invariant. For every $x\in X$,
\begin{equation*}
        \pi((U+K)x)=\pi(Ux)+\pi(Kx)=\pi(Ux).
\end{equation*}
Thus $U$ and $U+K$ induce the same operator on $X/N$. The norm of an induced operator is bounded above by the norm of any operator inducing it, and the estimate follows.
\end{proof}

We shall also need a simple separability device which absorbs countably many compact ranges while remaining invariant under a prescribed countable family of operators.

\begin{lemma}[Absorbing countably many compact errors]
\label{p2-first:lem:absorb}
Let $X$ be a Banach space, let $(R_n)_{n=1}^{\infty}\subset\PtwoB(X)$, and let $(K_m)_{m=1}^{\infty}\subset\PtwoK(X)$. Then there is a separable closed subspace $N\subset X$ such that
\begin{equation*}
        K_m(X)\subset N \qquad(m\in\N),
\end{equation*}
and
\begin{equation*}
        R_n(N)\subset N \qquad(n\in\N).
\end{equation*}
\end{lemma}

\begin{proof}
For each compact operator $K_m$, the closed linear span of $K_m(X)$ is separable. Indeed, $K_m(B_X)$ has compact norm closure, hence is separable, and
\begin{equation*}
        K_m(X)=\bigcup_{r=1}^{\infty}rK_m(B_X).
\end{equation*}
Let $N_0$ be the closed linear span of $\bigcup_mK_m(X)$. Then $N_0$ is separable. Let $N$ be the closed linear span of all vectors of the form
\begin{equation*}
        R_{i_1}R_{i_2}\ldots R_{i_k}x,
\end{equation*}
where $k\geq0$, $i_1,\ldots,i_k\in\N$, and $x\in N_0$. There are only countably many finite words in the countable family $(R_n)_{n \in \N}$, and the image of a separable space under a bounded operator is separable. Hence $N$ is separable. By construction, $N$ contains every $K_m(X)$ and is invariant under every $R_n$.
\end{proof}

\subsection{The Leavitt-type quotient}
\label{p2-first:sec:leavitt-construction}

Let
\begin{equation*}
        \PtwoC=2^{\aleph_0}, \qquad \kappa=\PtwoC^+.
\end{equation*}
We shall use a very special Leavitt path algebra. The general theory may be found in \cite[Chapter 1]{AbramsAraSiles2017} and in the original papers \cite{AbramsArandaPino2005,AraMorenoPardo2007}; for arbitrary graphs and infinite emitters, see also \cite{Goodearl2009,Tomforde2007}. In the present case, however, all the needed features are quite elementary, and we record the details.

Let $E_\kappa$ be the directed graph with one vertex $v$ and $\kappa$ loops. Informally, the associated Leavitt path algebra is the algebra generated by the loops, together with formal reverse loops, subject to the relations saying that a reverse loop cancels the corresponding loop and annihilates the others.

We first describe the algebra explicitly. Let $L_\kappa$ be the unital complex algebra generated by symbols
\begin{equation*}
        s_\alpha,\ t_\alpha \qquad(\alpha<\kappa),
\end{equation*}
subject to the relations
\begin{equation}\label{eq:levit-relations}
        t_\alpha s_\beta=\delta_{\alpha\beta}1 \qquad(\alpha,\beta<\kappa).
\end{equation}
Equivalently, $L_\kappa$ is the quotient of the free unital complex algebra on the symbols $s_\alpha,t_\alpha$, $\alpha<\kappa$, by the two-sided ideal generated by the elements
\begin{equation*}
        t_\alpha s_\beta-\delta_{\alpha\beta}1 \qquad(\alpha,\beta<\kappa).
\end{equation*}
Thus $t_\alpha$ is a left inverse for $s_\alpha$, while $t_\alpha$ annihilates $s_\beta$ whenever $\beta\ne\alpha$.

This is precisely the Leavitt path algebra $L_{\C}(E_\kappa)$ of the graph $E_\kappa$ with one vertex $v$ and $\kappa$ loops. In that graph-theoretic language, the loop corresponding to $\alpha<\kappa$ is $s_\alpha$, and its formal reverse edge is $t_\alpha$. The vertex $v$ gives the identity element, which we denote by $1$. Since there is only one vertex, all paths begin and end at $v$, so the usual source and range relations simply say that
\begin{equation*}
        1s_\alpha=s_\alpha1=s_\alpha, \qquad 1t_\alpha=t_\alpha1=t_\alpha \qquad(\alpha<\kappa).
\end{equation*}
The Cuntz--Krieger cancellation relation is exactly
\begin{equation*}
        t_\alpha s_\beta=\delta_{\alpha\beta}1 \qquad(\alpha,\beta<\kappa).
\end{equation*}

The only possible additional Cuntz--Krieger relation would be the finite-emitter relation
\begin{equation*}
        1=\sum_{s(e)=v}ee^*.
\end{equation*}
However, the unique vertex emits infinitely many loops. In the definition of a Leavitt path algebra, this relation is imposed only at vertices emitting a finite nonzero number of edges. Therefore no relation of the form
\begin{equation*}
        1=\sum_{\alpha\in F}s_\alpha t_\alpha
\end{equation*}
is imposed for a finite set $F\subset\kappa$.

We shall use the following notation for words. Let $W_\kappa$ denote the set of all finite sequences
\begin{equation*}
        \mu=(\alpha_1,\ldots,\alpha_n)
\end{equation*}
with entries in $\kappa$, together with the empty word $\varnothing$. We write such a word as $\mu=\alpha_1\ldots\alpha_n$, and define its length by $|\mu|=n$. For $\mu=\alpha_1\ldots\alpha_n$, set
\begin{equation*}
        s_\mu=s_{\alpha_1}\ldots s_{\alpha_n}, \qquad t_\mu=t_{\alpha_n}\ldots t_{\alpha_1}.
\end{equation*}
For the empty word, set
\begin{equation*}
        s_\varnothing=t_\varnothing=1.
\end{equation*}
The order in the definition of $t_\mu$ is reversed because $t_\mu$ represents the formal inverse path to $\mu$.

We shall use the standard spanning description of Leavitt path algebras; see \cite[Lemma 1.2.12(i),(iii)]{AbramsAraSiles2017}. In the present one-vertex case, it says that $L_\kappa$ is linearly spanned by the elements
\begin{equation*}
        s_\mu t_\nu \qquad(\mu,\nu\in W_\kappa).
\end{equation*}
Concretely, products of these monomials are computed by repeatedly replacing adjacent subwords $t_\alpha s_\beta$ with
\begin{equation*}
        t_\alpha s_\beta=\delta_{\alpha\beta}1.
\end{equation*}
Thus every product in the generators $s_\alpha,t_\alpha$ reduces to a finite linear combination of terms of the form
\begin{equation*}
        s_{\alpha_1}\ldots s_{\alpha_m}t_{\beta_n}\ldots t_{\beta_1}=s_\mu t_\nu,
\end{equation*}
where $\mu=\alpha_1\ldots\alpha_m$ and $\nu=\beta_1\ldots\beta_n$.

\begin{lemma}
\label{p2-first:lem:leavitt-simple}
The algebra $L_\kappa$ is algebraically simple.
\end{lemma}

\begin{proof}
This also follows immediately from the simplicity criterion for Leavitt path algebras of arbitrary graphs~\cite[Theorem~2.9.1]{AbramsAraSiles2017}: since $E_\kappa^0=\{v\}$, its only hereditary saturated subsets are $\varnothing$ and $E_\kappa^0$, and every cycle has an exit, because it uses only finitely many of the $\kappa$ loops and hence there is another loop available.  We give the direct argument in this special case.

Let $I$ be a nonzero two-sided ideal of $L_\kappa$, and choose $0\ne a\in I$. Write $a$ in normal form,
\begin{equation*}
        a=\sum_{r=1}^{n} c_r s_{\mu_r}t_{\nu_r},
\end{equation*}
where $c_r\ne0$ and the pairs $(\mu_r,\nu_r)$ are distinct. Choose an index $r_0$ such that $|\mu_{r_0}|+|\nu_{r_0}|$ is minimal among the pairs appearing in this expression. Put $\mu=\mu_{r_0}$ and $\nu=\nu_{r_0}$. Then
\begin{equation*}
        t_\mu a s_\nu=c_{r_0}1+b,
\end{equation*}
where $b$ is a finite linear combination of monomials $s_\sigma t_\tau$ with $(\sigma,\tau)\ne(\varnothing,\varnothing)$. Indeed, the chosen term gives $c_{r_0}1$, and the minimality of $|\mu|+|\nu|$ prevents any other term from contributing another scalar term.

Only finitely many non-empty words occur in the monomials appearing in $b$. Since $\kappa$ is infinite, choose a letter $\lambda<\kappa$ which is not the first letter of any of these non-empty words. Then every non-scalar term in $b$ is killed by multiplying on the left by $t_\lambda$ and on the right by $s_\lambda$. Hence
\begin{equation*}
        t_\lambda(t_\mu a s_\nu)s_\lambda=c_{r_0}t_\lambda s_\lambda=c_{r_0}1.
\end{equation*}
Thus $1\in I$, and so $I=L_\kappa$. Therefore $L_\kappa$ is algebraically simple.
\end{proof}

Now put
\begin{equation*}
        Y_\kappa=\ell_1(W_\kappa).
\end{equation*}
For $\alpha<\kappa$, define operators $S_\alpha,T_\alpha\in\PtwoB(Y_\kappa)$ on the unit vector basis by
\begin{equation*}
        S_\alpha\delta_w=\delta_{\alpha w} \qquad(w\in W_\kappa),
\end{equation*}
and
\begin{equation*}
        T_\alpha\delta_{\beta w}=
        \begin{cases}
        \delta_w, & \beta=\alpha,\\
        0, & \beta\ne\alpha,
        \end{cases}
        \qquad
        T_\alpha\delta_\varnothing=0.
\end{equation*}
Thus $S_\alpha$ prefixes a word by the letter $\alpha$, while $T_\alpha$ removes an initial $\alpha$ when it is present and sends the vector to zero otherwise.

The operator $S_\alpha$ is an isometry on $\ell_1(W_\kappa)$. The operator $T_\alpha$ is contractive, since it keeps only the coordinates indexed by words beginning with $\alpha$ and then relabels them. Since $T_\alpha S_\alpha=I$, both operators have norm one:
\begin{equation*}
        \|S_\alpha\|=\|T_\alpha\|=1.
\end{equation*}
Moreover, for $\alpha,\beta<\kappa$,
\begin{equation*}
        T_\alpha S_\beta=\delta_{\alpha\beta}I.
\end{equation*}
Indeed, if $\alpha=\beta$, then $S_\beta$ adds the first letter $\beta$ and $T_\alpha$ removes it; if $\alpha\ne\beta$, then $T_\alpha$ kills the vector.

By the universal property of $L_\kappa$, these operators induce a unital homomorphism
\begin{equation*}
        \rho\colon L_\kappa\to\PtwoB(Y_\kappa),
        \qquad
        \rho(s_\alpha)=S_\alpha,
        \qquad
        \rho(t_\alpha)=T_\alpha.
\end{equation*}
Since $\rho(1)=I\ne0$, the kernel of $\rho$ is a proper two-sided ideal of $L_\kappa$. By \Cref{p2-first:lem:leavitt-simple}, this kernel must be zero. Hence $\rho$ is faithful.

Let
\begin{equation*}
        B_\kappa=\overline{\rho(L_\kappa)}\subset\PtwoB(Y_\kappa)
\end{equation*}
be the norm-closed unital Banach algebra generated by $\rho(L_\kappa)$.

We shall use the standard Zorn-lemma fact that every unital Banach algebra has maximal proper closed two-sided ideals; for completeness we recall the proof.

\begin{lemma}
\label{p2-first:lem:maximal-ideal}
Let $A$ be a unital Banach algebra. Then $A$ has a maximal proper closed two-sided ideal.
\end{lemma}

\begin{proof}
Partially order the proper closed two-sided ideals of $A$ by inclusion. Let $(I_j)_{j\in J}$ be a chain of proper closed two-sided ideals, and set
\begin{equation*}
        I=\overline{\bigcup_{j\in J} I_j}.
\end{equation*}
Then $I$ is a closed two-sided ideal of $A$. We claim that $I$ is proper. If $1\in I$, then there is some $x\in\bigcup_{j\in J}I_j$ such that
\begin{equation*}
        \|1-x\|<1.
\end{equation*}
But then $x$ is invertible, so the ideal containing $x$ must be all of $A$, contradicting the properness of the ideals in the chain. Hence $I$ is proper. Zorn's lemma gives a maximal proper closed two-sided ideal.
\end{proof}

Applying Lemma~\ref{p2-first:lem:maximal-ideal} to $A=B_\kappa$, choose a maximal proper closed two-sided ideal $M$ of $B_\kappa$ and define
\begin{equation*}
        A_\kappa=B_\kappa/M.
\end{equation*}
Let
\begin{equation*}
        q \colon B_\kappa\to A_\kappa
\end{equation*}
be the quotient map. Then $A_\kappa$ is a unital Banach algebra. Moreover, it has no nonzero proper closed two-sided ideals. Indeed, if $J$ is a closed two-sided ideal of $A_\kappa$, then $q^{-1}(J)$ is a closed two-sided ideal of $B_\kappa$ containing $M$. By maximality of $M$, we have $q^{-1}(J)=M$ or $q^{-1}(J)=B_\kappa$, and hence $J=\{0\}$ or $J=A_\kappa$.

We now verify that the maximal ideal chosen above does not meet the embedded Leavitt algebra nontrivially, so the Leavitt algebra survives faithfully in the quotient.

\begin{lemma}
\label{p2-first:lem:core-survives}
We have
\begin{equation*}
        M\cap\rho(L_\kappa)=\{0\}.
\end{equation*}
Consequently, $q\circ\rho:L_\kappa\to A_\kappa$ is injective.
\end{lemma}

\begin{proof}
Suppose that $0\ne a\in M\cap\rho(L_\kappa)$. Since $\rho$ is faithful and $L_\kappa$ is algebraically simple, the algebra $\rho(L_\kappa)$ is algebraically simple. Therefore the algebraic two-sided ideal generated by $a$ inside $\rho(L_\kappa)$ is all of $\rho(L_\kappa)$. Hence there exist $x_1,\ldots,x_n,y_1,\ldots,y_n\in\rho(L_\kappa)$ such that
\begin{equation*}
        1=\sum_{r=1}^{n}x_ray_r.
\end{equation*}
The right-hand side belongs to $M$, because $a\in M$ and $M$ is a two-sided ideal of $B_\kappa$. Thus $1\in M$, contradicting the properness of $M$. Hence $M\cap\rho(L_\kappa)=\{0\}$. Finally, if $z\in\ker(q\circ\rho)$, then $\rho(z)\in M\cap\rho(L_\kappa)=\{0\}$. Since $\rho$ is faithful, $z=0$. Thus $q\circ\rho$ is injective.
\end{proof}

\subsection{Density of the quotient algebra}
\label{p2-first:sec:density}

We now prove that passing to the quotient has not changed the intended size of the algebra. The next lemma computes the density character of $A_\kappa$.

\begin{lemma}
\label{p2-first:lem:density}
The Banach algebra $A_\kappa$ has density character
\begin{equation*}
        \Ptwodens(A_\kappa)=\kappa=\PtwoC^+.
\end{equation*}
\end{lemma}

\begin{proof}
For $\alpha<\kappa$, put $p_\alpha=s_\alpha t_\alpha\in L_\kappa$, and use the same notation for its image in $A_\kappa$. By \Cref{p2-first:lem:core-survives}, each $p_\alpha$ is nonzero in $A_\kappa$. The relations give pairwise orthogonal idempotents:
\begin{equation*}
        p_\alpha p_\beta=\delta_{\alpha\beta}p_\alpha.
\end{equation*}
If $p$ and $q$ are nonzero orthogonal idempotents in a normed algebra, then $\|p-q\|\geq1$, because $p(p-q)=p$. Thus $A_\kappa$ contains a $1$-separated set of cardinality $\kappa$, and hence $\Ptwodens(A_\kappa)\geq\kappa$.

Conversely, $L_\kappa$ has cardinality at most $\kappa$: its elements are finite complex linear combinations of finite words in $\kappa$ many generators, and $\kappa\geq|\C|$. Since $B_\kappa$ is the norm closure of $\rho(L_\kappa)$, we have $\Ptwodens(B_\kappa)\leq\kappa$, and therefore $\Ptwodens(A_\kappa)\leq\kappa$.
\end{proof}

We shall use the following elementary fact about the Calkin algebra of a separable Banach space.

\begin{lemma}
\label{p2-first:lem:separable-density}
If $X$ is separable, then
\begin{equation*}
        \Ptwodens(\PtwoQ(X))\leq\PtwoC.
\end{equation*}
\end{lemma}

\begin{proof}
A separable Banach space has cardinality at most $\PtwoC$. Fix a countable dense subset $D\subset X$. Every bounded operator $T\in\PtwoB(X)$ is determined by its restriction to $D$, since $T$ is continuous. Therefore
\begin{equation*}
        |\PtwoB(X)|\leq |X|^{|D|}\leq \PtwoC^{\aleph_0}=\PtwoC.
\end{equation*}
It follows that $\Ptwodens(\PtwoQ(X))\leq|\PtwoQ(X)|\leq\PtwoC$.
\end{proof}

\subsection{The matrix-unit obstruction}
\label{p2-first:sec:matrix-obstruction}

We isolate the obstruction which will rule out representing the quotient algebra on a nonseparable Calkin algebra. In a topologically simple Calkin algebra, a uniformly controlled system of matrix units would force an impossible separable reduction.

\begin{lemma}[The $\ell_1$ matrix-unit obstruction]
\label{p2-first:lem:matrix-obstruction}
Let $X$ be a nonseparable Banach space. Suppose that $\PtwoQ(X)$ has no nonzero proper closed two-sided ideals. Then $\PtwoQ(X)$ cannot contain elements $(e_{ij})_{i,j\in\N}$ satisfying the following conditions:
\begin{enumerate}[label=\textup{(\roman*)}]
\item\label{p2-first:matrix-obstruction:matrix-units} $e_{ij}e_{kl}=\delta_{jk}e_{il}$ for all $i,j,k,l\in\N$;
\item\label{p2-first:matrix-obstruction:nonzero-corner} $e_{11}\ne0$;
\item\label{p2-first:matrix-obstruction:row-bound} there is $C_r<\infty$ such that, for every finite $F\subset\N$ and every scalar family $(a_j)_{j\in F}$,
\begin{equation*}
        \left\|\sum_{j\in F}a_je_{1j}\right\|\leq C_r\sup_{j\in F}|a_j|;
\end{equation*}
\item\label{p2-first:matrix-obstruction:column-bound} there is $C_c<\infty$ such that $\sup_{i \in \N} \|e_{i1}\|\leq C_c$.
\end{enumerate}
\end{lemma}

\begin{proof}
We proceed by contradiction, and suppose that such a family $(e_{ij})_{i,j\in\N}$ exists. Let $q \colon \PtwoB(X)\to\PtwoQ(X)$ be the quotient map. Define
\begin{equation*}
        \PtwoS(X)=\{T\in\PtwoB(X):\overline{T(X)}\text{ is separable}\}.
\end{equation*}
This is a closed two-sided ideal of $\PtwoB(X)$, it contains $\PtwoK(X)$, and it is proper because $X$ is nonseparable. Hence $\PtwoS(X)/\PtwoK(X)$ is a closed ideal of $\PtwoQ(X)$. Since $\PtwoQ(X)$ is topologically simple, this ideal is either $0$ or all of $\PtwoQ(X)$. It is not all of $\PtwoQ(X)$, because its inverse image is $\PtwoS(X)\ne\PtwoB(X)$. Therefore
\begin{equation}
\label{p2-first:eq:separable-range-compact}
        \PtwoS(X)=\PtwoK(X).
\end{equation}

Choose arbitrary lifts $E_{ij}\in\PtwoB(X)$ with $q(E_{ij})=e_{ij}$. By \ref{p2-first:matrix-obstruction:matrix-units},
\begin{equation*}
        D_{ij,kl}=E_{ij}E_{kl}-\delta_{jk}E_{il}\in\PtwoK(X)
\end{equation*}
for all $i,j,k,l\in\N$. For each finite $F\subset\N$ and each rational complex tuple $a=(a_j)_{j\in F}\in(\mathbb Q+i\mathbb Q)^F$, set
\begin{equation*}
        U_{F,a}=\sum_{j\in F}a_jE_{1j}.
\end{equation*}
By \ref{p2-first:matrix-obstruction:row-bound} and the definition of the quotient norm, choose $C_{F,a}\in\PtwoK(X)$ such that
\begin{equation}
\label{p2-first:eq:row-correction}
        \|U_{F,a}+C_{F,a}\|\leq(C_r+1)\sup_{j\in F}|a_j|.
\end{equation}
Similarly, by \ref{p2-first:matrix-obstruction:column-bound}, for each $i\in\N$ choose $H_i\in\PtwoK(X)$ such that
\begin{equation}
\label{p2-first:eq:column-correction}
        \|E_{i1}+H_i\|\leq C_c+1.
\end{equation}

The compact operators
\begin{equation*}
        D_{ij,kl}\ (i,j,k,l\in\N),\qquad
        C_{F,a}\ (F\subset\N\text{ finite},\ a\in(\mathbb Q+i\mathbb Q)^F),\qquad
        H_i\ (i\in\N)
\end{equation*}
form a countable family. By \Cref{p2-first:lem:absorb}, applied to the countable family $(E_{ij})_{i,j \in \N}$ and to this countable family of compact operators, there is a separable closed subspace $N\subset X$ such that $E_{ij}(N)\subset N$ for all $i,j$, and such that the ranges of all compact errors just listed are contained in $N$.

Let $Z=X/N$, and let $\pi \colon X\to Z$ be the quotient map. Since each $E_{ij}$ leaves $N$ invariant, it induces an operator $\overline E_{ij}\in\PtwoB(Z)$ defined by
\begin{equation*}
        \overline E_{ij}\pi x=\pi E_{ij}x
        \qquad(x\in X).
\end{equation*}
Since $D_{ij,kl}(X)\subset N$, the induced operators satisfy exact matrix-unit relations:
\begin{equation}
\label{p2-first:eq:exact-matrix-units}
        \overline E_{ij}\overline E_{kl}=\delta_{jk}\overline E_{il}.
\end{equation}
By \Cref{p2-first:lem:quotient-correction}, \eqref{p2-first:eq:row-correction} implies, first for rational scalar families and then by density, that
\begin{equation}
\label{p2-first:eq:induced-row-bound}
        \left\|\sum_{j\in F}a_j\overline E_{1j}\right\|\leq(C_r+1)\sup_{j\in F}|a_j|
\end{equation}
for every finite $F\subset\N$ and every scalar family $(a_j)_{j\in F}$. Likewise, \eqref{p2-first:eq:column-correction} gives
\begin{equation}
\label{p2-first:eq:induced-column-bound}
        \sup_{i \in \N} \|\overline E_{i1}\|\leq C_c+1.
\end{equation}

We next show that $\overline E_{11}\ne0$. If $\overline E_{11}=0$, then $E_{11}(X)\subset N$, so $E_{11}$ has separable range. By \eqref{p2-first:eq:separable-range-compact}, $E_{11}$ is compact, and hence $e_{11}=q(E_{11})=0$, contradicting \ref{p2-first:matrix-obstruction:nonzero-corner}.

Choose $y\in Z$ such that $\overline E_{11}y\ne0$, and put $z_0=\overline E_{11}y$. Then $z_0\ne0$ and $\overline E_{11}z_0=z_0$, by \eqref{p2-first:eq:exact-matrix-units}. Choose $\varphi\in Z^*$ with $\varphi(z_0)=1$.

We claim that, if we define
\begin{equation*}
        \Theta x=\bigl(\varphi(\overline E_{1j}\pi x)\bigr)_{j=1}^{\infty},
\end{equation*}
then $\Theta$ is a well-defined bounded operator from $X$ to $\ell_1$. Indeed, for $x\in X$ and $n\in\N$, \eqref{p2-first:eq:induced-row-bound} gives
\begin{equation*}
\begin{aligned}
        \sum_{j=1}^{n}|\varphi(\overline E_{1j}\pi x)|
        &=\sup_{|a_j|\leq1}\left|\varphi\left(\sum_{j=1}^{n}a_j\overline E_{1j}\pi x\right)\right|\\
        &\leq(C_r+1)\|\varphi\|\,\|\pi x\|
        \leq(C_r+1)\|\varphi\|\,\|x\|.
\end{aligned}
\end{equation*}
Taking the supremum over $n$ shows that $\Theta x\in\ell_1$ and that
\begin{equation*}
        \|\Theta x\|_{\ell_1}\leq(C_r+1)\|\varphi\|\,\|x\|.
\end{equation*}
Thus $\Theta\colon X\to\ell_1$ is bounded, and linearity is trivial.

For $i\in\N$, put $z_i=\overline E_{i1}z_0$. By \eqref{p2-first:eq:induced-column-bound}, the sequence $(z_i)_{i \in \N}$ is bounded. Choose lifts $x_i\in X$ with $\pi x_i=z_i$ and $\|x_i\|\leq\|z_i\|+1$. Then $(x_i)_{i \in \N}$ is bounded. For $i,k\in\N$, \eqref{p2-first:eq:exact-matrix-units} gives
\begin{equation*}
        \varphi(\overline E_{1k}z_i)=\varphi(\overline E_{1k}\overline E_{i1}z_0)=\delta_{ki}\varphi(\overline E_{11}z_0) = \delta_{ki} \varphi(z_0) =\delta_{ki}.
\end{equation*}
Hence $\Theta x_i$ is the $i$-th unit vector of $\ell_1$. Since $(x_i)_{i \in \N}$ is bounded and the unit vector basis of $\ell_1$ has no norm-convergent subsequence, $\Theta$ is not compact.

Choose a bounded sequence $(w_i)_{i=1}^{\infty}\subset X$ with no norm-convergent subsequence, and define
\begin{equation*}
        R:\ell_1\to X,\qquad R((a_i)_{i \in \N})=\sum_{i=1}^{\infty}a_iw_i.
\end{equation*}
Then $R$ is bounded. The operator $T=R\Theta$ has separable range, since $T(X)\subset\overline{\spann}\{w_i:i\in\N\}$. However, $Tx_i=w_i$ for every $i$, so $T$ is not compact. This contradicts \eqref{p2-first:eq:separable-range-compact}, which finishes the proof.
\end{proof}

It remains to connect the concrete algebra $A_\kappa$ with the abstract obstruction above. We do this by extracting a countable system of matrix units from the Leavitt generators.

\begin{lemma}
\label{p2-first:lem:Akappa-has-matrix-units}
The algebra $A_\kappa$ contains elements $(e_{ij})_{i,j\in\N}$ satisfying the hypotheses of \Cref{p2-first:lem:matrix-obstruction}.
\end{lemma}

\begin{proof}
Choose distinct $\alpha_1,\alpha_2,\ldots\in\kappa$, viewed as loop labels in the one-vertex graph defining $L_\kappa$, and define
\begin{equation*}
        e_{ij}=q\rho(s_{\alpha_i}t_{\alpha_j})\in A_\kappa
        \qquad(i,j\in\N).
\end{equation*}

We verify the four hypotheses of \Cref{p2-first:lem:matrix-obstruction}. Since $q\circ\rho$ is injective by \Cref{p2-first:lem:core-survives}, we have $e_{11}\ne0$, so \ref{p2-first:matrix-obstruction:nonzero-corner} holds. The Leavitt relations \eqref{eq:levit-relations} give
\begin{equation*}
        e_{ij}e_{kl}=\delta_{jk}e_{il}
        \qquad(i,j,k,l\in\N),
\end{equation*}
so \ref{p2-first:matrix-obstruction:matrix-units} holds.

We next verify \Cref{p2-first:matrix-obstruction:row-bound}. Fix a finite set $F\subset\N$ and scalars $(a_j)_{j\in F}$. Consider the concrete operator
\begin{equation*}
        U=\sum_{j\in F}a_jS_{\alpha_1}T_{\alpha_j}\in\PtwoB(Y_\kappa).
\end{equation*}
Since $\rho(s_{\alpha_1}t_{\alpha_j})=S_{\alpha_1}T_{\alpha_j}$ and $e_{1j}=q\rho(s_{\alpha_1}t_{\alpha_j})$, we have
\begin{equation*}
        q(U)=\sum_{j\in F}a_je_{1j}.
\end{equation*}
Thus, it suffices to estimate $\|U\|$.
If $x=(c_w)_{w\in W_\kappa}\in\ell_1(W_\kappa)$, then
\begin{equation*}
        \|Ux\|_1\leq \sup_{j\in F}|a_j|\sum_{w\in W_\kappa}\sum_{j\in F}|c_{\alpha_jw}|\leq \sup_{j\in F}|a_j|\,\|x\|_1.
\end{equation*}
Hence
\begin{equation*}
        \left\|\sum_{j\in F}a_je_{1j}\right\|\leq \sup_{j\in F}|a_j|.
\end{equation*}
This proves \ref{p2-first:matrix-obstruction:row-bound}, with $C_r=1$.

Finally, since $\|S_{\alpha_i}T_{\alpha_1}\|\leq1$ for every $i\in\N$ and $e_{i1}=q\rho(s_{\alpha_i}t_{\alpha_1})$, we have
\begin{equation*}
        \sup_{i\in\N}\|e_{i1}\|\leq1.
\end{equation*}
This proves \ref{p2-first:matrix-obstruction:column-bound}, with $C_c=1$. Thus the elements $(e_{ij})_{i,j\in\N}$ satisfy the hypotheses of \Cref{p2-first:lem:matrix-obstruction}.
\end{proof}

\subsection{Conclusion of the first proof}
\label{p2-first:sec:first-conclusion}

We can now finish the first construction: density rules out separable Calkin algebras, while the matrix-unit obstruction rules out the nonseparable case.

\begin{proof}[First proof of \Cref{p2:thm:main}]
Let $A=A_\kappa$. Assume, toward a contradiction, that $A_\kappa\cong\PtwoQ(X)$ for some Banach space $X$.

If $X$ is separable, then \Cref{p2-first:lem:separable-density} and \Cref{p2-first:lem:density} give
\begin{equation*}
        \Ptwodens(\PtwoQ(X))\leq\PtwoC<\PtwoC^+=\Ptwodens(A_\kappa),
\end{equation*}
which is a contradiction.

Thus $X$ is nonseparable. Since $A_\kappa$ is topologically simple, so is $\PtwoQ(X)$. Let $\Phi\colon A_\kappa\to\PtwoQ(X)$ be a Banach-algebra isomorphism. By \Cref{p2-first:lem:Akappa-has-matrix-units}, choose elements $(e_{ij})_{i,j\in\N}\subset A_\kappa$ satisfying the hypotheses of \Cref{p2-first:lem:matrix-obstruction}. The elements $\Phi(e_{ij})$ are matrix units in $\PtwoQ(X)$, and $\Phi(e_{11})\ne0$. Moreover, the row and column estimates are preserved up to the factor $\|\Phi\|$, that is
\begin{equation*}
        \left\|\sum_{j\in F}a_j\Phi(e_{1j})\right\|\leq \|\Phi\|\sup_{j\in F}|a_j|,
\end{equation*}
and
\begin{equation*}
        \sup_{i \in \N} \|\Phi(e_{i1})\|\leq\|\Phi\|.
\end{equation*}
This contradicts \Cref{p2-first:lem:matrix-obstruction} and finishes the proof.
\end{proof}

\section{Second proof}

\subsection{Proof strategy and organisation}

The second proof is organised as follows. In \Cref{p2-second:sec:shift-algebra}, we avoid the Leavitt-path algebra core and instead construct a very large shift algebra on a space of the form $c_0(W)$. In \Cref{p2-second:sec:density-trap}, the density of this algebra is chosen so large that, if it were isomorphic to a Calkin algebra $\PtwoQ(X)$, then the Banach space $X$ would have to have density strictly larger than the continuum.

The next step is to exploit the shift relations inside the alleged Calkin algebra. In \Cref{p2-second:sec:compact-killing}, we remove countably many compact errors so that the relevant relations can be witnessed on a carefully chosen subspace. In \Cref{p2-second:sec:column-forces-c0}, these relations are used to force $X$ to contain a copy of $c_0$.

Finally, \Cref{p2-second:sec:c0-destroys-simplicity} shows that a nonseparable Banach space containing $c_0$ has a nonsimple Calkin algebra: the separable-range operators give a nonzero proper closed ideal in $\PtwoQ(X)$. This contradicts the topological simplicity of the algebra constructed in \Cref{p2-second:sec:shift-algebra}. The contradiction is assembled in \Cref{p2-second:sec:second-conclusion}.

\subsection{The shift algebra}
\label{p2-second:sec:shift-algebra}

Let $(\beth_n)_{n=0}^{\infty}$ be the Beth sequence, defined by
\begin{equation*}
        \beth_0=\aleph_0, \qquad \beth_{n+1}=2^{\beth_n}.
\end{equation*}
Set
\begin{equation*}
        \lambda=\beth_\omega=\sup_{n<\omega}\beth_n.
\end{equation*}
Then $\lambda$ is a strong-limit cardinal, and in particular
\begin{equation*}
        \lambda>2^{\aleph_0}.
\end{equation*}

Let $\Gamma$ be a set of cardinality $\lambda$. We write
\begin{equation*}
        W=\Gamma^{<\omega}
\end{equation*}
for the set of finite words over the alphabet $\Gamma$, including the empty word $\varnothing$. Put
\begin{equation*}
        E=c_0(W),
\end{equation*}
and denote by $e_w$, $w\in W$, the canonical unit vector basis of $c_0(W)$.

For each $\alpha\in\Gamma$, define two operators $S_\alpha,T_\alpha\in\PtwoB(E)$ as follows. The operator $S_\alpha$ prefixes a word by $\alpha$:
\begin{equation*}
        S_\alpha e_w=e_{\alpha w} \qquad(w\in W).
\end{equation*}
The operator $T_\alpha$ removes an initial $\alpha$ when it is present, and sends the vector to zero otherwise:
\begin{equation*}
        T_\alpha e_{\beta w}=
        \begin{cases}
        e_w, & \beta=\alpha,\\
        0, & \beta\ne\alpha,
        \end{cases}
        \qquad
        T_\alpha e_\varnothing=0.
\end{equation*}
Thus $S_\alpha$ is an isometric embedding of $c_0(W)$ onto the coordinate subspace supported on the cylinder $\alpha W$, while $T_\alpha$ is the corresponding contractive retraction. Hence
\begin{equation*}
        \|S_\alpha\|=\|T_\alpha\|=1.
\end{equation*}
Moreover,
\begin{equation*}
        T_\alpha S_\beta=\delta_{\alpha\beta}I_E \qquad(\alpha,\beta\in\Gamma).
\end{equation*}

We shall use the following simple consequence of the $c_0$ geometry. If $F\subset\Gamma$ is finite and $(a_\alpha)_{\alpha\in F}\subset\C$, then
\begin{equation}
\label{p2-second:eq:shift-c0-estimate}
        \left\|\sum_{\alpha\in F}a_\alpha S_\alpha\right\|=\max_{\alpha\in F}|a_\alpha|.
\end{equation}
Indeed, the vectors $S_\alpha x$, $\alpha\in F$, have pairwise disjoint supports, since they are supported on the disjoint sets $\alpha W$. Hence, whenever $\|x\|\leq1$,
\begin{equation*}
        \left\|\sum_{\alpha\in F}a_\alpha S_\alpha x\right\|\leq \max_{\alpha\in F}|a_\alpha|.
\end{equation*}
This gives one inequality. For the reverse inequality, choose $\alpha_0\in F$ such that
\begin{equation*}
        |a_{\alpha_0}|=\max_{\alpha\in F}|a_\alpha|.
\end{equation*}
Then, for any $w\in W$,
\begin{equation*}
        \left\|\sum_{\alpha\in F}a_\alpha S_\alpha e_w\right\|\geq |a_{\alpha_0}|.
\end{equation*}
Therefore equality holds in \eqref{p2-second:eq:shift-c0-estimate}.

We now define the Banach algebra, which will be the counterexample. Let $B_\lambda$ be the closed unital subalgebra of $\PtwoB(E)$ generated by
\begin{equation*}
        I_E,\qquad S_\alpha,\qquad T_\alpha \qquad(\alpha\in\Gamma).
\end{equation*}
Thus $B_\lambda$ is the norm-closed algebra generated by the identity and by the forward and backward shifts attached to the alphabet $\Gamma$.

By \Cref{p2-first:lem:maximal-ideal}, choose a maximal proper closed two-sided ideal $M$ of $B_\lambda$, and define
\begin{equation}
\label{p2-second:eq:definition-of-A-lambda}
        A_\lambda=B_\lambda/M.
\end{equation}
This quotient is the Banach algebra which will serve as the counterexample. By the maximality of $M$, the algebra $A_\lambda$ is topologically simple.

\begin{lemma}[Density of the quotient algebra]
\label{p2-second:lem:density-A}
The Banach algebra $A_\lambda$ has density character
\begin{equation*}
        \Ptwodens(A_\lambda)=\lambda.
\end{equation*}
\end{lemma}

\begin{proof}
The algebra $B_\lambda$ is generated, as a closed algebra, by $\lambda$ many elements. Noncommutative polynomials with coefficients in $\mathbb Q+i\mathbb Q$ and variables from finite subsets of the generators form a dense subset of cardinality at most $\lambda$. Hence $\Ptwodens(B_\lambda)\leq\lambda$, and therefore $\Ptwodens(A_\lambda)\leq\lambda$.

For the reverse inequality, write
\begin{equation*}
        \sigma_\alpha=S_\alpha+M,\qquad \tau_\alpha=T_\alpha+M
\end{equation*}
inside $A_\lambda$. If $\alpha\ne\beta$, then
\begin{equation*}
        \tau_\alpha(\sigma_\alpha-\sigma_\beta)=1_{A_\lambda}.
\end{equation*}
Also $\|\tau_\alpha\|\leq1$ and $\|1_{A_\lambda}\|=1$. Therefore
\begin{equation*}
        1\leq\|\sigma_\alpha-\sigma_\beta\|.
\end{equation*}
Thus $A_\lambda$ contains a $1$-separated set of cardinality $\lambda$, so $\Ptwodens(A_\lambda)\geq\lambda$.
\end{proof}

\subsection{A density trap}
\label{p2-second:sec:density-trap}

We next record the cardinal estimate which explains the choice of $\lambda=\beth_\omega$. A Banach space of density $\kappa$ has at most $2^\kappa$ many bounded operators up to density, and hence its Calkin algebra also has density at most $2^\kappa$. Since $\lambda$ is a strong-limit cardinal and $\Ptwodens(A_\lambda)=\lambda$, this forces any Banach space $X$ with $\PtwoQ(X)\cong A_\lambda$ to have density at least $\lambda$.

\begin{lemma}
\label{p2-second:lem:density-trap}
Let $X$ be a Banach space. If $\Ptwodens(X)=\kappa$ is infinite, then
\begin{equation*}
        \Ptwodens(\PtwoQ(X))\leq2^\kappa.
\end{equation*}
Consequently, if $\PtwoQ(X)$ is Banach-algebra isomorphic to $A_\lambda$, then
\begin{equation*}
        \Ptwodens(X)\geq\lambda>2^{\aleph_0}.
\end{equation*}
\end{lemma}

\begin{proof}
A metric space of density $\kappa$ has cardinality at most $\kappa^{\aleph_0}$. Thus $|X|\leq\kappa^{\aleph_0}$. Every operator in $\PtwoB(X)$ is determined by its values on a fixed dense subset of cardinality $\kappa$, and so
\begin{equation*}
        |\PtwoB(X)|\leq |X|^\kappa\leq(\kappa^{\aleph_0})^\kappa=\kappa^\kappa=2^\kappa.
\end{equation*}
Therefore $\Ptwodens(\PtwoQ(X))\leq2^\kappa$.

Now suppose that $\PtwoQ(X)$ is Banach-algebra isomorphic to $A_\lambda$. Since density character is preserved by Banach-space isomorphism, \Cref{p2-second:lem:density-A} gives
\begin{equation*}
        \Ptwodens(\PtwoQ(X))=\Ptwodens(A_\lambda)=\lambda.
\end{equation*}
On the other hand, the estimate just proved gives
\begin{equation*}
        \lambda=\Ptwodens(\PtwoQ(X))\leq 2^\kappa.
\end{equation*}
We claim that this forces $\kappa\geq\lambda$. Indeed, if $\kappa<\lambda$, then the strong-limit property of $\lambda=\beth_\omega$ gives
\begin{equation*}
        2^\kappa<\lambda.
\end{equation*}
This contradicts $\lambda\leq2^\kappa$. Therefore $\kappa\geq\lambda$, as required.
\end{proof}

\subsection{Killing countably many compact errors}
\label{p2-second:sec:compact-killing}

We shall next need a simple consequence of nonseparability at very large density. Compact operators have separable behaviour on bounded sets: the image of the unit ball under a compact operator is norm-compact, hence separable. Therefore, a countable family of compact operators can only see a separable amount of information. If the ambient space has density strictly larger than the continuum, we can choose a unit vector on which all of these compact operators are simultaneously small.

\begin{lemma}
\label{p2-second:lem:compact-killing}
Let $X$ be a Banach space with $\Ptwodens(X)>2^{\aleph_0}$. Let $(K_n)_{n=1}^{\infty}$ be a countable family of compact operators on $X$, and let $(\eps_n)_{n=1}^{\infty}$ be a sequence of positive real numbers. Then there is $x\in X$ such that $\|x\|=1$ and
\begin{equation*}
        \|K_nx\|\leq\eps_n \qquad(n\in\N).
\end{equation*}
\end{lemma}

\begin{proof}
Fix a compact operator $K\in\PtwoK(X)$ and $\eps>0$. By Schauder's theorem, $K^*:X^*\to X^*$ is compact. Hence $K^*(B_{X^*})$ is norm-totally bounded. Choose $g_1,\ldots,g_m\in B_{X^*}$ such that for every $f\in B_{X^*}$ there is $1\leq r\leq m$ with
\begin{equation*}
        \|K^*f-K^*g_r\|<\eps.
\end{equation*}
Set $Y=\bigcap_{r=1}^{m}\ker(K^*g_r)$. Then $Y$ has finite codimension, and for $y\in Y$,
\begin{equation*}
        \|Ky\|=\sup_{f\in B_{X^*}}|K^*f(y)|\leq\eps\|y\|.
\end{equation*}
Thus for each $K_n$ there is a finite-codimensional closed subspace $Y_n\subset X$ such that $\|K_n|_{Y_n}\|\leq\eps_n$.

Let $Y_\infty=\bigcap_{n=1}^{\infty}Y_n$. If $Y_\infty=\{0\}$, then the map
\begin{equation*}
        X\to\prod_{n=1}^{\infty}X/Y_n,\qquad x\mapsto(x+Y_n)_n
\end{equation*}
is injective. Each quotient $X/Y_n$ is finite-dimensional over $\C$, so the product has cardinality at most $(2^{\aleph_0})^{\aleph_0}=2^{\aleph_0}$. This contradicts $\Ptwodens(X)>2^{\aleph_0}$. Hence $Y_\infty\ne\{0\}$. Choose $0\ne x\in Y_\infty$, and normalise it; this gives the result.
\end{proof}

\subsection{The column system forces $c_0$}
\label{p2-second:sec:column-forces-c0}

We now turn the shift relations in the Calkin algebra back into geometry inside the Banach space $X$. The elements $s_j$ should be thought of as the first column of a system of matrix units, while the elements $t_j$ are left inverses modulo compact operators. The estimate on the sums of the $s_j$ says that this column behaves like the unit vector basis of $c_0$ in the quotient algebra. The point of the next lemma is that, when the density of $X$ is large enough, the compact errors can be made simultaneously small; the quotient-level $c_0$ behaviour then lifts to an actual copy of $c_0$ inside $X$.

\begin{lemma}
\label{p2-second:lem:column-forces-c0}
Let $X$ be a Banach space with $\Ptwodens(X)>2^{\aleph_0}$. Suppose that there are elements $(s_j)_{j=1}^{\infty}$ and $(t_j)_{j=1}^{\infty}$ in $\PtwoQ(X)$ satisfying the following conditions.

\begin{enumerate}[label=\textup{(\roman*)}]
\item\label{p2-second:item:column-relations} We have
\begin{equation*}
        t_is_j=\delta_{ij}1_{\PtwoQ(X)} \qquad(i,j\in\N).
\end{equation*}

\item\label{p2-second:item:column-c0-upper} There is a constant $C<\infty$ such that
\begin{equation*}
        \left\|\sum_{j=1}^{m}a_js_j\right\|\leq C\max_{1\leq j\leq m}|a_j|
\end{equation*}
for every $m\in\N$ and every choice of scalars $a_1,\ldots,a_m\in\C$.

\item\label{p2-second:item:left-inverses-bounded} We have
\begin{equation*}
        \sup_{j\in\N}\|t_j\|<\infty.
\end{equation*}
\end{enumerate}

Then $X$ contains an isomorphic copy of $c_0$.
\end{lemma}

\begin{proof}
Choose lifts $S_j\in\PtwoB(X)$ of $s_j$. By \ref{p2-second:item:left-inverses-bounded}, we may choose lifts $T_j\in\PtwoB(X)$ of $t_j$ such that
\begin{equation*}
        \sup_{j\in\N}\|T_j\|<\infty.
\end{equation*}
Put
\begin{equation*}
        D=\max\left\{1,\sup_{j\in\N}\|T_j\|\right\}<\infty.
\end{equation*}
For $i,j\in\N$, define
\begin{equation*}
        R_{ij}=T_iS_j-\delta_{ij}I_X.
\end{equation*}
By \ref{p2-second:item:column-relations}, we have
\begin{equation*}
        R_{ij}\in\PtwoK(X) \qquad(i,j\in\N).
\end{equation*}

Let $\mathbb Q(i)=\mathbb Q+i\mathbb Q$. For every nonzero finitely supported family
\begin{equation*}
        a=(a_1,\ldots,a_m,0,0,\ldots)\in c_{00}(\mathbb Q(i)),
\end{equation*}
condition \ref{p2-second:item:column-c0-upper} gives
\begin{equation}
\label{p2-second:eq:quotient-column-bound}
        \left\|\sum_{j=1}^{m}a_js_j\right\|_{\PtwoQ(X)}
        \leq C\max_{1\leq j\leq m}|a_j|.
\end{equation}
Since $\sum_{j=1}^{m}a_jS_j$ is a lift of $\sum_{j=1}^{m}a_js_j$, the definition of the quotient norm and \eqref{p2-second:eq:quotient-column-bound} allow us to choose $K_a\in\PtwoK(X)$ such that
\begin{equation}
\label{p2-second:eq:Ka-choice}
        \left\|\sum_{j=1}^{m}a_jS_j+K_a\right\|
        \leq (C+1)\max_{1\leq j\leq m}|a_j|.
\end{equation}

Apply \Cref{p2-second:lem:compact-killing} to the countable family consisting of all operators $R_{ij}$, $i,j\in\N$, and all operators $K_a$, $0\ne a\in c_{00}(\mathbb Q(i))$.

For $i,j\in\N$, set
\begin{equation*}
        \varepsilon_{ij}=2^{-j-1}.
\end{equation*}
For $0\ne a=(a_1,\ldots,a_m,0,0,\ldots)\in c_{00}(\mathbb Q(i))$, set
\begin{equation*}
        \varepsilon_a=\max_{1\leq j\leq m}|a_j|.
\end{equation*}
The lemma gives a unit vector $x\in X$ such that
\begin{equation}
\label{p2-second:eq:Rij-small}
        \|R_{ij}x\|\leq \varepsilon_{ij}=2^{-j-1} \qquad(i,j\in\N),
\end{equation}
and
\begin{equation}
\label{p2-second:eq:Ka-small}
        \|K_ax\|\leq \varepsilon_a=\max_{1\leq j\leq m}|a_j|
\end{equation}
whenever $0\ne a=(a_1,\ldots,a_m,0,0,\ldots)\in c_{00}(\mathbb Q(i))$.

Set
\begin{equation*}
        x_j=S_jx \qquad(j\in\N).
\end{equation*}
We claim that $(x_j)_{j=1}^{\infty}$ is equivalent to the canonical basis of $c_0$.

Let
\begin{equation*}
        0\ne a=(a_1,\ldots,a_m,0,0,\ldots)\in c_{00}(\mathbb Q(i)).
\end{equation*}
By \eqref{p2-second:eq:Ka-choice} and \eqref{p2-second:eq:Ka-small},
\begin{equation*}
\begin{aligned}
        \left\|\sum_{j=1}^{m}a_jx_j\right\|
        &=\left\|\sum_{j=1}^{m}a_jS_jx\right\|\\
        &\leq\left\|\left(\sum_{j=1}^{m}a_jS_j+K_a\right)x\right\|+\|K_ax\|\\
        &\leq (C+2)\max_{1\leq j\leq m}|a_j|.
\end{aligned}
\end{equation*}
For the lower estimate, choose $i\in\{1,\ldots,m\}$ such that
\begin{equation*}
        |a_i|=\max_{1\leq j\leq m}|a_j|.
\end{equation*}
Then
\begin{equation*}
        T_i\left(\sum_{j=1}^{m}a_jx_j\right)=a_ix+\sum_{j=1}^{m}a_jR_{ij}x.
\end{equation*}
Using \eqref{p2-second:eq:Rij-small} and $\|x\|=1$, we get
\begin{equation*}
        \left\|T_i\left(\sum_{j=1}^{m}a_jx_j\right)\right\|\geq \max_{1\leq j\leq m}|a_j|-\max_{1\leq j\leq m}|a_j|\sum_{j=1}^{m}2^{-j-1}.
\end{equation*}
Since
\begin{equation*}
        \sum_{j=1}^{m}2^{-j-1}\leq \sum_{j=1}^{\infty}2^{-j-1}=\frac12,
\end{equation*}
we obtain
\begin{equation*}
        \left\|T_i\left(\sum_{j=1}^{m}a_jx_j\right)\right\|\geq \frac12\max_{1\leq j\leq m}|a_j|.
\end{equation*}
Since $\|T_i\|\leq D$, it follows that
\begin{equation*}
        \left\|\sum_{j=1}^{m}a_jx_j\right\|\geq \frac{1}{2D}\max_{1\leq j\leq m}|a_j|.
\end{equation*}

Since the estimates are continuous in the finitely many coefficients involved, they extend from $c_{00}(\mathbb Q(i))$ to all finitely supported complex scalar families. Hence, for every $m\in\N$ and every $a_1,\ldots,a_m\in\C$,
\begin{equation*}
        \frac{1}{2D}\max_{1\leq j\leq m}|a_j|
        \leq
        \left\|\sum_{j=1}^{m}a_jx_j\right\|
        \leq
        (C+2)\max_{1\leq j\leq m}|a_j|.
\end{equation*}
Define $U:c_{00}\to X$ by
\begin{equation*}
        Ue_j=x_j \qquad(j\in\N).
\end{equation*}
The upper estimate shows that $U$ is bounded for the $c_0$ norm, so it extends uniquely to a bounded operator $U:c_0\to X$. The lower estimate shows that $U$ is bounded below, hence injective with closed range. Therefore $U(c_0)$ is a subspace of $X$ isomorphic to $c_0$. This finishes the proof.
\end{proof}

\subsection{A nonseparable space containing $c_0$ has nonsimple Calkin algebra}
\label{p2-second:sec:c0-destroys-simplicity}

The previous subsection shows that the shift relations force a copy of $c_0$ inside any sufficiently large Banach space whose Calkin algebra realises the model algebra. We now explain why this is incompatible with topological simplicity of the Calkin algebra. The point is that, in a nonseparable Banach space, the separable-range operators form a proper closed operator ideal larger than the compact operators as soon as $X$ contains a copy of $c_0$. After quotienting by the compact operators, this gives a nonzero proper closed two-sided ideal in the Calkin algebra.

\begin{lemma}
\label{p2-second:lem:c0-destroys-simplicity}
Let $X$ be a nonseparable Banach space. If $X$ contains an isomorphic copy of $c_0$, then $\PtwoQ(X)$ has a nonzero proper closed two-sided ideal.
\end{lemma}

\begin{proof}
Let $i:c_0\hookrightarrow X$ be an isomorphic embedding. By the Josefson-Nissenzweig theorem, choose $(f_n)_{n=1}^{\infty}\subset X^*$ such that $\|f_n\|=1$ for every $n$ and $f_n(x)\to0$ for every $x\in X$. Define
\begin{equation*}
        J:X\to c_0,\qquad Jx=(f_n(x))_{n=1}^{\infty}.
\end{equation*}
This is bounded. It is not compact: compact subsets of $c_0$ have uniformly vanishing tails, whereas
\begin{equation*}
        \sup_{\|x\|\leq1}|(Jx)_n|=\|f_n\|=1
\end{equation*}
for every $n\in\N$. Hence $iJ:X\to X$ is a noncompact operator with separable range.

Let
\begin{equation*}
        \PtwoS(X)=\{T\in\PtwoB(X):\overline{T(X)}\text{ is separable}\}.
\end{equation*}
Then $\PtwoS(X)$ is a closed two-sided ideal of $\PtwoB(X)$, and $\PtwoK(X)\subset\PtwoS(X)$. Since $X$ is nonseparable, $I_X\notin\PtwoS(X)$, so $\PtwoS(X)\ne\PtwoB(X)$. Since $iJ\in\PtwoS(X)\setminus\PtwoK(X)$, we have
\begin{equation*}
        \PtwoK(X)\subsetneq\PtwoS(X)\subsetneq\PtwoB(X).
\end{equation*}
Thus $\PtwoS(X)/\PtwoK(X)$ is a nonzero proper closed two-sided ideal of $\PtwoQ(X)$.
\end{proof}

\subsection{Conclusion of the second proof}
\label{p2-second:sec:second-conclusion}

\begin{proof}[Second proof of \Cref{p2:thm:main}]
Let $A=A_\lambda$. By \Cref{p2-second:lem:density-A}, $\Ptwodens(A)=\lambda$, and by construction $A$ is topologically simple. Suppose, toward a contradiction, that there is a Banach-algebra isomorphism
\begin{equation*}
        \Phi:A\to\PtwoQ(X)
\end{equation*}
for some Banach space $X$. Then $\PtwoQ(X)$ is topologically simple. \Cref{p2-second:lem:density-trap} gives
\begin{equation*}
        \Ptwodens(X)\geq\lambda>2^{\aleph_0}.
\end{equation*}

Choose pairwise distinct elements $\alpha_j\in\Gamma$, $j\in\N$, and define
\begin{equation*}
        \sigma_j=S_{\alpha_j}+M, \qquad \tau_j=T_{\alpha_j}+M \qquad(j\in\N),
\end{equation*}
as elements of $A=B_\lambda/M$. Then, for all $i,j\in\N$,
\begin{equation*}
        \tau_i\sigma_j=\delta_{ij}1_A.
\end{equation*}
Moreover, if $m\in\N$ and $a_1,\ldots,a_m\in\C$, then \eqref{p2-second:eq:shift-c0-estimate}, followed by passage to the quotient, gives
\begin{equation*}
        \left\|\sum_{j=1}^{m}a_j\sigma_j\right\|
        \leq
        \left\|\sum_{j=1}^{m}a_jS_{\alpha_j}\right\|
        =
        \max_{1\leq j\leq m}|a_j|.
\end{equation*}
Also,
\begin{equation*}
        \sup_{j\in\N}\|\tau_j\|\leq1.
\end{equation*}

Set
\begin{equation*}
        s_j=\Phi(\sigma_j), \qquad t_j=\Phi(\tau_j) \qquad(j\in\N).
\end{equation*}
Then, for all $i,j\in\N$,
\begin{equation*}
        t_is_j=\delta_{ij}1_{\PtwoQ(X)}.
\end{equation*}
Furthermore, for every $m\in\N$ and every $a_1,\ldots,a_m\in\C$,
\begin{equation*}
        \left\|\sum_{j=1}^{m}a_js_j\right\|
        =
        \left\|\Phi\left(\sum_{j=1}^{m}a_j\sigma_j\right)\right\|
        \leq
        \|\Phi\|\max_{1\leq j\leq m}|a_j|.
\end{equation*}
Finally,
\begin{equation*}
        \sup_{j\in\N}\|t_j\|\leq \|\Phi\|.
\end{equation*}
\Cref{p2-second:lem:column-forces-c0} implies that $X$ contains an isomorphic copy of $c_0$. Since $\Ptwodens(X)>2^{\aleph_0}$, the space $X$ is nonseparable. Therefore \Cref{p2-second:lem:c0-destroys-simplicity} implies that $\PtwoQ(X)$ is not topologically simple. This contradicts the topological simplicity of $\PtwoQ(X)$, and completes the proof.
\end{proof}

\end{problempaperbody}

\clearpage
\providecommand{\PfourField}{\mathbb K}
\providecommand{\PfourF}{\mathbb F}
\providecommand{\Pfourran}{\operatorname{ran}}
\providecommand{\Pfourcodim}{\operatorname{codim}}

\problempaper{Problem 3. Strict cosingularity and adjoints}
\label{probpaper:separable-range-strict-cosingularity}

\begin{problemabstract}
Let $X$ and $Y$ be Banach spaces and let $T \colon X\to Y$ be bounded. We prove that, when $Y$ is separable, $T$ is strictly cosingular if and only if $T^*:Y^*\to X^*$ is strictly singular. 
\end{problemabstract}

\problemcontents

\begin{problempaperbody}

\section{Statement and notation}
\label{p4:sec:statement}

Throughout, we work over the scalar field $\mathbb K$, which is either $\R$ or $\C$. Unless explicitly stated otherwise, the infinite-dimensional subspaces used in the operator-theoretic definitions and theorem statements are norm closed.

If $E\subset X^*$ is a subspace, its \emph{preannihilator} in $X$ is
\begin{equation*}
        E^\perp=\{x\in X:e(x)=0\text{ for all }e\in E\}.
\end{equation*}
If $M\subset X$ is a subspace, its \emph{annihilator} in $X^*$ is
\begin{equation*}
        M^\perp=\{x^*\in X^*:x^*(m)=0\text{ for all }m\in M\}.
\end{equation*}
We write $\sigma(X^*,X)$ for the weak-star topology on $X^*$.

\begin{definition}
\label{p4:def:strictly-singular-cosingular}
Let $S \colon E\to F$ be a bounded operator between Banach spaces. We say that $S$ is \emph{strictly singular} if $S$ is not bounded below on any infinite-dimensional subspace of $E$.

Let $T \colon X\to Y$ be bounded. We say that $T$ is \emph{strictly cosingular} if, for every infinite-codimensional subspace $M\subset Y$, the operator
\begin{equation*}
        Q_M T \colon X\to Y/M
\end{equation*}
is not surjective, where $Q_M:Y\to Y/M$ denotes the quotient map.
\end{definition}

The terminology goes back to Pe{\l}czy\'nski \cite{Pelczynski1965}. We shall use the standard dual formulation of strict cosingularity: $T \colon X\to Y$ is strictly cosingular if and only if $T^*$ is not bounded below on any infinite-dimensional weak-star closed subspace of $Y^*$. This is recalled, for instance, in \cite[Section~3]{AndroulakisBeanland2008}. In particular,
\begin{equation*}
        T^*\text{ strictly singular}\quad\Longrightarrow\quad T\text{ strictly cosingular}.
\end{equation*}
The converse is false in general. The classical counterexample, due to Pe{\l}czy\'nski, is the canonical inclusion $c_0\hookrightarrow \ell_\infty$: it is strictly cosingular, while its adjoint is not strictly singular; see \cite{Pelczynski1965} and also \cite[Section~3]{AndroulakisBeanland2008}. The result below shows that the converse does hold when the range space is separable.

\begin{theorem}
\label{p4:thm:main}
    Let $X$ be a Banach space, let $Y$ be a separable Banach space, and let $T \colon X\to Y$ be bounded. Then $T$ is strictly cosingular if and only if $T^*$ is strictly singular.
\end{theorem}

\section{Proof strategy and organisation}
\label{p4:sec:strategy}

The proof is organised around the following idea. Strict cosingularity is naturally detected by the adjoint on weak-star closed subspaces of $Y^*$, while strict singularity asks about arbitrary infinite-dimensional subspaces of $Y^*$. Thus the difficult implication is to start with an arbitrary subspace on which $T^*$ is bounded below and manufacture a weak-star closed witness.

In \Cref{p4:sec:preliminaries}, we record two elementary reductions. First, surjectivity of an operator is equivalent to its adjoint being bounded below. Second, strict cosingularity of $T$ is equivalent to saying that $T^*$ is not bounded below on any infinite-dimensional weak-star closed subspace of $Y^*$.

In \Cref{p4:sec:preadjoint-extraction}, we prove the main technical lemma. It uses separability of $Y$ to diagonalise over a countable dense linear subspace of $Y$. The conclusion is that, after passing to an infinite-dimensional subspace $E_0\subset X^*$, a given operator $U \colon E\to Y^*$ becomes the restriction of an adjoint $A^*$, where $A \colon Y\to X/(E_0)^\perp$.

In \Cref{p4:sec:factorization}, we apply this lemma with $E=T^*(Z)$, where $T^*$ is bounded below on an infinite-dimensional subspace $Z\subset Y^*$. The extracted adjoint identity gives a quotient factorization
\begin{equation*}
        AT=q:X\to X/(E_0)^\perp.
\end{equation*}
Since $q$ is onto an infinite-dimensional quotient, this produces an infinite-codimensional quotient of $Y$ on which $T$ is onto. Hence $T$ is not strictly cosingular.

\section{Preliminaries}
\label{p4:sec:preliminaries}

We start recording two standard facts which will be used throughout the proof. Both are classical and elementary consequences of basic Banach space duality; we include the details for completeness and to fix the precise form in which they will be used.

\begin{lemma}[Surjectivity and the adjoint]
\label{p4:lem:surjectivity-adjoint}
Let $S \colon X\to Z$ be bounded. Then $S$ is surjective if and only if $S^*$ is bounded below.
\end{lemma}

\begin{proof}
Suppose first that $S$ is surjective. By the open mapping theorem, there is $C>0$ such that, for every $z\in Z$, one can choose $x\in X$ with $Sx=z$ and $\|x\|\leq C\|z\|$. Hence, for every $z^*\in Z^*$,
\begin{equation*}
        \|z^*\|=\sup_{\|z\|\leq 1}|z^*(z)|\leq C\sup_{\|x\|\leq 1}|z^*(Sx)|=C\|S^*z^*\|.
\end{equation*}
Thus $S^*$ is bounded below.

Conversely, suppose that $S^*$ is bounded below. Then $S^*$ is injective and has closed range. By the closed range theorem, $\Pfourran S$ is closed. Also
\begin{equation*}
        (\Pfourran S)^\perp=\ker S^*=\{0\}.
\end{equation*}
Therefore $\overline{\Pfourran S}=Z$, and since $\Pfourran S$ is closed, $S$ is surjective.
\end{proof}

\begin{proposition}[Weak-star characterization of strict cosingularity]
\label{p4:prop:wstar-characterization}
Let $T \colon X\to Y$ be bounded. Then the following are equivalent.
\begin{enumerate}[label=\textup{(\roman*)}]
\item\label{p4:item:T-strictly-cosingular} The operator $T$ is strictly cosingular.
\item\label{p4:item:Tstar-no-wstar-below} The operator $T^*$ is not bounded below on any infinite-dimensional weak-star closed subspace of $Y^*$.
\end{enumerate}
\end{proposition}

\begin{proof}
Let $M\subset Y$ be closed, and let
\begin{equation*}
        Q_M \colon Y\to Y/M
\end{equation*}
be the quotient map. The adjoint $Q_M^*$ identifies $(Y/M)^*$ isometrically and weak-star homeomorphically with $M^\perp\subset Y^*$. Moreover,
\begin{equation*}
        (Q_MT)^*=T^*Q_M^*.
\end{equation*}
By \Cref{p4:lem:surjectivity-adjoint}, the map $Q_MT$ is surjective if and only if $(Q_MT)^*$ is bounded below. Through the above identification of $(Y/M)^*$ with $M^\perp$, this is equivalent to $T^*$ being bounded below on $M^\perp$.

It remains only to translate the subspaces. Every $\sigma(Y^*,Y)$-closed subspace $W\subset Y^*$ has the form
\begin{equation*}
        W=(W^\perp)^\perp,
\end{equation*}
and $W$ is infinite-dimensional if and only if $Y/W^\perp$ is infinite-dimensional. Thus closed subspaces $M\subset Y$ of infinite codimension correspond exactly to infinite-dimensional weak-star closed subspaces of $Y^*$. The equivalence follows.
\end{proof}

The next lemma is a finite-dimensional form of Goldstine's theorem. We include a short direct proof.

\begin{lemma}[Finite-dimensional interpolation]
\label{p4:lem:finite-dimensional-interpolation}
Let $F\subset X^*$ be finite-dimensional and let $\varphi\in F^*$. For every $\eta>0$, there exists $x\in X$ such that
\begin{equation*}
        f(x)=\varphi(f) \qquad(f\in F),
\end{equation*}
and
\begin{equation*}
        \|x\|\leq (1+\eta)\|\varphi\|.
\end{equation*}
If $\varphi=0$, one may take $x=0$.
\end{lemma}

\begin{proof}
Define
\begin{equation*}
        R_F \colon X\to F^*, \qquad (R_Fx)(f)=f(x).
\end{equation*}
The adjoint $R_F^* \colon F^{**}\to X^*$ is, under the canonical identification $F^{**}=F$, the inclusion of $F$ into $X^*$. In particular, $R_F^*$ is an isometry.

First, $R_F$ is onto. Indeed, if $\operatorname{ran}R_F$ were a proper subspace of the finite-dimensional space $F^*$, then there would be a nonzero element $f\in (F^*)^*=F$ which annihilates $\operatorname{ran}R_F$. This would mean that
\begin{equation*}
        0=f(R_Fx)=(R_Fx)(f)=f(x) \qquad(x\in X),
\end{equation*}
which is impossible for a nonzero element $f\in F\subset X^*$.

Equip $F^*$ with the quotient norm
\begin{equation*}
        \|\psi\|_q=\inf\{\|x\|:R_Fx=\psi\}.
\end{equation*}
Thus $\|\psi\|_q$ is the least norm, up to an arbitrarily small error, of a vector $x\in X$ which represents $\psi$ by evaluation on $F$.

We compare this quotient norm with the original dual norm on $F^*$. Let $f\in F$. The dual norm of $f$, when $F^*$ is equipped with $\|\cdot\|_q$, is
\begin{equation*}
        \|f\|_{q,*}=\sup_{\|\psi\|_q\leq 1}|\psi(f)|.
\end{equation*}
By the definition of the quotient norm, taking the supremum over $\|\psi\|_q\leq1$ is the same as taking the supremum over vectors of the form $\psi=R_Fx$ with $\|x\|\leq1$, up to an arbitrarily small enlargement of the unit ball. Therefore
\begin{equation*}
        \|f\|_{q,*}=\sup_{\|x\|\leq 1}|(R_Fx)(f)|.
\end{equation*}
Since $(R_Fx)(f)=f(x)$, we get
\begin{equation*}
        \|f\|_{q,*}=\sup_{\|x\|\leq 1}|f(x)|=\|f\|.
\end{equation*}
Thus the quotient norm on $F^*$ induces on $(F^*)^*=F$ exactly the original norm inherited from $X^*$.

The original norm on $F^*$ also induces this same dual norm on $F$. Since $F^*$ is finite-dimensional, a norm is determined by its dual norm. Hence the quotient norm $\|\cdot\|_q$ and the original norm on $F^*$ coincide. Consequently,
\begin{equation*}
        \inf\{\|x\|:R_Fx=\varphi\}=\|\varphi\|.
\end{equation*}
The desired choice of $x$ follows.
\end{proof}

\section{The preadjoint extraction lemma}
\label{p4:sec:preadjoint-extraction}

The next lemma is the main point of the proof. It says that, after passing to a further infinite-dimensional subspace $E_0\subset E$, the operator $U \colon E\to Y^*$ behaves as though it had a preadjoint. More precisely, once we quotient $X$ by the common kernel $(E_0)^\perp$, each functional $e\in E_0$ becomes a functional $Je$ on $P=X/(E_0)^\perp$, and the lemma constructs an operator $A \colon Y\to P$ such that
\begin{equation*}
        A^*Je=Ue \qquad(e\in E_0).
\end{equation*}
Thus, on $E_0$, the map $U$ is realised by the adjoint of an operator defined on $Y$.

The idea of the proof is a diagonal construction. Since $Y$ is separable, we choose a countable dense linear set $\{y_1,y_2,\ldots\}\subset Y$. At stage $n$, we choose a vector $x_n\in X$ which represents the values of $Ue$ at $y_n$ for all previously chosen functionals $e_1,\ldots,e_{n-1}$. We then choose the next functional $e_n\in E$ so that all earlier interpolation identities remain valid. After this recursive construction, the rule $Ay_n=qx_n$ is well defined on the dense set, because all linear relations among the $y_n$ are killed by the functionals in $E_0$. The resulting operator $A \colon Y\to P$ has adjoint agreeing with $U$ on $E_0$.

\begin{lemma}[Preadjoint extraction]
\label{p4:lem:preadjoint-extraction}
Let $Y$ be separable, $E\subset X^*$ be a closed infinite-dimensional subspace, and $U \colon E\to Y^*$ be bounded. Then there are a closed infinite-dimensional subspace $E_0\subset E$ and a bounded operator
\begin{equation*}
        A\colon Y\to P:=X/(E_0)^\perp
\end{equation*}
such that $\|A\|\leq 2\|U\|$ and
\begin{equation}
\label{p4:eq:preadjoint-identity}
        A^*J e=Ue \qquad(e\in E_0),
\end{equation}
where $J \colon E_0\to P^*$ is the canonical map defined by
\begin{equation*}
        (Je)(qx)=e(x) \qquad(e\in E_0,\ x\in X),
\end{equation*}
and $q \colon X\to P$ is the quotient map.
\end{lemma}

\begin{proof}
Let
\begin{equation*}
        \PfourF=
        \begin{cases}
        \mathbb Q, & \PfourField=\R,\\
        \mathbb Q+i\mathbb Q, & \PfourField=\C.
        \end{cases}
\end{equation*}
Choose a countable dense $\PfourF$-linear subspace $D\subset Y$, and enumerate it as
\begin{equation*}
        D=\{y_1,y_2,\ldots\}.
\end{equation*}

We recursively construct linearly independent unit vectors $(e_n)_{n=1}^{\infty}$ in $E$ and vectors $(x_n)_{n=1}^{\infty}$ in $X$. Suppose that $e_1,\ldots,e_{n-1}$ have already been chosen, and put
\begin{equation*}
        F_{n-1}=\operatorname{span}\{e_1,\ldots,e_{n-1}\}.
\end{equation*}
Define $\varphi_n\in F_{n-1}^*$ by
\begin{equation*}
        \varphi_n(e)=(Ue)(y_n) \qquad(e\in F_{n-1}).
\end{equation*}
Then
\begin{equation*}
        \|\varphi_n\|\leq \|U\|\|y_n\|.
\end{equation*}
By \Cref{p4:lem:finite-dimensional-interpolation}, choose $x_n\in X$ such that
\begin{equation}
\label{p4:eq:interpolation-step}
        e(x_n)=(Ue)(y_n) \qquad(e\in F_{n-1}),
\end{equation}
and
\begin{equation}
\label{p4:eq:xn-bound}
        \|x_n\|\leq 2\|U\|\|y_n\|.
\end{equation}
If $y_n=0$, we take $x_n=0$.

Having already defined $h_1, \dots, h_{n-1} \in E^*$, define $h_n\in E^*$ by
\begin{equation*}
        h_n(e)=(Ue)(y_n)-e(x_n) \qquad(e\in E).
\end{equation*}
The subspace
\begin{equation*}
        H_n=\bigcap_{j=1}^{n}\ker h_j
\end{equation*}
has finite codimension in the infinite-dimensional space $E$, and is therefore infinite-dimensional. Choose
\begin{equation*}
        e_n\in H_n\setminus F_{n-1}, \qquad \|e_n\|=1.
\end{equation*}
This completes the recursion.

Set
\begin{equation*}
        E_0=\overline{\operatorname{span}}\{e_n:n\in\N\}\subset E.
\end{equation*}
Then $E_0$ is closed and infinite-dimensional. We claim that
\begin{equation}
\label{p4:eq:diagonal-identity}
        (Ue)(y_j)=e(x_j) \qquad(e\in E_0,\ j\in\N).
\end{equation}
Equivalently, for each fixed $j\in\N$, we must show that $h_j$ vanishes on $E_0$.

It is enough first to check this on the vectors $e_i$. If $i<j$, then $e_i\in F_{j-1}$, so \eqref{p4:eq:interpolation-step} gives
\begin{equation*}
        e_i(x_j)=(Ue_i)(y_j),
\end{equation*}
that is, $h_j(e_i)=0$. If $i\geq j$, then $e_i\in H_i\subset\ker h_j$, so again $h_j(e_i)=0$. Thus each $h_j$ vanishes on every $e_i$, and by continuity each $h_j$ vanishes on
\begin{equation*}
        E_0=\overline{\operatorname{span}}\{e_i:i\in\N\}.
\end{equation*}
This proves \eqref{p4:eq:diagonal-identity}.

Let
\begin{equation*}
        P=X/(E_0)^\perp
\end{equation*}
and let $q\colon X\to P$ be the quotient map. Define $A_0 \colon D\to P$ by
\begin{equation}
\label{p4:eq:A0-definition}
        A_0y_n=qx_n \qquad(n\in\N).
\end{equation}
We verify that this is $\PfourF$-linear. In other words, we must check that every $\PfourF$-linear relation among the vectors $y_n$ is respected by the images $qx_n$. Therefore, whenever
\begin{equation*}
        \sum_{j=1}^{m}a_jy_{n_j}=0,
\end{equation*}
with $a_1,\ldots,a_m\in\PfourF$, we must prove that
\begin{equation*}
        \sum_{j=1}^{m}a_jqx_{n_j}=0
\end{equation*}
in $P$. Thus, suppose that
\begin{equation*}
        \sum_{j=1}^{m}a_jy_{n_j}=0, \qquad a_1,\ldots,a_m\in\PfourF.
\end{equation*}
For every $e\in E_0$, equation \eqref{p4:eq:diagonal-identity} gives
\begin{equation*}
        e\left(\sum_{j=1}^{m}a_jx_{n_j}\right)=\sum_{j=1}^{m}a_j(Ue)(y_{n_j})=(Ue)\left(\sum_{j=1}^{m}a_jy_{n_j}\right)=0.
\end{equation*}
Therefore $\sum_{j=1}^{m}a_jx_{n_j}\in(E_0)^\perp$, and hence
\begin{equation*}
        \sum_{j=1}^{m}a_jqx_{n_j}=0.
\end{equation*}
Thus all $\PfourF$-linear relations in $D$ are respected.

If $y\in D$, choose $n\in\N$ with $y=y_n$. By \eqref{p4:eq:xn-bound},
\begin{equation*}
        \|A_0y\|=\|qx_n\|\leq \|x_n\|\leq 2\|U\|\|y\|.
\end{equation*}
Thus $A_0$ extends uniquely to a bounded $\PfourF$-linear map
\begin{equation*}
        A \colon Y\to P
\end{equation*}
with $\|A\|\leq 2\|U\|$. The extension is $\PfourField$-linear: in the real case this follows by approximating real scalars by rationals, and in the complex case by approximating complex scalars by elements of $\mathbb Q+i\mathbb Q$.

Recall that $J\colon E_0\to P^*$ is defined by
\begin{equation*}
        (Je)(qx)=e(x) \qquad(e\in E_0,\ x\in X).
\end{equation*}
This is well defined: if $qx=qx'$, then $x-x'\in(E_0)^\perp$, and since $e\in E_0$ we have $e(x-x')=0$. Hence $e(x)=e(x')$.

Now let $e\in E_0$. For every $n\in\N$, since $A$ extends $A_0$ and $A_0y_n=qx_n$, we have
\begin{equation*}
        (A^*Je)(y_n)=(Je)(Ay_n)=(Je)(qx_n)=e(x_n)=(Ue)(y_n).
\end{equation*}
The set $D= \{y_n:n\in\N\}$ is dense in $Y$, so the continuous functionals $A^*Je$ and $Ue$ agree on all of $Y$. This proves \eqref{p4:eq:preadjoint-identity}.
\end{proof}

\begin{remark}
\label{p4:rem:no-basis-needed}
No basis property is required of the sequence $(e_n)_{n=1}^{\infty}$. Linear independence is enough, because the argument only uses that $E_0=\overline{\operatorname{span}}\{e_n:n\in\N\}$ is infinite-dimensional.
\end{remark}

\section{Quotient factorization and proof of the theorem}
\label{p4:sec:factorization}

We first prove the nontrivial implication in a factorized form.

\begin{theorem}[Quotient factorization]
\label{p4:thm:quotient-factorization}
Let $Y$ be separable and let $T\colon X\to Y$ be bounded. If $T^*$ is bounded below on some infinite-dimensional subspace of $Y^*$, then there exist an infinite-dimensional Banach space $P$, a quotient map $q\colon X\to P$, and an operator $A\colon Y\to P$ such that
\begin{equation*}
        AT=q.
\end{equation*}
In particular, $T$ is not strictly cosingular.
\end{theorem}

\begin{proof}
Suppose that $T^*$ is bounded below on an infinite-dimensional subspace of $Y^*$. Passing to the norm closure of that subspace, choose a closed infinite-dimensional subspace $Z\subset Y^*$ and a constant $c>0$ such that
\begin{equation*}
        \|T^*z\|\geq c\|z\| \qquad(z\in Z).
\end{equation*}
Then
\begin{equation*}
        E=T^*(Z)\subset X^*
\end{equation*}
is closed and infinite-dimensional, and
\begin{equation*}
        U=(T^*|_Z)^{-1}:E\to Z\subset Y^*
\end{equation*}
is bounded.

By \Cref{p4:lem:preadjoint-extraction}, there is a closed infinite-dimensional subspace $E_0\subset E$ such that, if
\begin{equation*}
        P=X/(E_0)^\perp
\end{equation*}
and $q\colon X\to P$ denotes the quotient map, then there is an operator $A\colon Y\to P$ satisfying
\begin{equation}
\label{p4:eq:preadjoint-identity-used}
        A^*Je=Ue \qquad(e\in E_0).
\end{equation}

For $e\in E_0$, it follows from \eqref{p4:eq:preadjoint-identity-used} that
\begin{equation}
\label{p4:eq:adjoint-equality-on-E0}
        (AT)^*Je=T^*A^*Je=T^*Ue=e=q^*Je.
\end{equation}
Here the last equality follows directly from the definition of $J$.

It remains to extend \eqref{p4:eq:adjoint-equality-on-E0} from $E_0$ to all of $P^*$. The adjoint $q^*$ is an isometric weak-star homeomorphism from $P^*$ onto
\begin{equation*}
        q^*(P^*)=((E_0)^\perp)^\perp.
\end{equation*}

Since $q^*Je=e$ for every $e\in E_0$, we have $q^*(J(E_0))=E_0$. Also $q^*$ identifies $P^*$ weak-star homeomorphically with $((E_0)^\perp)^\perp$, and by the annihilator bipolar identity this latter space is $\overline{E_0}^{\,\sigma(X^*,X)}$. Hence $q^*(J(E_0))$ is weak-star dense in $q^*(P^*)$. Pulling back through the weak-star homeomorphism $q^*$, we conclude that $J(E_0)$ is weak-star dense in $P^*$.

The operators $(AT)^*\colon P^*\to X^*$ and $q^*\colon P^*\to X^*$ are weak-star-to-weak-star continuous. By \eqref{p4:eq:adjoint-equality-on-E0}, they agree on $J(E_0)$. Since $J(E_0)$ is weak-star dense in $P^*$, they agree on all of $P^*$. Hence
\begin{equation*}
        (AT)^*p^*=q^*p^* \qquad(p^*\in P^*).
\end{equation*}
Since two bounded operators into $P$ are equal whenever their adjoints agree, it follows that
\begin{equation*}
        AT=q.
\end{equation*}
This proves the first part of the theorem. We now show that this factorization implies that $T$ is not strictly cosingular.

The quotient map $q$ is surjective, so $AT=q$ forces $A$ to be surjective. Put
\begin{equation*}
        M=\ker A.
\end{equation*}
Then $A$ induces an isomorphism
\begin{equation*}
        \widetilde A\colon Y/M\to P.
\end{equation*}
The space $P$ is infinite-dimensional, since $P^*$ contains the infinite-dimensional subspace $J(E_0)$. Hence $M$ has infinite codimension in $Y$. Let $Q_M\colon Y\to Y/M$ be the quotient map. Then we have
\begin{equation*}
        \widetilde A Q_M T=AT=q.
\end{equation*}
It follows that
\begin{equation*}
        Q_M T=\widetilde A^{-1}q.
\end{equation*}
Since $q$ is surjective and $\widetilde A^{-1}$ is onto $Y/M$, the operator $Q_MT$ is surjective. Thus $T$ is not strictly cosingular.
\end{proof}

Using the previous result, we can now prove the main duality result. 

\begin{proof}[Proof of \Cref{p4:thm:main}]
Assume first that $T^*$ is strictly singular. Then $T^*$ is not bounded below on any infinite-dimensional subspace of $Y^*$, and in particular it is not bounded below on any infinite-dimensional weak-star closed subspace of $Y^*$. By \Cref{p4:prop:wstar-characterization}, $T$ is strictly cosingular.

For the converse, we prove the contrapositive. Suppose that $T^*$ is not strictly singular. Then $T^*$ is bounded below on some infinite-dimensional subspace of $Y^*$. By the quotient factorization theorem, \Cref{p4:thm:quotient-factorization}, the operator $T$ is not strictly cosingular. Therefore, if $T$ is strictly cosingular, then $T^*$ must be strictly singular. This proves the equivalence.
\end{proof}

The same argument gives the following geometric formulation.

\begin{corollary}[Weak-star extraction]
\label{p4:cor:wstar-extraction}
Let $Y$ be separable and let $T\colon X\to Y$ be bounded. If $T^*$ is bounded below on an infinite-dimensional norm-closed subspace of $Y^*$, then $T^*$ is bounded below on an infinite-dimensional weak-star closed subspace of $Y^*$.
\end{corollary}

\begin{proof}
Use the notation obtained in the proof of \Cref{p4:thm:quotient-factorization}. Since $A\colon Y\to P$ is surjective,
\begin{equation*}
        W=A^*(P^*)=(\ker A)^\perp
\end{equation*}
is an infinite-dimensional weak-star closed subspace of $Y^*$. The identity $AT=q$ gives
\begin{equation*}
        T^*A^*=q^*.
\end{equation*}
For $p^*\in P^*$,
\begin{equation*}
        \|T^*A^*p^*\|=\|q^*p^*\|=\|p^*\|\geq \frac{1}{\|A\|}\|A^*p^*\|.
\end{equation*}
Thus $T^*$ is bounded below on $W$.
\end{proof}

\medskip

Before we finish, let us point out where the separability hypothesis enters the argument. It is used only in \Cref{p4:lem:preadjoint-extraction}, through the choice of a countable dense $\PfourF$-linear subspace of $Y$. No separability assumption on $X$ is used. This also explains why the classical example of the inclusion $c_0\hookrightarrow \ell_\infty$ is not covered by the theorem: the range space $\ell_\infty$ is nonseparable, so the countable diagonal construction used above is unavailable.

Finally, the proof does not require the weak-star closed witness for $T^*$ to lie inside the original subspace on which $T^*$ is bounded below. Starting from such a subspace, the argument constructs a new weak-star closed subspace $W\subset Y^*$ on which $T^*$ is bounded below.

\end{problempaperbody}
\clearpage
\providecommand{\Pfiveaco}{\operatorname{aco}}
\providecommand{\PfiveDelta}{\Delta}

\problempaper{Problem 4. Weakly compact basis factorization}
\label{probpaper:weakly-compact-basis-factorization}

\begin{problemabstract}
We prove a basis-valued refinement of the Davis--Figiel--Johnson--Pelczynski factorization theorem. If $Y$ is separable and has the bounded approximation property, then every weakly compact operator $T\colon X\to Y$ factors through a reflexive Banach space with a Schauder basis. In particular, every weakly compact operator into a Banach space with a Schauder basis admits such a factorization.
\end{problemabstract}

\problemcontents

\begin{problempaperbody}

\section{Statement and notation}

Throughout, Banach spaces are over the real or complex scalar field. We write $B_E$ for the closed unit ball of a Banach space $E$, and we identify $E$ with its canonical image in $E^{**}$. Our main findings are the following.

\begin{theorem}[Weakly compact basis factorization]
\label{p5:thm:main}
Let $Y$ be a separable Banach space with the bounded approximation property. If $T\colon X\to Y$ is weakly compact, then $T$ factors through a reflexive Banach space with a Schauder basis.
\end{theorem}

As an automatic consequence we get.

\begin{corollary}
\label{p5:cor:basis-range}
Let $Y$ have a Schauder basis. If $T\colon X\to Y$ is weakly compact, then $T$ factors through a reflexive Banach space with a Schauder basis.
\end{corollary}

\section{Proof strategy and organisation}

The proof is organised in two stages, after the preliminary results recorded in \Cref{p5:sec:preliminaries}. In \Cref{p5:sec:approximants}, we use weak compactness of $T$ and the bounded approximation property of $Y$ to produce finite-rank operators $A_n\colon X\to Y$ which converge to $T$ in two dual senses:
\begin{equation*}
A_n^*y^*\to T^*y^* \quad (y^*\in Y^*),
\qquad
A_n^{**}x^{**}\to T^{**}x^{**} \quad (x^{**}\in X^{**}).
\end{equation*}
The adjoint convergence is the only delicate point. The map $T^*\colon B_{Y^*}\to X^*$ is Baire-one for the weak-star topology on $B_{Y^*}$, while the finite-rank approximants initially give continuous maps which converge only weakly. A tail-convexification argument, using Mazur's theorem, converts this weak convergence into pointwise norm convergence.

In \Cref{p5:sec:bridge}, the sequence $(A_n)_{n=1}^\infty$ is used to encode $T$ as a map into
\begin{equation*}
c(Y)=\{(y_n)_{n=1}^{\infty}\colon y_n\text{ converges in norm in }Y\}.
\end{equation*}
The ``freezing'' projections
\begin{equation*}
Q_m(y_1,y_2,\ldots)=(y_1,\ldots,y_m,y_m,y_m,\ldots)
\end{equation*}
are arranged to be finite-rank on the eventual DFJP interpolation space. They approximate the identity there and determine a finite-dimensional decomposition. The reflexive basisification theorem from \Cref{p5:sec:preliminaries} then embeds the interpolation space as a complemented subspace of a reflexive space with a Schauder basis. The final assembly is given in \Cref{p5:sec:conclusion}.

\section{Preliminaries}
\label{p5:sec:preliminaries}

We shall use two standard structural results. We state them in the precise forms needed below. The first is the Davis--Figiel--Johnson--Pelczynski interpolation construction; see the construction preceding \cite[Lemma~1]{DavisFigielJohnsonPelczynski1974}, together with parts \textup{(i)}, \textup{(ii)}, and \textup{(iv)} of that lemma.

\begin{theorem}[Davis--Figiel--Johnson--Pelczynski interpolation]
\label{p5:thm:dfjp-interpolation}
Let $E$ be a Banach space and let $K\subset E$ be bounded, closed, absolutely convex and weakly compact. For $j\in\mathbb Z$, set
\begin{equation*}
        U_j=2^jK+2^{-j}B_E,
\end{equation*}
let $|\cdot|_j$ be the Minkowski functional of $U_j$, and define
\begin{equation*}
        \PfiveDelta(K)=\left\{z\in E\colon \sum_{j\in\mathbb Z}|z|_j^2<\infty\right\}
\end{equation*}
with norm
\begin{equation*}
        \|z\|_{\PfiveDelta(K)}=\left(\sum_{j\in\mathbb Z}|z|_j^2\right)^{1/2}.
\end{equation*}
Then $\PfiveDelta(K)$ is a reflexive Banach space, the inclusion $\PfiveDelta(K)\to E$ is bounded, and $K$ is bounded as a subset of $\PfiveDelta(K)$.
\end{theorem}

The second result is a reflexive basisification result for spaces with finite-dimensional decompositions. It follows from the finite-dimensional stabilisation theorem of Johnson--Rosenthal--Zippin; see \cite[Corollary~4.12(a)]{JohnsonRosenthalZippin1971}. It is in the same spirit as Pelczynski's basisification theorem \cite{Pelczynski1971}.

\begin{theorem}[Reflexive basisification from an FDD]
\label{p5:thm:jrz-basisification}
Let $Z$ be a reflexive Banach space with a finite-dimensional decomposition. Then $Z$ is isomorphic to a complemented subspace of a reflexive Banach space with a Schauder basis.
\end{theorem}

\begin{proof}
Let $(E_n)_{n=1}^{\infty}$ be a finite-dimensional decomposition of $Z$, and let $(P_N)_{N=1}^{\infty}$ be its natural projections. By \cite[Corollary~4.12(a)]{JohnsonRosenthalZippin1971}, there is an absolute constant $K$ such that, for every $n\in\N$, there is a finite-dimensional Banach space $G_n$ containing $E_n$ as a $1$-complemented subspace and having a basis with basis constant at most $K$. Let $\pi_n\colon G_n\to E_n$ be a norm-one projection and put $H_n=\ker\pi_n$, equipped with the norm inherited from $G_n$. Thus $G_n=E_n\oplus H_n$ algebraically.

For $e\in E_n$ and $h\in H_n$, put $\|(e,h)\|_2=(\|e\|^2+\|h\|^2)^{1/2}$. Since $\|\pi_n\|=1$ and $\|I-\pi_n\|\leq2$, we have
\begin{equation*}
        \|(e,h)\|_2\leq\sqrt5\,\|e+h\|_{G_n},
        \qquad
        \|e+h\|_{G_n}\leq\sqrt2\,\|(e,h)\|_2.
\end{equation*}
Consequently, when the chosen basis of $G_n$ is regarded as a basis of $E_n\oplus_2H_n$, its basis constant is at most $\sqrt{10}K$.

Set
\begin{equation*}
        H=\left(\sum_{n=1}^{\infty}\oplus H_n\right)_2
        \qquad\text{and}\qquad
        V=Z\oplus_2H.
\end{equation*}
Then $V$ is reflexive. Moreover, the spaces $E_n\oplus_2H_n$ form a finite-dimensional decomposition of $V$, whose natural projections have norms bounded by
\begin{equation*}
        \max\left\{\sup_{N\in\N}\|P_N\|,1\right\}.
\end{equation*}
Since the chosen block bases have uniformly bounded basis constants, their concatenation is a Schauder basis of $V$. Finally, the embedding $\iota\colon Z\to V$ defined by $\iota z=(z,0)$ is complemented by the projection $(z,h)\mapsto(z,0)$.
\end{proof}

\section{Finite-rank approximants}
\label{p5:sec:approximants}

We shall use the bounded approximation property in the following sequential form. This is well-known, and we include a short proof for completeness.

\begin{lemma}
\label{p5:lem:strong-approximation}
Let $Y$ be separable and have the bounded approximation property. Then there are finite-rank operators $(P_n)_{n=1}^{\infty}$, $P_n\colon Y\to Y$, such that
\begin{equation*}
        \sup_{n\in\N}\|P_n\|<\infty \qquad\text{and}\qquad
        P_ny\to y \quad(y\in Y).
\end{equation*}
\end{lemma}

\begin{proof}
Choose a dense sequence $(y_j)_{j=1}^{\infty}$ in $Y$. Since $Y$ has the bounded approximation property, there is a constant $\lambda<\infty$ such that, for every $n\in\N$, there is a finite-rank operator $P_n\colon Y\to Y$ satisfying
\begin{equation*}
        \|P_n\|\leq\lambda
        \qquad\text{and}\qquad
        \|P_ny_j-y_j\|<\frac{1}{n} \quad(1\leq j\leq n).
\end{equation*}
It follows that $P_ny_j\to y_j$ for every fixed $j\in\N$. Since the operators $(P_n)_{n=1}^{\infty}$ are uniformly bounded and the set $\{y_j:j\in\N\}$ is dense in $Y$, we get $P_ny\to y$ for every $y\in Y$.
\end{proof}

We shall also use the following elementary metrizability fact for weakly compact sets in separable spaces.

\begin{lemma}
\label{p5:lem:weak-compact-metrizable}
Let $Y$ be separable and let $W\subset Y$ be weakly compact. Then $W$, endowed with the weak topology, is compact metrizable.
\end{lemma}

\begin{proof}
Let $Y_0=\overline{\operatorname{span}} W$. Then $Y_0$ is separable. If $Y_0=\{0\}$, there is nothing to prove, so assume that $Y_0\ne\{0\}$. Choose a norm-dense sequence $(s_n)_{n=1}^{\infty}$ in the unit sphere of $Y_0$. By Hahn--Banach, choose functionals $(y_n^*)_{n=1}^{\infty}\subset B_{Y^*}$ such that
\begin{equation*}
        |y_n^*(s_n)|>\frac{1}{2} \qquad(n\in\N).
\end{equation*}

We claim that $(y_n^*)_{n=1}^{\infty}$ separates the points of $Y_0$. Indeed, if $0\ne y\in Y_0$, choose $n\in\N$ such that
\begin{equation*}
        \left\|s_n-\frac{y}{\|y\|}\right\|<\frac{1}{4}.
\end{equation*}
Then
\begin{equation*}
        \left|y_n^*\left(\frac{y}{\|y\|}\right)\right|\geq |y_n^*(s_n)|-\left\|s_n-\frac{y}{\|y\|}\right\|>\frac{1}{4}.
\end{equation*}
Since $W$ is weakly compact, it is norm-bounded, and let
\begin{equation*}
        M=\sup_{w\in W}\|w\|<\infty.
\end{equation*}
Thus the map
\begin{equation*}
        \Phi\colon W\to \prod_{n=1}^{\infty}\{z\in\mathbb K:|z|\leq M\}, \qquad \Phi(w)=(y_n^*(w))_{n=1}^{\infty},
\end{equation*}
is a weakly continuous injection into a compact metric space. Since $W$ is weakly compact, $\Phi$ is a homeomorphism of $W$ onto its image. Hence, $W$ is compact metrizable in its weak topology.
\end{proof}

We record a simple regularity consequence of weak compactness that will be used
below.

\begin{lemma}
\label{p5:lem:baire-one-adjoint}
Let $Y$ be separable and let $T\colon X\to Y$ be weakly compact. Let $K$ be
the closed unit ball of $Y^*$, equipped with the weak-star topology, and define
\begin{equation*}
        F\colon K\to X^*, \qquad F(y^*)=T^*y^*.
\end{equation*}
Then $F$, regarded as an $X^*$-valued map with the norm topology on $X^*$, is
Baire-one.
\end{lemma}

\begin{proof}
Let
\begin{equation*}
        W=\overline{T(B_X)}^{\,w}\subset Y .
\end{equation*}
By weak compactness of $T$, the set $W$ is weakly compact. By
\Cref{p5:lem:weak-compact-metrizable}, it is compact metrizable in its weak
topology. Choose a weakly dense sequence $(w_j)_{j=1}^{\infty}$ in $W$. For
$y^*,z^*\in K$,
\begin{equation*}
\begin{aligned}
        \|T^*y^*-T^*z^*\|
        &= \sup_{x\in B_X}|(y^*-z^*)(Tx)|  \\
        &= \sup_{w\in W}|(y^*-z^*)(w)|     \\
        &= \sup_{j\in\mathbb N}|(y^*-z^*)(w_j)|.
\end{aligned}
\end{equation*}
The second equality uses the weak density of $T(B_X)$ in $W$, and the third uses
the weak density of $(w_j)_{j=1}^{\infty}$ in $W$, together with weak continuity
on $W$.

It follows that $F(K)$ is norm separable, because it embeds isometrically into
the separable space $C(W)$ by restriction to $W$. Also, for fixed
$z^*\in K$ and $r\geq0$,
\begin{equation*}
\begin{aligned}
        F^{-1}\bigl(\overline B_{F(K)}(F(z^*),r)\bigr)
        &= \{y^*\in K\colon\|F(y^*)-F(z^*)\|\leq r\} \\
        &=
        \bigcap_{j\in\mathbb N}
        \{y^*\in K\colon |(y^*-z^*)(w_j)|\leq r\}.
\end{aligned}
\end{equation*}
The last expression is weak-star closed in $K$. Hence the inverse image under $F$ of every closed norm ball in $F(K)$ is closed in $K$.

Since $F(K)$ is norm separable, the Banach space
\begin{equation*}
        Z=\overline{\operatorname{span}}\,F(K)\subset X^*
\end{equation*}
is separable. Since $Y$ is separable, $K=(B_{Y^*},w^*)$ is compact metrizable. Let $U$ be a norm-open subset of $F(K)$. Since $F(K)$ is separable metric, $U$ is a countable union of closed norm balls in $F(K)$; since the inverse image under $F$ of each such closed ball is closed in $K$, it follows that $F^{-1}(U)$ is an $F_\sigma$ subset of $K$. The same conclusion holds for norm-open subsets of $Z$, because $F$ takes its values in $F(K)$. Equivalently, inverse images of norm-closed subsets of $Z$ are $G_\delta$ subsets of $K$. Hence, by \cite[Theorem~4]{Stegall1991}, applied to $F\colon K\to Z$, the map $F$ is Baire-one.
\end{proof}

We shall use the following parametrized tail form of Mazur's theorem, which converts pointwise weak convergence into pointwise norm convergence by convexifying sufficiently far out in the sequence.

\begin{lemma}[Tail convexification]
\label{p5:lem:tail-convexification}
Let $K$ be compact metrizable, let $E$ be a Banach space, and let $f_n\colon K\to E$ be uniformly bounded norm-continuous maps. Suppose that $f_n(t)\xrightarrow{w} f(t)$ for every $t\in K$, and that $f\colon K\to E$ is norm-Baire-one. Then there are $g_k\in \operatorname{conv}\{f_n\colon n\geq k\}$ such that $\|g_k(t)-f(t)\|\to0$ for every $t\in K$.
\end{lemma}

\begin{proof}
Since $f$ is bounded and Baire-one, choose uniformly bounded norm-continuous maps $u_n\colon K\to E$ such that $u_n(t)\to f(t)$ in norm for every $t\in K$. This can be done by starting with any norm-continuous approximating sequence and composing with the radial retraction onto a closed ball containing $f(K)$; the retraction fixes $f(K)$, so pointwise convergence to $f$ is preserved.

Set $h_n=f_n-u_n$ and
\begin{equation*}
        L=K\times(B_{E^*},w^*).
\end{equation*}
For $(t,e^*)\in L$, define
\begin{equation*}
        \widehat h_n(t,e^*)=e^*(h_n(t)).
\end{equation*}
Then $\widehat h_n\in C(L)$, the sequence $(\widehat h_n)_{n=1}^\infty$ is uniformly bounded, and $\widehat h_n\to0$ pointwise on $L$. By dominated convergence against regular Borel measures on $L$, we have $\widehat h_n\to0$ weakly in $C(L)$.

Fix $k$. Since $0$ lies in the weak closure of $\operatorname{conv}\{\widehat h_n\colon n\geq k\}$, Mazur's theorem gives a finite convex combination
\begin{equation*}
        \widehat v_k=\sum_{n\geq k}\alpha_{k,n}\widehat h_n
\end{equation*}
such that $\|\widehat v_k\|_{C(L)}<1/k$. Define
\begin{equation*}
        g_k=\sum_{n\geq k}\alpha_{k,n}f_n,
        \qquad
        v_k=\sum_{n\geq k}\alpha_{k,n}u_n.
\end{equation*}
Then $\sup_{t\in K}\|g_k(t)-v_k(t)\|<1/k$. For each fixed $t\in K$, the vectors $v_k(t)$ are convex combinations of tails of the norm-convergent sequence $(u_n(t))_{n = 1}^\infty$, so $v_k(t)\to f(t)$. Hence $g_k(t)\to f(t)$ in norm. This finishes the proof.
\end{proof}

We now combine the preceding Baire-one regularity with the bounded approximation property to obtain finite-rank approximants whose adjoints converge on both sides.

\begin{proposition}
\label{p5:prop:two-sided-approximants}
Let $Y$ be separable with the bounded approximation property, and let $T\colon X\to Y$ be weakly compact. Then there are finite-rank operators $A_k\colon X\to Y$ such that
\begin{equation*}
        A_k^*y^*\to T^*y^* \qquad (y^*\in Y^*)
\end{equation*}
and
\begin{equation*}
        A_k^{**}x^{**}\to T^{**}x^{**} \qquad (x^{**}\in X^{**})
\end{equation*}
in norm.
\end{proposition}
\begin{proof}
Choose finite-rank operators $P_n\colon Y\to Y$ as in \Cref{p5:lem:strong-approximation}. Let $K$ be the closed unit ball of $Y^*$, equipped with the weak-star topology. Define
\begin{equation*}
        F_n\colon K\to X^*, \qquad F_n(y^*)=T^*P_n^*y^*,
\end{equation*}
and
\begin{equation*}
        F\colon K\to X^*, \qquad F(y^*)=T^*y^*.
\end{equation*} 
Each $F_n\colon K\to X^*$ is continuous from the weak-star topology on $K$ to the norm topology on $X^*$, because $P_n$ has finite rank. Moreover, the sequence $(F_n)_{n=1}^\infty$ is uniformly bounded.

For $x^{**}\in X^{**}$ and $y^*\in K$,
\begin{equation*}
\begin{aligned}
        \langle F_n(y^*)-F(y^*),x^{**}\rangle
        &= \langle P_n^*y^*-y^*,T^{**}x^{**}\rangle \\
        &= \langle y^*,(P_n^{**}-I_{Y^{**}})T^{**}x^{**}\rangle .
\end{aligned}
\end{equation*}
Since $T$ is weakly compact, $T^{**}x^{**}$ belongs to the canonical copy of $Y$ in $Y^{**}$. Therefore $P_nT^{**}x^{**}\to T^{**}x^{**}$ in norm, and so
\begin{equation*}
        F_n(y^*)\xrightarrow{w}F(y^*)\qquad (y^*\in K).
\end{equation*}
By \Cref{p5:lem:baire-one-adjoint}, $F$ is norm-Baire-one. Applying \Cref{p5:lem:tail-convexification}, choose
\begin{equation*}
        S_k\in\operatorname{conv}\{P_n\colon n\geq k\}
\end{equation*}
such that
\begin{equation*}
        \|T^*S_k^*y^*-T^*y^*\|\to0\qquad (y^*\in B_{Y^*}).
\end{equation*}
By homogeneity this convergence holds for every $y^*\in Y^*$. Set
\begin{equation*}
        A_k=S_kT .
\end{equation*}
Then each $A_k$ is finite-rank and $A_k^*y^*=T^*S_k^*y^*\to T^*y^*$.

Finally, $S_ky\to y$ for every $y\in Y$, because $S_k$ is a convex combination of the tail $\{P_n\colon n\geq k\}$. For $x^{**}\in X^{**}$, weak compactness of $T$ gives $T^{**}x^{**}\in Y$. Since $S_k^{**}$ agrees with $S_k$ on the canonical copy of $Y$ in $Y^{**}$, and since $S_k\to I_Y$ strongly, we have
\begin{equation*}
        A_k^{**}x^{**}=S_k^{**}T^{**}x^{**}=S_kT^{**}x^{**}\to T^{**}x^{**}.
\end{equation*}
This finishes the proof.
\end{proof}

\section{The bridge theorem}
\label{p5:sec:bridge}

We now isolate the main factorization step. The previous results produce finite-rank approximants whose adjoints converge pointwise in norm, both on $Y^*$ and after passing to second adjoints on $X^{**}$. The bridge theorem shows that this two-sided approximation is already enough to force a genuine factorization of $T$ through a reflexive Banach space with a Schauder basis. In this way the analytic approximation information is converted into the structural conclusion needed below.

\begin{proposition}[Bridge theorem]
\label{p5:prop:bridge}
Let $T\colon X\to Y$ be a bounded operator. Suppose that there are finite-rank operators $A_n\colon X\to Y$ such that
\begin{equation}
\label{p5:eq:bridge-adjoint-convergence}
        A_n^*y^*\to T^*y^* \qquad (y^*\in Y^*)
\end{equation}
and
\begin{equation}
\label{p5:eq:bridge-biadjoint-convergence}
        A_n^{**}x^{**}\to T^{**}x^{**} \qquad (x^{**}\in X^{**})
\end{equation}
in norm. Then $T$ factors through a reflexive Banach space with a Schauder basis.
\end{proposition}

\begin{proof}
By \eqref{p5:eq:bridge-biadjoint-convergence} and uniform boundedness, we have
\begin{equation*}
        M=\max\{\|T\|,\sup_{n\in\mathbb N}\|A_n\|\}
        =\max\{\|T\|,\sup_{n\in\mathbb N}\|A_n^{**}\|\}<\infty .
\end{equation*}
Applying \eqref{p5:eq:bridge-biadjoint-convergence} to the canonical image of $X$ in $X^{**}$ shows that $A_nx\to Tx$ for every $x\in X$. It also shows that $T$ is weakly compact, since $T^{**}x^{**}$ is a norm limit in $Y^{**}$ of vectors in the canonical copy of $Y$.

Let
\begin{equation*}
        E=c(Y)=\{(y_n)_{n=1}^{\infty}\colon (y_n)_{n=1}^{\infty}\text{ converges in norm in }Y\}
\end{equation*}
with the supremum norm. Define
\begin{equation*}
        R\colon X\to E,\qquad Rx=(A_nx)_{n=1}^{\infty},
\end{equation*}
and
\begin{equation*}
        L\colon E\to Y,\qquad L((y_n)_{n=1}^{\infty})=\lim_{n\to\infty}y_n .
\end{equation*}
Then $T=LR$.

We first prove that $R$ is weakly compact. Every functional $\Phi\in E^*$ has a representation
\begin{equation*}
        \Phi((z_n)_{n=1}^{\infty})=\eta^*\left(\lim_{n\to\infty} z_n\right)+\sum_{n=1}^{\infty} y_n^*(z_n),
\end{equation*}
where $\eta^*\in Y^*$ and $(y_n^*)_{n=1}^{\infty}\in\ell_1(Y^*)$. This follows from the isomorphism
\begin{equation*}
        c(Y)\cong c_0(Y)\oplus_{\infty}Y,\qquad (z_n)_{n=1}^{\infty}\mapsto \left((z_n-\lim_{m\to\infty} z_m)_{n=1}^{\infty},\lim_{m\to\infty} z_m\right).
\end{equation*}

For $x^{**}\in X^{**}$, \eqref{p5:eq:bridge-biadjoint-convergence} shows that the sequence $(A_n^{**}x^{**})_{n=1}^{\infty}$ belongs to $E$ and has limit $T^{**}x^{**}$. Thus we may define
\begin{equation*}
        \widehat R\colon X^{**}\to E,\qquad \widehat R x^{**}=(A_n^{**}x^{**})_{n=1}^{\infty}.
\end{equation*}
We claim that $\widehat R$ represents $R^{**}$ inside the canonical copy of $E$ in $E^{**}$. Indeed, for $\Phi\in E^*$ and $x^{**}\in X^{**}$,
\begin{equation*}
        \Phi(\widehat R x^{**})=\eta^*(T^{**}x^{**})+\sum_{n=1}^{\infty} y_n^*(A_n^{**}x^{**})=x^{**}(R^*\Phi).
\end{equation*}
The right-hand side is precisely $(R^{**}x^{**})(\Phi)$, while the left-hand side is $(J_E\widehat R x^{**})(\Phi)$, where $J_E\colon E\to E^{**}$ is the canonical embedding. Hence
\begin{equation*}
        R^{**}x^{**}=J_E\widehat R x^{**}\qquad (x^{**}\in X^{**}).
\end{equation*}
Thus $R^{**}(X^{**})\subset J_E(E)$, and therefore $R$ is weakly compact.

For $m\in\mathbb N$, define
\begin{equation*}
        Q_m\colon E\to E,\qquad Q_m(y_1,y_2,\ldots)=(y_1,\ldots,y_m,y_m,y_m,\ldots).
\end{equation*}
Then
\begin{equation*}
        \|Q_m\|\leq1, \qquad\text{and} \qquad Q_mQ_k=Q_{\min(m,k)}.
\end{equation*}
Moreover, if $z=(y_n)_{n=1}^{\infty}\in E$, then
\begin{equation*}
        \|Q_mz-z\|_E=\sup_{n>m}\|y_m-y_n\|\to0
\end{equation*}
as $m\to\infty$, because $(y_n)_{n=1}^{\infty}$ is norm-convergent in $Y$.

Let
\begin{equation*}
        W=\overline{R(B_X)}^{\,w}\subset E .
\end{equation*}
Since $R$ is weakly compact, $W$ is weakly compact. We record the key shrinking estimate. For $x\in X$,
\begin{equation*}
        (I-Q_m)Rx=(0,\ldots,0,(A_{m+1}-A_m)x,(A_{m+2}-A_m)x,\ldots),
\end{equation*}
and this sequence has limit $(T-A_m)x$. Thus, if $\Phi\in E^*$ is represented as above then
\begin{equation*}
        R^*(I-Q_m^*)\Phi=(T^*-A_m^*)\eta^*+\sum_{n=m+1}^{\infty}(A_n^*-A_m^*)y_n^* .
\end{equation*}
Equivalently,
\begin{equation*}
        \sup_{v\in W}|\Phi(Q_mv-v)|\to0\qquad (\Phi\in E^*).
\end{equation*}

Define $K$ to be the norm-closed absolutely convex hull of the sets $Q_mW$, that is
\begin{equation*}
        K=\overline{\operatorname{aco}\bigl(\bigcup_{m=1}^{\infty}Q_mW\bigr)}^{\|\cdot\|_E}.
\end{equation*}
We claim that $K$ is weakly compact. By the Krein--Smulian theorem, it suffices to show that
\begin{equation*}
        \bigcup_{m=1}^{\infty}Q_mW
\end{equation*}
is relatively weakly compact. By the Eberlein--Smulian theorem, it is enough to prove that this union is relatively weakly sequentially compact.

Take a sequence $(Q_{m_j}w_j)_{j=1}^{\infty}\subseteq \bigcup_{m=1}^{\infty}Q_mW$, with $w_j\in W$. Passing to a subsequence, we may suppose that either $(m_j)_{j=1}^{\infty}$ is constant or $m_j\to\infty$.

Suppose first that $m_j=m$ for all $j$. Then $(Q_mw_j)_{j=1}^{\infty}$ lies in the weakly compact set $Q_mW$, and hence has a weakly convergent subsequence.

It remains to consider the case $m_j\to\infty$. Since $W$ is weakly compact, by passing to a further subsequence we may suppose that $w_j\to w$ weakly in $W$. For every $\Phi\in E^*$,
\begin{equation*}
        |\Phi(Q_{m_j}w_j-w_j)|\leq \sup_{v\in W}|\Phi(Q_{m_j}v-v)|\to0,
\end{equation*}
and therefore $Q_{m_j}w_j\to w$ weakly. This proves weak compactness of $K$.

We claim that $W\subset K$ and that each $Q_m$ leaves $K$ invariant. For the first part, note that if $w\in W$, then $Q_mw\in \bigcup_{k=1}^{\infty}Q_kW$ for every $m\in\mathbb N$, and $Q_mw\to w$ in norm. Hence $w\in K$.

For the invariance assertion, note that for $w\in W$ and $k,m\in\mathbb N$,
\begin{equation*}
        Q_m(Q_kw)=Q_{\min(m,k)}w\in \bigcup_{\ell=1}^{\infty}Q_\ell W.
\end{equation*}
Thus $Q_m$ maps the generating set $\bigcup_{k=1}^{\infty}Q_kW$ into itself. By linearity, it maps its absolutely convex hull into itself, and by norm-continuity it maps the norm closure of that hull, namely $K$, into $K$.

Apply \Cref{p5:thm:dfjp-interpolation} to $K$. For $j\in\mathbb Z$, set
\begin{equation*}
        U_j=2^jK+2^{-j}B_E,
\end{equation*}
let $|\cdot|_j$ be the Minkowski functional of $U_j$, and define
\begin{equation*}
        Z=\left\{z\in E\colon \|z\|_Z^2=\sum_{j\in\mathbb Z}|z|_j^2<\infty\right\}.
\end{equation*}
By \Cref{p5:thm:dfjp-interpolation}, $Z$ is reflexive, the inclusion $J\colon Z\to E$ is bounded, and $K$ is bounded as a subset of $Z$. Since $R(B_X)\subset W\subset K$, the map $R\colon X\to E$ lifts to a bounded operator
\begin{equation*}
        \widetilde R\colon X\to Z,\qquad J\widetilde R=R .
\end{equation*}
Therefore $T=LJ\widetilde R$, so $T$ already factors through the reflexive space $Z$.

It remains to show that $Z$ has the bounded approximation property. Since $Q_mK\subset K$ and $\|Q_m\|\leq1$ on $E$, we have
\begin{equation*}
        Q_mU_j\subset U_j \qquad (j\in\mathbb Z).
\end{equation*}
Therefore $Q_m$ is contractive for each gauge $|\cdot|_j$, and hence restricts to a contraction on $Z$.

We now show that $Q_mz\to z$ in the $Z$-norm as $m\to\infty$. For fixed $j\in\mathbb Z$ and $z\in Z$, we have
\begin{equation*}
        |(I-Q_m)z|_j\leq 2^j\|(I-Q_m)z\|_E\to0
\end{equation*}
as $m\to\infty$. On the other hand, since $Q_m$ is contractive for each gauge $|\cdot|_j$, the triangle inequality gives
\begin{equation*}
        |(I-Q_m)z|_j\leq |z|_j+|Q_mz|_j\leq2|z|_j.
\end{equation*}
Since $z\in Z$, the sequence $(|z|_j)_{j\in\mathbb Z}$ belongs to $\ell_2(\mathbb Z)$. Thus dominated convergence in the $\ell_2(\mathbb Z)$-sum gives
\begin{equation*}
        \|Q_mz-z\|_Z^2=\sum_{j\in\mathbb Z}|(I-Q_m)z|_j^2\to0
\end{equation*}
as $m\to\infty$. Hence $Q_mz\to z$ in $Z$ for every $z\in Z$.

We now show that $Q_m|_Z$ has finite rank. Let $H$ be the norm closure in $E$ of the linear span of $K$, that is
\begin{equation*}
        H=\overline{\operatorname{span} K}^{\|\cdot\|_E}.
\end{equation*}
First, we claim that $Z\subset H$. Indeed, let $z\in Z$ and let $q\colon E\to E/H$ be the quotient map. Since $q(K)=0$, we have $q(U_j)\subset 2^{-j}B_{E/H}$ for every $j\in\mathbb Z$. Hence
\begin{equation*}
        |z|_j\geq 2^j\|qz\| \qquad (j\in\mathbb Z).
\end{equation*}
If $qz\ne0$, then
\begin{equation*}
        \sum_{j\in\mathbb Z}|z|_j^2\geq \sum_{j\geq0}|z|_j^2\geq \|qz\|^2\sum_{j\geq0}4^j=\infty,
\end{equation*}
contradicting $z\in Z$. Therefore $qz=0$, and hence $z\in H$.

Fix $m\in\mathbb N$ and put
\begin{equation*}
        F_m=\sum_{i=1}^m Q_iR(X).
\end{equation*}
Each $Q_iR(X)$ is finite-dimensional, because
\begin{equation*}
        Q_iRx=(A_1x,\ldots,A_ix,A_ix,A_ix,\ldots)
\end{equation*}
and $A_1,\ldots,A_i$ have finite-dimensional ranges. Hence $F_m$ is finite-dimensional. Moreover, since $Q_iR(X)$ is finite-dimensional, it is weakly closed in $E$. Since $Q_i$ is weak-to-weak continuous and $W=\overline{R(B_X)}^{\,w}$, we have
\begin{equation*}
        Q_iW\subset \overline{Q_iR(B_X)}^{\,w}\subset \overline{Q_iR(X)}^{\,w}=Q_iR(X) \qquad (1\leq i\leq m).
\end{equation*}
Hence, for $w\in W$ and $k\geq1$, we have
\begin{equation*}
        Q_mQ_kw=Q_{\min(m,k)}w\in F_m,
\end{equation*}
because $\min(m,k)\leq m$ and $Q_{\min(m,k)}w\in Q_{\min(m,k)}R(X)\subset F_m$.

Since $Q_mQ_kw\in F_m$ for every $w\in W$ and $k\geq1$, and since $F_m$ is a linear subspace, $Q_m$ maps the absolutely convex hull of $\bigcup_{k=1}^{\infty}Q_kW$ into $F_m$. Because $F_m$ is finite-dimensional, it is norm closed in $E$; hence, by the definition of $K$ as the norm-closed absolutely convex hull,
\begin{equation*}
        Q_mK\subset F_m.
\end{equation*}
Now $Z\subset H=\overline{\operatorname{span}K}^{\|\cdot\|_E}$. Therefore, using the norm-continuity of $Q_m$ on $E$,
\begin{equation*}
        Q_mZ\subseteq Q_mH \subseteq \overline{\operatorname{span}(Q_mK)}^{\|\cdot\|_E}\subseteq F_m.
\end{equation*}
Thus $Q_m|_Z$ has range contained in the finite-dimensional space $F_m$, so $Q_m|_Z$ is finite-rank.

The finite-rank contractions $Q_m|_Z$ converge strongly to the identity on $Z$. Thus $Z$ has the metric approximation property. It is also separable, because $\bigcup_{m=1}^{\infty}Q_mZ$ is dense in $Z$ and each $Q_mZ$ is finite-dimensional. Put $Q_0=0$ and define
\begin{equation*}
        E_m=(Q_m-Q_{m-1})Z \qquad (m\in\mathbb N).
\end{equation*}
Since $Q_mQ_k=Q_{\min(m,k)}$ and $Q_mz\to z$ for every $z\in Z$, the sequence $(E_m)_{m=1}^{\infty}$ is a finite-dimensional decomposition of $Z$.

By \Cref{p5:thm:jrz-basisification}, there is a reflexive Banach space $V$ with a Schauder basis such that $Z$ is isomorphic to a complemented subspace of $V$. Let $i\colon Z\to V$ be the embedding and let $P\colon V\to iZ$ be a bounded projection. Define
\begin{equation*}
        B\colon X\to V,\qquad B=i\widetilde R,
\end{equation*}
and
\begin{equation*}
        C\colon V\to Y,\qquad C=(LJ)i^{-1}P.
\end{equation*}
Then $CB=T$. Hence $T$ factors through the reflexive Banach space $V$ with a Schauder basis.
\end{proof}

\section{Proof of the main theorem}
\label{p5:sec:conclusion}

We now put the preceding ingredients together. The approximation result supplies the two-sided finite-rank approximants required by the bridge theorem, and the bridge theorem converts them into the desired factorization.

\begin{proof}[Proof of \Cref{p5:thm:main}]
By \Cref{p5:prop:two-sided-approximants}, there are finite-rank operators $A_n\colon X\to Y$ such that
\begin{equation*}
        A_n^*y^*\to T^*y^* \qquad (y^*\in Y^*)
\end{equation*}
and
\begin{equation*}
        A_n^{**}x^{**}\to T^{**}x^{**} \qquad (x^{**}\in X^{**})
\end{equation*}
in norm. Thus the hypotheses of \Cref{p5:prop:bridge} are satisfied, and thus we obtain a factorization of $T$ through a reflexive Banach space with a Schauder basis, finishing the proof.
\end{proof}

\end{problempaperbody}

\newpage

\clearpage
\providecommand{\PthreeD}{\mathcal D}
\providecommand{\PthreeL}{\mathcal B}
\providecommand{\PthreeE}{\mathbb E}
\providecommand{\PthreeBV}{\operatorname{BV}_{\mathrm{br}}}
\providecommand{\PthreeTV}{\operatorname{TV}}
\providecommand{\Pthreeone}{\mathbf 1}
\providecommand{\Pthreeweak}{\mathrm{weak}}
\providecommand{\Pthreeip}[2]{\langle #1,#2\rangle}
\providecommand{\PthreeId}{\operatorname{Id}}

\providecommand{\crefname}[3]{}
\providecommand{\Crefname}[3]{}
\makeatletter
\@ifundefined{p3claim}{\newtheorem{p3claim}{Claim}}{}
\makeatother
\crefname{p3claim}{Claim}{Claims}
\Crefname{p3claim}{Claim}{Claims}

\problempaper{\texorpdfstring{Problem 5. Primariness of $L_p(L_1)$}{Problem 5. Primariness of Lp(L1)}}

\begin{problemabstract}
We prove that $L_p(L_1)$ has the uniform primary factorization property, and thus it is primary for every $1<p<\infty$. We also prove the corresponding scalar-compression result for the doubly cancellative product-Haar part of $L_p(L_1)$.
\end{problemabstract}

\problemcontents

\begin{problempaperbody}

\section{Main theorem, proof structure, and notation}
\label{p3:sec:statement}

Recall that a Banach space $Y$ has the \emph{uniform primary factorization property} if there is a constant $K\geq1$ such that, for every $T\in\PthreeL(Y)$, the identity on $Y$ factors through either $T$ or $\PthreeId_Y-T$ with factorization constant at most $K$. That is, for every $T\in\PthreeL(Y)$, there are $A,B\in\PthreeL(Y)$ with $\|A\|\|B\|\leq K$ such that
\begin{equation*}
        ATB=\PthreeId_Y
        \qquad\text{or}\qquad
        A(\PthreeId_Y-T)B=\PthreeId_Y.
\end{equation*}

Our main result is the following.

\begin{theorem}[Main theorem]
\label{p3:thm:main}
For every $1<p<\infty$, the space $X=L_p(L_1)$ has the uniform primary factorization property. In particular, $X$ is primary.
\end{theorem}

In fact, \Cref{p3:thm:main} is obtained as a consequence of a stronger factorisation result on the doubly cancellative product-Haar part $X_{00}$. We prove that for every operator $T\in\PthreeL(X_{00})$ and every $\varepsilon>0$, there are operators $A,B\in\PthreeL(X_{00})$ and a scalar $c\in\mathbb R$ such that
\begin{equation*}
        AB=\PthreeId_{X_{00}},
        \qquad
        \|ATB-c\PthreeId_{X_{00}}\|<\varepsilon .
\end{equation*}
This scalar-compression statement is then applied to the cancellative compression of an arbitrary projection on $L_p(L_1)$, and the usual Pe{\l}czy\'nski decomposition argument gives primarity.

\subsection{Organization of the proof}
\label{p3:sec:proof-organization}

The main step is to show that every operator $T\in\PthreeL(X_{00})$ can be reduced, after passing to suitable faithful product Haar systems, to a product Haar multiplier. This is the key reduction identified by Lechner, Motakis, M\"uller, and Schlumprecht \cite{LechnerMotakisMullerSchlumprecht2024}: once arbitrary operators can be compressed to multipliers, the scalar compression theorem for multipliers can be applied. Thus the central task of the proof is to construct faithful Haar systems for which all off diagonal coefficients of the compressed operator become small.

We introduce the doubly cancellative subspace $X_{00}$ in \Cref{p3:sec:notation}, and record the one and two parameter multiplier facts needed later in \Cref{p3:sec:multiplier-terminology,p3:sec:su-result,p3:sec:lmms-result}. The first new ingredient is the local $L_1$ trace theorem, proved in \Cref{p3:sec:local-trace}. It assigns to every operator on $L_1^0$ a limiting averaged Haar diagonal on each finite equal-dyadic set, and represents these limits by an $L_\infty$ density. We then introduce faithful outer and inner Haar systems in \Cref{p3:sec:faithful-systems}. The associated exterior maps allow us to compress an operator to a faithful product Haar copy while keeping a left inverse.

The construction of the faithful systems is carried out in \Cref{p3:sec:multiplier-reduction-construction}. Before that, the finite one step tools needed for the construction are established in \Cref{p3:sec:finite-one-step-constructions}, with the outer and inner versions separated in \Cref{p3:sec:outer-one-step-construction,p3:sec:inner-one-step-construction}. The construction itself is an alternating outer and inner tree construction. At each half stage, signs and depths are chosen so as to control three types of off diagonal coefficients: same outer strips, same inner strips, and mixed coefficients. The finite choices required at the two half stages are isolated as the technical claims in \Cref{p3:sec:technical-claims-multiplier-reduction}. Once the limiting faithful systems are built, the off diagonal part of the compressed operator is the sum of three controlled pieces, while the remaining diagonal part is a bounded product Haar multiplier. This proves the reduction theorem, \Cref{p3:thm:reduction-to-multiplier}.

Finally, the LMMS scalar compression theorem is applied to the resulting product Haar multiplier on $X_{00}$. This gives scalar compression on the doubly cancellative part. The passage from $X_{00}$ back to the full mixed norm space is then obtained using a complemented copy of $L_p(L_1)$ inside $X_{00}$, and Pe{\l}czy\'nski's decomposition method yields primarity.

\section{Notation and preliminary results}
\label{p3:sec:preliminary-results}

We adhere to standard notation and record below the conventions needed throughout the proof. For notational convenience, all Banach spaces are over the real scalar field, and all operators are assumed to be bounded and linear unless explicitly stated otherwise. The complex case is obtained in the usual way, by inserting conjugates in coefficient pairings and using the complex versions of the multiplier results quoted below.

\subsection{Notation and basic definitions}
\label{p3:sec:notation}

Here and below, $L_p(L_1)$ denotes the Bochner space $L_p([0,1];L_1[0,1])$, and we write $X=L_p(L_1)$. For nonnegative quantities $A$ and $B$, we write $A\lesssim B$ if $A\leq CB$ for an absolute constant $C$, and $A\approx B$ if $A\lesssim B$ and $B\lesssim A$. Dependence of the implicit constant on additional parameters is indicated by subscripts; for example, $A\lesssim_p B$ allows the constant to depend on $p$.

Let $\PthreeD$ be the collection of dyadic intervals in $[0,1)$. For $m\geq 0$, we write
\begin{equation*}
        \PthreeD_m=\{I\in\PthreeD\colon |I|=2^{-m}\}
\end{equation*}
for the dyadic intervals at level $m$. If $B\in\PthreeD$ and $n\geq 0$, we write
\begin{equation*}
        \PthreeD_n(B)=\{J\in\PthreeD\colon J\subset B,\ |J|=2^{-n}|B|\}
\end{equation*}
for the dyadic descendants of $B$ at relative depth $n$.

A measurable set $B\subset[0,1)$, together with a specified integer $m_B\geq 0$, is called a \emph{finite equal-dyadic set} if $B$ is a finite disjoint union of intervals in $\PthreeD_{m_B}$. Thus
\begin{equation*}
        B=\bigsqcup_{\ell=1}^r B_\ell,\qquad B_\ell\in\PthreeD_{m_B}.
\end{equation*}
The integer $m_B$ is called the \emph{decomposition level} of $B$. We usually suppress $m_B$ from the notation and write simply $B$, but the decomposition level is always recorded implicitly. We extend the descendant notation relative to this specified level by
\begin{equation*}
        \PthreeD_n(B)=\bigcup_{\ell=1}^r\PthreeD_n(B_\ell)=\{J\in\PthreeD_{m_B+n}:J\subset B\}\qquad(n\geq 0).
\end{equation*}
 
For $I\in\PthreeD$, let $I^+$ and $I^-$ denote the two dyadic children of $I$, chosen so that the usual sign-valued Haar function supported on $I$ satisfies
\begin{equation*}
        h_I=\Pthreeone_{I^+}-\Pthreeone_{I^-}.
\end{equation*}
For $1\leq q<\infty$, we write
\begin{equation*}
        L_q^0[0,1]=\overline{\operatorname{span}\{h_I\colon I\in\PthreeD\}}^{\|\cdot\|_{L_q}}=\left\{f\in L_q[0,1]\colon \int_0^1 f=0\right\}.
\end{equation*}
When no confusion is possible, we write $L_q^0$.

If $f$ and $g$ are functions on $[0,1]$, then
\begin{equation*}
        (f\otimes g)(s,t)=f(s)g(t)\qquad (s,t\in[0,1]).
\end{equation*}
In particular, $h_I\otimes h_J$ denotes the function $(s,t)\longmapsto h_I(s)h_J(t)$. We write
\begin{equation*}
        X_{00}=\overline{\operatorname{span}}\{h_I\otimes h_J\colon I,J\in\PthreeD\}
        \subset X .
\end{equation*}
This is the doubly cancellative product-Haar part of $X$.

We use the normalized vectors
\begin{equation*}
        e_I=|I|^{-1/p}h_I,\qquad e_I^*=|I|^{-1/p'}h_I
\end{equation*}
in the outer coordinate, and
\begin{equation*}
        f_J=|J|^{-1}h_J,\qquad f_J^*=h_J
\end{equation*}
in the inner coordinate. Thus
\begin{equation*}
        u_{I,J}=e_I\otimes f_J,\qquad u_{I,J}^*=e_I^*\otimes f_J^*
\end{equation*}
is a biorthogonal product-Haar system on the algebraic product Haar span.

\begin{lemma}[The doubly cancellative projection]
\label{p3:lem:x00-projection}
For $f\in X=L_p(L_1)$, define
\begin{equation*}
        (\mathbb E_2f)(s,t)=\int_0^1 f(s,u)\,du
\end{equation*}
and
\begin{equation*}
        (\mathbb E_1f)(s,t)=\int_0^1 f(r,t)\,dr,
\end{equation*}
where the second integral is a Bochner integral in $L_1[0,1]$.  Then $\mathbb E_1$ and $\mathbb E_2$ are commuting contractive projections on $X$, and
\begin{equation*}
        P_{00}=(\PthreeId-\mathbb E_1)(\PthreeId-\mathbb E_2)
\end{equation*}
is a projection with $\|P_{00}\|\leq4$.  Moreover,
\begin{equation*}
        X_{00}=P_{00}X
        =\{f\in X:\mathbb E_1f=0\text{ and }\mathbb E_2f=0\}.
\end{equation*}
In particular, if $x\in L_p^0$ and $g\in L_1^0$, then $x\otimes g\in X_{00}$.
\end{lemma}

\begin{proof}
For almost every $s$,
\begin{equation*}
        \|\mathbb E_2f(s,\cdot)\|_1
        =\left|\int_0^1 f(s,u)\,du\right|
        \leq \|f(s,\cdot)\|_1,
\end{equation*}
so $\mathbb E_2$ is contractive on $X$.  Also,
\begin{equation*}
        \|\mathbb E_1f\|_{L_p(L_1)}
        =\left\|\int_0^1 f(r,\cdot)\,dr\right\|_1
        \leq\int_0^1\|f(r,\cdot)\|_1\,dr
        \leq\|f\|_{L_p(L_1)},
\end{equation*}
because the outer measure is one.  Thus $\mathbb E_1$ is contractive.  Clearly $\mathbb E_1^2=\mathbb E_1$ and $\mathbb E_2^2=\mathbb E_2$, and the two projections commute by Fubini's theorem. Hence $P_{00}$ is a bounded projection and $\|P_{00}\|\leq4$. 

The range of $P_{00}$ is exactly the intersection of the kernels of $\mathbb E_1$ and $\mathbb E_2$.  Every product Haar function $h_I\otimes h_J$ belongs to this intersection, so $X_{00}\subset P_{00}X$.

Conversely, let $f\in P_{00}X$.  Choose dyadic rectangular simple functions $g_n$ with $g_n\to f$ in $L_p(L_1)$.  Since $P_{00}$ is bounded, $P_{00}g_n\to f$.  It is enough to check that $P_{00}g_n$ belongs to the closed span of the product Haar functions.  By linearity it suffices to consider a rectangle $\Pthreeone_A\otimes\Pthreeone_B$, where $A$ and $B$ are dyadic intervals.  Then
\begin{equation*}
        P_{00}(\Pthreeone_A\otimes\Pthreeone_B)
        =(\Pthreeone_A-|A|\Pthreeone_{[0,1)})
        \otimes
        (\Pthreeone_B-|B|\Pthreeone_{[0,1)}).
\end{equation*}
Each factor is a dyadic step function of integral zero and hence lies in the finite linear span of the one-parameter Haar functions.  Therefore the tensor above lies in the algebraic span of $\{h_I\otimes h_J:I,J\in\PthreeD\}$.  This proves $P_{00}X\subset X_{00}$.

Finally, if $x\in L_p^0$ and $g\in L_1^0$, then both one-variable means vanish, so $x\otimes g\in P_{00}X=X_{00}$.
\end{proof}

\subsection{Multiplier terminology}
\label{p3:sec:multiplier-terminology}

We shall need the following definition.

\begin{definition}
Let $a=(a_I)_{I\in\PthreeD}$ be a scalar family. The \emph{scalar Haar multiplier with symbol $a$} is the linear map
\begin{equation*}
        M_a\colon \operatorname{span}\{h_I\colon I\in\PthreeD\}\longrightarrow
        \operatorname{span}\{h_I\colon I\in\PthreeD\}
\end{equation*}
defined by
\begin{equation*}
        M_a h_I=a_Ih_I\qquad (I\in\PthreeD).
\end{equation*}
We say that $M_a$ is \emph{bounded on $L_q^0$} if it extends to a bounded operator $L_q^0\to L_q^0$; in that case we use the same symbol for the extension, $1 \leq q < \infty$.

Similarly, for a scalar array $d=(d_{I,J})_{I,J\in\PthreeD}$, the \emph{product Haar multiplier with symbol $d$} is the linear map
\begin{equation*}
        M_d\colon \operatorname{span}\{h_I\otimes h_J\colon I,J\in\PthreeD\}\longrightarrow
        \operatorname{span}\{h_I\otimes h_J\colon I,J\in\PthreeD\}
\end{equation*}
defined by
\begin{equation*}
        M_d(h_I\otimes h_J)=d_{I,J}h_I\otimes h_J
        \qquad (I,J\in\PthreeD).
\end{equation*}
We say that $M_d$ is \emph{bounded on $X_{00}$} if it extends to a bounded operator on $X_{00}$. Equivalently, the normalized vectors $u_{I,J}$ are eigenvectors with eigenvalues $d_{I,J}$.
\end{definition}

We have the following elementary observation.

\begin{lemma}[Bounded scalar multipliers on $L_p^0$]
\label{p3:lem:lp-haar-multiplier}
Let $1<p<\infty$. There is a constant $U_p$ such that, for every bounded scalar family $a=(a_I)_{I\in\PthreeD}$, the scalar Haar multiplier $M_a$ extends to a bounded operator on $L_p^0$ and
\begin{equation*}
        \|M_a\|_{\PthreeL(L_p^0)}\leq U_p\sup_{I\in\PthreeD}|a_I|.
\end{equation*}
\end{lemma}

\begin{proof}
This is the classical unconditionality of the Haar system in $L_p[0,1]$ for $1<p<\infty$, restricted to the closed codimension-one subspace $L_p^0$; see, for instance, \cite[Chapter~6]{AlbiacKalton2016}.
\end{proof}

We also use the factorization terminology of Lechner--Motakis--M\"uller--Schlumprecht \cite[Definition~2.2]{LechnerMotakisMullerSchlumprecht2024}. 

\begin{definition}
    Let $E$ be a Banach space, let $S,T\in\PthreeL(E)$, let $C\geq 1$, and let $\eta\geq 0$. We say that $S$ \emph{projectionally $C$-factors through $T$ with error $\eta$} if there are operators $A,B\in\PthreeL(E)$ such that
    \begin{equation*}
            AB=\PthreeId_E,\qquad
            \|S-ATB\|\leq\eta,\qquad
            \|A\|\,\|B\|\leq C .
    \end{equation*}
\end{definition}

We will recall the two main multiplier results used in the proof: the one-parameter theorem of Semenov--Uksusov \cite[Theorem~3]{SemenovUksusov2012} and the bi-parameter scalar-compression theorem of Lechner--Motakis--M\"uller--Schlumprecht \cite[Theorem~2.3]{LechnerMotakisMullerSchlumprecht2024}. We abbreviate the latter authors as LMMS.

\subsection{The one-parameter $L_1$ Haar multiplier theorem}
\label{p3:sec:su-result}

Let $M_a$ be a scalar Haar multiplier on $L_1^0[0,1]$, so that
\begin{equation*}
        M_a h_I=a_Ih_I\qquad (I\in\PthreeD).
\end{equation*}
For the coefficient family $a=(a_I)_{I\in\PthreeD}$, write
\begin{equation*}
        \|a\|_\infty=\sup_{I\in\PthreeD}|a_I|,
\end{equation*}
and define the finite-chain variation norm
\begin{equation*}
        \|a\|_W =
        \sup\left\{
        \sum_{r=0}^{m-1}|a_{I_{r+1}}-a_{I_r}|:
        I_0\supset I_1\supset\ldots\supset I_m,\ m\geq 1
        \right\}.
\end{equation*}

We start by recalling the one-parameter multiplier theorem of Semenov and Uksusov \cite[Theorem~3]{SemenovUksusov2012}, restricted to the cancellative Haar subspace $L_1^0$. This only removes the one-dimensional constant coordinate and does not change the boundedness criterion, up to absolute constants.

\begin{theorem}[Semenov--Uksusov]
\label{p3:thm:su-finite-chain}
Let $M_a$ be a scalar Haar multiplier on $L_1^0$. Then $M_a$ is bounded on $L_1^0$ if and only if
\begin{equation*}
        \|a\|_W+\|a\|_\infty<\infty .
\end{equation*}
Moreover,
\begin{equation*}
        \|M_a\|\approx \|a\|_W+\|a\|_\infty,
\end{equation*}
where $\|M_a\|$ denotes the operator norm on $L_1^0$.
\end{theorem}

The norm $\|\cdot\|_W$ is the Semenov--Uksusov variation norm, rewritten with dyadic intervals as indices. We shall use the following equivalent branch formulation.

\begin{corollary}[Branch variation form]
\label{p3:cor:su-branch}
For a scalar Haar multiplier $M_a$ on $L_1^0$, boundedness of $M_a$ is equivalent to finiteness of
\begin{equation*}
        \PthreeBV(a)=\sup_{I_0\supset I_1\supset\ldots}\left(|a_{I_0}|+\sum_{k=0}^{\infty}|a_{I_{k+1}}-a_{I_k}|\right),
\end{equation*}
where the supremum is over infinite dyadic branches starting at an arbitrary dyadic interval. Moreover,
\begin{equation*}
        \|M_a\|\approx \PthreeBV(a),
\end{equation*}
where $\|M_a\|$ denotes the operator norm on $L_1^0$.
\end{corollary}

\begin{proof}
If $\PthreeBV(a)<\infty$, then $\|a\|_\infty\leq \PthreeBV(a)$ because the branch may start at any dyadic interval.  Every finite chain can be extended to an infinite branch, so $\|a\|_W\leq \PthreeBV(a)$.

Conversely, assume $\|a\|_W+\|a\|_\infty<\infty$.  Fix an infinite branch $I_0\supset I_1\supset\ldots$.  For every $N$,
\begin{equation*}
        |a_{I_0}|+\sum_{k=0}^{N-1}|a_{I_{k+1}}-a_{I_k}|
        \leq \|a\|_\infty+\|a\|_W .
\end{equation*}
Passing to the supremum over $N$ gives
\begin{equation*}
        |a_{I_0}|+\sum_{k=0}^{\infty}|a_{I_{k+1}}-a_{I_k}|
        \leq \|a\|_\infty+\|a\|_W .
\end{equation*}
Taking the supremum over all branches yields
\begin{equation*}
        \PthreeBV(a)\leq \|a\|_\infty+\|a\|_W .
\end{equation*}
The norm equivalence follows from \Cref{p3:thm:su-finite-chain}.
\end{proof}

\subsection{The bi-parameter multiplier reduction}
\label{p3:sec:lmms-result}

We also need the following scalar compression result of Lechner--Motakis--M\"uller--Schlumprecht \cite[Theorem~2.3]{LechnerMotakisMullerSchlumprecht2024}. It says that, after passing through uniformly controlled complemented copies of $X_{00}$, every bounded product Haar multiplier is arbitrarily close to a scalar multiple of the identity.

\begin{theorem}[Lechner--Motakis--M\"uller--Schlumprecht]
\label{p3:thm:lmms}
Let $M_d$ be a bounded product Haar multiplier on $X_{00}$. Given $\eta>0$, there are
\begin{equation*}
        A_d,B_d\in\PthreeL(X_{00}),\qquad \lambda\in\mathbb R
\end{equation*}
such that
\begin{equation*}
\begin{gathered}
        A_dB_d=\PthreeId_{X_{00}}, \qquad
        \|A_d M_d B_d-\lambda\PthreeId_{X_{00}}\|<\eta,\\
        \|A_d\|\|B_d\|\leq 1+\eta .
\end{gathered}
\end{equation*}
\end{theorem}

In the notation of \cite{LechnerMotakisMullerSchlumprecht2024}, $\mathcal V(\delta^2)$ denotes the algebraic span of the bi-parameter Haar system \cite[Notation~2.1(e)]{LechnerMotakisMullerSchlumprecht2024}, and its completion in the $L^p(L^1)$ mixed norm is $X_{00}$ in our notation. By \cite[Theorem~2.10]{LechnerMotakisMullerSchlumprecht2024}, Capon's projection \cite{Capon1983} is unbounded on $L^p(L^1)$. Consequently, the $L^p(L^1)$ case of \cite[Definition~2.2(b),(c) and Theorem~2.3]{LechnerMotakisMullerSchlumprecht2024} gives precisely the scalar approximate projectional factorization stated above. Finally, passing from $h_I\otimes h_J$ to the normalized vectors $e_I\otimes f_J$ does not change the coefficient array of a diagonal multiplier.

\section{The local $L_1$ trace}
\label{p3:sec:local-trace}

The $L_1$ coordinate is the delicate one. We begin with a technical device that will be used repeatedly below, extracting from an arbitrary operator on $L_1^0$ a limiting averaged diagonal trace on every finite equal-dyadic set. We shall need the following definition.

\begin{definition}[Finite dyadic symmetry group]\label{p3:def:finite-dyadic-symmetry-group}
    Let $B$ be a dyadic interval and let $N\geq 0$. We denote by $G_N(B)$ the group of all bijections $\gamma$ of $[0,1)$ satisfying the following conditions.
    \begin{enumerate}[label=\textup{(\roman*)}]
        \item \label{p3:item:dyadic-symmetry-outside} $\gamma$ fixes $[0,1)\setminus B$ pointwise.
        \item \label{p3:item:dyadic-symmetry-levels} For each $0\leq k\leq N+1$, $\gamma$ maps every interval in $\PthreeD_k(B)$ onto an interval in $\PthreeD_k(B)$.
    \item \label{p3:item:dyadic-symmetry-terminal} If $L=[a,a+\ell)\in\PthreeD_{N+1}(B)$ and $\gamma L=L'=[b,b+\ell)$, then
    \begin{equation*}
            \gamma(t)=b+(t-a)\qquad (t\in L).
    \end{equation*}
    \end{enumerate}
    Equivalently, $G_N(B)$ is generated by independent swaps of the two children of each dyadic interval $K\in\PthreeD_k(B)$, $0\leq k\leq N$.
\end{definition}

For example, if $B=[0,1)$ and $N=1$, then $G_1(B)$ is generated by swapping the two halves of $[0,1)$ and, independently, swapping the two quarters inside each half. Hence, its elements permute the four intervals in $\PthreeD_2([0,1))$ in a way compatible with the dyadic tree structure. We will also need the following elementary observation, where we write
\begin{equation*}
        \PthreeD_{\leq N}(B)=\bigcup_{k=0}^N\PthreeD_k(B)
\end{equation*}
for the dyadic intervals in the finite tree below $B$ up to level $N$. For $\gamma\in G_N(B)$ and $J\in\PthreeD_{\leq N}(B)$, let $\varepsilon_\gamma(J)\in\{-1,1\}$ be defined by
\begin{equation*}
        U_\gamma f_J=\varepsilon_\gamma(J)f_{\gamma(J)},
\end{equation*}
where $U_\gamma x=x\circ\gamma^{-1}$.

\begin{lemma}[Sign flip in the finite dyadic tree]
\label{p3:lem:finite-tree-sign-flip}
Let $B$ be a dyadic interval, let $N\geq 0$, and let $J,L\in\PthreeD_{\leq N}(B)$ with $J\ne L$. Then there is an involution $\rho\in G_N(B)$ such that $\rho(J)=J$, $\rho(L)=L$, and
\begin{equation*}
        \varepsilon_\rho(J)\varepsilon_\rho(L)=-1.
\end{equation*}
\end{lemma}

\begin{proof}
Since $J$ and $L$ are dyadic intervals, they are either disjoint or one is properly contained in the other. If $J$ and $L$ are disjoint, let $\rho$ be the symmetry which flips the two children of $J$ and is trivial at all other nodes. Then $\rho$ fixes both $J$ and $L$ as sets, changes the sign of $f_J$, and leaves $f_L$ unchanged. Hence $\varepsilon_\rho(J)\varepsilon_\rho(L)=-1$.

If, say, $L\subsetneq J$, let $\rho$ flip the two children of $L$ and be trivial at all other nodes. Then $\rho$ fixes both $J$ and $L$ as sets, changes the sign of $f_L$, and leaves $f_J$ unchanged, because $f_J$ is constant on $L$. Again $\varepsilon_\rho(J)\varepsilon_\rho(L)=-1$. The case $J\subsetneq L$ is symmetric.
\end{proof}

We are now ready for the proof of our next result.

\begin{proposition}[Local $L_1$ trace theorem]
\label{p3:prop:local-l1-trace}
Let $R\in\PthreeL(L_1^0)$ and let $B$ be a finite equal-dyadic set. Then
\begin{equation*}
        \beta_n^B(R) = \frac1{\#\PthreeD_n(B)} \sum_{J\in\PthreeD_n(B)} \Pthreeip{f_J^*}{Rf_J}
\end{equation*}
converges as $n\to\infty$. Denote its limit by $\lambda_B(R)$. Then there is a function $\varphi_R\in L_\infty[0,1]$ with
\begin{equation*}
        \|\varphi_R\|_\infty\lesssim \|R\|
\end{equation*}
such that
\begin{equation*}
        \lambda_B(R)=\frac1{|B|}\int_B\varphi_R
\end{equation*}
for every finite equal-dyadic set $B$.
\end{proposition}

\begin{proof}
First assume that $B$ is a dyadic interval. For $N\geq 0$, let $P_N^B$ denote the projection onto the finite-dimensional subspace
\begin{equation*}
         V_N(B) = \operatorname{span}\{f_J\colon J\in\PthreeD_k(B),\ 0\leq k\leq N\}.
\end{equation*}
Equivalently, $P_N^B$ is the local martingale-difference projection onto the first $N+1$ Haar levels below $B$. Let $E_{N+1}^B$ denote conditional expectation onto the functions supported on $B$ which are constant on each interval in $\PthreeD_{N+1}(B)$. For $x\in L_1[0,1]$, this projection is given by
\begin{equation*}
        P_N^Bx=E_{N+1}^B(\Pthreeone_Bx)-\frac1{|B|}\left(\int_Bx\right)\Pthreeone_B,
\end{equation*}
thus $\|P_N^B\|\leq 2$ as an operator on $L_1$, uniformly in $B$ and $N$, and $P_N^B$ maps $L_1^0$ into $L_1^0$.
 
Let $G_N(B)$ be the finite dyadic symmetry group from \Cref{p3:def:finite-dyadic-symmetry-group}. Each $\gamma\in G_N(B)$ induces an isometry $U_\gamma$ on $L_1^0$ by $U_\gamma x=x\circ\gamma^{-1}$. Moreover, by \ref{p3:item:dyadic-symmetry-levels} and \ref{p3:item:dyadic-symmetry-terminal}, for every $J\in\PthreeD_k(B)$ with $0\leq k\leq N$, we have
\begin{equation*}
        U_\gamma f_J=\varepsilon_\gamma(J)f_{\gamma(J)}, \qquad \text{for some} \qquad \varepsilon_\gamma(J)\in\{-1,1\}.
\end{equation*}
Define the Reynolds average
\begin{equation*}
        R_{N,B}=\frac1{|G_N(B)|}\sum_{\gamma\in G_N(B)}U_\gamma^{-1}P_N^BRP_N^BU_\gamma .
\end{equation*}
Since each $U_\gamma$ is an isometry and $\|P_N^B\|\leq 2$, we have
\begin{equation*}
        \|R_{N,B}\|\leq 4\|R\|.
\end{equation*}
Moreover, $R_{N,B}$ maps $V_N(B)$ into itself and annihilates the complementary Haar levels. Fix $J,L\in\PthreeD_{\leq N}(B)$. Then
\begin{equation*}
        \Pthreeip{f_L^*}{R_{N,B}f_J}=\frac1{|G_N(B)|}\sum_{\gamma\in G_N(B)}\varepsilon_\gamma(J)\varepsilon_\gamma(L)\Pthreeip{f_{\gamma(L)}^*}{Rf_{\gamma(J)}}.
\end{equation*}
Suppose first that $J\ne L$. By \Cref{p3:lem:finite-tree-sign-flip}, choose an involution $\rho\in G_N(B)$ such that $\rho(J)=J$, $\rho(L)=L$, and
\begin{equation*}
        \varepsilon_\rho(J)\varepsilon_\rho(L)=-1.
\end{equation*}
We pair the term indexed by $\gamma$ with the term indexed by $\gamma\rho$. Since $\rho$ fixes $J$ and $L$ as sets, we have $(\gamma\rho)(J)=\gamma(J)$ and $(\gamma\rho)(L)=\gamma(L)$, while
\begin{equation*}
        \varepsilon_{\gamma\rho}(J)\varepsilon_{\gamma\rho}(L)=-\varepsilon_\gamma(J)\varepsilon_\gamma(L).
\end{equation*}
Thus the two terms cancel. Therefore
\begin{equation*}
        \Pthreeip{f_L^*}{R_{N,B}f_J}=0\qquad (J,L\in\PthreeD_{\leq N}(B),\ J\ne L).
\end{equation*}

It remains to identify the diagonal coefficients. Let $J\in\PthreeD_k(B)$ for some $0\leq k\leq N$. The orbit of $J$ under $G_N(B)$ is precisely $\PthreeD_k(B)$, with each interval occurring equally often. Hence
\begin{equation*}
        \Pthreeip{f_J^*}{R_{N,B}f_J}=\frac1{\#\PthreeD_k(B)}\sum_{M\in\PthreeD_k(B)}\Pthreeip{f_M^*}{Rf_M}=\beta_k^B(R).
\end{equation*}
Thus $R_{N,B}$ is the finite Haar multiplier with coefficient $\beta_k^B(R)$ on every node in $\PthreeD_k(B)$, $0\leq k\leq N$, and coefficient zero on Haar functions outside $\PthreeD_{\leq N}(B)$.

Regard $R_{N,B}$ as a Haar multiplier on the whole space $L_1^0$, with symbol equal to $\beta_k^B(R)$ on every node in $\PthreeD_k(B)$ for $0\leq k\leq N$, and equal to zero on all other Haar nodes. Choose any infinite dyadic branch beginning at $B$. Along this branch, the symbol has one entry jump from $0$ to $\beta_0^B(R)$, then the successive jumps
\begin{equation*}
        |\beta_1^B(R)-\beta_0^B(R)|,\ldots,|\beta_N^B(R)-\beta_{N-1}^B(R)|,
\end{equation*}
and finally the jump from $\beta_N^B(R)$ to $0$ below the truncated tree. Hence this branch has contribution
\begin{equation*}
        |\beta_0^B(R)|+\sum_{k=0}^{N-1}|\beta_{k+1}^B(R)-\beta_k^B(R)|+|\beta_N^B(R)|.
\end{equation*}
By \Cref{p3:cor:su-branch},
\begin{equation*}
        |\beta_0^B(R)|+\sum_{k=0}^{N-1}|\beta_{k+1}^B(R)-\beta_k^B(R)|+|\beta_N^B(R)|\lesssim \|R_{N,B}\|.
\end{equation*}
Since $\|R_{N,B}\|\lesssim\|R\|$, with an absolute constant independent of $B$ and $N$, we obtain
\begin{equation*}
        |\beta_0^B(R)|+\sum_{k=0}^{N-1}|\beta_{k+1}^B(R)-\beta_k^B(R)|+|\beta_N^B(R)|\lesssim\|R\|.
\end{equation*}
Hence the finite initial variations of the sequence $(\beta_n^B(R))_{n\geq 0}$ are uniformly bounded, and therefore
\begin{equation*}
        \sum_{k=0}^{\infty}|\beta_{k+1}^B(R)-\beta_k^B(R)|<\infty .
\end{equation*}
It follows that $(\beta_n^B(R))_{n\geq 0}$ is Cauchy. We write
\begin{equation*}
        \lambda_B(R)=\lim_{n\to\infty}\beta_n^B(R).
\end{equation*}
If $B_0,B_1$ are the dyadic children of $B$, then for every $n$,
\begin{equation*}
        \beta_{n+1}^B(R)=\frac12\beta_n^{B_0}(R)+\frac12\beta_n^{B_1}(R).
\end{equation*}
Passing to the limit gives
\begin{equation}
\label{p3:eq:lambda-dyadic-additivity}
        \lambda_B(R)=\frac12\lambda_{B_0}(R)+\frac12\lambda_{B_1}(R).
\end{equation}
Equivalently,
\begin{equation*}
        |B|\lambda_B(R)=|B_0|\lambda_{B_0}(R)+|B_1|\lambda_{B_1}(R).
\end{equation*}

For $m\geq 0$, define
\begin{equation*}
        \varphi_m=\sum_{B\in\PthreeD_m}\lambda_B(R)\Pthreeone_B.
\end{equation*}
By \eqref{p3:eq:lambda-dyadic-additivity}, $(\varphi_m)_{m\geq 0}$ is a dyadic martingale. Moreover,
\begin{equation*}
        \sup_{m\geq 0}\|\varphi_m\|_\infty\lesssim \|R\|.
\end{equation*}
By Banach--Alaoglu, choose a weak-star cluster point $\varphi_R\in L_\infty[0,1]=(L_1[0,1])^*$ of $(\varphi_m)_{m\geq 0}$. Fix a dyadic interval $B$. If $m$ is at least the level of $B$, then, by iterating \eqref{p3:eq:lambda-dyadic-additivity}, we have
\begin{equation*}
        \lambda_B(R)=\frac1{\#\PthreeD_{m-\ell(B)}(B)}\sum_{A\in\PthreeD_{m-\ell(B)}(B)}\lambda_A(R),
\end{equation*}
where $\ell(B)$ denotes the dyadic level of $B$. Since $\varphi_m$ is equal to $\lambda_A(R)$ on each such $A$, we get
\begin{equation*}
        \int_B\varphi_m=\sum_{A\in\PthreeD_{m-\ell(B)}(B)}|A|\lambda_A(R)=|B|\lambda_B(R).
\end{equation*}
Since weak-star convergence preserves the pairing with $\Pthreeone_B\in L_1[0,1]$, passing to the cluster point gives
\begin{equation*}
        \int_B\varphi_R=|B|\lambda_B(R).
\end{equation*}
Thus
\begin{equation*}
        \lambda_B(R)=\frac1{|B|}\int_B\varphi_R .
\end{equation*}

Finally, if $B$ is a finite equal-dyadic set, write it as a disjoint union of dyadic intervals $B_1,\ldots,B_r$ of the same length. The descendants of $B$ are the union of the descendants of the $B_\ell$, and therefore
\begin{equation*}
        \lambda_B(R) = \frac1r\sum_{\ell=1}^r\lambda_{B_\ell}(R) =
        \frac1{|B|}\int_B\varphi_R .
\end{equation*}
\end{proof}

\section{Faithful Haar systems}
\label{p3:sec:faithful-systems}

We shall replace the standard Haar system by copies supported on dyadic sets which may be geometrically scattered, but which preserve the dyadic tree structure exactly. The following definitions separate the tree structure from the normalizations used in the outer $L_p$ and inner $L_1$ coordinates.

\begin{definition}[Faithful trees]
\label{p3:def:-faithful-tree}
A \emph{rooted dyadic subtree} is a set $\mathscr F\subset\PthreeD$ such that $[0,1)\in\mathscr F$, every predecessor of every $I\in\mathscr F$ also belongs to $\mathscr F$, and, for every $I\in\mathscr F$, either both children $I^+$ and $I^-$ belong to $\mathscr F$, or neither child belongs to $\mathscr F$. If $\mathscr F$ is finite, we call it a \emph{finite rooted dyadic subtree}. We write
\begin{equation*}
        \partial\mathscr F=\{I\in\mathscr F\colon I^+\notin\mathscr F \text{ and } I^-\notin\mathscr F\}
\end{equation*}
for the \emph{terminal nodes}, or \emph{leaves}, of $\mathscr F$.

A \emph{ faithful tree indexed by $\mathscr F$} is a family $(\Gamma_I)_{I\in\mathscr F}$ of finite equal-dyadic sets such that $\Gamma_{[0,1)}=[0,1)$, $|\Gamma_I|=|I|$, and, whenever $I,I^+,I^-\in\mathscr F$, the sets $\Gamma_{I^+}$ and $\Gamma_{I^-}$ partition $\Gamma_I$ into two equal-measure subsets. If $\mathscr F$ is finite, we call $(\Gamma_I)_{I\in\mathscr F}$ a \emph{finite  faithful tree}. The \emph{terminal sets} of $(\Gamma_I)_{I\in\mathscr F}$ are the sets $\Gamma_I$ with $I\in\partial\mathscr F$.
\end{definition}

\begin{example}
A faithful tree need not be the standard dyadic tree. For instance, at the first level one may take
\begin{equation*}
        \Gamma_{[0,1)}=[0,1),
\end{equation*}
and
\begin{equation*}
        \Gamma_{[0,1/2)}=[0,1/4)\cup[1/2,3/4),\qquad
        \Gamma_{[1/2,1)}=[1/4,1/2)\cup[3/4,1).
\end{equation*}
These two sets are finite equal-dyadic sets, they partition $[0,1)$, and both have measure $1/2$, but $\Gamma_{[0,1/2)}\neq[0,1/2)$. Continuing one more level, we may split
\begin{equation*}
        \Gamma_{[0,1/2)}=\Gamma_{[0,1/4)}\sqcup\Gamma_{[1/4,1/2)}
\end{equation*}
by setting
\begin{equation*}
        \Gamma_{[0,1/4)}=[0,1/8)\cup[1/2,5/8),\qquad
        \Gamma_{[1/4,1/2)}=[1/8,1/4)\cup[5/8,3/4).
\end{equation*}
Again, both children are finite equal-dyadic sets, and both have measure $1/4$. One can split $\Gamma_{[1/2,1)}$ in the same way. Thus the indexing tree is the usual dyadic tree, but the sets $\Gamma_I$ may be scattered finite unions of equal-length dyadic intervals.
\end{example}

We shall use faithful trees in both coordinates. The underlying tree structure is the same in the outer and inner variables, but the Haar blocks are normalized according to the ambient space: the outer coordinate uses the $L_p$ normalization, while the inner coordinate uses the $L_1$ normalization.

\begin{definition}[Outer faithful Haar blocks]
\label{p3:def:outer-faithful-blocks}
Let $(\Gamma_I)_{I\in\PthreeD}$ be a faithful tree in the outer coordinate. The associated outer Haar blocks are
\begin{equation*}
        H_I=|\Gamma_I|^{-1/p}(\Pthreeone_{\Gamma_{I^+}}-\Pthreeone_{\Gamma_{I^-}}),\qquad H_I^*=|\Gamma_I|^{-1/p'}(\Pthreeone_{\Gamma_{I^+}}-\Pthreeone_{\Gamma_{I^-}}).
\end{equation*}
\end{definition}

\begin{definition}[Inner faithful Haar blocks]
\label{p3:def:inner-faithful-blocks}
Let $(\Delta_J)_{J\in\PthreeD}$ be a faithful tree in the inner coordinate. The associated inner Haar blocks are
\begin{equation*}
        K_J=|\Delta_J|^{-1}(\Pthreeone_{\Delta_{J^+}}-\Pthreeone_{\Delta_{J^-}}),\qquad K_J^*=\Pthreeone_{\Delta_{J^+}}-\Pthreeone_{\Delta_{J^-}}.
\end{equation*}
\end{definition}

The faithful trees give copies of the product Haar basis inside $X_{00}$. We shall use the following exterior maps to pass between the original product-Haar coordinates and these  copies. The map $B$ embeds the original basis into the  Haar blocks, while $A$ is the corresponding coordinate projection back onto the original basis.

\begin{lemma}[Exterior maps]
\label{p3:lem:exterior}
Let $(\Gamma_I)_{I\in\PthreeD}$ be a  faithful outer tree, with associated blocks $(H_I,H_I^*)_{I\in\PthreeD}$, and let $(\Delta_J)_{J\in\PthreeD}$ be a  faithful inner tree, with associated blocks $(K_J,K_J^*)_{J\in\PthreeD}$. Then the map
\begin{equation*}
        B u_{I,J}=H_I\otimes K_J
\end{equation*}
extends to an isometric embedding $B\colon X_{00}\to X_{00}$. Moreover, there is a contraction $A\colon X_{00}\to X_{00}$ such that $A$ restricts to the inverse isometry on $B(X_{00})$; in particular,
\begin{equation*}
        AB=\PthreeId_{X_{00}}.
\end{equation*}
For every bounded operator $T$ on $X_{00}$,
\begin{equation*}
        \Pthreeip{u_{I,J}^*}{ATB u_{I',J'}}=\Pthreeip{H_I^*\otimes K_J^*}{T(H_{I'}\otimes K_{J'})}.
\end{equation*}
\end{lemma}

\begin{proof}
Let $\Sigma_\Gamma$ and $\Sigma_\Delta$ be the $\sigma$-algebras generated by the faithful trees $(\Gamma_I)_{I\in\PthreeD}$ and $(\Delta_J)_{J\in\PthreeD}$, respectively. Let
\begin{equation*}
        U_\Gamma^{(r)}\colon L_r[0,1]\longrightarrow L_r(\Sigma_\Gamma)
\end{equation*}
be the isometric lattice embedding determined on dyadic simple functions by
\begin{equation*}
        U_\Gamma^{(r)}\Pthreeone_I=\Pthreeone_{\Gamma_I}\qquad (I\in\PthreeD),
\end{equation*}
and define $U_\Delta^{(r)}$ similarly. Let $V_\Gamma^{(r)}$ and $V_\Delta^{(r)}$ denote the inverse isometries on the corresponding ranges.

Define
\begin{equation*}
        B=U_\Gamma^{(p)}\otimes U_\Delta^{(1)}
\end{equation*}
on the algebraic product Haar span. Since $U_\Delta^{(1)}$ preserves the $L_1$-norm on each inner fibre and $U_\Gamma^{(p)}$ is induced by a measure-preserving lattice embedding in the outer variable, for every $f\in X_{00}$ we have
\begin{equation*}
     \bigl\|(U_\Gamma^{(p)}\otimes U_\Delta^{(1)})f\bigr\|_{L_p(L_1)} =\|f\|_{L_p(L_1)}.
\end{equation*}
Hence $B=U_\Gamma^{(p)}\otimes U_\Delta^{(1)}$ extends uniquely to an isometric embedding of $X_{00}$ into $X_{00}$. Moreover,
\begin{equation*}
        Bu_{I,J}=B(e_I\otimes f_J)=H_I\otimes K_J.
\end{equation*}

Let $E_{\Delta}^{(t)}$ be conditional expectation in the inner variable onto $\Sigma_\Delta$, and let $E_{\Gamma}^{(s)}$ be Bochner conditional expectation in the outer variable onto $\Sigma_\Gamma$.  Put
\begin{equation*}
        E_{\Gamma,\Delta}=E_{\Gamma}^{(s)}E_{\Delta}^{(t)}=E_{\Delta}^{(t)}E_{\Gamma}^{(s)}.
\end{equation*}
For almost every $s$,
\begin{equation*}
        \|E_{\Delta}^{(t)}f(s,\cdot)\|_1\leq\|f(s,\cdot)\|_1,
\end{equation*}
and therefore $E_{\Delta}^{(t)}$ is contractive on $L_p(L_1)$.  Also $E_{\Gamma}^{(s)}$ is the usual Bochner conditional expectation on $L_p(L_1)$ and is contractive.  Hence $E_{\Gamma,\Delta}$ is contractive on $L_p(L_1)$.

Define
\begin{equation*}
        A=(V_\Gamma^{(p)}\otimes V_\Delta^{(1)})E_{\Gamma,\Delta}.
\end{equation*}
The map $V_\Gamma^{(p)}\otimes V_\Delta^{(1)}$ is an isometry on the range of $E_{\Gamma,\Delta}$, so $A$ is contractive on $L_p(L_1)$, hence on $X_{00}$. Since $E_{\Gamma,\Delta}$ is the identity on $B(X_{00})$, the definition gives $AB=\PthreeId_{X_{00}}$.

Finally, for every $x\in X_{00}$,
\begin{equation*}
        \Pthreeip{u_{I,J}^*}{Ax}=\Pthreeip{H_I^*\otimes K_J^*}{x}.
\end{equation*}
Applying this with $x=TB u_{I',J'}=T(H_{I'}\otimes K_{J'})$ gives
\begin{equation*}
        \Pthreeip{u_{I,J}^*}{ATB u_{I',J'}}=\Pthreeip{H_I^*\otimes K_J^*}{T(H_{I'}\otimes K_{J'})}. \qedhere
\end{equation*}
\end{proof}

\section{Potential shadows and trace operators}\label{p3:sec:potential-packet-shadows}

The shadow averages used below must be defined before the future faithful tree has been chosen. We therefore define them from the ordinary dyadic descendants of a current terminal set, viewed as a finite equal-dyadic set.

Let $A\subset[0,1)$ be a finite equal-dyadic set and $m \geq 0$.
For $L\in\PthreeD_m(A)$, we use the standard normalized Haar vectors
\begin{equation*}
        e_L=|L|^{-1/p}h_L, \qquad e_L^*=|L|^{-1/p'}h_L.
\end{equation*}
For $x\in X_{00}$ define the outer coefficient
\begin{equation*}
        \pi_Lx=\int_0^1 e_L^*(s)x(s,\cdot)\,ds\in L_1^0.
\end{equation*}
Thus $\pi_Lx$ is the $e_L$-coefficient of $x$ in the outer variable, regarded as an element of the inner space.

\begin{definition}[Potential traces and shadows]
\label{p3:def:potential-packet-traces-shadows}
Let $A$ be a finite equal-dyadic set. For $Q\in\PthreeL(L_p^0)$ and $m\geq0$, define
\begin{equation*}
        \alpha_m^A(Q) = \frac1{\#\PthreeD_m(A)} \sum_{L\in\PthreeD_m(A)} \Pthreeip{e_L^*}{Qe_L}.
\end{equation*}

For $T\in\PthreeL(X_{00})$ and $m\geq0$, define $\Omega_{A,m}^T\in\PthreeL(L_1^0)$ by
\begin{equation*}
        \Omega_{A,m}^Tg = \frac1{\#\PthreeD_m(A)} \sum_{L\in\PthreeD_m(A)} \pi_L\bigl(T(e_L\otimes g)\bigr) \qquad (g\in L_1^0).
\end{equation*}
\end{definition}
Indeed, for each $L\in\PthreeD_m(A)$, the map $g\mapsto\pi_L(T(e_L\otimes g))$ is a bounded operator $L_1^0\to L_1^0$ of norm at most $\|T\|$, since $\|e_L\|_{L_p}=\|e_L^*\|_{L_{p'}}=1$. Hence the finite average defining $\Omega_{A,m}^T$ belongs to $\PthreeL(L_1^0)$ and satisfies $\|\Omega_{A,m}^T\|\leq\|T\|$.

These quantities depend only on the current terminal set $A$, viewed as a finite equal-dyadic set, and on the standard dyadic structure below it. They do not depend on any future choices in the faithful tree.

For a dyadic interval $L$ and $\sigma\in\{-1,1\}$, write $L^\sigma=L^+$ if $\sigma=1$ and $L^\sigma=L^-$ if $\sigma=-1$. If $A$ is a finite equal-dyadic set and $n\geq0$, let
\begin{equation*}
        \varepsilon=(\varepsilon_L)_{L\in\PthreeD_n(A)}\in\{-1,1\}^{\PthreeD_n(A)}
\end{equation*}
be a sign selection on the depth-$n$ descendants of $A$. Define
\begin{equation}
\label{p3:eq:random-split-sets}
        A_+(\varepsilon)=\bigcup_{L\in\PthreeD_n(A)}L^{\varepsilon_L},\qquad A_-(\varepsilon)=\bigcup_{L\in\PthreeD_n(A)}L^{-\varepsilon_L}.
\end{equation}
Both $A_+(\varepsilon)$ and $A_-(\varepsilon)$ are regarded as finite equal-dyadic sets with decomposition level $ m_{A_+(\varepsilon)}=m_{A_-(\varepsilon)} = m_A+n+1 $. The associated \emph{outer random Haar block} is
\begin{equation}
\label{p3:eq:outer-random-block-definition}
        H_A(\varepsilon)=|A|^{-1/p}\bigl(\Pthreeone_{A_+(\varepsilon)}-\Pthreeone_{A_-(\varepsilon)}\bigr), \hspace{3pt} H_A^*(\varepsilon)=|A|^{-1/p'}\bigl(\Pthreeone_{A_+(\varepsilon)}-\Pthreeone_{A_-(\varepsilon)}\bigr).
\end{equation}
The associated \emph{inner random Haar block} is
\begin{equation}
\label{p3:eq:inner-random-block-definition}
        K_A(\varepsilon)=|A|^{-1}\bigl(\Pthreeone_{A_+(\varepsilon)}-\Pthreeone_{A_-(\varepsilon)}\bigr),\qquad K_A^*(\varepsilon)=\Pthreeone_{A_+(\varepsilon)}-\Pthreeone_{A_-(\varepsilon)}.
\end{equation}
If $N=\#\PthreeD_n(A)$, then
\begin{equation}
\label{p3:eq:outer-random-block-expansion}
    H_A(\varepsilon)=N^{-1/p}\sum_{L\in\PthreeD_n(A)}\varepsilon_Le_L,\qquad H_A^*(\varepsilon)=N^{-1/p'}\sum_{L\in\PthreeD_n(A)}\varepsilon_Le_L^*.
\end{equation}
Similarly,
\begin{equation}
\label{p3:eq:inner-random-block-expansion}
        K_A(\varepsilon)=N^{-1}\sum_{L\in\PthreeD_n(A)}\varepsilon_Lf_L,\qquad K_A^*(\varepsilon)=\sum_{L\in\PthreeD_n(A)}\varepsilon_Lf_L^*.
\end{equation}
Consequently, for every $Q\in\PthreeL(L_p^0)$, expanding the random block, we have
\begin{equation*}
        \Pthreeip{H_A^*(\varepsilon)}{QH_A(\varepsilon)}=\frac1N\sum_{L,M\in\PthreeD_n(A)}\varepsilon_L\varepsilon_M\Pthreeip{e_L^*}{Qe_M}.
\end{equation*}
Since the signs $(\varepsilon_L)_{L\in\PthreeD_n(A)}$ are independent and symmetric,
\begin{equation*}
        \PthreeE_\varepsilon[\varepsilon_L\varepsilon_M]=\begin{cases}1, & L=M,\\ 0, & L\neq M.\end{cases}
\end{equation*}
Thus, all off-diagonal terms vanish after taking expectation, and only the diagonal terms remain, which gives
\begin{equation}
\label{p3:eq:outer-random-block-expectation}
        \PthreeE_\varepsilon\left[\Pthreeip{H_A^*(\varepsilon)}{QH_A(\varepsilon)}\right]=\frac1N\sum_{L\in\PthreeD_n(A)}\Pthreeip{e_L^*}{Qe_L}=\alpha_n^A(Q).
\end{equation}
Moreover, for $r\geq0$,
\begin{equation}
\label{p3:eq:child-trace-expectation}
        \PthreeE_\varepsilon\left[\alpha_r^{A_+(\varepsilon)}(Q)\right]=\PthreeE_\varepsilon\left[\alpha_r^{A_-(\varepsilon)}(Q)\right]=\alpha_{n+1+r}^A(Q).
\end{equation}
Indeed, $\alpha_r^{A_+(\varepsilon)}(Q)$ is the average, over $L\in\PthreeD_n(A)$, of the average of the diagonal coefficients below the randomly selected half $L^{\varepsilon_L}$. Averaging in the signs therefore averages over both halves of every $L$, which is precisely the average over $\PthreeD_{n+1+r}(A)$.

The next elementary estimate says that, after a sufficiently fine random split, the child traces concentrate around their expected values. This is the probabilistic result used later to preserve infinite tails through the construction.

\begin{lemma}[Variance of child traces]
\label{p3:lem:child-trace-variance}
Let $A$ be a finite equal-dyadic set, let $Q\in\PthreeL(L_p^0)$, and let $n,r\geq0$. Let
\begin{equation*}
        \varepsilon=(\varepsilon_L)_{L\in\PthreeD_n(A)}
\end{equation*}
be uniformly distributed on $\{-1,1\}^{\PthreeD_n(A)}$. For each sign $\theta\in\{+,-\}$, we have
\begin{equation}
\label{p3:eq:child-trace-variance}
        \operatorname{Var}_\varepsilon\left(\alpha_r^{A_\theta(\varepsilon)}(Q)\right)\leq\frac{\|Q\|^2}{\#\PthreeD_n(A)}.
\end{equation}
\end{lemma}
\begin{proof}
We prove the estimate for $A_+$. Write
\begin{equation*}
        \alpha_r^{A_+(\varepsilon)}(Q)=\frac1{\#\PthreeD_n(A)}\sum_{L\in\PthreeD_n(A)}X_L(\varepsilon_L),
\end{equation*}
where
\begin{equation*}
        X_L(1)=2^{-r}\sum_{M\in\PthreeD_r(L^+)}\Pthreeip{e_M^*}{Qe_M},\qquad X_L(-1)=2^{-r}\sum_{M\in\PthreeD_r(L^-)}\Pthreeip{e_M^*}{Qe_M}.
\end{equation*}
The random variables $X_L(\varepsilon_L)$ are independent as $L$ varies. Moreover, since $\|e_M\|_{L_p}=\|e_M^*\|_{L_{p'}}=1$, we have $|\Pthreeip{e_M^*}{Qe_M}|\leq\|Q\|$, and hence $|X_L(\varepsilon_L)|\leq\|Q\|$. Therefore
\begin{equation*}
        \operatorname{Var}_\varepsilon\left(\alpha_r^{A_+(\varepsilon)}(Q)\right)=\frac1{\#\PthreeD_n(A)^2}\sum_{L\in\PthreeD_n(A)}\operatorname{Var}_\varepsilon\left(X_L(\varepsilon_L)\right)\leq\frac{\|Q\|^2}{\#\PthreeD_n(A)}.
\end{equation*}
The proof for $A_-$ is identical.
\end{proof}

We will need the following well-known elementary result; we include a proof for completeness.

\begin{lemma}[$L_1$ disjointification criterion]
\label{p3:lem:l1-disjointification}
Let $\mathcal A$ be a bounded subset of $L_1$.  If $\mathcal A$ is not uniformly integrable, then there are a sequence $(f_k)_{k\geq1}$ in $\mathcal A$, pairwise disjoint measurable sets $E_k\subset[0,1]$, and a number $\delta>0$ such that
\begin{equation*}
        \int_{E_k}|f_k|>\delta\qquad(k\geq1).
\end{equation*}
\end{lemma}

\begin{proof}
Since $\mathcal A$ is not uniformly integrable, there are $\eta>0$, functions $f_n\in\mathcal A$, and measurable sets $A_n\subset[0,1]$ such that
\begin{equation*}
        \mu(A_n)<2^{-n},\qquad \int_{A_n}|f_n|>\eta\qquad(n\geq1).
\end{equation*}
We pass to a subsequence, still denoted by $(f_n,A_n)$, with the following extra property.  After $f_1,A_1,\ldots,f_k,A_k$ have been chosen, choose $\delta_i>0$ for $1\leq i\leq k$ so that
\begin{equation*}
        \mu(E)<\delta_i\quad\Longrightarrow\quad \int_E|f_i|<\eta/2
\end{equation*}
for every measurable $E$.  Since the original witnessing sets have measures tending to zero, the remaining subsequence may be chosen so that, for every $j>i$,
\begin{equation*}
        \mu(A_j)<2^{-j}\delta_i .
\end{equation*}
Consequently, for every fixed $i$,
\begin{equation*}
        \mu\left(\bigcup_{j>i}A_j\right)\leq\sum_{j>i}\mu(A_j)<\delta_i.
\end{equation*}
Now define
\begin{equation*}
        E_i=A_i\setminus\bigcup_{j>i}A_j .
\end{equation*}
The sets $E_i$ are pairwise disjoint.  Moreover,
\begin{equation*}
        \int_{E_i}|f_i|
        \geq\int_{A_i}|f_i|-\int_{\bigcup_{j>i}A_j}|f_i|
        >\eta-\eta/2=\eta/2.
\end{equation*}
Thus the conclusion holds with $\delta=\eta/2$.
\end{proof}

The next proposition is the compactness step behind the shadow construction. Although $\Omega_{A,m}^T$ is defined by averaging more and more outer dyadic descendants of $A$, its orbit on each fixed $g\in L_1^0$ remains uniformly integrable. Thus, along any free ultrafilter supported on an infinite set of depths, the weak limit is still represented by an element of $L_1^0$, rather than only by an element of $L_1^{**}$.

\begin{proposition}[Potential outer shadows are $L_1$-valued]
\label{p3:prop:shadow}
Let $A$ be a finite equal-dyadic set, let $T\in\PthreeL(X_{00})$, and let $S\subset\N$ be infinite. For every $g\in L_1^0$, the set
\begin{equation*}
        \{\Omega_{A,m}^Tg:m\in S\}
\end{equation*}
is uniformly integrable in $L_1$. Consequently, if $\mathcal U$ is a free ultrafilter containing $S$, then
\begin{equation*}
        \Omega_A^{T,\mathcal U}g = \Pthreeweak\!\lim_{m\to\mathcal U}\Omega_{A,m}^Tg
\end{equation*}
exists in $L_1^0$, and $\Omega_A^{T,\mathcal U}\colon L_1^0\to L_1^0$ is bounded with
\begin{equation*}
        \|\Omega_A^{T,\mathcal U}\|\leq \|T\| .
\end{equation*}
\end{proposition}

\begin{proof}
Fix $g\in L_1^0$ and $m\in S$, and put $N_m=\#\PthreeD_m(A)$. We record the randomization identity used below. For a sign family $\sigma=(\sigma_L)_{L\in\PthreeD_m(A)}$, set
\begin{equation*}
        x_m^\sigma=N_m^{-1/p}\sum_{L\in\PthreeD_m(A)}\sigma_Le_L,\qquad (x_m^\sigma)^*=N_m^{-1/p'}\sum_{L\in\PthreeD_m(A)}\sigma_Le_L^*.
\end{equation*}
The supports are pairwise disjoint, so $\|x_m^\sigma\|_{L_p}=1$ and $\|(x_m^\sigma)^*\|_{L_{p'}}=1$. For every $\phi\in L_\infty$,
\begin{equation}
\label{p3:eq:shadow-sign-average}
        \Pthreeip{\phi}{\Omega_{A,m}^Tg}=\PthreeE_\sigma\left[\Pthreeip{(x_m^\sigma)^*\otimes\phi}{T(x_m^\sigma\otimes g)}\right].
\end{equation}
Indeed, after expanding the right hand side, the terms with $L\ne L'$ have mean zero, while the terms with $L=L'$ give exactly the defining average of $\Omega_{A,m}^Tg$. In particular,
\begin{equation}
\label{p3:eq:shadow-uniform-l1-bound}
        \|\Omega_{A,m}^Tg\|_1\leq\|T\|\,\|g\|_1.
\end{equation}

Suppose that $\{\Omega_{A,m}^Tg:m\in S\}$ is not uniformly integrable for some $\|g\|_1=1$. By \Cref{p3:lem:l1-disjointification}, after passing to a sequence $m_1<m_2<\ldots$ in $S$, there are pairwise disjoint measurable sets $E_k\subset[0,1]$ and a number $\delta>0$ such that
\begin{equation*}
        \int_{E_k}|\Omega_{A,m_k}^Tg|>\delta\qquad(k\geq1).
\end{equation*}
Choose $\psi_k\in L_\infty$ supported on $E_k$, with $\|\psi_k\|_\infty\leq1$, such that
\begin{equation*}
        \Pthreeip{\psi_k}{\Omega_{A,m_k}^Tg}>\delta.
\end{equation*}
Fix $M\geq 1$. For each $1\leq k\leq M$, let $\sigma^k$ be an independent sign family on $\PthreeD_{m_k}(A)$, and let $\eta_1,\ldots,\eta_M$ be independent signs, independent also of all the families $\sigma^k$. Put
\begin{equation*}
        a_k(\sigma)=(x_{m_k}^{\sigma^k})^*\otimes\psi_k,\qquad b_k(\sigma)=x_{m_k}^{\sigma^k}\otimes g.
\end{equation*}
By \eqref{p3:eq:shadow-sign-average}, for each $k$ we have
\begin{equation*}
        \delta<\PthreeE_{\sigma^k}\left[\Pthreeip{a_k(\sigma)}{Tb_k(\sigma)}\right].
\end{equation*}
Summing over $k$ gives
\begin{equation*}
        \delta M<\PthreeE_\sigma\left[\sum_{k=1}^M\Pthreeip{a_k(\sigma)}{Tb_k(\sigma)}\right].
\end{equation*}
For fixed $\sigma$, averaging in the signs $\eta_1,\ldots,\eta_M$ gives
\begin{equation*}
        \PthreeE_\eta\left[\Pthreeip{\sum_{k=1}^M\eta_ka_k(\sigma)}{T\left(\sum_{\ell=1}^M\eta_\ell b_\ell(\sigma)\right)}\right]=\sum_{k,\ell=1}^M\PthreeE_\eta[\eta_k\eta_\ell]\Pthreeip{a_k(\sigma)}{Tb_\ell(\sigma)}=\sum_{k=1}^M\Pthreeip{a_k(\sigma)}{Tb_k(\sigma)}.
\end{equation*}
Therefore
\begin{equation*}
        \delta M<\PthreeE_{\sigma,\eta}\left[\Pthreeip{\sum_{k=1}^M\eta_k(x_{m_k}^{\sigma^k})^*\otimes\psi_k}{T\left(\sum_{k=1}^M\eta_kx_{m_k}^{\sigma^k}\otimes g\right)}\right].
\end{equation*}
For fixed signs, the first vector has norm at most one in $L_{p'}(L_\infty)$. Indeed, the functions $\psi_k$ have disjoint supports in the inner variable and
\begin{equation*}
        |(x_{m_k}^{\sigma^k})^*|=|A|^{-1/p'}\Pthreeone_A.
\end{equation*}
Hence, using the normalization of $g$,
\begin{equation*}
        \left|\Pthreeip{\sum_{k=1}^M\eta_k(x_{m_k}^{\sigma^k})^*\otimes\psi_k}{T\left(\sum_{k=1}^M\eta_kx_{m_k}^{\sigma^k}\otimes g\right)}\right|\leq\|T\|\left\|\sum_{k=1}^M\eta_kx_{m_k}^{\sigma^k}\right\|_{L_p}.
\end{equation*}
The Khintchine inequality, applied pointwise in the outer variable, gives
\begin{equation*}
        \PthreeE_{\sigma,\eta}\left\|\sum_{k=1}^M\eta_kx_{m_k}^{\sigma^k}\right\|_{L_p}\leq C_pM^{1/2}.
\end{equation*}
Consequently,
\begin{equation*}
        \delta M\leq C_p\|T\|M^{1/2}\qquad(M\geq1),
\end{equation*}
which is impossible for large $M$. This proves uniform integrability.

By the Dunford--Pettis criterion for $L_1$, see for example \cite[Theorem~5.2.9]{AlbiacKalton2016}, uniform integrability implies relative weak compactness in $L_1$. Hence, the ultrafilter weak limit belongs to $L_1$. Since $L_1^0$ is weakly closed and each $\Omega_{A,m}^Tg$ lies in $L_1^0$, the limit lies in $L_1^0$. Linearity is inherited from the scalar ultralimits, and \eqref{p3:eq:shadow-uniform-l1-bound} gives
\begin{equation*}
        \|\Omega_A^{T,\mathcal U}g\|_1\leq\|T\|\,\|g\|_1.
\end{equation*}
\end{proof}

The next lemma establishes the compatibility between outer-averaged diagonals and shadow operators. If the inner coordinate is fixed as $v$ and then tested against $v^*$, the averaged outer diagonal of the resulting one-parameter operator is exactly the corresponding coefficient of the shadow.

\begin{lemma}[The same-outer identity]
\label{p3:lem:same-outer-identity}
Let $A$ be a finite equal-dyadic set in the outer variable, let $T\in\PthreeL(X_{00})$, $v\in L_1^0$, and $v^*\in L_\infty$. Define
\begin{equation*}
        Q_{v^*,v}x=(\PthreeId\otimes v^*)T(x\otimes v)\qquad (x\in L_p^0).
\end{equation*}
Then $Q_{v^*,v}\in\PthreeL(L_p^0)$ and, for every $m\geq0$,
\begin{equation}
\label{p3:eq:same-outer-identity}
        \alpha_m^A(Q_{v^*,v})=\Pthreeip{v^*}{\Omega_{A,m}^Tv}.
\end{equation}
Consequently, if $\mathcal U$ is a free ultrafilter for which $\Omega_A^{T,\mathcal U}v$ is defined, then
\begin{equation*}
        \lim_{m\to\mathcal U}\alpha_m^A(Q_{v^*,v})=\Pthreeip{v^*}{\Omega_A^{T,\mathcal U}v}.
\end{equation*}
\end{lemma}

\begin{proof}
The boundedness follows from
\begin{equation*}
        \|Q_{v^*,v}x\|_p
        \leq
        \|v^*\|_\infty\|T\|\|x\|_p\|v\|_1 .
\end{equation*}
Since $T(x\otimes v)\in X_{00}$, taking the inner $v^*$-coefficient leaves an element of $L_p^0$. For $L\in\PthreeD_m(A)$,
\begin{equation*}
        \Pthreeip{e_L^*}{Q_{v^*,v}e_L} = \Pthreeip{e_L^*\otimes v^*}{T(e_L\otimes v)} = \Pthreeip{v^*}{\pi_L(T(e_L\otimes v))}.
\end{equation*}
Averaging over $L\in\PthreeD_m(A)$ gives \eqref{p3:eq:same-outer-identity}. The ultrafilter statement follows by applying $v^*$ to the weak ultralimit.
\end{proof}

We shall also need the complementary construction, where the local trace is taken in the inner $L_1$ coordinate and the outer variable is left alive. Fixing an inner finite equal-dyadic set $B$, the local $L_1$ trace theorem turns each outer matrix coefficient of $T$ into a scalar. The next lemma packages these scalars as a bounded operator on $L_p^0$.

\begin{lemma}[Outer trace operator associated with an inner set]
\label{p3:lem:inner-set-outer-trace-operator}
Let $B$ be a finite equal-dyadic set in the inner variable and let $T\in\PthreeL(X_{00})$. There is a unique operator
\begin{equation*}
        \mathcal I_B^T\in\PthreeL(L_p^0)
\end{equation*}
such that, for all $x\in L_p^0$ and all $x^*\in L_{p'}$,
\begin{equation}
\label{p3:eq:definition-of-IBT}
        \Pthreeip{x^*}{\mathcal I_B^Tx}=\lambda_B\bigl(R_{x^*,x}^T\bigr),
\end{equation}
where
\begin{equation*}
        R_{x^*,x}^Tg=\int_0^1 x^*(s)T(x\otimes g)(s,\cdot)\,ds\qquad(g\in L_1^0).
\end{equation*}
Moreover,
\begin{equation*}
        \|\mathcal I_B^T\|\lesssim_p\|T\|.
\end{equation*}
\end{lemma}

\begin{proof}
For $g\in L_1^0$,
\begin{equation*}
        \|R_{x^*,x}^Tg\|_1\leq\|x^*\|_{p'}\|T(x\otimes g)\|_{L_p(L_1)}\leq\|x^*\|_{p'}\|T\|\|x\|_p\|g\|_1.
\end{equation*}
Since $T(x\otimes g)\in X_{00}$ is cancellative in the inner variable, $R_{x^*,x}^Tg$ has inner integral zero. Hence $R_{x^*,x}^T\in\PthreeL(L_1^0)$. By \Cref{p3:prop:local-l1-trace},
\begin{equation*}
        |\lambda_B(R_{x^*,x}^T)|\lesssim\|R_{x^*,x}^T\|\leq\|x^*\|_{p'}\|T\|\|x\|_p.
\end{equation*}
Thus, for fixed $x$, the map $x^*\mapsto\lambda_B(R_{x^*,x}^T)$ is a bounded functional on $L_{p'}$. Since $1<p<\infty$, it is represented by a unique element $y\in L_p$.

It remains to check that $y$ belongs to $L_p^0$. This uses the outer cancellation. If $x^*=\Pthreeone_{[0,1)}$, then
\begin{equation*}
        R_{\Pthreeone_{[0,1)},x}^Tg=\int_0^1T(x\otimes g)(s,\cdot)\,ds=0,
\end{equation*}
because $T(x\otimes g)\in X_{00}$ has zero outer integral. Therefore $\Pthreeip{\Pthreeone_{[0,1)}}{y}=0$, so $y\in L_p^0$.

Define $\mathcal I_B^Tx=y$. Linearity follows from the linearity of $x\mapsto\lambda_B(R_{x^*,x}^T)$ and uniqueness of the representing vector. The displayed estimate gives $\|\mathcal I_B^T\|\lesssim_p\|T\|$, and uniqueness follows because an element of $L_p^0$ annihilated by all $x^*\in L_{p'}$ is zero.
\end{proof}

\section{Admissible constraints and outer tail preservation}
\label{p3:sec:admissible-constraints-and-tail-preservation}

We now isolate the side constraints that will be imposed during the random splittings introduced above. These are the coefficients that can be made small by randomness. The self-diagonal coefficients are excluded, because their averages are the trace quantities controlled separately.

In the finite splitting estimates below, where $E$ is either $L_p^0$ or $L_1^0$, we adopt a generic notation in order to treat the two normalizations simultaneously. A \emph{random normalized Haar block} means the block constructed from the split \eqref{p3:eq:random-split-sets}, with the normalization appropriate to $E$. More precisely, when $E=L_p^0$, we use the outer normalization \eqref{p3:eq:outer-random-block-definition} and write
\begin{equation*}
        u_A(\varepsilon)=H_A(\varepsilon),\qquad u_A^*(\varepsilon)=H_A^*(\varepsilon).
\end{equation*}
When $E=L_1^0$, we use the inner normalization \eqref{p3:eq:inner-random-block-definition} and write
\begin{equation*}
        u_A(\varepsilon)=K_A(\varepsilon),\qquad u_A^*(\varepsilon)=K_A^*(\varepsilon).
\end{equation*}

\begin{definition}[Admissible finite constraints]
\label{p3:def:admissible-side-constraints}
Let $E=L_p^0$ or $E=L_1^0$. Let $A_1,\ldots,A_d$ be pairwise disjoint finite equal-dyadic sets. For each $a=1,\ldots,d$, let $u_a,u_a^*$ be a random normalized Haar block supported on $A_a$, defined by an independent sign selection on sufficiently fine dyadic descendants of $A_a$.

A finite collection $\mathscr C$ of inequalities involving the random blocks $u_a,u_a^*$ is called \emph{admissible} if each element $C\in\mathscr C$ is an inequality of the form
\begin{equation*}
        |X_C|<\eta_C,
\end{equation*}
where $\eta_C>0$ is called the tolerance of $C$, and where the scalar-valued random variable $X_C$ has one of the following forms:
\begin{equation*}
        X_C=\Pthreeip{u_a^*}{S_Cz_C},\qquad X_C=\Pthreeip{z_C^*}{S_Cu_a},
\end{equation*}
or
\begin{equation*}
        X_C=\Pthreeip{u_a^*}{S_Cu_b}\qquad(a\ne b).
\end{equation*}
Here $S_C\in\PthreeL(E)$, $z_C\in E$, $z_C^*\in E^*$, the indices $a,b$, and the number $\eta_C>0$ are fixed before the random signs defining the blocks $u_a$ are chosen.
\end{definition}

The self-diagonal inequality
\begin{equation*}
        |\Pthreeip{u_a^*}{Su_a}|<\eta
\end{equation*}
is deliberately not included. Such terms are not side constraints: their expectations are averaged diagonal traces, and they are controlled separately by the trace estimates.

\section{Finite splitting estimates}
\label{p3:sec:finite-splitting}

At each finite stage of the construction, only finitely many current finite equal-dyadic sets are split. The estimates below are finite-dimensional probability estimates for the Haar blocks produced by such splits. In particular, we have the following elementary observation.

\begin{lemma}[Coordinate projections on dyadic descendants]
\label{p3:lem:packet-coordinates}
Let $E=L_p^0$, $1<p<\infty$, or $E=L_1^0$. Let $L_1,\ldots,L_N$ be pairwise disjoint dyadic intervals of the same length. In the case $E=L_p^0$, put
\begin{equation*}
        u_i=e_{L_i},\qquad u_i^*=e_{L_i}^*\qquad(1\leq i\leq N),
\end{equation*}
and in the case $E=L_1^0$, put
\begin{equation*}
        u_i=f_{L_i},\qquad u_i^*=f_{L_i}^*\qquad(1\leq i\leq N).
\end{equation*}
Define
\begin{equation*}
        Px=\sum_{i=1}^N\Pthreeip{u_i^*}{x}u_i .
\end{equation*}
Then $P\colon E\to E$ is a contraction. Moreover, the map
\begin{equation*}
        (a_i)_{i=1}^N\longmapsto \sum_{i=1}^N a_iu_i
\end{equation*}
identifies the span of $u_1,\ldots,u_N$ isometrically with $\ell_p^N$ in the $L_p$ case and with $\ell_1^N$ in the $L_1$ case.
\end{lemma}

\begin{proof}
The supports of the $u_i$ are disjoint, and the chosen normalizations give the displayed $\ell_p$ and $\ell_1$ norms by direct integration. The projection is the sum of the local Haar-coordinate projections on the intervals $L_i$. On each $L_i$ it has the form
\begin{equation*}
        x\longmapsto |L_i|^{-1}h_{L_i}\int_{L_i}h_{L_i}x .
\end{equation*}
For $E=L_p^0$, H\"older's inequality gives
\begin{equation*}
        \left\||L_i|^{-1}h_{L_i}\int_{L_i}h_{L_i}x\right\|_p\leq \|\Pthreeone_{L_i}x\|_p .
\end{equation*}
For $E=L_1^0$, the same formula gives
\begin{equation*}
        \left\||L_i|^{-1}h_{L_i}\int_{L_i}h_{L_i}x\right\|_1\leq \int_{L_i}|x|.
\end{equation*}
Since the intervals $L_i$ are disjoint, summing over $i$ gives $\|Px\|_E\leq\|x\|_E$.
\end{proof}

The next estimate says that the random self-diagonal coefficient of a new Haar block is concentrated around the average of the diagonal coefficients on the dyadic descendants from which the block is built.

\begin{lemma}[Centered quadratic estimate]
\label{p3:lem:centered}
Let $E=L_p^0$ or $E=L_1^0$, let $Q\in\PthreeL(E)$, and let $A$ be a finite equal-dyadic set. Fix $n\geq0$, write $\PthreeD_n(A)=\{L_1,\ldots,L_N\}$, and let $\varepsilon=(\varepsilon_i)_{i=1}^N$ be uniformly distributed on $\{-1,1\}^N$. Let $u(\varepsilon),u^*(\varepsilon)$ be the random normalized Haar block and its biorthogonal functional associated with $A$, $n$, and $\varepsilon$. For each $i$, let $u_i,u_i^*$ be the normalized Haar vector and biorthogonal functional associated with $L_i$, using the normalization of $E$. Then
\begin{equation*}
        \PthreeE_\varepsilon\left[\left|\Pthreeip{u^*(\varepsilon)}{Qu(\varepsilon)}-\frac1N\sum_{i=1}^N\Pthreeip{u_i^*}{Qu_i}\right|^2\right]\lesssim_p \frac{\|Q\|^2}{N}.
\end{equation*}
In the case $E=L_1^0$, the subscript $p$ is omitted.
\end{lemma}

\begin{proof}
Write $q_{ij}=\Pthreeip{u_i^*}{Qu_j}$. In the $L_p$ case,
\begin{equation*}
        u(\varepsilon)=N^{-1/p}\sum_{i=1}^N\varepsilon_iu_i,\qquad u^*(\varepsilon)=N^{-1/p'}\sum_{i=1}^N\varepsilon_iu_i^*.
\end{equation*}
In the $L_1$ case,
\begin{equation*}
        u(\varepsilon)=N^{-1}\sum_{i=1}^N\varepsilon_iu_i,\qquad u^*(\varepsilon)=\sum_{i=1}^N\varepsilon_iu_i^*.
\end{equation*}
Thus, in both cases,
\begin{equation*}
        \Pthreeip{u^*(\varepsilon)}{Qu(\varepsilon)}-\frac1N\sum_{i=1}^Nq_{ii}=\frac1N\sum_{i\ne j}\varepsilon_i\varepsilon_jq_{ij}.
\end{equation*}
By independence of the signs,
\begin{equation*}
        \PthreeE_\varepsilon\left[\left|\frac1N\sum_{i\ne j}\varepsilon_i\varepsilon_jq_{ij}\right|^2\right]\leq 2N^{-2}\sum_{i\ne j}|q_{ij}|^2.
\end{equation*}
It remains to bound the Hilbert--Schmidt sum. For fixed $j$, the column $(q_{ij})_{i=1}^N$ is the coordinate vector of $PQu_j$, where $P$ is the coordinate projection from \Cref{p3:lem:packet-coordinates}. In the $L_1$ case, its $\ell_1$ norm is at most $\|Q\|$, hence its $\ell_2$ norm is at most $\|Q\|$. In the $L_p$ case with $1<p\leq2$, its $\ell_p$ norm is at most $\|Q\|$, hence its $\ell_2$ norm is at most $\|Q\|$. Summing over $j$ gives
\begin{equation*}
        \sum_{i,j=1}^N|q_{ij}|^2\lesssim_p N\|Q\|^2.
\end{equation*}
For $p\geq2$, the same argument is applied to the rows, equivalently to $Q^*$ on $L_{p'}$ with $p'\leq2$. Combining this estimate with the preceding second-moment bound proves the lemma.
\end{proof}

The next estimates handle the genuinely one-sided coefficients. If one side of the pairing is fixed before the signs are chosen, then a sufficiently fine random Haar block is almost orthogonal to it on average.

\begin{lemma}[One-sided estimates]
\label{p3:lem:one-sided}
Let $E=L_p^0$ or $E=L_1^0$, and let $A$ be a finite equal-dyadic set. Then, for every $\eta>0$, every pair of integers $M_1,M_2\geq0$, and every pair of finite families
\begin{equation*}
        \mathcal F_1=\{(S_b,z_b):1\leq b\leq M_1\}\subset \PthreeL(E)\times E
\end{equation*}
and
\begin{equation*}
        \mathcal F_2=\{(R_a,z_a^*):1\leq a\leq M_2\}\subset \PthreeL(E)\times E^*,
\end{equation*}
there is $n_0\geq0$ such that, for every $n\geq n_0$, the following holds. If $\PthreeD_n(A)=\{L_1,\ldots,L_N\}$ and $u(\varepsilon),u^*(\varepsilon)$ is the random normalized Haar block and its biorthogonal functional associated with $A$, $n$, and $\varepsilon\in\{-1,1\}^N$, then
\begin{equation*}
        \PthreeE_\varepsilon\left[\left|\Pthreeip{u^*(\varepsilon)}{S_bz_b}\right|^2\right]<\eta^2\qquad (1\leq b\leq M_1),
\end{equation*}
and
\begin{equation*}
        \PthreeE_\varepsilon\left[\left|\Pthreeip{R_a^*z_a^*}{u(\varepsilon)}\right|^2\right]<\eta^2\qquad (1\leq a\leq M_2).
\end{equation*}
\end{lemma}

\begin{proof}
For each depth $n$, write $\PthreeD_n(A)=\{L_1,\ldots,L_N\}$ and denote the corresponding random block by $u_n(\varepsilon),u_n^*(\varepsilon)$. It suffices to prove that, for every fixed $y\in E$ and every fixed $y^*\in E^*$,
\begin{equation*}
        \lim_{n\to\infty}\PthreeE_\varepsilon\left[\left|\Pthreeip{u_n^*(\varepsilon)}{y}\right|^2\right]=0,\qquad \lim_{n\to\infty}\PthreeE_\varepsilon\left[\left|\Pthreeip{y^*}{u_n(\varepsilon)}\right|^2\right]=0.
\end{equation*}
Indeed, after applying these two limits to the finitely many fixed vectors $y_b=S_bz_b$ and fixed functionals $y_a^*=R_a^*z_a^*$, one chooses a single depth $n$ large enough for all of them.

First consider $E=L_1^0$. If $z\in L_1^0$, then
\begin{equation*}
        \PthreeE_\varepsilon\left[\left|\Pthreeip{u^*(\varepsilon)}{z}\right|^2\right]=\sum_{i=1}^N\left|\int_{L_i}f_{L_i}^*z\right|^2\leq \left(\max_{1\leq i\leq N}\int_{L_i}|z|\right)\|z\|_1.
\end{equation*}
The last expression tends to zero as $n\to\infty$ by absolute continuity of the integral. If $z^*\in (L_1^0)^*$, choose an $L_\infty$ representative, still denoted by $z^*$; then
\begin{equation*}
        \PthreeE_\varepsilon\left[\left|\Pthreeip{z^*}{u(\varepsilon)}\right|^2\right]=N^{-2}\sum_{i=1}^N|\Pthreeip{z^*}{f_{L_i}}|^2\leq N^{-1}\|z^*\|_\infty^2,
\end{equation*}
which tends to zero.

Now consider $E=L_p^0$. In this case
\begin{equation*}
        u(\varepsilon)=N^{-1/p}\sum_{i=1}^N\varepsilon_iu_i,\qquad u^*(\varepsilon)=N^{-1/p'}\sum_{i=1}^N\varepsilon_iu_i^*.
\end{equation*}
If $z\in L_p^0$, then
\begin{equation*}
        \PthreeE_\varepsilon\left[\left|\Pthreeip{u^*(\varepsilon)}{z}\right|^2\right]=N^{-2/p'}\sum_{i=1}^N|\Pthreeip{u_i^*}{z}|^2,\qquad |\Pthreeip{u_i^*}{z}|\leq\|\Pthreeone_{L_i}z\|_p.
\end{equation*}
For $1<p\leq2$, this tends to zero because
\begin{equation*}
        \sum_{i=1}^N\|\Pthreeone_{L_i}z\|_p^2\leq \left(\max_{1\leq i\leq N}\|\Pthreeone_{L_i}z\|_p\right)^{2-p}\|z\|_p^p.
\end{equation*}
For $p>2$, H\"older's inequality gives
\begin{equation*}
        \sum_{i=1}^N\|\Pthreeone_{L_i}z\|_p^2\leq N^{1-2/p}\|z\|_p^2,
\end{equation*}
and hence
\begin{equation*}
        \PthreeE_\varepsilon\left[\left|\Pthreeip{u^*(\varepsilon)}{z}\right|^2\right]\leq N^{-1}\|z\|_p^2.
\end{equation*}
Thus, the vector estimate tends to zero in all cases. The estimate for fixed functionals is the same argument with $p$ replaced by $p'$, after representing the functional by an element of $L_{p'}$ modulo constants. Applying these estimates to the fixed vectors $S_bz_b$ and the fixed functionals $R_a^*z_a^*$, and then choosing $n_0$ large enough so that all estimates hold for every $n\geq n_0$, proves the lemma.
\end{proof}

The next estimate is the two-sided random analogue of the preceding one-sided estimates. When two new Haar blocks are built on disjoint finite equal-dyadic sets, and their signs are chosen independently, the mixed coefficient has second moment of order the reciprocal of the number of descendants used in the split.

\begin{lemma}[Bilinear estimate for two new blocks]
\label{p3:lem:bilinear}
Let $E=L_p^0$ or $E=L_1^0$, and let $S\in\PthreeL(E)$. Let $A_a$ and $A_b$ be disjoint finite equal-dyadic sets. Fix depths $n_a,n_b\geq0$, write
\begin{equation*}
        \PthreeD_{n_a}(A_a)=\{L_1^a,\ldots,L_{N_a}^a\},\qquad \PthreeD_{n_b}(A_b)=\{L_1^b,\ldots,L_{N_b}^b\},
\end{equation*}
and let $N_0 = \min \{ N_a,N_b \}$. Let $u_a(\varepsilon^a),u_a^*(\varepsilon^a)$ and $u_b(\varepsilon^b),u_b^*(\varepsilon^b)$ be the random normalized Haar blocks and their biorthogonal functionals associated with these two splittings, where the sign selections $\varepsilon^a\in\{-1,1\}^{N_a}$ and $\varepsilon^b\in\{-1,1\}^{N_b}$ are independent. Then
\begin{equation*}
        \PthreeE_{\varepsilon^a,\varepsilon^b}\left[\left|\Pthreeip{u_a^*(\varepsilon^a)}{S u_b(\varepsilon^b)}\right|^2\right]\lesssim_p \frac{\|S\|^2}{N_0}.
\end{equation*}
In the case $E=L_1^0$ the subscript $p$ is omitted.
\end{lemma}

\begin{proof}
Let $u_{a,i},u_{a,i}^*$ be the normalized Haar vector and biorthogonal functional associated with $L_i^a$, and let $u_{b,j},u_{b,j}^*$ be the corresponding objects associated with $L_j^b$. Write
\begin{equation*}
        s_{ij}=\Pthreeip{u_{a,i}^*}{S u_{b,j}}\qquad(1\leq i\leq N_a,\ 1\leq j\leq N_b).
\end{equation*}
The products $\varepsilon_i^a\varepsilon_j^b$ are orthogonal in $L_2$ of the product sign space, so the mixed terms vanish after averaging.

First consider $E=L_1^0$. Then
\begin{equation*}
        u_a^*(\varepsilon^a)=\sum_{i=1}^{N_a}\varepsilon_i^a u_{a,i}^*,\qquad u_b(\varepsilon^b)=N_b^{-1}\sum_{j=1}^{N_b}\varepsilon_j^b u_{b,j}.
\end{equation*}
Hence
\begin{equation*}
        \PthreeE_{\varepsilon^a,\varepsilon^b}\left[\left|\Pthreeip{u_a^*(\varepsilon^a)}{S u_b(\varepsilon^b)}\right|^2\right]=N_b^{-2}\sum_{i=1}^{N_a}\sum_{j=1}^{N_b}|s_{ij}|^2.
\end{equation*}
For each fixed $j$, the column $(s_{ij})_{i=1}^{N_a}$ is the coordinate vector of the projection of $S u_{b,j}$ onto the span of $u_{a,1},\ldots,u_{a,N_a}$. By \Cref{p3:lem:packet-coordinates}, its $\ell_1$ norm is at most $\|S\|$, hence its $\ell_2$ norm is at most $\|S\|$. Therefore
\begin{equation*}
        \sum_{i=1}^{N_a}\sum_{j=1}^{N_b}|s_{ij}|^2\leq N_b\|S\|^2,
\end{equation*}
and so
\begin{equation*}
        \PthreeE_{\varepsilon^a,\varepsilon^b}\left[\left|\Pthreeip{u_a^*(\varepsilon^a)}{S u_b(\varepsilon^b)}\right|^2\right]\leq N_b^{-1}\|S\|^2\leq N_0^{-1}\|S\|^2.
\end{equation*}

Now consider $E=L_p^0$. Then
\begin{equation*}
        u_a^*(\varepsilon^a)=N_a^{-1/p'}\sum_{i=1}^{N_a}\varepsilon_i^a u_{a,i}^*,\qquad u_b(\varepsilon^b)=N_b^{-1/p}\sum_{j=1}^{N_b}\varepsilon_j^b u_{b,j}.
\end{equation*}
Thus
\begin{equation*}
        \PthreeE_{\varepsilon^a,\varepsilon^b}\left[\left|\Pthreeip{u_a^*(\varepsilon^a)}{S u_b(\varepsilon^b)}\right|^2\right]=N_a^{-2/p'}N_b^{-2/p}\sum_{i=1}^{N_a}\sum_{j=1}^{N_b}|s_{ij}|^2.
\end{equation*}
If $1<p\leq2$, then for each fixed $j$, the column $(s_{ij})_{i=1}^{N_a}$ has $\ell_p$ norm at most $\|S\|$, hence $\ell_2$ norm at most $\|S\|$. Therefore
\begin{equation*}
        \sum_{i=1}^{N_a}\sum_{j=1}^{N_b}|s_{ij}|^2\leq N_b\|S\|^2,
\end{equation*}
and consequently
\begin{equation*}
        \PthreeE_{\varepsilon^a,\varepsilon^b}\left[\left|\Pthreeip{u_a^*(\varepsilon^a)}{S u_b(\varepsilon^b)}\right|^2\right]\leq N_a^{-2/p'}N_b^{1-2/p}\|S\|^2\leq N_0^{-1}\|S\|^2.
\end{equation*}
If $p\geq2$, we use rows instead. For each fixed $i$, the row $(s_{ij})_{j=1}^{N_b}$ has $\ell_{p'}$ norm at most $\|S\|$, hence $\ell_2$ norm at most $\|S\|$. Thus
\begin{equation*}
        \sum_{i=1}^{N_a}\sum_{j=1}^{N_b}|s_{ij}|^2\leq N_a\|S\|^2,
\end{equation*}
and therefore
\begin{equation*}
        \PthreeE_{\varepsilon^a,\varepsilon^b}\left[\left|\Pthreeip{u_a^*(\varepsilon^a)}{S u_b(\varepsilon^b)}\right|^2\right]\leq N_a^{1-2/p'}N_b^{-2/p}\|S\|^2\leq N_0^{-1}\|S\|^2.
\end{equation*}
This proves the lemma.
\end{proof}

\section{Finite one-step constructions}
\label{p3:sec:finite-one-step-constructions}

The preceding estimates are probabilistic results. We now use them to prove two finite splitting results that are used in the recursive construction. Each result starts from finitely many current terminal sets and finitely many constraints fixed before the random signs are chosen.

\subsection{Outer one-step construction}
\label{p3:sec:outer-one-step-construction}

The outer coordinate is constructed by repeatedly splitting finite equal-dyadic sets. At a single outer step, one has finitely many current outer sets, one-parameter operators on $L_p^0$, and admissible side constraints. The step has two tasks. First, the newly born outer Haar blocks must satisfy the prescribed side constraints. Second, each child must inherit an infinite tail of depths on which the relevant averaged diagonal estimates remain valid. The next lemma is the precise one-step form of this operation.

\begin{lemma}[Additive preservation of outer tails]
\label{p3:lem:outer-tail}
Let $A_1,\ldots,A_d$ be pairwise disjoint finite equal-dyadic sets in the outer coordinate. Assume that the following is given.

\begin{enumerate}[label=\textup{(\alph*)}]
\item \label{p3:item:outer-tail-hyp-operators} For each $1\leq a\leq d$, a finite index set $F_a$ and a family $(Q_{a,s})_{s\in F_a}$ in $\PthreeL(L_p^0)$.

\item \label{p3:item:outer-tail-hyp-tails} For each $1\leq a\leq d$, an infinite set $S_a\subset\N$.

\item \label{p3:item:outer-tail-hyp-bounds} Positive numbers $\theta_{a,s}$ for $1\leq a\leq d$ and $s\in F_a$ such that
\begin{equation}
\label{p3:eq:outer-tail-hypothesis}
        |\alpha_m^{A_a}(Q_{a,s})|<\theta_{a,s}\qquad(1\leq a\leq d,\ m\in S_a,\ s\in F_a).
\end{equation}

\item \label{p3:item:outer-tail-hyp-tolerances} Positive numbers $\eta_{a,s}$ and $\delta_{a,s}$ for $1\leq a\leq d$ and $s\in F_a$.

\item \label{p3:item:outer-tail-hyp-side} A finite collection $\mathscr C$ of admissible side constraints, in the sense of \Cref{p3:def:admissible-side-constraints}, involving the new outer blocks to be constructed on $A_1,\ldots,A_d$.
\end{enumerate}

Then there exist depths $n_a\in S_a$ and sign selections
\begin{equation*}
        \varepsilon^a\in\{-1,1\}^{\PthreeD_{n_a}(A_a)}\qquad(1\leq a\leq d)
\end{equation*}
with the following properties. Put
\begin{equation*}
        A_{a,+}=A_{a,+}(\varepsilon^a),\qquad A_{a,-}=A_{a,-}(\varepsilon^a)
\end{equation*}
and let $H_{A_a}(\varepsilon^a),H_{A_a}^*(\varepsilon^a)$ be the associated new outer block and its biorthogonal functional.

\begin{enumerate}[label=\textup{(\roman*)}]
\item \label{p3:item:outer-tail-concl-side} The chosen signs make every admissible side constraint in $\mathscr C$ true for the blocks $H_{A_a}(\varepsilon^a),H_{A_a}^*(\varepsilon^a)$.

\item \label{p3:item:outer-tail-concl-diagonal} For every $1\leq a\leq d$ and every $s\in F_a$,
\begin{equation}
\label{p3:eq:outer-tail-new-diagonal}
        \left|\Pthreeip{H_{A_a}^*(\varepsilon^a)}{Q_{a,s}H_{A_a}(\varepsilon^a)}-\alpha_{n_a}^{A_a}(Q_{a,s})\right|<\eta_{a,s}.
\end{equation}

\item \label{p3:item:outer-tail-concl-child-tail} For every $1\leq a\leq d$ and every $\vartheta\in\{+,-\}$, there is an infinite set $S_{a,\vartheta}\subset\N$ such that, for every $r\in S_{a,\vartheta}$ and every $s\in F_a$,
\begin{equation}
\label{p3:eq:outer-tail-child-tail}
        n_a+1+r\in S_a\quad\text{and}\quad |\alpha_r^{A_{a,\vartheta}}(Q_{a,s})|<\theta_{a,s}+\delta_{a,s}.
\end{equation}
\end{enumerate}
\end{lemma}
\begin{proof}
Throughout the proof, $\mathbb P$ denotes probability with respect to the random sign choices. For every finite equal-dyadic set $A$, write $N_n(A)=\#\PthreeD_n(A)$ for the number of depth $n$ descendants of $A$. Choose numbers $0<\gamma<1/4$ and $0<\kappa<1/(8d)$. 

The proof has four steps: first we choose the depths so that the diagonal and side-constraint failures are rare; then we estimate the probability that a child tail fails at a fixed future depth; then we use Markov's inequality to force many future depths to be good; finally we use finiteness of the sign space to pass from arbitrarily many good depths to infinitely many good depths. 

\begin{description}

\item[\textup{Step 1: Choosing the splitting depths}]
We first choose depths $n_a\in S_a$, $1\leq a\leq d$. Since each $S_a$ is infinite and $N_n(A_a)\to\infty$ as $n\to\infty$ along $S_a$, we may choose the depths so large that the following three requirements hold.

Let $\mathcal E_{\mathrm{diag}}$ be the event that conclusion \ref{p3:item:outer-tail-concl-diagonal} holds for all $1\leq a\leq d$ and $s\in F_a$. For $1\leq a\leq d$ and $s\in F_a$, let $\mathcal B_{a,s}$ be the event
\begin{equation*}
        \left|\Pthreeip{H_{A_a}^*(\varepsilon^a)}{Q_{a,s}H_{A_a}(\varepsilon^a)}-\alpha_{n_a}^{A_a}(Q_{a,s})\right|\geq \eta_{a,s}.
\end{equation*}
By \Cref{p3:lem:centered},
\begin{equation*}
        \PthreeE\left[\left|\Pthreeip{H_{A_a}^*(\varepsilon^a)}{Q_{a,s}H_{A_a}(\varepsilon^a)}-\alpha_{n_a}^{A_a}(Q_{a,s})\right|^2\right]\lesssim_p \frac{\|Q_{a,s}\|^2}{N_{n_a}(A_a)}.
\end{equation*}
Hence Chebyshev's inequality gives
\begin{equation*}
        \mathbb P(\mathcal B_{a,s})\lesssim_p \frac{\|Q_{a,s}\|^2}{N_{n_a}(A_a)\eta_{a,s}^2}.
\end{equation*}
Since $\mathcal E_{\mathrm{diag}}^c\subseteq\bigcup_{a=1}^d\bigcup_{s\in F_a}\mathcal B_{a,s}$, the finite union bound gives
\begin{equation*}
        \mathbb P(\mathcal E_{\mathrm{diag}}^c)\lesssim_p \sum_{a=1}^d\sum_{s\in F_a}\frac{\|Q_{a,s}\|^2}{N_{n_a}(A_a)\eta_{a,s}^2}.
\end{equation*}
Increasing the depths $n_a\in S_a$ if necessary, we may arrange that
\begin{equation*}
        \mathbb P(\mathcal E_{\mathrm{diag}}^c)<\gamma/2.
\end{equation*}

Let $\mathcal E_{\mathrm{side}}$ be the event that conclusion \ref{p3:item:outer-tail-concl-side} holds for the finite admissible collection $\mathscr C$. For each constraint $C\in\mathscr C$, let $X_C$ denote the random scalar appearing in that constraint, and let $\tau_C>0$ denote its tolerance. By admissibility, each $X_C$ is of one of the following forms:
\begin{equation*}
        \Pthreeip{H_{A_a}^*(\varepsilon^a)}{S_Cz_C},\qquad \Pthreeip{R_C^*z_C^*}{H_{A_a}(\varepsilon^a)},\qquad \Pthreeip{H_{A_a}^*(\varepsilon^a)}{S_CH_{A_b}(\varepsilon^b)}\quad(a\neq b).
\end{equation*}
Choose numbers $\rho_C>0$, $C\in\mathscr C$, such that
\begin{equation*}
        \sum_{C\in\mathscr C}\rho_C<\gamma/2.
\end{equation*}
By \Cref{p3:lem:one-sided,p3:lem:bilinear}, after increasing the depths $n_a\in S_a$ if necessary, we may arrange that
\begin{equation*}
        \PthreeE[|X_C|^2]<\rho_C\tau_C^2\qquad(C\in\mathscr C).
\end{equation*}
Let $\mathcal B_C=\{|X_C|\geq\tau_C\}$ be the event that the constraint $C$ fails. By Markov's inequality,
\begin{equation*}
        \mathbb P(\mathcal B_C)\leq \frac{\PthreeE[|X_C|^2]}{\tau_C^2}<\rho_C.
\end{equation*}
Since $\mathcal E_{\mathrm{side}}^c\subseteq\bigcup_{C\in\mathscr C}\mathcal B_C$, the finite union bound gives
\begin{equation*}
        \mathbb P(\mathcal E_{\mathrm{side}}^c)\leq \sum_{C\in\mathscr C}\mathbb P(\mathcal B_C)<\sum_{C\in\mathscr C}\rho_C<\gamma/2.
\end{equation*}

Finally, we choose the same depths so large that, for every $1\leq a\leq d$,
\begin{equation}
\label{p3:eq:outer-tail-child-failure-small}
        \sum_{s\in F_a}\frac{2\|Q_{a,s}\|^2}{N_{n_a}(A_a)\delta_{a,s}^2}<\kappa.
\end{equation}
Set $\mathcal E=\mathcal E_{\mathrm{diag}}\cap\mathcal E_{\mathrm{side}}$. Then
\begin{equation}
\label{p3:eq:event-E-large}
        \mathbb P(\mathcal E)>1-\gamma.
\end{equation}

\item[\textup{Step 2: Estimating child tail failures}]
Fix $1\leq a\leq d$. Since $S_a$ is infinite, the set of integers $r\geq0$ such that $n_a+1+r\in S_a$ is infinite. Enumerate it increasingly as
\begin{equation*}
        r_{a,1}<r_{a,2}<\ldots,\qquad n_a+1+r_{a,\ell}\in S_a\quad(\ell\geq1).
\end{equation*}
For each $\ell\geq1$, let $\mathcal H_{a,\ell}$ be the event that, for both children and every $s\in F_a$,
\begin{equation*}
        |\alpha_{r_{a,\ell}}^{A_{a,+}}(Q_{a,s})|<\theta_{a,s}+\delta_{a,s},\qquad |\alpha_{r_{a,\ell}}^{A_{a,-}}(Q_{a,s})|<\theta_{a,s}+\delta_{a,s}.
\end{equation*}
Fix $\ell\geq1$ and put $r=r_{a,\ell}$. For $s\in F_a$, set
\begin{equation*}
        X_s^+=\alpha_r^{A_{a,+}}(Q_{a,s}),\qquad X_s^-=\alpha_r^{A_{a,-}}(Q_{a,s}).
\end{equation*}
By \eqref{p3:eq:child-trace-expectation}, both random variables have the same expectation
\begin{equation*}
        \PthreeE_{\varepsilon^a}[X_s^+]=\PthreeE_{\varepsilon^a}[X_s^-]=\alpha_{n_a+1+r}^{A_a}(Q_{a,s}).
\end{equation*}
Denote this common scalar by
\begin{equation*}
        \mu_s=\alpha_{n_a+1+r}^{A_a}(Q_{a,s}).
\end{equation*}
Since $n_a+1+r\in S_a$, the hypothesis \eqref{p3:eq:outer-tail-hypothesis} gives
\begin{equation*}
        |\mu_s|<\theta_{a,s}.
\end{equation*}
Hence, if either child fails the desired estimate for this $s$, then the corresponding random variable must deviate from its expectation by more than $\delta_{a,s}$. Indeed,
\begin{equation*}
        |X_s^\vartheta|\geq\theta_{a,s}+\delta_{a,s}\quad\Longrightarrow\quad |X_s^\vartheta-\mu_s|\geq |X_s^\vartheta|-|\mu_s|>\delta_{a,s}\qquad(\vartheta\in\{+,-\}).
\end{equation*}
Therefore
\begin{equation*}
        \mathcal H_{a,\ell}^c\subseteq\bigcup_{s\in F_a}\bigcup_{\vartheta\in\{+,-\}}\{|X_s^\vartheta-\mu_s|>\delta_{a,s}\}.
\end{equation*}
By \eqref{p3:eq:child-trace-variance} and Chebyshev's inequality, for each $s\in F_a$ and $\vartheta\in\{+,-\}$,
\begin{equation*}
        \mathbb P(|X_s^\vartheta-\mu_s|>\delta_{a,s})\leq\frac{\|Q_{a,s}\|^2}{N_{n_a}(A_a)\delta_{a,s}^2}.
\end{equation*}
Taking the union bound over both children and over $s\in F_a$, and using \eqref{p3:eq:outer-tail-child-failure-small}, we obtain
\begin{equation*}
        \mathbb P(\mathcal H_{a,\ell}^c)\leq \sum_{s\in F_a}\frac{2\|Q_{a,s}\|^2}{N_{n_a}(A_a)\delta_{a,s}^2}<\kappa.
\end{equation*}

\item[\textup{Step 3: Forcing many good future depths}]
Fix $M\geq1$. For each $1\leq a\leq d$, let $\mathcal G_{a,M}$ be the event that at least $M$ of the events $\mathcal H_{a,1},\ldots,\mathcal H_{a,2M}$ occur. Let
\begin{equation*}
        Z_{a,M}=\sum_{\ell=1}^{2M}\Pthreeone_{\mathcal H_{a,\ell}^c}
\end{equation*}
be the number of failures among these $2M$ events. By Step 2,
\begin{equation*}
        \PthreeE[Z_{a,M}]=\sum_{\ell=1}^{2M}\mathbb P(\mathcal H_{a,\ell}^c)<2M\kappa.
\end{equation*}
If $\mathcal G_{a,M}$ fails, then fewer than $M$ of the $2M$ events occur, and hence at least $M+1$ of them fail. Therefore
\begin{equation*}
        \mathcal G_{a,M}^c=\{Z_{a,M}\geq M+1\}\subseteq\{Z_{a,M}\geq M\}.
\end{equation*}
By Markov's inequality,
\begin{equation*}
        \mathbb P(\mathcal G_{a,M}^c)\leq \mathbb P(Z_{a,M}\geq M)\leq \frac{\PthreeE[Z_{a,M}]}{M}<2\kappa.
\end{equation*}
Using \eqref{p3:eq:event-E-large} and a union bound over $1\leq a\leq d$, we get
\begin{equation*}
        \mathbb P\left(\mathcal E\cap\bigcap_{a=1}^d\mathcal G_{a,M}\right)\geq 1-\mathbb P(\mathcal E^c)-\sum_{a=1}^d\mathbb P(\mathcal G_{a,M}^c)>1-\gamma-2d\kappa>0.
\end{equation*}
Thus, for every $M\geq1$, there is a choice of signs for which $\mathcal E$ and all the events $\mathcal G_{a,M}$ hold.

\item[\textup{Step 4: Extracting one sign choice with infinitely many good depths}]
The depths $n_a$ are now fixed, so the product sign space
\begin{equation*}
        \prod_{a=1}^d\{-1,1\}^{\PthreeD_{n_a}(A_a)}
\end{equation*}
is finite. For each $M\geq1$, choose one sign choice for which
\begin{equation*}
        \mathcal E\cap\bigcap_{a=1}^d\mathcal G_{a,M}
\end{equation*}
holds. Since the product sign space is finite, one of these sign choices occurs for infinitely many values of $M$. Fix such a sign choice, and let $\mathcal M\subset\N$ be an infinite set such that this fixed choice works for every $M\in\mathcal M$.

Since $\mathcal E$ holds for this choice, conclusions \ref{p3:item:outer-tail-concl-side} and \ref{p3:item:outer-tail-concl-diagonal} hold. Moreover, for every $1\leq a\leq d$ and every $M\in\mathcal M$, the event $\mathcal G_{a,M}$ holds. Since $\mathcal M$ is infinite, it is unbounded. Hence, for each fixed $a$, infinitely many of the events $\mathcal H_{a,\ell}$ hold. Indeed, if only finitely many of them held, say $R_a$ of them, then choosing $M\in\mathcal M$ with $M>R_a$ would contradict the fact that $\mathcal G_{a,M}$ holds.

For each $1\leq a\leq d$, define
\begin{equation*}
        S_{a,+}=S_{a,-}=\{r_{a,\ell}:\mathcal H_{a,\ell}\text{ holds}\}.
\end{equation*}
These sets are infinite by the preceding paragraph. The same infinite set works for both children because $\mathcal H_{a,\ell}$ requires the estimates for $A_{a,+}$ and $A_{a,-}$ simultaneously. If $r\in S_{a,\vartheta}$, then $r=r_{a,\ell}$ for some $\ell$ with $\mathcal H_{a,\ell}$ true, so
\begin{equation*}
        n_a+1+r\in S_a
\end{equation*}
and, for every $s\in F_a$,
\begin{equation*}
        |\alpha_r^{A_{a,\vartheta}}(Q_{a,s})|<\theta_{a,s}+\delta_{a,s}.
\end{equation*}
This proves conclusion \ref{p3:item:outer-tail-concl-child-tail}, and hence the lemma.
\end{description}
\end{proof}

\subsection{Inner one-step construction}
\label{p3:sec:inner-one-step-construction}

The inner coordinate is controlled differently. The relevant one-parameter multipliers live on $L_1^0$, and the local trace from \Cref{p3:prop:local-l1-trace} replaces the outer tail averages. At one inner step, the new block is chosen so that its diagonal coefficient is close to the trace over the parent set, and the two children are chosen so that their traces remain close to the same value.

\begin{lemma}[One-step $L_1$ split with trace control]
\label{p3:lem:l1-one-step}
Suppose that the following objects are fixed.

\begin{enumerate}[label=\textup{(\alph*)}]
\item \label{p3:item:l1-one-step-hyp-operators} Operators $R_1,\ldots,R_m\in\PthreeL(L_1^0)$.

\item \label{p3:item:l1-one-step-hyp-set} A finite equal-dyadic set $B$ in the inner coordinate.

\item \label{p3:item:l1-one-step-hyp-error} A number $\delta>0$.

\item \label{p3:item:l1-one-step-hyp-side} A finite admissible collection $\mathscr C$ of one-sided side constraints, to be imposed on the new inner Haar block produced by the split of $B$.
\end{enumerate}

For a split at depth $n$, choose independent signs $\varepsilon=(\varepsilon_J)_{J\in\PthreeD_n(B)}$ and let $B_\pm(\varepsilon)$ and $K_B(\varepsilon),K_B^*(\varepsilon)$ be the associated random split and inner Haar block. The constraints in $\mathscr C$ are interpreted with this block $K_B(\varepsilon)$ and this functional $K_B^*(\varepsilon)$. Then, for all sufficiently large $n$, there is a choice of signs $\varepsilon$ such that the following hold.

\begin{enumerate}[label=\textup{(\roman*)}]
\item \label{p3:item:l1-one-step-diagonal} For every $1\leq r\leq m$,
\begin{equation}
\label{p3:eq:l1-new-diagonal-trace}
        \left|\Pthreeip{K_B^*(\varepsilon)}{R_rK_B(\varepsilon)}-\lambda_B(R_r)\right|<\delta.
\end{equation}

\item \label{p3:item:l1-one-step-child-traces} For every $1\leq r\leq m$,
\begin{equation}
\label{p3:eq:l1-child-traces}
        |\lambda_{B_+(\varepsilon)}(R_r)-\lambda_B(R_r)|<\delta,\qquad |\lambda_{B_-(\varepsilon)}(R_r)-\lambda_B(R_r)|<\delta.
\end{equation}

\item \label{p3:item:l1-one-step-side} All one-sided side constraints in $\mathscr C$ hold.
\end{enumerate}
\end{lemma}

\begin{proof}
We show that, for all sufficiently large depths $n$, the desired set of sign choices has positive probability. Throughout the proof, probability and expectation are taken with respect to the random signs defining the split of $B$.

By \Cref{p3:prop:local-l1-trace}, for each $1\leq r\leq m$ we have
\begin{equation*}
        \beta_n^B(R_r)\longrightarrow\lambda_B(R_r)\qquad(n\to\infty).
\end{equation*}
Since there are only finitely many operators in \ref{p3:item:l1-one-step-hyp-operators}, there is $n_0$ such that, for every $n\geq n_0$,
\begin{equation}
\label{p3:eq:beta-close-lambda}
        |\beta_n^B(R_r)-\lambda_B(R_r)|<\delta/2\qquad(1\leq r\leq m).
\end{equation}

Fix $n\geq n_0$ and put $N=\#\PthreeD_n(B)$. By the expansion of the random inner block,
\begin{equation*}
        K_B(\varepsilon)=N^{-1}\sum_{J\in\PthreeD_n(B)}\varepsilon_Jf_J,\qquad K_B^*(\varepsilon)=\sum_{J\in\PthreeD_n(B)}\varepsilon_Jf_J^*.
\end{equation*}
Therefore, for each $1\leq r\leq m$,
\begin{equation*}
        \PthreeE_\varepsilon\left[\Pthreeip{K_B^*(\varepsilon)}{R_rK_B(\varepsilon)}\right]=\beta_n^B(R_r).
\end{equation*}
Moreover, applying \Cref{p3:lem:centered} in the case $E=L_1^0$ gives
\begin{equation*}
        \PthreeE_\varepsilon\left[\left|\Pthreeip{K_B^*(\varepsilon)}{R_rK_B(\varepsilon)}-\beta_n^B(R_r)\right|^2\right]\lesssim\frac{\|R_r\|^2}{N}.
\end{equation*}
Thus, by Chebyshev's inequality and a finite union bound over $1\leq r\leq m$,
\begin{equation*}
        \mathbb P\left(\exists\,1\leq r\leq m:\left|\Pthreeip{K_B^*(\varepsilon)}{R_rK_B(\varepsilon)}-\beta_n^B(R_r)\right|\geq\delta/2\right)\longrightarrow0
\end{equation*}
as $n\to\infty$. Increasing $n_0$ if necessary, this failure probability is smaller than $1/4$ for all $n\geq n_0$. On the complementary event, \eqref{p3:eq:beta-close-lambda} implies \ref{p3:item:l1-one-step-diagonal}.

Next we control the traces of the two children. For each $1\leq r\leq m$, let $\varphi_r\in L_\infty[0,1]$ be supplied by \Cref{p3:prop:local-l1-trace}, so that
\begin{equation*}
        \lambda_C(R_r)=\frac1{|C|}\int_C\varphi_r
\end{equation*}
for every finite equal-dyadic set $C$, and
\begin{equation*}
        \|\varphi_r\|_\infty\lesssim\|R_r\|.
\end{equation*}
For the plus child,
\begin{equation*}
        \lambda_{B_+(\varepsilon)}(R_r)=\frac2{|B|}\sum_{J\in\PthreeD_n(B)}\int_{J^{\varepsilon_J}}\varphi_r.
\end{equation*}
The summands are independent as $J$ varies, and
\begin{equation*}
        \PthreeE_\varepsilon\left[\lambda_{B_+(\varepsilon)}(R_r)\right]=\lambda_B(R_r).
\end{equation*}
Since every $J\in\PthreeD_n(B)$ has measure $|B|/N$, we also have
\begin{equation*}
        \operatorname{Var}_\varepsilon\left(\lambda_{B_+(\varepsilon)}(R_r)\right)\leq\sum_{J\in\PthreeD_n(B)}\left(\frac{\|\varphi_r\|_\infty}{N}\right)^2\lesssim\frac{\|R_r\|^2}{N}.
\end{equation*}
The same argument gives
\begin{equation*}
        \PthreeE_\varepsilon\left[\lambda_{B_-(\varepsilon)}(R_r)\right]=\lambda_B(R_r),\qquad \operatorname{Var}_\varepsilon\left(\lambda_{B_-(\varepsilon)}(R_r)\right)\lesssim\frac{\|R_r\|^2}{N}.
\end{equation*}
Chebyshev's inequality and a finite union bound over $1\leq r\leq m$ and the two signs $\pm$ show that the failure probability of \ref{p3:item:l1-one-step-child-traces} tends to zero as $n\to\infty$. Increasing $n_0$ again, this failure probability is strictly smaller than $1/4$ for all $n\geq n_0$.

It remains to impose the side constraints from \ref{p3:item:l1-one-step-hyp-side}. Since $\mathscr C$ is finite and admissible, every constraint is one of the one-sided forms handled by \Cref{p3:lem:one-sided}. For each constraint, Markov's inequality applied to the corresponding second-moment estimate shows that its failure probability can be made arbitrarily small by taking the splitting depth sufficiently large. Hence, after increasing $n_0$ once more and using a finite union bound over all constraints in $\mathscr C$, the probability that some side constraint fails is strictly smaller than $1/4$ for all $n\geq n_0$.

For every $n\geq n_0$, the probability that any of the three required conclusions fails is therefore strictly smaller than $3/4$. Hence the probability that \ref{p3:item:l1-one-step-diagonal}, \ref{p3:item:l1-one-step-child-traces}, and \ref{p3:item:l1-one-step-side} all hold is positive. Thus there is a choice of signs satisfying all three conclusions.
\end{proof}

The previous lemma handles one inner set and one new inner block. We now record the finite simultaneous version needed at an inner splitting stage. The only extra point is that bilinear constraints between two distinct new blocks are allowed; these are handled by the bilinear estimate, using independence of the sign choices on distinct sets.

\begin{lemma}[Simultaneous finite $L_1$ splitting]
\label{p3:lem:l1-simultaneous}
Let $B_1,\ldots,B_d$ be pairwise disjoint finite equal-dyadic sets in the inner coordinate. Suppose that the following objects are fixed.

\begin{enumerate}[label=\textup{(\alph*)}]
\item \label{p3:item:l1-sim-hyp-operators} For each $1\leq a\leq d$, a finite family $\mathcal R_a\subset\PthreeL(L_1^0)$.

\item \label{p3:item:l1-sim-hyp-tolerances} For each $1\leq a\leq d$ and each $R\in\mathcal R_a$, a positive tolerance $\delta_{a,R}>0$.

\item \label{p3:item:l1-sim-hyp-side} A finite admissible collection $\mathscr C$ of side constraints involving the new inner blocks to be constructed on $B_1,\ldots,B_d$.
\end{enumerate}

Then there are depths $n_1,\ldots,n_d$ and independent sign choices
\begin{equation*}
        \varepsilon^a=(\varepsilon_J^a)_{J\in\PthreeD_{n_a}(B_a)}\qquad(1\leq a\leq d)
\end{equation*}
such that, writing $B_{a,\pm}=B_{a,\pm}(\varepsilon^a)$ and denoting the associated inner Haar block by $K_{B_a}(\varepsilon^a)$, with functional $K_{B_a}^*(\varepsilon^a)$, the following hold.

\begin{enumerate}[label=\textup{(\roman*)}]
\item \label{p3:item:l1-sim-concl-diagonal} For every $1\leq a\leq d$ and every $R\in\mathcal R_a$,
\begin{equation*}
        \left|\Pthreeip{K_{B_a}^*(\varepsilon^a)}{RK_{B_a}(\varepsilon^a)}-\lambda_{B_a}(R)\right|<\delta_{a,R}.
\end{equation*}

\item \label{p3:item:l1-sim-concl-child-traces} For every $1\leq a\leq d$, every $R\in\mathcal R_a$, and every $\vartheta\in\{+,-\}$,
\begin{equation*}
        |\lambda_{B_{a,\vartheta}}(R)-\lambda_{B_a}(R)|<\delta_{a,R}.
\end{equation*}

\item \label{p3:item:l1-sim-concl-side} All side constraints in $\mathscr C$ hold.
\end{enumerate}
\end{lemma}

\begin{proof}
We choose the depths probabilistically. Throughout the proof, probability and expectation are taken with respect to the independent sign choices
\begin{equation*}
        \varepsilon^a=(\varepsilon_J^a)_{J\in\PthreeD_{n_a}(B_a)}\qquad(1\leq a\leq d).
\end{equation*}
Since only finitely many operators and side constraints occur, it is enough to choose the depths so that each of the relevant failure probabilities is small.

First consider the trace estimates. Fix $1\leq a\leq d$ and $R\in\mathcal R_a$. By the proof of \Cref{p3:lem:l1-one-step}, applied to the set $B_a$ and the operator $R$, the failure probability of
\begin{equation*}
        \left|\Pthreeip{K_{B_a}^*(\varepsilon^a)}{RK_{B_a}(\varepsilon^a)}-\lambda_{B_a}(R)\right|<\delta_{a,R}
\end{equation*}
and
\begin{equation*}
        |\lambda_{B_{a,+}(\varepsilon^a)}(R)-\lambda_{B_a}(R)|<\delta_{a,R},\qquad |\lambda_{B_{a,-}(\varepsilon^a)}(R)-\lambda_{B_a}(R)|<\delta_{a,R}
\end{equation*}
tends to zero as $n_a\to\infty$. Hence, by increasing the finitely many depths $n_a$, we may assume that the probability that any trace estimate in \ref{p3:item:l1-sim-concl-diagonal} or \ref{p3:item:l1-sim-concl-child-traces} fails is less than $1/3$.

It remains to impose the side constraints in $\mathscr C$. Since $\mathscr C$ is admissible, each constraint is either one-sided or bilinear between two distinct new blocks. For every one-sided constraint, \Cref{p3:lem:one-sided} and Markov's inequality show that its failure probability can be made arbitrarily small by increasing the depth of the corresponding split. For every bilinear constraint between the blocks over $B_a$ and $B_b$, with $a\ne b$, \Cref{p3:lem:bilinear} and Markov's inequality show that its failure probability can be made arbitrarily small by increasing both depths $n_a$ and $n_b$. Because there are only finitely many constraints, we may increase the depths again so that the probability that any constraint in $\mathscr C$ fails is less than $1/3$.

For these choices of depths, the probability that one of the trace estimates or one of the side constraints fails is less than $2/3$. Therefore the probability that all conclusions \ref{p3:item:l1-sim-concl-diagonal}, \ref{p3:item:l1-sim-concl-child-traces}, and \ref{p3:item:l1-sim-concl-side} hold is positive. Thus, there is a choice of signs satisfying all the required estimates simultaneously.
\end{proof}

\section{The multiplier reduction construction}
\label{p3:sec:multiplier-reduction-construction}
\label{p3:sec:strips-and-partial-constructions}
\label{p3:sec:reduction-to-multiplier}

We now prove the reduction from an arbitrary operator on $X_{00}$ to a product Haar multiplier. Fix
\begin{equation*}
        T\in\PthreeL(X_{00}).
\end{equation*}
We construct an outer faithful Haar system
\begin{equation*}
        (H_I,H_I^*)_{I\in\PthreeD}
\end{equation*}
and an inner faithful Haar system
\begin{equation*}
        (K_J,K_J^*)_{J\in\PthreeD}
\end{equation*}
so that, after passing to the associated faithful product-Haar copy, the matrix of $T$ is small away from the product diagonal. The product matrix coefficients are
\begin{equation}
\label{p3:eq:product-matrix-coefficients}
        a_{J,J'}^{I,I'}=\Pthreeip{H_I^*\otimes K_J^*}{T(H_{I'}\otimes K_{J'})}.
\end{equation}
The diagonal coefficients are $a_{J,J}^{I,I}$, and they define the multiplier which remains after the construction. The off-diagonal coefficients have three types:
\begin{equation*}
        I=I',\ J\ne J',\qquad I\ne I',\ J=J',\qquad I\ne I',\ J\ne J'.
\end{equation*}
The first type is controlled by outer tail preservation, the second by inner trace variation, and the third by summable coefficient estimates.

We shall use the following terminology throughout the construction. We say that an outer block $H_I$ is \emph{born} at the outer half stage where the set $\Gamma_I$ is split into its two children. We say that an inner block $K_J$ is \emph{born} at the inner half stage where the set $\Delta_J$ is split into its two children. A product coefficient is \emph{available} once all Haar blocks appearing in it have been born. A \emph{frontier} is the finite family of indices whose sets have already been chosen but have not yet been split. Thus, at a full state $\mathfrak S_n^0$, both frontiers are indexed by $\PthreeD_n$, while at a half state $\mathfrak S_n^{1/2}$ the outer frontier is indexed by $\PthreeD_{n+1}$ and the inner frontier is indexed by $\PthreeD_n$. An element of one of these finite generations is called a \emph{frontier index}. A \emph{reservoir tail} for an outer frontier index $A$ is an infinite set $S_A\subset\mathbb N$ along which the outer tail estimates over $\Gamma_A$ are imposed.

For $J\ne J'$, the family
\begin{equation*}
        (a_{J,J'}^{I,I})_{I\in\PthreeD}
\end{equation*}
is called the \emph{same outer strip} indexed by $(J,J')$. This strip is \emph{registered} at the inner half stage where the later of the two blocks $K_J$ and $K_{J'}$ is born. From that half stage onward, its outer coefficient estimates and reservoir-tail estimates are included in every later legal state. For $I\ne I'$, the family
\begin{equation*}
        (a_{J,J}^{I,I'})_{J\in\PthreeD}
\end{equation*}
is called the \emph{same inner strip} indexed by $(I,I')$. This strip is \emph{registered} at the outer half stage where the later of the two blocks $H_I$ and $H_{I'}$ is born. From that half stage onward, its inner trace and variation estimates are included in every later legal state. The coefficients $a_{J,J'}^{I,I'}$ with $I\ne I'$ and $J\ne J'$ are called \emph{mixed coefficients}. A mixed coefficient is \emph{registered} at the half stage where the last of the four blocks
\begin{equation*}
        H_I,\qquad H_{I'},\qquad K_J,\qquad K_{J'}
\end{equation*}
is born. At that half stage its assigned summable estimate is imposed. Thorough the proof, $\alpha_m^A(Q)$ denotes the potential outer trace average from \Cref{p3:def:potential-packet-traces-shadows}, and $\lambda_B(R)$ denotes the local $L_1$ trace from \Cref{p3:prop:local-l1-trace}.

The proof below invokes construction claims exactly at the points where choices have to be made: registrations are locally finite, the outer half stage can be performed, and the inner half stage can be performed. The final norm estimate also uses a few elementary auxiliary lemmas. The formal statements and proofs are given in \Cref{p3:sec:technical-claims-multiplier-reduction}.

\begin{theorem}[Arbitrary operators reduce to product Haar multipliers]
\label{p3:thm:reduction-to-multiplier}
Let $1<p<\infty$. For every $T\in\PthreeL(X_{00})$ and every $\varepsilon>0$, there are faithful outer and inner Haar systems, exterior maps $A,B\in\PthreeL(X_{00})$, and a bounded product Haar multiplier $M$ on $X_{00}$ such that
\begin{equation*}
        AB=\PthreeId_{X_{00}},\qquad \|ATB-M\|_{\PthreeL(X_{00})}<\varepsilon.
\end{equation*}
Moreover,
\begin{equation*}
        \|M\|_{\PthreeL(X_{00})}\leq \|T\|_{\PthreeL(X_{00})}+\varepsilon.
\end{equation*}
If the faithful systems are denoted by $(H_I,H_I^*)_{I\in\PthreeD}$ and $(K_J,K_J^*)_{J\in\PthreeD}$, then
\begin{equation*}
        Mu_{I,J}=\Pthreeip{H_I^*\otimes K_J^*}{T(H_I\otimes K_J)}u_{I,J}\qquad(I,J\in\PthreeD).
\end{equation*}
\end{theorem}

\begin{proof}
Fix $T\in\PthreeL(X_{00})$ and $\varepsilon>0$. Let $U_p$ be the constant from \Cref{p3:lem:lp-haar-multiplier}, and let $C_{\rm SU}$ be such that the estimate $\|M_a\|_{\PthreeL(L_1^0)}\leq C_{\rm SU}\PthreeBV(a)$ holds in \Cref{p3:cor:su-branch}. Put
\begin{equation*}
        C_0=8,\qquad C_{\rm out}=U_p,\qquad C_{\rm in}=C_{\rm SU}.
\end{equation*}
Choose positive summable families
\begin{equation*}
        (\sigma_{J,J'})_{J\ne J'},\qquad (\tau_{I,I'})_{I\ne I'},\qquad (\mu_{I,I',J,J'})_{I\ne I',\,J\ne J'}
\end{equation*}
such that
\begin{equation}
\label{p3:eq:multiplier-reduction-budgets}
        C_{\rm out}\sum_{J\ne J'}\sigma_{J,J'}<\frac{\varepsilon}{3},\qquad C_{\rm in}\sum_{I\ne I'}\tau_{I,I'}<\frac{\varepsilon}{3},\qquad \sum_{I\ne I',\,J\ne J'}\mu_{I,I',J,J'}<\frac{\varepsilon}{3}.
\end{equation}
The number $\sigma_{J,J'}$ is reserved for the same outer strip indexed by $(J,J')$, the number $\tau_{I,I'}$ is reserved for the same inner strip indexed by $(I,I')$, and the number $\mu_{I,I',J,J'}$ is reserved for the mixed coefficient indexed by $(I,I',J,J')$.

Put
\begin{equation*}
        \PthreeD_{\leq n}=\bigcup_{m=0}^n\PthreeD_m,\qquad \PthreeD_{<n}=\bigcup_{m=0}^{n-1}\PthreeD_m.
\end{equation*}
The recursion alternates outer and inner half stages:
\begin{equation*}
        \mathfrak S_0^0\xrightarrow{\mathrm{outer}}\mathfrak S_0^{1/2}\xrightarrow{\mathrm{inner}}\mathfrak S_1^0\xrightarrow{\mathrm{outer}}\mathfrak S_1^{1/2}\xrightarrow{\mathrm{inner}}\mathfrak S_2^0\xrightarrow{\mathrm{outer}}\ldots .
\end{equation*}
At the beginning of stage $n$, denoted by $\mathfrak S_n^0$, the outer sets $\Gamma_I$ and inner sets $\Delta_J$ have been chosen for $I,J\in\PthreeD_{\leq n}$, while the outer and inner Haar blocks have been born for $I,J\in\PthreeD_{<n}$. Thus both current frontiers are indexed by $\PthreeD_n$. The construction is by induction on $n$. In the induction hypothesis we first list the objects in place, and then the conditions which the construction has achieved for the corresponding coefficients and strip operators. A full state $\mathfrak S_n^0$ is called \emph{legal} if the objects in \textup{(\ref{p3:item:quantity-sets})}--\textup{(\ref{p3:item:quantity-mixed})} have been specified and the estimates in \textup{(\ref{p3:item:estimate-outer-coefficients})}--\textup{(\ref{p3:item:estimate-mixed})} hold. A half state $\mathfrak S_n^{1/2}$ is \emph{legal} when the analogous assertion holds with outer frontier $\PthreeD_{n+1}$ and inner frontier $\PthreeD_n$. In the base case these estimates are empty because no strip or mixed coefficient has yet been registered.

\medskip
\noindent\textbf{Base case.} Choose
\begin{equation*}
        \Gamma_{[0,1)}=[0,1),\qquad \Delta_{[0,1)}=[0,1),\qquad S_{[0,1)}=\mathbb N.
\end{equation*}
No Haar block has yet been born, no same outer or same inner strip has been registered, and no mixed coefficient has been registered. The estimates in the induction hypothesis below are therefore vacuous, except for the reservoir tail $S_{[0,1)}$. Hence $\mathfrak S_0^0$ is a legal initial state.

\medskip
\noindent\textbf{Induction hypothesis at stage $n$.} Suppose that $\mathfrak S_n^0$ has been constructed and is legal. The following objects are in place.
\begin{enumerate}[label=\textup{(Q\arabic*)},ref=Q\arabic*]
\item\label{p3:item:quantity-sets} The outer and inner sets $\Gamma_I$ and $\Delta_J$ for $I,J\in\PthreeD_{\leq n}$, and the born blocks $H_I,H_I^*,K_J,K_J^*$ for $I,J\in\PthreeD_{<n}$.
\item\label{p3:item:quantity-reservoirs} For every current outer frontier index $A\in\PthreeD_n$, an infinite reservoir tail $S_A\subset\mathbb N$.
\item\label{p3:item:quantity-same-outer} For every same outer strip indexed by $(J,J')$ which has already been registered and every current outer frontier index $A\in\PthreeD_n$, positive numbers $\theta_A^{J,J'}$ and $\omega_A^{J,J'}$.
\item\label{p3:item:quantity-same-inner} For every same inner strip indexed by $(I,I')$ which has already been registered, the finite inner tree $\mathcal T_{I,I'}^{\rm old}$ and the finite inner frontier $\mathcal R_{I,I'}$ present when $(I,I')$ was registered, the future forest
\begin{equation*}
        \mathcal F_{I,I'}=\{C\in\PthreeD:C\subseteq B\text{ for some }B\in\mathcal R_{I,I'}\},
\end{equation*}
and positive future tolerances $\eta_C^{I,I'}$ and $\rho_{C,C'}^{I,I'}$ for every pair $C,C'\in\mathcal F_{I,I'}$ with $C$ a dyadic parent of $C'$, which we will denote by $C\to C'$.
\item\label{p3:item:quantity-mixed} The finite collection of mixed coefficients whose registration half stage has already occurred.
\end{enumerate}
For a same outer strip indexed by $(J,J')$ which has already been registered, we write
\begin{equation}
\label{p3:eq:same-outer-operator}
        Q_{J,J'}x=(\PthreeId\otimes K_J^*)T(x\otimes K_{J'})\qquad(x\in L_p^0).
\end{equation}
For a same inner strip indexed by $(I,I')$ which has already been registered, we write
\begin{equation}
\label{p3:eq:same-inner-operator}
        R_{I,I'}v=(H_I^*\otimes\PthreeId)T(H_{I'}\otimes v)\qquad(v\in L_1^0).
\end{equation}
The construction has been carried out so that the following conditions hold for these operators and coefficients.
\begin{enumerate}[label=\textup{(E\arabic*)},ref=E\arabic*]
\item\label{p3:item:estimate-outer-coefficients} (Existing same outer-strip are small) For every same outer strip indexed by $(J,J')$ which has already been registered and every already born outer index $I$,
\begin{equation}
\label{p3:eq:same-outer-old-coefficients}
        |a_{J,J'}^{I,I}|<\sigma_{J,J'}.
\end{equation}
\item\label{p3:item:estimate-outer-tails}(Room to make future same outer-strip choices small) For every same outer strip indexed by $(J,J')$ which has already been registered and every current outer frontier index $A$,
\begin{equation}
\label{p3:eq:outer-tail-invariant}
        |\alpha_m^{\Gamma_A}(Q_{J,J'})|<\theta_A^{J,J'}\qquad(m\in S_A)
\end{equation}
and
\begin{equation}
\label{p3:eq:outer-room-invariant}
        \theta_A^{J,J'}+\omega_A^{J,J'}<\sigma_{J,J'}.
\end{equation}

\item\label{p3:item:estimate-inner-old-part} (Existing same inner-strip is small) For every same inner-strip indexed by $(I, I')$ which has already been registered, set
\begin{equation}
\label{p3:eq:old-inner-variation}
\begin{aligned}
    G_{I,I'}={}&
    \mathbf 1_{\{[0,1)\in\mathcal T_{I,I'}^{\rm old}\}}
        |a_{[0,1),[0,1)}^{I,I'}|
    +\sum_{\substack{P\to Q\\ P,Q\in\mathcal T_{I,I'}^{\rm old}}}
        |a_{P,P}^{I,I'}-a_{Q,Q}^{I,I'}| \\
    &+\sum_{\substack{B\in\mathcal R_{I,I'}\\ p(B)\in\mathcal T_{I,I'}^{\rm old}}}
        |a_{p(B),p(B)}^{I,I'}|,
    \end{aligned}
\end{equation}
here $p(B)$ denotes the dyadic parent of $B$, when this parent belongs to the old inner tree $\mathcal T_{I,I'}^{\rm old}$.
The part of the strip which was already present at the registration half stage is controlled by
\begin{equation}
\label{p3:eq:same-inner-birth-condition}
        G_{I,I'}+\sum_{B\in\mathcal R_{I,I'}}|\lambda_{\Delta_B}(R_{I,I'})|<\frac{\tau_{I,I'}}{C_0}.
\end{equation}
Here $G_{I,I'}$ measures the variation on the inner tree which had already been built, while the trace terms $\lambda_{\Delta_B}(R_{I,I'})$, $B\in\mathcal R_{I,I'}$, measure the size at the roots where the future tree will be attached.
\item\label{p3:item:estimate-inner-future-part} (Room to make future same inner-strip choices small) For every same inner strip indexed by $(I,I')$ which has already been registered, the future tolerances are chosen so that
\begin{equation}
\label{p3:eq:same-inner-future-budget}
        4\sum_{C\in\mathcal F_{I,I'}}\eta_C^{I,I'}+\sum_{C\to C',\ C,C'\in\mathcal F_{I,I'}}\rho_{C,C'}^{I,I'}<\frac{\tau_{I,I'}}{C_0}.
\end{equation}
After the registration half stage, whenever a future inner block $K_C$ is born after $(I, I')$ is born, that is, $C$ is born but $C \not \in \mathcal T_{I,I'}^{\rm old}$, we impose
\begin{equation}
\label{p3:eq:same-inner-node-estimate}
        |\Pthreeip{K_C^*}{R_{I,I'}K_C}-\lambda_{\Delta_C}(R_{I,I'})|<\eta_C^{I,I'},
\end{equation}
and whenever a future edge $C\to C'$ in $\mathcal F_{I,I'}$ has been created, we impose
\begin{equation}
\label{p3:eq:same-inner-edge-estimate}
        |\lambda_{\Delta_{C'}}(R_{I,I'})-\lambda_{\Delta_C}(R_{I,I'})|<\rho_{C,C'}^{I,I'}.
\end{equation}
Thus \textup{(\ref{p3:item:estimate-inner-old-part})} controls the part of the strip present at registration, while \textup{(\ref{p3:item:estimate-inner-future-part})} controls the coefficients born after registration.
\item\label{p3:item:estimate-mixed} (Mixed coefficients are small) Every registered mixed coefficient satisfies
\begin{equation}
\label{p3:eq:mixed-coefficient-estimate}
        |a_{J,J'}^{I,I'}|<\mu_{I,I',J,J'}.
\end{equation}
\end{enumerate}
By Claim~\ref{p3:claim:registration-finiteness}, every off-diagonal coefficient is assigned to exactly one of these mechanisms, and only finitely many new strips, coefficients, and estimates occur at each half stage. We now argue how to continue the inductive process, provided we are starting from a legal state at stage $n$.

\medskip
\noindent\textbf{Outer half stage at level $n$.} We split every current outer set $\Gamma_I$, $I\in\PthreeD_n$, birth $H_I,H_I^*$ for $I\in\PthreeD_n$, and produce a half state $\mathfrak S_n^{1/2}$ whose outer frontier is indexed by $\PthreeD_{n+1}$ and whose inner frontier is indexed by $\PthreeD_n$. By Claim~\ref{p3:claim:outer-half-stage}, the splits can be chosen together with the reservoir tails $S_{A^+}$ and $S_{A^-}$ below each $A\in\PthreeD_n$, the numbers $\theta_{A^\vartheta}^{J,J'}$ and $\omega_{A^\vartheta}^{J,J'}$ for every same outer strip registered by the half state, and, for every same inner strip registered at this half stage, the objects listed in \textup{(\ref{p3:item:quantity-same-inner})}. These choices can be made so that, \textup{(\ref{p3:item:estimate-outer-coefficients})} and \textup{(\ref{p3:item:estimate-outer-tails})} hold for the same outer strips registered by the half state, \textup{(\ref{p3:item:estimate-inner-old-part})} and \textup{(\ref{p3:item:estimate-inner-future-part})} hold for every same inner strip registered by the half state, and \textup{(\ref{p3:item:estimate-mixed})} holds for every mixed coefficient registered at this half stage.

\medskip
\noindent\textbf{Inner half stage at level $n$.} We split every current inner set $\Delta_J$, $J\in\PthreeD_n$, birth $K_J,K_J^*$ for $J\in\PthreeD_n$, and produce the next full state $\mathfrak S_{n+1}^0$. By Claim~\ref{p3:claim:inner-half-stage}, the splits can be chosen together with the possibly shrunk reservoir tails $S_A$ for $A\in\PthreeD_{n+1}$, the initial numbers $\theta_A^{J,J'}$ and $\omega_A^{J,J'}$ for every same outer strip registered at this half stage, and the future estimates in \textup{(\ref{p3:item:estimate-inner-future-part})} for all same inner strips registered by that time whose relevant future blocks or future edges are created at this half stage. These choices can be made so that, \textup{(\ref{p3:item:estimate-outer-coefficients})} and \textup{(\ref{p3:item:estimate-outer-tails})} hold for all same outer strips registered by the new full state, \textup{(\ref{p3:item:estimate-inner-old-part})} and \textup{(\ref{p3:item:estimate-inner-future-part})} hold for all same inner strips registered by that time, and \textup{(\ref{p3:item:estimate-mixed})} holds for every mixed coefficient registered at this half stage.

This proves the induction step from $\mathfrak S_n^0$ to $\mathfrak S_{n+1}^0$. Repeating the two half stages gives the sequence
\begin{equation*}
        \mathfrak S_0^0\xrightarrow{\mathrm{outer}}\mathfrak S_0^{1/2}\xrightarrow{\mathrm{inner}}\mathfrak S_1^0\xrightarrow{\mathrm{outer}}\mathfrak S_1^{1/2}\xrightarrow{\mathrm{inner}}\ldots .
\end{equation*}
Therefore every dyadic index is eventually split in both coordinates, and we obtain faithful outer and inner Haar systems
\begin{equation*}
        (H_I,H_I^*)_{I\in\PthreeD},\qquad (K_J,K_J^*)_{J\in\PthreeD}.
\end{equation*}

Let $A$ and $B$ be the associated exterior maps. By \Cref{p3:lem:exterior},
\begin{equation*}
        \|A\|\leq1,\qquad \|B\|\leq1,\qquad AB=\PthreeId_{X_{00}},
\end{equation*}
and
\begin{equation}
\label{p3:eq:compressed-matrix-coefficients}
        \Pthreeip{u_{I,J}^*}{ATB u_{I',J'}}=a_{J,J'}^{I,I'}\qquad(I,I',J,J'\in\PthreeD).
\end{equation}
On the algebraic product Haar span define
\begin{equation*}
        M_0u_{I,J}=a_{J,J}^{I,I}u_{I,J}.
\end{equation*}
The matrix coefficients of $ATB-M_0$ are the off-diagonal coefficients from \eqref{p3:eq:product-matrix-coefficients}. We estimate the three off-diagonal pieces.

For $J\ne J'$, let $O_{J,J'}^{\rm out}$ be given on product Haar vectors by
\begin{equation*}
        O_{J,J'}^{\rm out}u_{I,L}=\begin{cases} a_{J,J'}^{I,I}u_{I,J},&L=J',\\ 0,&L\ne J'. \end{cases}
\end{equation*}
The construction gives $|a_{J,J'}^{I,I}|<\sigma_{J,J'}$ for all $I\in\PthreeD$, and \Cref{p3:lem:one-coordinate-strip-estimates} gives
\begin{equation*}
        \|O_{J,J'}^{\rm out}\|\leq C_{\rm out}\sigma_{J,J'}.
\end{equation*}
Thus
\begin{equation}
\label{p3:eq:same-outer-total-estimate}
        \left\|\sum_{J\ne J'}O_{J,J'}^{\rm out}\right\|\leq C_{\rm out}\sum_{J\ne J'}\sigma_{J,J'}.
\end{equation}

For $I\ne I'$, let $O_{I,I'}^{\rm in}$ be given on product Haar vectors by
\begin{equation*}
        O_{I,I'}^{\rm in}u_{L,J}=\begin{cases} a_{J,J}^{I,I'}u_{I,J},&L=I',\\ 0,&L\ne I'. \end{cases}
\end{equation*}
The estimates \eqref{p3:eq:same-inner-birth-condition}, \eqref{p3:eq:same-inner-future-budget}, \eqref{p3:eq:same-inner-node-estimate}, and \eqref{p3:eq:same-inner-edge-estimate}, together with \Cref{p3:lem:inner-strip-branch-variation}, imply that the branch variation of $(a_{J,J}^{I,I'})_{J\in\PthreeD}$ is less than $\tau_{I,I'}$. Hence \Cref{p3:cor:su-branch} and \Cref{p3:lem:one-coordinate-strip-estimates} gives
\begin{equation*}
        \|O_{I,I'}^{\rm in}\|\leq C_{\rm in}\tau_{I,I'}.
\end{equation*}
Therefore
\begin{equation}
\label{p3:eq:same-inner-total-estimate}
        \left\|\sum_{I\ne I'}O_{I,I'}^{\rm in}\right\|\leq C_{\rm in}\sum_{I\ne I'}\tau_{I,I'}.
\end{equation}

Finally, define
\begin{equation*}
        O^{\rm mix}=\sum_{I\ne I',\,J\ne J'}a_{J,J'}^{I,I'}u_{I,J}\otimes u_{I',J'}^*.
\end{equation*}
Since $\|u_{I,J}\|=\|u_{I',J'}^*\|=1$ and \eqref{p3:eq:mixed-coefficient-estimate} holds for every mixed coefficient, the mixed series converges absolutely in operator norm and
\begin{equation}
\label{p3:eq:mixed-total-estimate}
        \|O^{\rm mix}\|\leq \sum_{I\ne I',\,J\ne J'}\mu_{I,I',J,J'}.
\end{equation}
Set
\begin{equation*}
        O=\sum_{J\ne J'}O_{J,J'}^{\rm out}+\sum_{I\ne I'}O_{I,I'}^{\rm in}+O^{\rm mix}.
\end{equation*}
Combining \eqref{p3:eq:multiplier-reduction-budgets}, \eqref{p3:eq:same-outer-total-estimate}, \eqref{p3:eq:same-inner-total-estimate}, and \eqref{p3:eq:mixed-total-estimate}, we obtain
\begin{equation*}
        \|O\|<\varepsilon.
\end{equation*}
Moreover, by construction and \eqref{p3:eq:compressed-matrix-coefficients}, the operators $ATB-M_0$ and $O$ have the same product Haar coefficients on the algebraic product Haar span. By \Cref{p3:lem:product-coefficient-separation}, they agree there. Define
\begin{equation*}
        M=ATB-O.
\end{equation*}
Then $M\in\PthreeL(X_{00})$, and $M$ agrees with $M_0$ on the algebraic product Haar span. Hence $M$ is the bounded product Haar multiplier with diagonal
\begin{equation*}
        Mu_{I,J}=a_{J,J}^{I,I}u_{I,J}=\Pthreeip{H_I^*\otimes K_J^*}{T(H_I\otimes K_J)}u_{I,J}.
\end{equation*}
Finally,
\begin{equation*}
        \|ATB-M\|=\|O\|<\varepsilon,
\end{equation*}
and, because $A$ and $B$ are contractions,
\begin{equation*}
        \|M\|\leq\|ATB\|+\|ATB-M\|\leq\|T\|+\varepsilon.
\end{equation*}
This completes the proof.
\end{proof}

\section{Technical claims for the multiplier reduction construction}
\label{p3:sec:technical-claims-multiplier-reduction}

In this section we prove the auxiliary lemmas and construction claims invoked in the proof of \Cref{p3:thm:reduction-to-multiplier}. The terminology is the one fixed at the beginning of the construction section.

\begin{lemma}[One-coordinate strip estimates]
\label{p3:lem:one-coordinate-strip-estimates}
Let $(d_I)_{I\in\PthreeD}$ define a scalar Haar multiplier $M_d$ on $L_p^0$, and let $J\ne J'$. The product strip operator
\begin{equation*}
        O_{J,J'}^{\rm out}u_{I,L}=\begin{cases} d_Iu_{I,J},&L=J',\\ 0,&L\ne J' \end{cases}
\end{equation*}
satisfies
\begin{equation*}
        \|O_{J,J'}^{\rm out}\|_{\PthreeL(X_{00})}\leq \|M_d\|_{\PthreeL(L_p^0)}.
\end{equation*}
In particular, with $U_p$ as in \Cref{p3:lem:lp-haar-multiplier},
\begin{equation*}
        \|O_{J,J'}^{\rm out}\|_{\PthreeL(X_{00})}\leq U_p\sup_{I\in\PthreeD}|d_I|.
\end{equation*}
Similarly, if $(e_J)_{J\in\PthreeD}$ defines a scalar Haar multiplier $N_e$ on $L_1^0$ and $I\ne I'$, then the product strip operator
\begin{equation*}
        O_{I,I'}^{\rm in}u_{L,J}=\begin{cases} e_Ju_{I,J},&L=I',\\ 0,&L\ne I' \end{cases}
\end{equation*}
satisfies
\begin{equation*}
        \|O_{I,I'}^{\rm in}\|_{\PthreeL(X_{00})}\leq \|N_e\|_{\PthreeL(L_1^0)}.
\end{equation*}
\end{lemma}

\begin{proof}
For the same outer strip, define $P_{J',J}\colon L_1^0\to L_1^0$ by
\begin{equation*}
        P_{J',J}v=\Pthreeip{f_{J'}^*}{v}f_J.
\end{equation*}
Then $\|P_{J',J}\|\leq1$. For $x\in X_{00}$, let $g(s)=\Pthreeip{f_{J'}^*}{x(s,\cdot)}$. Since $x$ has zero outer integral, $g\in L_p^0$. Also $\|g\|_{L_p}\leq\|x\|_{L_p(L_1)}$, and the same outer strip is $(M_dg)\otimes f_J$. Hence
\begin{equation*}
        \|O_{J,J'}^{\rm out}x\|_{L_p(L_1)}\leq\|M_d\|\|x\|_{L_p(L_1)}.
\end{equation*}
The estimate with $U_p$ follows from \Cref{p3:lem:lp-haar-multiplier}.

For the same inner strip, define $P_{I',I}\colon L_p^0\to L_p^0$ by
\begin{equation*}
        P_{I',I}x=\Pthreeip{e_{I'}^*}{x}e_I.
\end{equation*}
Then $\|P_{I',I}\|\leq1$. For $x\in X_{00}$, let
\begin{equation*}
        v(t)=\int_0^1 e_{I'}^*(s)x(s,t)\,ds.
\end{equation*}
Since $x$ has zero inner integral, $v\in L_1^0$. By H\"older's inequality, $\|v\|_1\leq\|x\|_{L_p(L_1)}$, and the same inner strip is $e_I\otimes N_ev$. Hence
\begin{equation*}
        \|O_{I,I'}^{\rm in}x\|_{L_p(L_1)}\leq\|N_e\|\|x\|_{L_p(L_1)}.
\end{equation*}
\end{proof}

\begin{lemma}[Product Haar coefficients separate points]
\label{p3:lem:product-coefficient-separation}
If $x\in X_{00}$ and
\begin{equation*}
        \Pthreeip{u_{I,J}^*}{x}=0\qquad(I,J\in\PthreeD),
\end{equation*}
then $x=0$. In particular, if two bounded operators on $X_{00}$ have the same product Haar coefficients on every product Haar vector, then they agree on the algebraic product Haar span.
\end{lemma}

\begin{proof}
Let $P_m$ be the finite product-Haar projection onto the span of $u_{I,J}$ with $I,J\in\bigcup_{k=0}^{m}\PthreeD_k$. These projections are obtained from the dyadic martingale difference projections in the two variables, and they converge strongly to the identity on $X_{00}\subset L_p(L_1)$. If all product Haar coefficients of $x$ vanish, then $P_mx=0$ for every $m$, hence $x=\lim_mP_mx=0$. The operator statement follows by applying this to the difference of the two images of each product Haar vector.
\end{proof}

\begin{lemma}[Inner strip branch variation estimate]
\label{p3:lem:inner-strip-branch-variation}
Let $(I,I')$ be a registered same inner strip, and put $d_J=a_{J,J}^{I,I'}$. If the old-part estimate in \textup{(\ref{p3:item:estimate-inner-old-part})} and the future-part estimates in \textup{(\ref{p3:item:estimate-inner-future-part})} hold for this strip, then
\begin{equation*}
        \PthreeBV((d_J)_{J\in\PthreeD})<\tau_{I,I'}.
\end{equation*}
\end{lemma}

\begin{proof}
Fix an infinite dyadic branch $J_0\supset J_1\supset\ldots$. The branch is the union of an initial part contained in the old finite tree $\mathcal T_{I,I'}^{\rm old}$ and, possibly after crossing one index $B\in\mathcal R_{I,I'}$, a tail contained in the future forest $\mathcal F_{I,I'}$. On the old part, the contribution is bounded by the corresponding terms in $G_{I,I'}$. If the branch crosses from an old parent $p(B)$ into $B\in\mathcal R_{I,I'}$, then
\begin{equation*}
        |d_{p(B)}-d_B|\leq |a_{p(B), p(B)}^{I,I'}|+|\lambda_{\Delta_B}(R_{I,I'})|+\eta_B^{I,I'}.
\end{equation*}
If the branch starts inside the future forest, let $B\in\mathcal R_{I,I'}$ be the frontier root containing $J_0$. The value $|d_{J_0}|$ is bounded by $|\lambda_{\Delta_B}(R_{I,I'})|$, the $\rho$-sum along the path from $B$ to $J_0$, and $\eta_{J_0}^{I,I'}$. Finally, for every future edge $C\to C'$ on the branch,
\begin{equation*}
        |d_C-d_{C'}|\leq \eta_C^{I,I'}+\rho_{C,C'}^{I,I'}+\eta_{C'}^{I,I'}.
\end{equation*}
Therefore the branch variation is bounded by
\begin{equation*}
        G_{I,I'}+3\sum_{B\in\mathcal R_{I,I'}}|\lambda_{\Delta_B}(R_{I,I'})|+4\sum_{C\in\mathcal F_{I,I'}}\eta_C^{I,I'}+\sum_{C\to C',\ C,C'\in\mathcal F_{I,I'}}\rho_{C,C'}^{I,I'}.
\end{equation*}
By \eqref{p3:eq:same-inner-birth-condition}, the first two terms satisfy
\begin{equation*}
        G_{I,I'}+3\sum_{B\in\mathcal R_{I,I'}}|\lambda_{\Delta_B}(R_{I,I'})|\leq 3\left(G_{I,I'}+\sum_{B\in\mathcal R_{I,I'}}|\lambda_{\Delta_B}(R_{I,I'})|\right)<\frac{3\tau_{I,I'}}{C_0}.
\end{equation*}
Together with \eqref{p3:eq:same-inner-future-budget} and $C_0=8$, this gives a bound smaller than $\tau_{I,I'}$. Taking the supremum over all branches gives the claim.
\end{proof}

\begin{p3claim}[Registration is well defined and locally finite]
\label{p3:claim:registration-finiteness}
Every off diagonal product coefficient is assigned to exactly one of the three mechanisms in the construction. At every half stage, only finitely many new strips and mixed coefficients are registered, only finitely many reservoir tails are updated, and the side constraint lists submitted to the finite splitting lemmas are finite. When a same inner strip is registered, the countable family of future tolerances attached to it is chosen after the finite constraints have been imposed.
\end{p3claim}

\begin{proof}
Let $a_{J,J'}^{I,I'}$ be an off-diagonal coefficient. If $I=I'$ and $J\ne J'$, the coefficient belongs to the same outer strip indexed by $(J,J')$, and this strip is registered when the later of $K_J$ and $K_{J'}$ is born. If $I\ne I'$ and $J=J'$, the coefficient belongs to the same inner strip indexed by $(I,I')$, and this strip is registered when the later of $H_I$ and $H_{I'}$ is born. If $I\ne I'$ and $J\ne J'$, the coefficient is mixed, and it is registered when the last of $H_I,H_{I'},K_J,K_{J'}$ is born. These three alternatives are disjoint and exhaust the off-diagonal cases.

At a fixed half stage, only finitely many blocks have been born and only finitely many new blocks are born. Hence only finitely many strips can be registered for the first time, and only finitely many mixed coefficients can have their last block born at that half stage. The current frontiers are finite dyadic generations, so only finitely many reservoir tails are created or shrunk at that half stage. Therefore every side constraint list submitted to a finite splitting lemma is finite. The future tolerance family attached to a newly registered same inner strip is countable, but it is chosen separately after the finite splitting step and is not part of the side constraint list.
\end{proof}

We now prove the main claims in the construction, in other words, that we can make appropriate choices in each of the half-stages.

\begin{p3claim}[The outer half stage is possible]
\label{p3:claim:outer-half-stage}
Assume that $\mathfrak S_n^0$ is legal. Then one can choose the outer splits, the reservoir tails for the new outer frontier, the propagated same outer numbers $\theta$ and $\omega$, and, for every same inner strip registered at this half stage, the objects in \textup{(\ref{p3:item:quantity-same-inner})}, so that the resulting half state $\mathfrak S_n^{1/2}$ is legal. The chosen outer splits determine the born blocks $H_I,H_I^*$ for $I\in\PthreeD_n$. More precisely, \textup{(\ref{p3:item:estimate-outer-coefficients})} and \textup{(\ref{p3:item:estimate-outer-tails})} hold for the same outer strips registered by the half state, \textup{(\ref{p3:item:estimate-mixed})} holds for every mixed coefficient registered at this half stage, and every same inner strip registered at this half stage satisfies the old-part condition \textup{(\ref{p3:item:estimate-inner-old-part})} and is assigned future tolerances so that the budget in \textup{(\ref{p3:item:estimate-inner-future-part})} holds.
\end{p3claim}

\begin{proof}
At the beginning of the outer half stage, the inner tree is fixed and the current outer frontier is indexed by $\PthreeD_n$. For each current outer frontier index $A$ and each same outer strip indexed by $(J,J')$ already registered, choose positive numbers $\eta_A^{J,J'}$ and $\delta_A^{J,J'}$ so small that
\begin{equation*}
        \eta_A^{J,J'}<\frac{\omega_A^{J,J'}}{2},\qquad \delta_A^{J,J'}<\frac{\omega_A^{J,J'}}{2}.
\end{equation*}
We shall apply \Cref{p3:lem:outer-tail} with the current common tail $S_A$, the bound $\theta_A^{J,J'}$, and the tolerances $\eta_A^{J,J'},\delta_A^{J,J'}$, together with the admissible side constrain now defined.

We form one finite admissible outer side constraint list. First, include every mixed coefficient whose registration time is the present outer half stage:
\begin{equation*}
        |a_{J,J'}^{I,I'}|<\mu_{I,I',J,J'}.
\end{equation*}
The inner blocks are fixed. If exactly one of $H_I,H_{I'}$ is new, the constraint is one-sided in the new outer block. If both are new, then $I\ne I'$, so the two new outer blocks are supported on distinct current frontier sets, and the constraint is bilinear between distinct new blocks. Thus the constraint is admissible.

Next consider a same inner strip indexed by $(I,I')$ whose registration time is the present outer half stage. Let $\mathcal T_{I,I'}^{\rm old}$ be the already born inner tree and let $\mathcal R_{I,I'}$ be the current inner frontier. For every $P\in\mathcal T_{I,I'}^{\rm old}$, the coefficient is
\begin{equation*}
        a_{P,P}^{I,I'}=\Pthreeip{H_I^*\otimes K_P^*}{T(H_{I'}\otimes K_P)}.
\end{equation*}
Choose tolerances for the finitely many old coefficients so small that, if all corresponding inequalities hold, then
\begin{equation*}
        G_{I,I'}<\frac{\tau_{I,I'}}{4C_0}.
\end{equation*}
This is possible because $G_{I,I'}$ is a finite sum of terms $|a_{P,P}^{I,I'}|$, $|a_{P,P}^{I,I'}-a_{Q,Q}^{I,I'}|$, and $|a_{p(B), p(B)}^{I,I'}|$, and making all finitely many coefficients $|a_{P,P}^{I,I'}|$ sufficiently small makes that finite sum small. These inequalities are admissible outer constraints, because the inner blocks are fixed and at least one of the two outer blocks is new. For each $B\in\mathcal R_{I,I'}$, the trace value is
\begin{equation*}
        \lambda_{\Delta_B}(R_{I,I'})=\Pthreeip{H_I^*}{\mathcal I_{\Delta_B}^T H_{I'}},
\end{equation*}
by \Cref{p3:lem:inner-set-outer-trace-operator}. Choose positive numbers $\zeta_B^{I,I'}$, $B\in\mathcal R_{I,I'}$, such that
\begin{equation*}
        \sum_{B\in\mathcal R_{I,I'}}\zeta_B^{I,I'}<\frac{\tau_{I,I'}}{4C_0},
\end{equation*}
and add the admissible constraints
\begin{equation*}
        |\lambda_{\Delta_B}(R_{I,I'})| = |\Pthreeip{H_I^*}{\mathcal I_{\Delta_B}^T H_{I'}}| <\zeta_B^{I,I'}\qquad(B\in\mathcal R_{I,I'}).
\end{equation*}

The side constraint list is finite by Claim~\ref{p3:claim:registration-finiteness}. Apply \Cref{p3:lem:outer-tail} simultaneously to the current outer frontier sets, the operators $Q_{J,J'}$ for the already registered same outer strips, the common tails $S_A$, the chosen tolerances, and this finite side constraint list.

All side constraints are then satisfied. Hence all mixed coefficients registered at this half stage satisfy their $\mu$ bounds, and every same inner strip indexed by $(I,I')$ registered at this half stage satisfies
\begin{equation*}
        G_{I,I'}+\sum_{B\in\mathcal R_{I,I'}}|\lambda_{\Delta_B}(R_{I,I'})|<\frac{\tau_{I,I'}}{C_0}.
\end{equation*}
This is precisely the old-part condition \textup{(\ref{p3:item:estimate-inner-old-part})} for the newly registered same inner strips.
For a same outer strip indexed by $(J,J')$ already registered, \Cref{p3:lem:outer-tail} gives
\begin{equation*}
        |\Pthreeip{H_A^*}{Q_{J,J'}H_A}|<\theta_A^{J,J'}+\eta_A^{J,J'}<\theta_A^{J,J'}+\omega_A^{J,J'}<\sigma_{J,J'}
\end{equation*}
for every newly born outer block $H_A$. For each child $A^\vartheta$, $\vartheta\in\{+,-\}$, the same lemma gives an infinite tail $S_{A^\vartheta}$ such that
\begin{equation*}
        |\alpha_m^{\Gamma_{A^\vartheta}}(Q_{J,J'})|<\theta_A^{J,J'}+\delta_A^{J,J'}\qquad(m\in S_{A^\vartheta}).
\end{equation*}
Define
\begin{equation*}
        \theta_{A^\vartheta}^{J,J'}=\theta_A^{J,J'}+\delta_A^{J,J'},\qquad \omega_{A^\vartheta}^{J,J'}=\omega_A^{J,J'}-\delta_A^{J,J'}.
\end{equation*}
Then $\omega_{A^\vartheta}^{J,J'}>0$ and
\begin{equation*}
        \theta_{A^\vartheta}^{J,J'}+\omega_{A^\vartheta}^{J,J'}=\theta_A^{J,J'}+\omega_A^{J,J'}<\sigma_{J,J'}.
\end{equation*}
Thus the same outer tail estimate is preserved.

For each same inner strip indexed by $(I,I')$ registered at this half stage, choose positive numbers $\eta_C^{I,I'}$ and $\rho_{C,C'}^{I,I'}$ on the future forest $\mathcal F_{I,I'}$ so that \eqref{p3:eq:same-inner-future-budget} holds. Thus the future tolerance condition in \textup{(\ref{p3:item:estimate-inner-future-part})} is initialized for the newly registered same inner strips; the corresponding node and edge estimates will be imposed in later inner half stages when those future blocks and edges are born. Therefore the half state is legal.
\end{proof}

\begin{p3claim}[The inner half stage is possible]
\label{p3:claim:inner-half-stage}
Assume that $\mathfrak S_n^{1/2}$ is legal. Then one can choose the inner splits, the possibly shrunk reservoir tails for the current outer frontier, and the initial same outer numbers $\theta$ and $\omega$ for every same outer strip registered at this half stage so that the resulting full state $\mathfrak S_{n+1}^0$ is legal. The chosen inner splits determine the born blocks $K_J,K_J^*$ for $J\in\PthreeD_n$. More precisely, the relevant node and edge estimates from the future-part condition \textup{(\ref{p3:item:estimate-inner-future-part})} are imposed for the same inner strips already registered, \textup{(\ref{p3:item:estimate-mixed})} holds for every mixed coefficient registered at this half stage, and every same outer strip registered at this inner half stage satisfies \textup{(\ref{p3:item:estimate-outer-coefficients})} and \textup{(\ref{p3:item:estimate-outer-tails})}.
\end{p3claim}

\begin{proof}
At the beginning of the inner half stage, the outer tree is fixed. We form one finite admissible inner side constraint list. First include every mixed coefficient whose registration time is the present inner half stage:
\begin{equation*}
        |a_{J,J'}^{I,I'}|<\mu_{I,I',J,J'}.
\end{equation*}
The outer blocks are fixed. If exactly one of $K_J,K_{J'}$ is new, the constraint is one-sided in the new inner block. If both are new, then $J\ne J'$, so the two new inner blocks are supported on distinct current frontier sets, and the constraint is bilinear between distinct new blocks. Thus the constraint is admissible.

For each current outer frontier index $A$ in the half state $\mathfrak S_n^{1/2}$, choose once and for all a free ultrafilter $\mathcal U_A$ containing the current reservoir tail $S_A$.

For each same outer strip indexed by $(J,J')$ whose registration time is the present inner half stage, add, for every already born outer index $I$, the constraint
\begin{equation*}
        |\Pthreeip{H_I^*\otimes K_J^*}{T(H_I\otimes K_{J'})}|<\sigma_{J,J'}.
\end{equation*}
The outer block is fixed, so this is an admissible inner constraint. If exactly one of $K_J,K_{J'}$ is new it is one-sided; if both are new it is bilinear between distinct current inner frontier sets.

For the same newly registered same outer strip, add, for every current outer frontier index $A$, the shadow constraint
\begin{equation*}
        |\Pthreeip{K_J^*}{\Omega_{\Gamma_A}^{T,\mathcal U_A}K_{J'}}|<\frac{\sigma_{J,J'}}{8}.
\end{equation*}
By \Cref{p3:prop:shadow}, $\Omega_{\Gamma_A}^{T,\mathcal U_A}\in\PthreeL(L_1^0)$, so this is an admissible inner constraint. The same ultrafilter $\mathcal U_A$ is used for every same outer strip registered at this half stage.

Now add the same inner estimates. For every same inner strip indexed by $(I,I')$ already registered and every current inner frontier index $C\in\mathcal F_{I,I'}$, include the operator $R_{I,I'}$ with tolerance
\begin{equation*}
        \delta_C^{I,I'}=\min\{\eta_C^{I,I'},\rho_{C,C^+}^{I,I'},\rho_{C,C^-}^{I,I'}\}.
\end{equation*}
The side constraint list and the list of already registered same inner strips are finite by Claim~\ref{p3:claim:registration-finiteness}. Apply \Cref{p3:lem:l1-simultaneous} to the current inner frontier sets, the finite families of operators $R_{I,I'}$ for the already registered same inner strips, the tolerances $\delta_C^{I,I'}$, and the finite admissible side constraint list.

The chosen signs impose all listed side constraints. Therefore every mixed coefficient registered at this half stage satisfies its $\mu$ bound. Since the old finite tree and registration frontier attached to a same inner strip do not change after that strip is registered, \textup{(\ref{p3:item:estimate-inner-old-part})} remains true for every previously registered same inner strip. The conclusions of \Cref{p3:lem:l1-simultaneous} also give
\begin{equation*}
        |\Pthreeip{K_C^*}{R_{I,I'}K_C}-\lambda_{\Delta_C}(R_{I,I'})|<\eta_C^{I,I'}
\end{equation*}
for every same inner strip indexed by $(I,I')$ already registered and every current inner frontier index $C\in\mathcal F_{I,I'}$, and
\begin{equation*}
        |\lambda_{\Delta_{C'}}(R_{I,I'})-\lambda_{\Delta_C}(R_{I,I'})|<\rho_{C,C'}^{I,I'}
\end{equation*}
for the two children $C'$ of each current $C$. These are exactly the node and edge estimates required in the future-part condition \textup{(\ref{p3:item:estimate-inner-future-part})} at this half stage.

Let $(J,J')$ index a same outer strip registered at this half stage, and write $Q_{J,J'}$ as in \eqref{p3:eq:same-outer-operator}. The old coefficient constraints give
\begin{equation*}
        |a_{J,J'}^{I,I}|<\sigma_{J,J'}
\end{equation*}
for every already born outer index $I$. For each current outer frontier index $A$, the shadow constraint and \Cref{p3:lem:same-outer-identity} give
\begin{equation*}
        \lim_{m\to\mathcal U_A}\alpha_m^{\Gamma_A}(Q_{J,J'})=\Pthreeip{K_J^*}{\Omega_{\Gamma_A}^{T,\mathcal U_A}K_{J'}}.
\end{equation*}
Hence
\begin{equation*}
        \{m\in\mathbb N:|\alpha_m^{\Gamma_A}(Q_{J,J'})|<\sigma_{J,J'}/4\}\in\mathcal U_A.
\end{equation*}
For fixed $A$, all such good sets, as $(J,J')$ ranges over the finitely many same outer strips registered at the present half stage, belong to the same ultrafilter $\mathcal U_A$. Hence their finite intersection with the old reservoir tail $S_A$ also belongs to $\mathcal U_A$, and is infinite. Replace $S_A$ by this finite intersection. For every newly registered same outer strip indexed by $(J,J')$, this reservoir satisfies
\begin{equation*}
        |\alpha_m^{\Gamma_A}(Q_{J,J'})|<\frac{\sigma_{J,J'}}{4}\qquad(m\in S_A).
\end{equation*}
We record this estimate by setting
\begin{equation*}
        \theta_A^{J,J'}=\frac{\sigma_{J,J'}}{4}
\end{equation*}
and choose, for instance,
\begin{equation*}
        \omega_A^{J,J'}=\frac{\sigma_{J,J'}}{2}.
\end{equation*}
Then $\theta_A^{J,J'}+\omega_A^{J,J'}=3\sigma_{J,J'}/4<\sigma_{J,J'}$. Thus \eqref{p3:eq:outer-tail-invariant} and \eqref{p3:eq:outer-room-invariant} hold for all newly registered same outer strips, while they remain true for the previously registered same outer strips because the common tail has only been shrunk. Hence $\mathfrak S_{n+1}^0$ is legal.
\end{proof}

\section{Scalar compression and primarity}
\label{p3:sec:scalar-compression}

We now combine the arbitrary-to-multiplier reduction with the LMMS scalar compression theorem for bounded product Haar multipliers. This gives the scalar compression statement on $X_{00}$.

\begin{theorem}[Scalar compression on $X_{00}$]
\label{p3:thm:scalar-x00}
For every $S\in\PthreeL(X_{00})$ and every $\varepsilon>0$, there are $A,B\in\PthreeL(X_{00})$ and a scalar $c\in\mathbb R$ such that
\begin{equation*}
        AB=\PthreeId_{X_{00}},\qquad \|ASB-c\PthreeId_{X_{00}}\|_{\PthreeL(X_{00})}<\varepsilon .
\end{equation*}
Moreover, $A$ and $B$ may be chosen so that
\begin{equation*}
        \|A\|\,\|B\|\leq 1+\varepsilon .
\end{equation*}
\end{theorem}

\begin{proof}
Choose $\delta>0$ small enough so that
\begin{equation*}
        (1+\delta)\delta+\delta<\varepsilon,\qquad 1+\delta<1+\varepsilon .
\end{equation*}
Apply \Cref{p3:thm:reduction-to-multiplier} to $S$ with accuracy $\delta$. We obtain $A_0,B_0\in\PthreeL(X_{00})$ and a bounded product Haar multiplier $M$ on $X_{00}$ such that
\begin{equation*}
        A_0B_0=\PthreeId_{X_{00}},\qquad \|A_0SB_0-M\|_{\PthreeL(X_{00})}<\delta .
\end{equation*}
By \Cref{p3:lem:exterior}, the maps $A_0$ and $B_0$ are contractions.

Apply \Cref{p3:thm:lmms} to $M$ with accuracy $\delta$. There are $A_1,B_1\in\PthreeL(X_{00})$ and $c\in\mathbb R$ such that
\begin{equation*}
        A_1B_1=\PthreeId_{X_{00}},\qquad \|A_1MB_1-c\PthreeId_{X_{00}}\|_{\PthreeL(X_{00})}<\delta,\qquad \|A_1\|\,\|B_1\|\leq 1+\delta .
\end{equation*}
Set
\begin{equation*}
        A=A_1A_0,\qquad B=B_0B_1 .
\end{equation*}
Then
\begin{equation*}
        AB=A_1A_0B_0B_1=A_1B_1=\PthreeId_{X_{00}}.
\end{equation*}
Moreover,
\begin{equation*}
\begin{aligned}
        \|ASB-c\PthreeId_{X_{00}}\|_{\PthreeL(X_{00})}
        &\leq \|A_1(A_0SB_0-M)B_1\|_{\PthreeL(X_{00})}+\|A_1MB_1-c\PthreeId_{X_{00}}\|_{\PthreeL(X_{00})} \\
        &< \|A_1\|\,\|B_1\|\delta+\delta \leq (1+\delta)\delta+\delta<\varepsilon .
\end{aligned}
\end{equation*}
Finally, since $A_0$ and $B_0$ are contractions,
\begin{equation*}
        \|A\|\,\|B\|\leq \|A_1\|\,\|B_1\|\leq 1+\delta<1+\varepsilon.
\end{equation*}
\end{proof}

It remains to pass from the doubly cancellative space back to $X$. The passage uses a standard complemented cancellative copy.

\begin{lemma}[A one-complemented cancellative copy]
\label{p3:lem:cancellative-copy}
The space $X_{00}$ contains an isometric copy of $X$ which is the range of a norm-one projection on $X$.
\end{lemma}

\begin{proof}
Let
\begin{equation*}
        r(t)=\Pthreeone_{[0,1/2)}(t)-\Pthreeone_{[1/2,1)}(t),\qquad \sigma(t)=2t\pmod 1 .
\end{equation*}
Define $U\colon X\to X$ by
\begin{equation*}
        (Uf)(s,t)=r(s)r(t)f(\sigma(s),\sigma(t)).
\end{equation*}
Since $\sigma$ preserves Lebesgue measure and $|r|=1$, we have
\begin{equation*}
        \|Uf\|_{L_p(L_1)}=\|f\|_{L_p(L_1)}\qquad(f\in X),
\end{equation*}
so $U$ is an isometry. Moreover, for almost every $s$,
\begin{equation*}
        \int_0^1 (Uf)(s,t)\,dt=0,
\end{equation*}
and, as an $L_1$-valued integral in the outer variable,
\begin{equation*}
        \int_0^1 (Uf)(s,\cdot)\,ds=0.
\end{equation*}
Thus $U(X)\subset X_{00}$.

Define $V\colon X\to X$ by
\begin{equation*}
        (Vg)(u,v)=\frac14\sum_{\alpha,\beta\in\{0,1\}}(-1)^{\alpha+\beta}g\left(\frac{u+\alpha}{2},\frac{v+\beta}{2}\right).
\end{equation*}
For $f\in X$ we have
\begin{equation*}
        U f\left(\frac{u+\alpha}{2},\frac{v+\beta}{2}\right)=(-1)^{\alpha+\beta}f(u,v),
\end{equation*}
and therefore
\begin{equation*}
        VUf=f.
\end{equation*}

It remains to check that $V$ is contractive. For $g\in X$, put
\begin{equation*}
        a(s)=\|g(s,\cdot)\|_{L_1}.
\end{equation*}
Then, for almost every $u$,
\begin{equation*}
        \|Vg(u,\cdot)\|_{L_1}\leq \frac12\left(a(u/2)+a((u+1)/2)\right).
\end{equation*}
By convexity and a change of variables,
\begin{equation*}
        \|Vg\|_{L_p(L_1)}^p\leq \int_0^1\frac12\left(a(u/2)^p+a((u+1)/2)^p\right)\,du=\|g\|_{L_p(L_1)}^p.
\end{equation*}
Thus $\|V\|\leq1$.

Now set
\begin{equation*}
        P=UV .
\end{equation*}
Since $VU=\PthreeId_X$, we have
\begin{equation*}
        P^2=UVUV=U(VU)V=UV=P,
\end{equation*}
and
\begin{equation*}
        \operatorname{ran} P=U(X)\subset X_{00}.
\end{equation*}
Also $\|P\|\leq\|U\|\|V\|\leq1$, and since $P$ is a non-zero projection, $\|P\|=1$. Hence $U(X)$ is an isometric copy of $X$ contained in $X_{00}$ and one-complemented in $X$.
\end{proof}

Finally, we are ready for the proof of our main result.

\begin{proof}[Proof of \Cref{p3:thm:main}]
Let
\begin{equation*}
        P_{00}=(\PthreeId-\mathbb E_1)(\PthreeId-\mathbb E_2)\colon X\to X_{00},
\end{equation*}
where $\mathbb E_i$ denotes integration in the $i$-th variable. By \Cref{p3:lem:x00-projection}, $P_{00}$ is a bounded projection onto $X_{00}$ and $\|P_{00}\|\leq4$. Let $\iota\colon X_{00}\to X$ denote the inclusion.

Fix $T\in\PthreeL(X)$, and define
\begin{equation*}
        S=P_{00}T\iota\in\PthreeL(X_{00}).
\end{equation*}
Apply \Cref{p3:thm:scalar-x00} to $S$ with accuracy $\eta=1/8$. We obtain $A_0,B_0\in\PthreeL(X_{00})$ and $c\in\mathbb R$ such that
\begin{equation*}
        A_0B_0=\PthreeId_{X_{00}},\qquad \|A_0SB_0-c\PthreeId_{X_{00}}\|_{\PthreeL(X_{00})}<1/8,\qquad \|A_0\|\,\|B_0\|\leq 9/8.
\end{equation*}
Either $|c|\geq1/2$ or $|1-c|\geq1/2$.

First suppose that $|c|\geq1/2$. Then $A_0SB_0$ is invertible on $X_{00}$, and
\begin{equation*}
        \|(A_0SB_0)^{-1}\|\leq \frac1{|c|-1/8}\leq \frac83.
\end{equation*}
Set
\begin{equation*}
        R=(A_0SB_0)^{-1}.
\end{equation*}
Since $S=P_{00}T\iota$, we have
\begin{equation*}
        RA_0P_{00}T\iota B_0=\PthreeId_{X_{00}}.
\end{equation*}

Now let $U\colon X\to X_{00}$ and $V\colon X_{00}\to X$ be the maps from \Cref{p3:lem:cancellative-copy}, so that $U$ is an isometry, $V$ is a contraction, and $VU=\PthreeId_X$. Composing the preceding factorization with $U$ on the right and $V$ on the left gives
\begin{equation*}
        VRA_0P_{00}T\iota B_0U=\PthreeId_X.
\end{equation*}
Thus $\PthreeId_X$ factors through $T$, with factorization constant at most
\begin{equation*}
        \|V\|\,\|R\|\,\|A_0\|\,\|P_{00}\|\,\|B_0\|\,\|U\|\leq \frac83\cdot\frac98\,\|P_{00}\|=3\|P_{00}\|.
\end{equation*}

Now suppose that $|1-c|\geq1/2$. Since
\begin{equation*}
        A_0(\PthreeId_{X_{00}}-S)B_0=\PthreeId_{X_{00}}-A_0SB_0,
\end{equation*}
we have
\begin{equation*}
        \|A_0(\PthreeId_{X_{00}}-S)B_0-(1-c)\PthreeId_{X_{00}}\|_{\PthreeL(X_{00})}<1/8.
\end{equation*}
Thus $A_0(\PthreeId_{X_{00}}-S)B_0$ is invertible and its inverse has norm at most $8/3$. Since
\begin{equation*}
        \PthreeId_{X_{00}}-S=P_{00}(\PthreeId_X-T)\iota,
\end{equation*}
the same argument gives
\begin{equation*}
        V\bigl(A_0(\PthreeId_{X_{00}}-S)B_0\bigr)^{-1}A_0P_{00}(\PthreeId_X-T)\iota B_0U=\PthreeId_X.
\end{equation*}
Hence $\PthreeId_X$ factors through $\PthreeId_X-T$, again with factorization constant at most $3\|P_{00}\|$.

Therefore, $X$ has the uniform primary factorization property, with constant $K=3\|P_{00}\|$. Since $\mathbb E_1$ and $\mathbb E_2$ are contractive projections, $\|P_{00}\|\leq4$, so one may take $K\leq12$. Primariness follows automatically from the UPFP and Pe\l czy\'nski's decomposition method.
\end{proof}

\end{problempaperbody}

\newpage
\part{The Automated Pipeline}\label{part3}
\markboth{The Automated Pipeline}{The Automated Pipeline}
\section{Technical and methodological considerations}
\label{sec:technical-details-pipeline}

\subsection{An experimental search process}

The computational part of the project began as an experimental search procedure rather than as a fully developed autonomous system. In the initial runs, a small number of agents inspected papers, extracted questions, attempted proofs or counterexamples, and retained promising outputs for later review. After these runs produced several promising proof candidates, we formalized the procedure into the pipeline described below.

The subsequent refinements were primarily organizational rather than mathematical. The agent instructions did not include problem-specific lemmas, proof strategies, or techniques. Agents were instructed to read the source paper, assess whether an extracted question was approachable, and attempt it using the local context and any references they chose to consult. Later changes concerned classification, record keeping, and coordination: distinguishing full from partial results, identifying answers already present in the literature, preserving unsuccessful attempts, preventing duplication across parallel agents, and standardizing the resulting packets. The pipeline supplied candidate targets, while the protocol provided a common structure in which attempts could be recorded and reviewed; neither prescribed a solution path for an individual problem. Human involvement remained substantial: we selected the broad mathematical area, revised the protocol in response to observed failure modes, chose packets for closer examination, and edited or rewrote the material retained for inclusion.

\subsection{Pipeline overview}

\Cref{fig:pipeline-overview} gives the operational view of the stable version of the experiment.  The diagram separates two automatic stages, candidate setup and agent work. Run memory denotes shared state consulted by agents, while human review sits outside the automatic loop.

\begin{figure}[t]
\centering
\definecolor{pipeblue}{RGB}{232,238,247}
\definecolor{pipegreen}{RGB}{232,243,235}
\definecolor{pipeamber}{RGB}{249,242,226}
\definecolor{pipeviolet}{RGB}{241,236,247}
\definecolor{pipegray}{RGB}{246,246,246}
\resizebox{\textwidth}{!}{%
\begin{tikzpicture}[
  font=\small,
  groupbox/.style={draw=black!35, rounded corners=4pt, dashed, line width=.55pt},
  grouplabel/.style={font=\scriptsize\bfseries, fill=white, inner sep=2pt},
  stagebox/.style={draw=black!55, rounded corners=2pt, minimum height=1.35cm,
    text width=3.35cm, align=center, inner sep=5pt},
  setupbox/.style={stagebox, text width=4.05cm},
  widebox/.style={stagebox, text width=3.75cm},
  arrow/.style={-{Latex[length=2.2mm]}, line width=0.55pt, draw=black!70},
  feedback/.style={-{Latex[length=2.1mm]}, line width=0.45pt,
    draw=black!55, dashed}
]
\node[grouplabel, anchor=west] at (-2.80,1.56) {Candidate setup};
\draw[groupbox, fill=pipeblue!20] (-2.85,-1.38) rectangle (12.75,1.38);
\node[grouplabel, anchor=west] at (2.20,-2.18) {Agent work};
\draw[groupbox, fill=pipeamber!20] (2.20,-5.10) rectangle (11.90,-2.40);

\node[setupbox, fill=pipeblue] (corpus) at (0,0)
  {\textbf{Corpus}\\arXiv metadata\\TeX/PDF source};
\node[setupbox, fill=pipeblue] (signals) at (4.95,0)
  {\textbf{Source signals}\\question and\\conjecture cues};
\node[setupbox, fill=pipegreen] (queue) at (9.90,0)
  {\textbf{Target queue}\\ranked targets\\lane assignment};

\node[stagebox, fill=pipegray] (memory) at (-0.85,-3.75)
  {\textbf{Run memory}\\global index\\past outcomes};
\node[stagebox, fill=pipeamber] (agent) at (4.65,-3.75)
  {\textbf{Agent attempt}\\read source\\try proof/example};
\node[widebox, fill=pipeviolet] (packet) at (9.25,-3.75)
  {\textbf{Make packet}\\result or\\attempt note};
\node[stagebox, fill=pipegray] (review) at (14.15,-3.75)
  {\textbf{Human review}};

\draw[arrow] (corpus.east) -- (signals.west);
\draw[arrow] (signals.east) -- (queue.west);
\draw[arrow] (queue.south) -- ++(0,-1.12) -| (agent.north);
\draw[arrow] (memory.east) -- (agent.west);
\draw[arrow] (agent.east) -- (packet.west);
\draw[arrow] (packet.east) -- (review.west);

\draw[feedback] (packet.south) -- ++(0,-.95) -| (memory.south);
\end{tikzpicture}
}
\caption{Operational view of the automatic proof-discovery pipeline. Candidate
setup builds the target queue; agent work consults run memory, makes a packet
or attempt note, and passes the result to human review.}
\label{fig:pipeline-overview}
\end{figure}
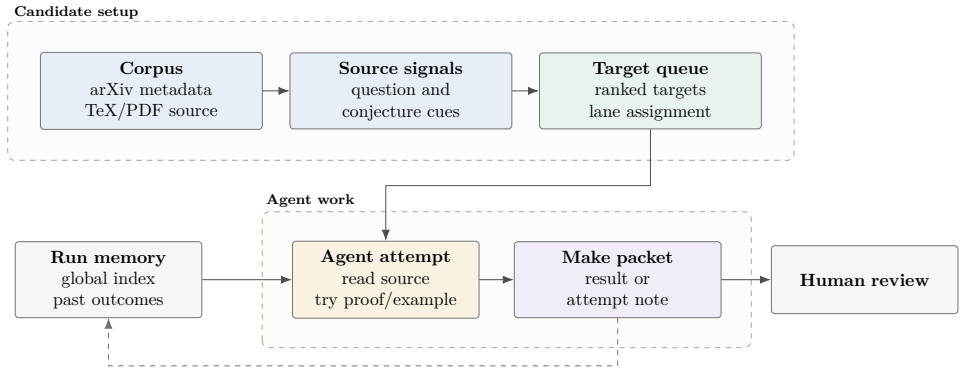

The implementation combined scripts with protocol files read by the model. Scripts handled the deterministic parts: collecting source material, ranking papers, assigning lanes, and rebuilding the global indexes. The mathematical choices were left to the agent. At the start of a run the agent read a protocol file, selected targets from its lane, inspected the source, and decided whether to attempt a proof, search for a counterexample, record a literature answer, or leave an attempt note.

This is not a fully scripted proof engine. It is a controlled way to run Codex sessions as mathematical agents, while preserving enough structure for their outputs to be compared across many papers and sessions. Once the protocols stabilized, most deviations concerned formatting or coordination rather than changes in the mathematical task.

\subsection{Candidate setup}

The candidate setup corresponds to the three boxes in the upper part of \Cref{fig:pipeline-overview}.  The \emph{corpus} consisted of arXiv metadata together with TeX source whenever it was available.  TeX source was important because open questions, theorem environments, definitions, and references can be located more reliably in source form than from PDF text extraction. 

The second box, \emph{source signals}, is a deliberately simple extraction stage. We searched the source for phrases such as ``Question'', ``Problem'', ``Conjecture'', ``we ask'', and ``we do not know''.  These signals did not define the target by themselves.  They only marked passages that an agent should read.  This distinction matters because many papers ask a question and
then answer it nearby, or quote an older problem only to report a known solution.

The final setup box is the \emph{target queue}. Ranking was a triage device, not a measure of absolute importance.  For the Banach-space run, the ranker favoured functional-analytic terminology, explicit question or conjecture language, publication and affiliation signals, and references to central work in the area.  The goal is practical: to give the agents a stream of papers likely to contain serious questions that were still specific enough for an autonomous attempt.

\subsection{Agent work}

The agent workflow in \Cref{fig:pipeline-overview} begins with the run memory. Before attacking a target, an agent checks whether the paper, the question, or a close variant has already appeared in the run. The same check helps avoid same paper false positives, in which a paper states a question and then answers it nearby. If the target has already been covered, the agent is instructed to record only what is useful for future avoidance and then move on.

The next stage is the mathematical attempt. After reading the relevant source passage, the agent decides whether to seek a proof or counterexample, establish a special case, isolate an obstruction, or identify an applicable theorem from the literature. The instructions give priority to full proofs and counterexamples, followed by substantial partial results and conditional theorems. Literature searches are used to identify directly relevant prior results, particularly before an output is classified as a full solution; they are not intended to constitute an exhaustive review. If an existing answer is found, the agent records it and moves to another target. Otherwise, it proceeds with the mathematical attempt.

The final stage in the agent workflow is \emph{make packet}. If the attempt produces a substantive mathematical claim, the agent writes either a packet or a compact attempt note. The protocol instructs agents not to call a proof complete when it depends on an unproved lemma, to label conditional arguments honestly, and to record novelty checks for claimed full solutions or counterexamples. These safeguards reduce overclaiming, but do not eliminate it. Several packets later changed status after review. Thus packets should be read as structured claims for inspection, not as formal certificates. In the case of partial results, agents can be encouraged to push the argument to a full result, in which case they can produce an additional packet.

The early runs used one or two agents under close human supervision. Later runs are organized to support parallel work. In the main Banach run, each paper is assigned to one of twenty possible lanes by a deterministic hash of its arXiv identifier. An agent launched on a lane asks a queue helper for suitable targets in that lane.

\subsection{Run memory and packet records}

Run memory is the shared state behind the experiment. Each substantial result or literature identification is recorded there, together with the papers and questions already associated with full solutions, counterexamples, partial results, conditional results, literature answers, proof gaps, or attempts. Agents are instructed to consult this memory before committing to a target and to update it after each attempt. In parallel runs, the same memory also records which targets are currently being pursued, reducing collisions among agents.

The output of an agent run is not necessarily a solution packet. If an attempt fails but contains useful information, the agent can leave an attempt note. If it produces a result, a serious reduction, a proof gap, or a literature identification, it updates the run memory and, when appropriate, writes a human readable packet. The packet is the persistent record. It identifies the source paper, the question or conjecture being addressed, the proposed result, and enough evidence for a human reader to reconstruct the claim. Once the packet format had been standardized, packets also included the original arXiv PDF, a brief account of the proof idea, the references consulted, and any code or exact computations used in the argument.

The packet categories exist because early runs made clear that ``solved'' was too coarse a label. We use the following distinctions.

\begin{enumerate}[label=\emph{(\roman*)}, leftmargin=2em]
\item \emph{Full solution.}  A claimed proof of the extracted question as stated in the source paper.  These packets are treated as high priority candidates for human review, not as formally certified proofs.

\item \emph{Counterexample.}  A construction satisfying the hypotheses of a question or conjecture while violating its proposed conclusion.  As with full solutions, these require later checks of both the proof and the novelty status.

\item \emph{Partial result.}  A solved subcase, theorem adjacent to the original problem, meaningful reduction, sharp obstruction, or quantitative improvement that does not settle the full source question.

\item \emph{Conditional result.}  An argument whose remaining dependency is isolated explicitly, for example an unproved lemma, a scalar inequality, or a computational check.

\item \emph{Literature answer.}  A record that the question is already answered elsewhere.  Within this category we distinguish two cases.  An \emph{explicit literature answer} means that another paper explicitly answers the original question.  An \emph{implied literature answer} means that an
existing theorem answers the question only after an identification or reformulation made by the agent.  Both kinds of record are useful for avoiding duplication, but neither is counted as a new result of the pipeline.

\item \emph{Proof gap.} A possible gap identified incidentally in a source paper or in one of our earlier packets. This category is included because agents began pointing out such gaps while attempting to solve other problems. We do not systematically search for proof gaps, and a gap is not counted as a solution unless it subsequently leads to a corrected theorem or a genuine counterexample. These gaps have not been independently verified and may therefore reflect errors made by the model.
\end{enumerate}

These categories respond to concrete failure modes. Sometimes an agent finds a theorem in the literature and initially writes as if it has proved something new.  Sometimes it solves a nearby subcase but not the actual question.  Sometimes it finds a real obstruction, but only under an additional assumption.  Sometimes, while attempting a target, it notices a possible gap in
the source paper or in an earlier packet.  The taxonomy is a way of making these outcomes visible rather than hiding them inside a single success/failure
label.

\subsection{Scale of the \texorpdfstring{\texttt{math.FA}}{math.FA} run}

The Banach-space run grew in stages.  The first small-scale pass began with a ranked metadata pool of 439 papers. From this pool we selected 35 papers for source probing, and the deterministic source scan recorded 166 open-problem signals across 28 of them.  This first pass was mainly a feasibility test of whether the model could find real questions and produce packets that a human would want to read.

The main pipeline run, from which the principal results reported in this paper are drawn, then used a source backed queue for \texttt{math.FA}. Candidate papers were ranked using deterministic metadata and source text signals. Papers were admitted to the queue if source text could be extracted, at least one deterministic open problem signal was detected, and the paper passed a functional analysis relevance filter. This procedure produced a queue of 1,433 papers. All aggregate counts, rates, and quantitative results reported in this paper refer to this 1,433 paper run.

After this run, we expanded the source corpus to include the full modern \texttt{math.FA} collection of 35,297 papers published from 2000 through 2026. Source text was successfully extracted from 34,890 of these papers. The deterministic scan identified 42,483 open problem signals across 15,666 papers. This expanded corpus will be used for live updates on the \website[project website]\footnote{Future live updates may include solutions generated by the newer GPT 5.6 model. None of the results discussed in this paper were generated by that model.}, \textbf{but it is not included in the aggregate counts reported in this paper}.

\subsection{Summary of Results}

This section summarizes the aggregate output of the discovery run. We separate active run records from packets that entered human review, since these represent different stages of the workflow. The tables below report both the total number of records or packets and the corresponding number of unique sources. Unique sources are counted by primary arXiv identifier, so the two columns need not agree: a single source paper may contribute more than one result, and multiple attempts may target the same paper.

\begin{table}[htbp]
\centering
\begin{tabular}{@{}p{0.60\textwidth}rr@{}}
\hline
\textbf{Category} & \textbf{Records} & \textbf{Unique sources} \\
\hline
Full solutions & 126 & 124 \\
Counterexamples & 102 & 101 \\
Partial results & 211 & 185 \\
Conditional results & 14 & 14 \\
Literature already answered & 140 & 137 \\
Literature implied answers & 116 & 114 \\
Proof gaps & 18 & 17 \\
\hline
Attempt records & 925 & 822 \\
All active solution packets & 709 & 653 \\
All registry records & 1127 & 945 \\
\hline
\end{tabular}
\caption{Aggregate counts for records in the active run indexes. Packets moved to human review are reported separately in \Cref{tab:human-review-counts}.}
\label{tab:active-run-counts}
\end{table}

\Cref{tab:active-run-counts} reports the active state of the run. The solution categories count packets that remain in the active solution folders; packets moved to human review are excluded from these counts and reported separately. Differences between the record count and the unique source count arise when a single source paper contributes more than one record. For instance, arXiv:1901.07866 contributes two records; arXiv:2312.14711 contributes two records in the full solution category. Partial results can likewise include multiple records for the same target, since distinct partial attacks may yield different intermediate results that the model judged worth preserving. Attempt records should not be interpreted as disjoint mathematical results. Rather, they are bookkeeping entries for searches, failed attacks, same paper triage, and partial investigations. The registry is broader still: it serves as the long term memory of the run, recording active packets, reviewed packets, failures, and other outcomes used to avoid rediscovering the same target.

\begin{table}[htbp]
\centering
\begin{tabular}{@{}p{0.54\textwidth}rr@{}}
\hline
\textbf{Human review outcome} & \textbf{Packets} & \textbf{Unique sources} \\
\hline
Verified & 31 & 29 \\
Rejected & 10 & 10 \\
\hline
\end{tabular}
\caption{Human reviewed packets grouped by recorded review outcome. These packets have been moved out of the active solution folders and are therefore not counted as active solution packets in \Cref{tab:active-run-counts}.}
\label{tab:human-review-counts}
\end{table}

\Cref{tab:human-review-counts} reports the corresponding outcomes for packets that entered human review. These counts are kept separate from the active run counts because review changes the status of a packet. Once a packet is moved into the human review workflow, it no longer represents an unresolved active solution packet. The table therefore records the assigned review outcome while preserving the same distinction between total packets and unique source papers.

\subsection{Scope, verification, novelty, and attribution}
\label{sec:scope-verification-attribution}

The labels ``full solution'' and ``counterexample'' describe the pipeline's current classification of an attempt. They should not be interpreted as guarantees that the argument is correct, that it addresses the source question exactly as intended, or that the result is novel. An attempt may contain a mathematical error, rely on an incorrect interpretation or extraction of the problem statement, overlook a hypothesis, or establish only a nearby result. As reflected in the verified and rejected outcomes reported in \Cref{tab:human-review-counts}, subsequent model and human review can reveal errors, gaps, or misunderstandings that were not apparent in the initial packet.

We conducted targeted literature searches and recorded known answers where they were found, but these searches were not exhaustive. In particular, the absence of a result from the references inspected by the pipeline should not be taken as evidence that the result is new. A proof or counterexample classified as potentially new may already appear, either explicitly or implicitly, in the existing literature, possibly in different terminology or as a consequence of a more general theorem. Claims of novelty therefore remain provisional until the relevant literature has been examined more thoroughly by researchers familiar with the area.

All reported attempts should be treated as candidates for mathematical inspection rather than as independently certified results. Establishing a result requires more than producing a plausible argument. It requires understanding the proof, checking that it addresses the intended problem, identifying and correcting errors where necessary, determining its relationship to existing work, and explaining its significance within the surrounding mathematical context.

For any result that is ultimately established, mathematical credit belongs to the people who formulated the question and to the mathematicians who understand, verify, correct, develop, and contextualize the argument. These contributions are central to both the mathematical value of the result and its proper attribution. The model and pipeline are best understood as research tools that can support the search for connections, the generation of candidate ideas, and the organization of possible arguments.

\section{Selected results from the automated pipeline}
\label{sec:selected-pipeline-results}

This section presents a small selection of results obtained by the automated pipeline described above. The examples were chosen arbitrarily by the first-named author and are intended only to illustrate the range and character of the mathematics produced by the pipeline; they should not be regarded as exhaustive, representative, or selected according to mathematical significance. For readability, the original arguments have been rewritten and reorganised to improve their exposition. The complete collection of results and their accompanying records is available at the \website. Each example below records its origin, the question addressed, and the resulting answer and proof.

\subsection{A tube construction for uniformly Lipschitz maps on decomposable Banach balls}
\label{sec:selected-decomposable-uniformly-lipschitz}

\noindent\textit{Origin.} Question~2 in the paper of Barroso and Ferreira \cite{BarrosoFerreira2026RetractionMethods} asks whether the closed unit ball of every infinite-dimensional Banach space admits a fixed-point-free uniformly Lipschitz self-map with null minimal displacement.

\begin{question}
Let $X$ be an infinite-dimensional Banach space. Does there exist a map $T\colon B_X\to B_X$ such that
\begin{equation*}
 \operatorname{Fix}(T)=\varnothing,
 \qquad
 d(T,B_X):=\inf_{x\in B_X}\|Tx-x\|=0,
 \qquad
 \sup_{n\geq1}\operatorname{Lip}(T^n)<\infty?
\end{equation*}
\end{question}

\begin{theorem}
\label{thm:selected-decomposable-uniformly-lipschitz}
Suppose that
\begin{equation*}
 X=Y\oplus Z
\end{equation*}
as a topological direct sum, where $Y$ and $Z$ are closed infinite-dimensional subspaces. Then there is a map $T\colon B_X\to B_X$ such that
\begin{equation*}
 \operatorname{Fix}(T)=\varnothing,
 \qquad
 d(T,B_X)=0,
 \qquad
 \sup_{n\geq1}\operatorname{Lip}(T^n)<\infty.
\end{equation*}
\end{theorem}

\begin{proof}
We begin with three auxiliary observations. For a normed space $E$, define the radial retraction $\kappa_E\colon E\to B_E$ by
\begin{equation*}
 \kappa_E(x)=
 \begin{cases}
 x, & \|x\|\leq1,\\[1mm]
 \dfrac{x}{\|x\|}, & \|x\|>1.
 \end{cases}
\end{equation*}
This map is $2$-Lipschitz. Indeed, suppose that $\|x\|\leq\|y\|$. The assertion is immediate if $x,y\in B_E$. If $\|x\|\leq1<\|y\|$, then
\begin{equation*}
 \left\|x-\frac{y}{\|y\|}\right\|
 \leq \|x-y\|+\left\|y-\frac{y}{\|y\|}\right\|
 =\|x-y\|+\|y\|-1
 \leq2\|x-y\|.
\end{equation*}
If $1<\|x\|\leq\|y\|$, then
\begin{equation*}
\begin{aligned}
 \left\|\frac{x}{\|x\|}-\frac{y}{\|y\|}\right\|
 &\leq \frac{\|x-y\|}{\|x\|}
     +\|y\|\left|\frac1{\|x\|}-\frac1{\|y\|}\right|\\
 &=\frac{\|x-y\|}{\|x\|}
     +\frac{\|y\|-\|x\|}{\|x\|}
 \leq2\|x-y\|.
\end{aligned}
\end{equation*}

We next record an estimate for selecting a point on a sphere whose radius varies. Let $R_Z\colon B_Z\to S_Z$ be a Lipschitz retraction, put
\begin{equation*}
 L=\operatorname{Lip}(R_Z),
\end{equation*}
and, for $r>0$ and $q\in Z$, define
\begin{equation*}
 H(r,q)=rR_Z\bigl(\kappa_Z(q/r)\bigr).
\end{equation*}
Then
\begin{equation*}
 \|H(r,q)-H(s,p)\|
 \leq2L\|q-p\|+(1+L)|r-s|
 \qquad(r,s>0,\ p,q\in Z).
\end{equation*}
For fixed $r>0$, the estimate for $\kappa_Z$ gives
\begin{equation*}
 \|H(r,q)-H(r,p)\|\leq2L\|q-p\|.
\end{equation*}
To compare the radii, suppose that $0<r\leq s$. For fixed $q\in Z$,
\begin{equation*}
 \|H(r,q)-H(s,q)\|
 \leq s-r+rL\|\kappa_Z(q/r)-\kappa_Z(q/s)\|.
\end{equation*}
Writing $c=\|q\|$, a direct consideration of the cases $c\leq r$, $r<c<s$, and $s\leq c$ gives
\begin{equation*}
 r\|\kappa_Z(q/r)-\kappa_Z(q/s)\|\leq s-r.
\end{equation*}
The claimed estimate follows.

Finally, every infinite-dimensional Banach space $Y$ admits a Lipschitz function
\begin{equation*}
 \eta\colon S_Y\to(0,1]
\end{equation*}
whose infimum is zero. To see this, use Riesz' lemma to choose a sequence $(u_n)_{n=1}^{\infty}\subset S_Y$ and $\delta>0$ such that
\begin{equation*}
 \|u_n-u_m\|\geq\delta \qquad(n\ne m).
\end{equation*}
Let $(\varepsilon_n)_{n=1}^{\infty}$ be a decreasing sequence of positive numbers converging to zero, and define
\begin{equation*}
 \eta(u)=\min\left\{1,\inf_{n\in\N}
       \bigl(\varepsilon_n+\|u-u_n\|\bigr)\right\}
 \qquad(u\in S_Y).
\end{equation*}
The function $\eta$ is $1$-Lipschitz and satisfies
\begin{equation*}
 \eta(u_n)\leq\varepsilon_n \qquad(n\in\N),
\end{equation*}
so its infimum is zero. It is nevertheless strictly positive at every point. Indeed, at most one member of the separated sequence can lie strictly within distance $\delta/2$ of a given $u\in S_Y$; all the remaining terms in the infimum are at least $\delta/2$, while the possible exceptional term is also positive.

We now construct the required map. Let $P\colon X\to Y$ and $Q\colon X\to Z$ be the bounded projections associated with the decomposition, and choose $a,b>0$ such that
\begin{equation*}
 a+b\leq1.
\end{equation*}
By the theorem of Benyamini and Sternfeld \cite{BenyaminiSternfeld1983}, there are Lipschitz retractions
\begin{equation*}
 R_Y\colon B_Y\to S_Y,
 \qquad
 R_Z\colon B_Z\to S_Z.
\end{equation*}
Define $A\colon Y\to aS_Y$ by
\begin{equation*}
 A(y)=aR_Y\bigl(\kappa_Y(y/a)\bigr).
\end{equation*}
Then $A$ is Lipschitz and fixes every point of $aS_Y$. Using the function $\eta$ constructed above, define
\begin{equation*}
 r(y)=b\eta(y/a)
 \qquad(y\in aS_Y).
\end{equation*}
Thus $r$ is Lipschitz,
\begin{equation*}
 0<r(y)\leq b \qquad(y\in aS_Y),
\end{equation*}
and
\begin{equation*}
 \inf_{y\in aS_Y}r(y)=0.
\end{equation*}

For $x\in B_X$, set
\begin{equation*}
 y(x)=A(Px),
 \qquad
 \rho(x)=r(y(x)),
\end{equation*}
and
\begin{equation*}
 \Phi(x)
 =H(\rho(x),Qx)
 =\rho(x)R_Z\bigl(\kappa_Z(Qx/\rho(x))\bigr).
\end{equation*}
Finally, define
\begin{equation*}
 T(x)=y(x)-\Phi(x).
\end{equation*}
The preceding estimates show that $T$ is Lipschitz. Moreover,
\begin{equation*}
 \|y(x)\|=a,
 \qquad
 \|\Phi(x)\|=\rho(x)\leq b,
\end{equation*}
and hence
\begin{equation*}
 \|T(x)\|\leq a+b\leq1.
\end{equation*}
Thus $T$ maps $B_X$ into itself.

Define also
\begin{equation*}
 S(x)=y(x)+\Phi(x).
\end{equation*}
We claim that
\begin{equation*}
 T^2=S,
 \qquad
 T^3=T.
\end{equation*}
Fix $x\in B_X$, and write
\begin{equation*}
 y=y(x),
 \qquad
 \rho=\rho(x),
 \qquad
 \Phi(x)=\rho z
\end{equation*}
for some $z\in S_Z$. Since $P(y-\rho z)=y$ and $A(y)=y$, the base point and the radius are unchanged after applying $T$. Furthermore,
\begin{equation*}
 \frac{Q(y-\rho z)}{\rho}=-z\in S_Z.
\end{equation*}
Both $\kappa_Z$ and $R_Z$ fix points of $S_Z$, and therefore
\begin{equation*}
 T(Tx)=y+\rho z=S(x).
\end{equation*}
The same argument applied to $S(x)=y+\rho z$ gives
\begin{equation*}
 T(Sx)=y-\rho z=T(x).
\end{equation*}
Thus the iterates of $T$ alternate between the two Lipschitz maps $T$ and $S$, and consequently
\begin{equation*}
 \sup_{n\geq1}\operatorname{Lip}(T^n)<\infty.
\end{equation*}

The map $T$ has no fixed point. Indeed, if $T(x)=x$, then $T^2(x)=T(x)$. On the other hand, the formulas above give
\begin{equation*}
 T(x)=y-\rho z,
 \qquad
 T^2(x)=y+\rho z.
\end{equation*}
This would imply $2\rho z=0$, which is impossible because $\rho>0$ and $z\in S_Z$.

It remains to prove that the minimal displacement is zero. The following scaling-and-radial-retraction argument follows \cite[Proposition~4.12]{BarrosoFerreira2026RetractionMethods}. Fix $z_0\in S_Z$ and use the sequence $(u_n)_{n=1}^{\infty}$ occurring in the construction of $\eta$. Put
\begin{equation*}
 x_n=au_n+b\eta(u_n)z_0
 \qquad(n\in\N).
\end{equation*}
Since $\|x_n\|\leq a+b\leq1$, the sequence lies in $B_X$. The defining formula for $T$ gives
\begin{equation*}
 T(x_n)=au_n-b\eta(u_n)z_0,
\end{equation*}
and therefore
\begin{equation*}
 \|T(x_n)-x_n\|
 =2b\eta(u_n)
 \leq2b\varepsilon_n
 \longrightarrow0.
\end{equation*}
Hence $d(T,B_X)=0$.
\end{proof}

\subsection{Uniform H\"older nonexpansive maps on decomposable Banach balls}
\label{sec:selected-holder-maps}

\noindent\textit{Origin.} The final open-question list in Barroso's paper \cite{BarrosoHolderContractive} asks whether the closed unit ball $B_X$ fails the fixed point property for uniformly $\alpha$-H\"older nonexpansive maps with null minimal displacement in several important classes of Banach spaces.

\begin{question}
Let $X$ be an infinite-dimensional Banach space and let $\alpha\in(0,1)$. Does there exist a map $T\colon B_X\to B_X$ such that
\begin{equation*}
 \operatorname{Fix}(T)=\varnothing,
 \qquad
 d(T,B_X):=\inf_{x\in B_X}\|Tx-x\|=0,
\end{equation*}
and
\begin{equation*}
 \|T^n x-T^n y\|\leq \|x-y\|^\alpha
 \qquad(x,y\in B_X,\ n\geq1)?
\end{equation*}
In particular, what happens when $X$ is isomorphic to a Hilbert space or when $X$ is reflexive and has an unconditional basis?
\end{question}

\begin{proof}[Answer and proof]
The answer is affirmative whenever
\begin{equation*}
 X=Y\oplus Z,
\end{equation*}
where $Y$ and $Z$ are closed infinite-dimensional subspaces; equivalently, whenever $X$ is decomposable. This includes spaces isomorphic to an infinite-dimensional Hilbert space and spaces with an unconditional basis.

By \Cref{thm:selected-decomposable-uniformly-lipschitz}, there is a map $U\colon B_X\to B_X$ such that
\begin{equation*}
 \operatorname{Fix}(U)=\varnothing,
 \qquad d(U,B_X)=0,
 \qquad M:=\sup_{n\geq1}\operatorname{Lip}(U^n)<\infty.
\end{equation*}
It remains to convert these uniformly Lipschitzian dynamics into uniformly $\alpha$-H\"older nonexpansive dynamics.

Put $C=2M$ and
\begin{equation*}
 \eta=\min\{2,C^{-1/(1-\alpha)}\}.
\end{equation*}
Choose $r>0$ sufficiently small that $2r\leq\eta^\alpha$. In particular, $r<1$. Let $R_r\colon B_X\to rB_X$ be the radial retraction
\begin{equation*}
 R_r(x)=
 \begin{cases}
 x, & \|x\|\leq r,\\[1mm]
 \dfrac{rx}{\|x\|}, & \|x\|>r.
 \end{cases}
\end{equation*}
This retraction is $2$-Lipschitz on every normed space. Define $S\colon rB_X\to rB_X$ by
\begin{equation*}
 S(z)=rU(z/r).
\end{equation*}
Then
\begin{equation*}
 S^n(z)=rU^n(z/r),
 \qquad \operatorname{Lip}(S^n)\leq M
 \quad(n\geq1),
\end{equation*}
and $S$ is fixed-point free with null minimal displacement on $rB_X$. Set
\begin{equation*}
 V=SR_r\colon B_X\longrightarrow B_X.
\end{equation*}
Because $S(rB_X)\subseteq rB_X$ and $R_r$ is the identity on $rB_X$,
\begin{equation*}
 V^n=S^nR_r \qquad(n\geq1).
\end{equation*}

For $x,y\in B_X$, write $d=\|x-y\|$. Since $V^nx,V^ny\in rB_X$,
\begin{equation*}
 \|V^nx-V^ny\| \leq \min\{M\|R_rx-R_ry\|,2r\}\leq \min\{Cd,2r\}.
\end{equation*}
If $d\leq\eta$, then the definition of $\eta$ gives $Cd\leq d^\alpha$. If $d\geq\eta$, then $2r\leq\eta^\alpha\leq d^\alpha$. Consequently,
\begin{equation*}
 \|V^nx-V^ny\|\leq d^\alpha
 \qquad(x,y\in B_X,\ n\geq1).
\end{equation*}

If $Vx=x$, then $x\in rB_X$, so $R_rx=x$ and $Sx=x$, contradicting $\operatorname{Fix}(S)=\varnothing$. Hence $V$ has no fixed point. Finally, because $d(S,rB_X)=0$, there are $z_j\in rB_X$ such that $\|Sz_j-z_j\|\to0$. For these points $R_rz_j=z_j$, and therefore
\begin{equation*}
 \|Vz_j-z_j\|=\|Sz_j-z_j\|\longrightarrow0.
\end{equation*}
Thus $d(V,B_X)=0$.

For the stated consequences, split a Hilbert space into two closed orthogonal summands of infinite dimension and transport the decomposition through an isomorphism. If $X$ has an unconditional basis, its odd and even coordinate subspaces give the required decomposition. This argument does not settle the arbitrary reflexive case.
\end{proof}

\subsection{Nearly isometric embeddability need not imply almost Lipschitz embeddability}
\label{sec:selected-nearly-isometric}

\noindent\textit{Origin.} This is Problem~4 in Section~5 of Baudier and Lancien \cite{BaudierLancienTightEmbeddability}.

\begin{question}
Exhibit metric spaces $X$ and $Y$ such that $X$ nearly isometrically embeds into $Y$, but $X$ does not almost Lipschitz embed into $Y$.

More precisely, $X$ \emph{almost Lipschitz embeds} into $Y$ if there are constants $r>0$ and $D\geq1$ such that, for every continuous function $\varphi\colon[0,\infty)\to[0,1)$ satisfying
\begin{equation*}
 \varphi(0)=0,
 \qquad \varphi(t)>0\quad(t>0),
\end{equation*}
there is a map $f_\varphi\colon X\to Y$ for which
\begin{equation*}
 r d_X(x,y)\varphi(d_X(x,y))
 \leq d_Y(f_\varphi(x),f_\varphi(y))
 \leq Drd_X(x,y).
\end{equation*}

For nearly isometric embeddability, let $\mathcal P$ consist of the continuous functions $\rho$ satisfying
\begin{equation*}
 \rho(t)=t\quad(0\leq t\leq1),
 \qquad \rho(t)\leq t\quad(t\geq1),
 \qquad \frac{\rho(t)}t\longrightarrow0\quad(t\to\infty),
\end{equation*}
and let $\Omega$ consist of the functions $\omega$ satisfying
\begin{equation*}
 \begin{aligned}
 \omega(0)&=0,
 & t&\leq\omega(t) &&(0\leq t\leq1),\\
 \omega(t)&=t &&(t\geq1),
 & \frac{\omega(t)}t&\longrightarrow\infty &&(t\downarrow0).
 \end{aligned}
\end{equation*}
The space $X$ \emph{nearly isometrically embeds} into $Y$ if, for every $(\rho,\omega)\in\mathcal P\times\Omega$, there is a map $f\colon X\to Y$ such that
\begin{equation*}
 \rho(d_X(x,y))\leq d_Y(f(x),f(y))\leq\omega(d_X(x,y)).
\end{equation*}
\end{question}

\begin{proof}[Answer and proof]
We first record a concave-majorant construction. Given $\rho\in\mathcal P$, there is a continuous, nondecreasing, concave, unbounded function $g\colon[0,\infty)\to[0,\infty)$ satisfying
\begin{equation*}
 g(0)=0,
 \qquad g(t)=t\quad(0\leq t\leq1),
 \qquad \rho(t)\leq g(t)\leq t,
 \qquad \frac{g(t)}t\longrightarrow0.
\end{equation*}
To see this, define
\begin{equation*}
 h(t)=
 \begin{cases}
 t, & 0\leq t\leq1,\\
 1+\log t, & t\geq1,
 \end{cases}
 \qquad
 \rho_0(t)=\max\{\rho(t),h(t)\}.
\end{equation*}
Choose a sequence $\varepsilon_j\downarrow0$. Since $\rho_0(t)/t\to0$, each number
\begin{equation*}
 A_j=\sup_{t\geq0}\bigl(\rho_0(t)-\varepsilon_jt\bigr)
\end{equation*}
is finite. Set
\begin{equation*}
 g(t)=\inf\bigl(\{t\}\cup
       \{A_j+\varepsilon_jt:j\geq1\}\bigr).
\end{equation*}
Every affine function in this infimum dominates $\rho_0$. Thus $\rho\leq\rho_0\leq g\leq t$. The function $g$ is concave and nondecreasing. For each fixed $j$,
\begin{equation*}
 \limsup_{t\to\infty}\frac{g(t)}t
 \leq \limsup_{t\to\infty}\frac{A_j+\varepsilon_jt}{t}
 =\varepsilon_j,
\end{equation*}
so $g(t)/t\to0$. Finally, $g\geq h$, and hence $g$ is unbounded.

A nondecreasing concave function with $g(0)=0$ is subadditive. Consequently,
\begin{equation*}
 d_g(m,n)=g(|m-n|)
\end{equation*}
defines a metric on $\mathbb N_0$.

Take $X=\mathbb N_0$ with its usual metric. For every $\rho\in\mathcal P$, choose a function $g_\rho$ as above, and let $Y_\rho$ be a copy of $\mathbb N_0$ with metric
\begin{equation*}
 d_\rho(m,n)=g_\rho(|m-n|).
\end{equation*}
Let $Y$ be the metric wedge of the family $(Y_\rho)_{\rho\in\mathcal P}$, obtained by identifying all zero points to a common basepoint $o$. If $a\in Y_\rho$ and $b\in Y_\sigma$, define
\begin{equation*}
 d_Y(a,b)=
 \begin{cases}
 g_\rho(|a-b|), & \rho=\sigma,\\[1mm]
 g_\rho(a)+g_\sigma(b), & \rho\neq\sigma.
 \end{cases}
\end{equation*}

We first show that $X$ nearly isometrically embeds into $Y$. Fix $\rho\in\mathcal P$ and $\omega\in\Omega$, and send $n\in\mathbb N_0$ to the point $n$ in the $\rho$-component. For $k=|m-n|\geq1$,
\begin{equation*}
 \rho(k) \leq g_\rho(k) =d_Y(F_{\rho,\omega}(m),F_{\rho,\omega}(n)) \leq k =\omega(k).
\end{equation*}
The case $k=0$ is immediate.

We next prove that $Y$ contains no bi-Lipschitz copy of the usual integer ray. Suppose that $F\colon\mathbb N_0\to Y$ and constants $c,C>0$ satisfy
\begin{equation*}
 c|m-n|\leq d_Y(F(m),F(n))\leq C|m-n|
 \qquad(m,n\in\mathbb N_0).
\end{equation*}
Put $R_n=d_Y(F(n),o)$. Then
\begin{equation*}
 R_n\geq d_Y(F(n),F(0))-R_0\geq cn-R_0,
\end{equation*}
so $R_n$ grows at least linearly. On the other hand, $d_Y(F(n+1),F(n))\leq C$. If $F(n)$ and $F(n+1)$ belong to different wedge components, their distance is $R_n+R_{n+1}$, which is eventually larger than $C$. Thus, from some index onward, the whole sequence lies in a single component $Y_{\rho_0}$.

Write $F(n)=a_n\in Y_{\rho_0}$ for $n\geq N$. Since
\begin{equation*}
 g_{\rho_0}(|a_{n+1}-a_n|)=d_Y(F(n+1),F(n))\leq C
\end{equation*}
and $g_{\rho_0}$ is nondecreasing and unbounded, there is $B<\infty$ such that
\begin{equation*}
 |a_{n+1}-a_n|\leq B \qquad(n\geq N).
\end{equation*}
Hence
\begin{equation*}
 |a_m-a_n|\leq B|m-n| \qquad(m,n\geq N),
\end{equation*}
and therefore
\begin{equation*}
 d_Y(F(m),F(n))\leq g_{\rho_0}(B|m-n|).
\end{equation*}
Because $g_{\rho_0}(t)/t\to0$, the right-hand side is $o(|m-n|)$ as $|m-n|\to\infty$, contradicting the lower estimate $d_Y(F(m),F(n))\geq c|m-n|$. Thus no bi-Lipschitz embedding $\mathbb N_0\to Y$ exists.

Finally, suppose that $X$ almost Lipschitz embeds into $Y$. Choose
\begin{equation*}
 \varphi(t)=
 \begin{cases}
 t/2, & 0\leq t\leq1,\\
 1/2, & t\geq1.
 \end{cases}
\end{equation*}
Every nonzero distance in $\mathbb N_0$ is at least $1$, so the corresponding map would satisfy
\begin{equation*}
    \frac r2|m-n| \leq d_Y(f_\varphi(m),f_\varphi(n)) \leq Dr|m-n| \qquad(m\neq n).
\end{equation*}
This would be a bi-Lipschitz embedding of the integer ray into $Y$, a contradiction. Hence $X$ does not almost Lipschitz embed into $Y$.
\end{proof}

\subsection{A \texorpdfstring{$c_0$}{c0}-vector outside the canonical closed span in a variable-exponent space}
\label{sec:selected-variable-exponent}

\noindent\textit{Origin.} This answers negatively the open question repeated on page~8 of Talponen's paper \cite{Talponen2009SequentialSpaces}.

\begin{question}
Let $p\colon\mathbb N\to[1,\infty]$. For $x=(x_n)\in\ell^\infty$, define
\begin{equation*}
    \lvert\!\lvert\!\lvert  x\rvert\!\rvert\!\rvert_{(1)}=|x_1|\boxplus_{p(1)}|x_2|, \qquad \lvert\!\lvert\!\lvert x\rvert\!\rvert\!\rvert_{(k)}    =\lvert\!\lvert\!\lvert x\rvert\!\rvert\!\rvert_{(k-1)}\boxplus_{p(k)}|x_{k+1}|,
\end{equation*}
where $a\boxplus_r b=(a^r+b^r)^{1/r}$ when $r<\infty$, and put
\begin{equation*}
 \Phi(x)=\lim_{k\to\infty} \lvert\!\lvert\!\lvert x\rvert\!\rvert\!\rvert_{(k)}, \qquad \ell^{p(\cdot)}=\{x\in\ell^\infty:\Phi(x)<\infty\}.
\end{equation*}
Writing $[(e_n)]$ for the closed linear span in $\ell^{p(\cdot)}$ of the canonical unit vectors, must one always have
\begin{equation*}
 \ell^{p(\cdot)}\cap c_0=[(e_n)]?
\end{equation*}
\end{question}

\begin{proof}[Answer and proof]
No. We construct a function $p\colon\mathbb N\to[2,\infty)$ taking only finite values and an element
\begin{equation*}
 x\in\ell^{p(\cdot)}\cap c_0\setminus[(e_n)].
\end{equation*}
Let $P_n$ denote truncation after the $n$th coordinate and let $Q_n=I-P_n$. We first note that, for every $y\in\ell^{p(\cdot)}$,
\begin{equation*}
    y\in[(e_n)] \quad\Longleftrightarrow\quad \|Q_ny\|_{\ell^{p(\cdot)}}\longrightarrow0.
\end{equation*}
Indeed, the reverse implication follows because $P_ny$ is finitely supported and $P_ny\to y$. Conversely, if $y\in[(e_n)]$ and $\varepsilon>0$, choose a finitely supported $z$, supported in $\{1,\ldots,N\}$, such that $\|y-z\|_{\ell^{p(\cdot)}}<\varepsilon$. The recursive norm is coordinatewise monotone in the absolute values, so every coordinate projection is contractive. Hence, for $n\geq N$,
\begin{equation*}
 \|Q_ny\|_{\ell^{p(\cdot)}}
 =\|Q_n(y-z)\|_{\ell^{p(\cdot)}}
 \leq\|y-z\|_{\ell^{p(\cdot)}}<\varepsilon.
\end{equation*}

Fix $0<c<1$, and for $j\geq1$ put
\begin{equation*}
 r_j=j+1,
 \qquad m_j=j^{r_j}.
\end{equation*}
Set $N_1=2$ and $N_{j+1}=N_j+m_j$, and let
\begin{equation*}
 B_j=\{N_j,N_j+1,\ldots,N_{j+1}-1\}.
\end{equation*}
Define $x\in\ell^\infty$ by
\begin{equation*}
 x_1=1,
 \qquad x_n=\frac cj\quad(n\in B_j).
\end{equation*}
Since $c/j\to0$, we have $x\in c_0$. Define $p\colon\mathbb N\to[2,\infty)$ by
\begin{equation*}
    p(k)=r_j \qquad\text{whenever}\qquad N_j-1\leq k\leq N_{j+1}-2.
\end{equation*}
These intervals partition $\mathbb N$, so $p$ is well defined.

Let
\begin{equation*}
 A_j=\bigl\|P_{N_{j+1}-1}x\bigr\|_{\ell^{p(\cdot)}}
\end{equation*}
be the norm of the initial segment consisting of $x_1$ and the first $j$ complete blocks. Then $A_0=1$, and all exponents used while adjoining block $j$ are equal to $r_j$. Therefore
\begin{equation*}
    A_j =\left(A_{j-1}^{r_j}+m_j\left(\frac cj\right)^{r_j}\right)^{1/r_j} =\left(A_{j-1}^{r_j}+c^{r_j}\right)^{1/r_j}.
\end{equation*}
Since $A_{j-1}\geq1$,
\begin{equation*}
    A_j=A_{j-1}\left(1+\left(\frac c{A_{j-1}}\right)^{r_j}\right)^{1/r_j} \leq A_{j-1}(1+c^{r_j})^{1/r_j},
\end{equation*}
and consequently
\begin{equation*}
    A_j\leq\prod_{i=1}^j(1+c^{r_i})^{1/r_i}.
\end{equation*}
This infinite product converges because
\begin{equation*}
    \sum_{i=1}^\infty\frac{\log(1+c^{r_i})}{r_i} \leq\sum_{i=1}^\infty\frac{c^{r_i}}{r_i}<\infty.
\end{equation*}
Every finite prefix ending inside a block has norm at most the norm after the complete block. Thus $\sup_n\|P_nx\|_{\ell^{p(\cdot)}}<\infty$, and hence $x\in\ell^{p(\cdot)}$.

It remains to show that $x\notin[(e_n)]$. Consider the tail $Q_{N_j-1}x$. Its projection onto $B_j$ has $m_j$ coordinates equal to $c/j$, with all relevant exponents equal to $r_j$, and therefore has norm
\begin{equation*}
 \left(m_j\left(\frac cj\right)^{r_j}\right)^{1/r_j}=c.
\end{equation*}
By coordinatewise monotonicity,
\begin{equation*}
    \|Q_{N_j-1}x\|_{\ell^{p(\cdot)}}\geq c \qquad(j\geq1).
\end{equation*}
The tails therefore do not converge to zero in norm, so the preceding equivalence yields $x\notin[(e_n)]$.
\end{proof}

\subsection{A radial-twist counterexample for a coarse sum}
\label{sec:selected-radial-twist}

\noindent\textit{Origin.} This is Problem~5.7 in Braga's paper \cite{Braga2017Asymptotic}. It appears on page~20 of the arXiv version.

\begin{question}
Let $X,Y_1,Y_2$ be Banach spaces, and let
\begin{equation*}
 f=(f_1,f_2)\colon X\longrightarrow Y_1\oplus Y_2
\end{equation*}
be a coarse Lipschitz embedding. Must there exist an infinite-dimensional subspace $X_0\subseteq X$ such that either $f_1|_{X_0}\colon X_0\to Y_1$ or $f_2|_{X_0}\colon X_0\to Y_2$ is a coarse Lipschitz embedding?
\end{question}

\begin{proof}[Answer and proof]
No. We construct a bi-Lipschitz embedding
\begin{equation*}
 F=(F_1,F_2)\colon\ell_2\longrightarrow\ell_2\oplus_2\ell_2
\end{equation*}
such that neither coordinate restricts to a coarse embedding on any nonzero linear subspace of $\ell_2$.

Let $H=\ell_2$ over the real scalars. For $t\geq0$, define
\begin{equation*}
 d_4(t)=\operatorname{dist}(t,4\mathbb Z),
 \qquad
 \varphi(t)=\frac\pi4\min\{d_4(t),2\},
\end{equation*}
and, for $r\geq0$, set
\begin{equation*}
 \theta(r)=\varphi(\log(1+r)).
\end{equation*}
The function $\varphi$ is $\pi/4$-Lipschitz and takes values in $[0,\pi/2]$. Moreover, for every $k\in\mathbb N_0$,
\begin{equation*}
 \theta(e^{4k}-1)=0, \qquad \theta(e^{4k+2}-1)=\frac\pi2.
\end{equation*}
Define
\begin{equation*}
 F(x)=\bigl(\cos(\theta(\|x\|))x,
             \sin(\theta(\|x\|))x\bigr)
 \in H\oplus_2H,
\end{equation*}
and write $F=(F_1,F_2)$.

We first prove that $F$ is bi-Lipschitz. For $t\in[0,\pi/2]$, let
\begin{equation*}
 J_t\colon H\longrightarrow H\oplus_2H,
 \qquad J_tz=(\cos(t)z,\sin(t)z).
\end{equation*}
Each $J_t$ is an isometry and $F(x)=J_{\theta(\|x\|)}x$. Let $r=\|x\|$ and $s=\|y\|$. If $x,y\neq0$, put
\begin{equation*}
 u=\frac xr,
 \qquad v=\frac ys,
 \qquad a=\langle u,v\rangle,
 \qquad c=\cos(\theta(r)-\theta(s)).
\end{equation*}
Since $\theta$ takes values in $[0,\pi/2]$, we have $c\in[0,1]$, and a direct calculation gives
\begin{equation*}
 \|F(x)-F(y)\|^2=r^2+s^2-2rsca,
 \qquad
 \|x-y\|^2=r^2+s^2-2rsa.
\end{equation*}
If $a\geq0$, then $ca\leq a$, so $\|F(x)-F(y)\|\geq\|x-y\|$. If $a<0$, then
\begin{equation*}
 \|F(x)-F(y)\|^2\geq r^2+s^2
 \geq\frac12\|x-y\|^2.
\end{equation*}
The same estimate is immediate if one of $x,y$ is zero. Hence
\begin{equation*}
 \|F(x)-F(y)\|\geq\frac1{\sqrt2}\|x-y\|.
\end{equation*}

For the upper estimate, assume without loss of generality that $r\geq s$. Since $\|J_t-J_{t'}\|\leq|t-t'|$, we obtain
\begin{equation*}
\begin{aligned}
 \|F(x)-F(y)\|
 &\leq\|J_{\theta(r)}(x-y)\|
   +\|(J_{\theta(r)}-J_{\theta(s)})y\|\\
 &\leq\|x-y\|+s|\theta(r)-\theta(s)|.
\end{aligned}
\end{equation*}
The Lipschitz estimate for $\varphi$ yields
\begin{equation*}
 |\theta(r)-\theta(s)|
 \leq\frac\pi4\log\frac{1+r}{1+s}.
\end{equation*}
Using $\log(1+u)\leq u$ with $u=(r-s)/(1+s)$, we get
\begin{equation*}
 s\log\frac{1+r}{1+s}
 \leq\frac{s}{1+s}(r-s)
 \leq r-s
 \leq\|x-y\|.
\end{equation*}
Therefore
\begin{equation*}
 \frac1{\sqrt2}\|x-y\|
 \leq\|F(x)-F(y)\|
 \leq\left(1+\frac\pi4\right)\|x-y\|,
\end{equation*}
so $F$ is bi-Lipschitz and hence a coarse Lipschitz embedding.

It remains to show that neither coordinate works on any nonzero subspace. Let $X_0\subseteq H$ be a nonzero linear subspace, and choose $u\in X_0$ with $\|u\|=1$. For
\begin{equation*}
 R_k=e^{4k+2}-1,
 \qquad x_k=R_ku,
\end{equation*}
we have $\theta(R_k)=\pi/2$, and hence
\begin{equation*}
 F_1(x_k)=F_1(-x_k)=0,
 \qquad \|x_k-(-x_k)\|=2R_k\longrightarrow\infty.
\end{equation*}
Thus $F_1|_{X_0}$ is not a coarse embedding. Similarly, with
\begin{equation*}
 S_k=e^{4k}-1,
 \qquad y_k=S_ku,
\end{equation*}
we have $\theta(S_k)=0$, and therefore
\begin{equation*}
 F_2(y_k)=F_2(-y_k)=0,
 \qquad \|y_k-(-y_k)\|=2S_k\longrightarrow\infty.
\end{equation*}
Thus $F_2|_{X_0}$ is not a coarse embedding either. This holds for every nonzero $X_0$, and gives a negative answer to the question.
\end{proof}

\newpage

\newpage

\noindent\textbf{Acknowledgements.}
This paper forms part of the first-named author's PhD research at Lancaster University, conducted under the supervision of Professor N. J. Laustsen. He gratefully acknowledges the support of the Engineering and Physical Sciences Research Council (EPSRC), grant number EP/W524438/1, which has supported his studies.

The authors would like to thank Kevin Beanland and Tomasz Kania for proposing problems and providing comments for this project, which helped to shape some of the mathematical directions pursued here. They are also grateful to Ond\v{r}ej Kalenda, Tommaso Russo, Bruno de Mendon\c{c}a Braga, Anna Pelczar-Barwacz, Jes\'us Mar\'ia Fern\'andez Castillo, and Samya Kumar Ray for their kind and careful responses to our original questions concerning the automatic pipeline, and for taking the time to look at the problems and material we sent them.

The first named author would like to express his special gratitude to Kevin Beanland for his enthusiastic response to the project, as well as for his encouragement and support. He is also especially grateful to Tomasz Kania for many long conversations about mathematics and artificial intelligence, and for his generosity in answering questions, even at inconvenient times. He would also like to thank Enrique Rozas García and Felix Schwarzfischer for their comments on a preliminary version of this manuscript and for their thoughtful and stimulating conversations about artificial intelligence and its role in mathematical research.

Finally, he would like to express his heartfelt thanks to his supervisor, Niels Laustsen, for giving him the freedom to explore, including when that exploration led in unconventional directions. \medskip

\noindent\textbf{AI usage statement.} Large language models were central to the exploratory and drafting stages of this project, and we refer to the main text for a full description of AI usage. The proofs printed here have been edited, checked against the cited literature, and integrated by the authors, who take responsibility for the mathematical content. \medskip

For the purpose of open access, the author has applied a Creative Commons Attribution (CC BY) licence to any Author Accepted Manuscript version arising. \medskip

\noindent\textbf{Data availability.} The source materials used in this study consist of publicly available arXiv papers and their associated source files. For papers for which the pipeline produced a  packet, the relevant paper information, packet, and model output are available at the \website[project website]. No proprietary, personal, or restricted data were used.

\noindent\textbf{Conflict of interest.} The authors declare that they have no conflict of interest.

\clearpage
\markboth{References}{References}
\providecommand{\bibdoi}[1]{%
  \href{https://doi.org/#1}{\nolinkurl{doi:#1}}}
\providecommand{\bibarxiv}[1]{%
  \href{https://arxiv.org/abs/#1}{\nolinkurl{arXiv:#1}}}

\end{document}